\documentclass[11pt,reqno]{amsart}
\PassOptionsToPackage{expansion=false}{microtype}
\usepackage[T1]{fontenc}
\usepackage{lmodern,amsmath,amssymb,mathtools,mathrsfs,booktabs,array,longtable,microtype}
\usepackage{etoolbox}
\makeatletter
\patchcmd{\LT@output}{\vss}
  {\vskip\z@\@plus1fil\@minus\normalbaselineskip}{}%
  {\PackageError{longtable-compat}{First output patch failed}{}}
\patchcmd{\LT@output}{\vss}
  {\vskip\z@\@plus1fil\@minus\normalbaselineskip}{}%
  {\PackageError{longtable-compat}{Second output patch failed}{}}
\makeatother
\usepackage[margin=1.05in]{geometry}
\usepackage[colorlinks=true,linkcolor=blue,citecolor=red,urlcolor=blue]{hyperref}
\usepackage[nameinlink,noabbrev]{cleveref}
\DeclareRobustCommand{\ChernFirst}{\eqref{eq:transport}}
\DeclareRobustCommand{\ChernCyclic}{\eqref{eq:cyclic}}
\DeclareRobustCommand{\ChernPure}{\eqref{eq:pure}}
\DeclareRobustCommand{\ChernBarred}{\eqref{eq:barred}}
\AtBeginDocument{\let\equationalias\label}
\usepackage{needspace}
\makeatletter
\def\subsubsection{\@startsection{subsubsection}{3}%
  \z@{.5\linespacing\@plus.7\linespacing}{-.5em}%
  {\normalfont\bfseries}}
\def\paragraph{\@startsection{paragraph}{4}%
  \z@\z@{-\fontdimen2\font}%
  {\normalfont\bfseries}}
\makeatother
\numberwithin{equation}{section}
\newtheorem{maintheorem}{Theorem}
\newtheorem{maincorollary}{Corollary}

\newtheorem{theorem}{Theorem}[section]
\newtheorem{lemma}[theorem]{Lemma}
\newtheorem{proposition}[theorem]{Proposition}
\newtheorem{corollary}[theorem]{Corollary}

\crefname{maintheorem}{Theorem}{Theorems}
\newcommand{\ampDisc}{\mathbb D}
\newcommand{\ddc}{i\partial\bar\partial}
\newcommand{\ddcn}{dd^c}
\newcommand{\dcn}{d^c}
\DeclareMathOperator{\Pole}{Pole}
\DeclareMathOperator{\Supp}{Supp}

\newcommand{\C}{\mathbb C}
\newcommand{\R}{\mathbb R}

\newcommand{\CP}{\mathbb P}

\newcommand{\Ric}{\operatorname{Ric}}
\newcommand{\tr}{\operatorname{tr}}

\newcommand{\Id}{\operatorname{Id}}
\newcommand{\eps}{\varepsilon}

\newcommand{\Rea}{\operatorname{Re}}

\newcommand{\Sym}{\operatorname{Sym}}

\title[Chern Holomorphic Sectional Curvature]
{Chern Holomorphic Sectional Curvature:\\
Canonical Ampleness, Rigidity, and Flatness Counterexamples}
\author[X. Qin]{Xiangsen Qin}
\address{Xiangsen Qin: Chern Institute of Mathematics and LPMC, Nankai University \\
Tianjin 300071, China}
\email{qinxiangsen@nankai.edu.cn}

\author[Y. Tian]{Yuanhong Tian}
\address{Yuanhong Tian: School of Mathematics and Statistics, Beijing Institute of Technology\\
Beijing, 100081, China}
\email{yuanhongtian@bit.edu.cn}
\hypersetup{pdftitle={Chern Holomorphic Sectional Curvature: Canonical Ampleness, Rigidity, and Flatness Counterexamples},pdfauthor={Xiangsen Qin and Yuanhong Tian}}

\newtheoremstyle{bolddefinition}
  {\topsep}{\topsep}{\normalfont}{}{\bfseries}{.}{.5em}{}
\theoremstyle{bolddefinition}
\newtheorem{definition}[theorem]{Definition}

\newcommand{\rank}{\operatorname{rank}}
\newcommand{\adj}{\operatorname{adj}}
\newcommand{\Span}{\operatorname{span}}
\newcommand{\im}{\operatorname{im}}
\newcommand{\End}{\operatorname{End}}

\usepackage{enumitem}

\allowdisplaybreaks[2]
\usepackage{fvextra,adjustbox,pdflscape}

\providecommand{\PP}{\mathbb P}

\providecommand{\Ann}{\operatorname{Ann}}

\providecommand{\ad}{\operatorname{ad}}

\providecommand{\rad}{\operatorname{rad}}

\newcounter{canproofstep}
\DeclareRobustCommand{\HscGenericEval}{\hyperref[sep:eval]{G1}}
\DeclareRobustCommand{\HscWeakPair}{\hyperref[wp:definition]{WP}}
\begin{document}
\raggedbottom

\begin{abstract}
On a compact Brody-hyperbolic complex manifold, the existence of a strictly positive plurinegative $(1,1)$-current with smooth exterior derivative implies that the canonical bundle is nef, and ample if every positive Lelong upper level set of the current has complex dimension at most two; in dimensions one and two, hyperbolicity alone suffices for ampleness. Next, we study Hermitian metrics of constant Chern holomorphic sectional curvature $c$. Locally and without compactness assumptions, any such metric with $c\ne0$ is K\"ahler in dimension three, on balanced fourfolds, and under rank bounds on the iterated Chern curvature hull of the $(1,0)$ torsion derivative. Globally, on compact manifolds in Fujiki's class, the metric is K\"ahler--Einstein when $c<0$ and K\"ahler flat when $c=0$. Without the Fujiki hypothesis, $c=0$ need not imply Chern flatness: for every dimension $n\ge7$, we construct compact balanced nonflat examples with all four Chern--Ricci contractions vanishing. Nevertheless, on a compact Hermitian manifold with $c=0$ and vanishing first Chern--Ricci form, K\"ahler flatness follows if there exists a smooth semipositive plurinegative $(1,1)$-form positive definite at one point.
\end{abstract}
\maketitle
\setcounter{tocdepth}{1}
\hypersetup{bookmarksdepth=2}
\tableofcontents

\section{Introduction}\label{sec:introduction}
A compact complex manifold $X$ admitting a Hermitian metric of
strictly negative Chern holomorphic sectional curvature (HSC) is Brody
hyperbolic: every holomorphic map $\C\to X$ is constant. By Brody's
theorem, this is equivalent to Kobayashi hyperbolicity
\cite{amp:old:Brody1978}. The canonical ampleness conjecture asks whether
this hyperbolicity alone forces the canonical bundle $K_X$ to be ample.
For the general compact complex formulation, see
\cite[Conjecture~1]{GolotaCanonical2017} and
\cite[Conjecture~1.1(a)]{YangZheng}; see also Kobayashi's discussion
\cite[p.~370]{KobayashiHCS1998}.
The corresponding questions for quasi-negative and strictly negative
Chern HSC appear in \cite[Conjecture~1.1(b) and (c)]{YangZheng}, respectively.

The constant-curvature conjecture concerns the metric itself. It predicts
that a Hermitian metric $h$ on a compact complex manifold with $H_h\equiv c$ is
K\"ahler if $c\ne0$ and Chern flat if $c=0$ \cite[Conjecture~1]{RaoZheng22}.
This holds automatically in dimension one and on compact complex surfaces
\cite{BalasGauduchon,ADM96}. The zero case behaves differently in
higher dimensions, as the compact counterexamples in Theorem~\ref{ce:thm:main} show.

\subsection{Canonical positivity and Fujiki rigidity}
Wu and Yau proved canonical ampleness for projective manifolds
with strictly negative K\"ahler HSC \cite{WuYau16}; Tosatti and Yang
removed projectivity from the hypotheses \cite[Corollary~1.2]{TosattiYang17}.
In a recent preprint, Wan obtains ampleness for quasi-negative
Chern HSC of a Hermitian metric on a compact K\"ahler manifold
\cite[Theorem~1.1]{WanQuasiNegative2026}. Here we start from
hyperbolicity and impose a differential condition on a separate
$(1,1)$-current.

For a smooth real $(1,1)$-form $\kappa$, plurinegativity means that
$i\partial\bar\partial\kappa$ is nonpositive on decomposable complex
two-vectors \cite[Definition~2.3]{extzero:AGL}. For currents we use the
corresponding distributional inequality, with
$\ddcn=(i/2\pi)\partial\bar\partial$.

For a real $(1,1)$ current $\mathscr T$ bounded below by a smooth form
and with $d\mathscr T$ smooth, write
$\mathscr T=\omega_0+\ddcn u$, where $\omega_0$ is smooth and $u$ is
quasi-plurisubharmonic (Section~\ref{sec:conventions}). Its Lelong
number at $x$ is $\nu(\mathscr T,x)=\nu(u,x)$. For $t>0$, its upper
level sets, defined by
\[
 E_t(\mathscr T)=\{x\in X:\nu(\mathscr T,x)\ge t\},
\]
are analytic by Siu's analyticity theorem after a local smooth correction
\cite[Main Theorem]{Siu74}\cite{SuVu25}.

\begin{definition}\label{cur:def:C}
A compact complex manifold $X$ satisfies $\mathrm{(C)}$ if there
exist a real $(1,1)$ current $\mathscr T$, a smooth Hermitian form $\gamma$, and
$\varepsilon>0$ such that
\begin{equation}\label{cur:eq:C}
 \mathscr T\ge\varepsilon\gamma,\qquad d\mathscr T\text{ is smooth},\qquad \ddcn \mathscr T\le0.
\end{equation}
It satisfies $\mathrm{(C_2)}$ if these data can be chosen so that, in addition,
\begin{equation}\label{cur:eq:intrinsic-C2}
 \dim_{\mathbb C}E_t(\mathscr T)\le2\qquad\text{for every }t>0.
\end{equation}
\end{definition}

A holomorphic line bundle
is \emph{nef} if its normalized Chern curvature admits smooth
representatives bounded below by $-\varepsilon\gamma$ for every
$\varepsilon>0$ and one fixed smooth Hermitian form $\gamma$.

\begin{maintheorem}\label{amp:main}
Let $X$ be a compact Brody-hyperbolic complex manifold satisfying $\mathrm{(C)}$. Then its canonical
bundle $K_X$ is nef. If, moreover, $X$ satisfies $\mathrm{(C_2)}$,
then $X$ is projective and $K_X$ is ample.
\end{maintheorem}

Condition $\mathrm{(C_2)}$ holds if $X$ admits a smooth plurinegative
metric or if a witness for $\mathrm{(C)}$ has zero Lelong numbers
everywhere. In dimension one the sign condition is automatic by bidegree.
On a compact complex surface a Gauduchon metric has fundamental
form $\omega$ with $i\partial\bar\partial\omega=0$. Thus in either
dimension $\mathrm{(C_2)}$ needs no curvature hypothesis.

\begin{maincorollary}\label{amp:cor:low-dimensional-ampleness}
Let $X$ be a compact complex manifold of complex dimension
$n\le2$. If $X$ is Brody hyperbolic, then $X$ is
projective and $K_X$ is ample. In particular, this holds if $X$
admits a smooth Hermitian metric of strictly negative Chern HSC.
\end{maincorollary}
This proves parts (a) and (c) of \cite[Conjecture~1.1]{YangZheng} in
dimension at most two. On a surface the Gauduchon metric is auxiliary;
a metric of strictly negative Chern HSC, when given, need not be
pluriclosed. The corollary also gives canonical ampleness without the
global spherical shell conjecture \cite[Working Hypothesis~5.5]{NakamuraGSS1984}.
Assuming that conjecture, Khelifati proved that every compact Kobayashi
hyperbolic surface is K\"ahler \cite[Theorem~3.20]{KhelifatiCY2026}.

We next ask when the metric itself is K\"ahler. For $X$ in
Fujiki's class $\mathcal C$, that is, bimeromorphic to a compact
K\"ahler manifold, nonpositive Chern HSC excludes rational curves
and hence gives a K\"ahler metric on $X$
(Lemmas~\ref{n:lem:fujiki-kahler} and~\ref{n:lem:no-rational-kahler}).
The next theorem concerns $h$ itself. Write $\omega_h$ and
$\Ric_h^{(1)}$ for its fundamental form and first Chern--Ricci form.
\begin{maintheorem}
\label{n:thm:main}
Let $X$ be a compact complex manifold of dimension $n\ge2$ in Fujiki's
class $\mathcal C$, and let $h$ be a smooth Hermitian metric with
$H_h\equiv c\le0$. Then $h$ is K\"ahler. More precisely:
\begin{enumerate}
\item If $c<0$, then $X$ is projective and $K_X$ is ample. With
$\lambda_c=(n+1)c/2$, $h$ is the unique K\"ahler--Einstein metric satisfying
\[
\Ric_h^{(1)}=\lambda_c\omega_h,\qquad
[\omega_h]=\frac{2\pi}{-\lambda_c}c_1(K_X).
\]
Its universal cover is complex hyperbolic space of holomorphic sectional
curvature $c$.
\item If $c=0$, then $h$ is K\"ahler flat. Its universal cover is
Euclidean, and a finite \'etale cover of $X$ is a complex torus on which
the lifted metric is translation invariant.
\end{enumerate}
\end{maintheorem}

For prescribed pluriclosed metrics on compact K\"ahler
manifolds, the negative and zero cases of constant Chern HSC are treated in
\cite[Theorem~1.5]{TangConstant} and \cite[Corollary~1.2]{BroderTang}.
With pseudoeffective $K_X$ and zero Chern HSC,
\cite[Theorem~1.1]{BroderTang} obtains conformal balancedness without
a K\"ahler background, and flat K\"ahlerness if the metric is
pluriclosed. Theorem~\ref{n:thm:main} requires no pluriclosedness in
Fujiki's class.
\subsection{Local rigidity and compact nonflat examples}
For $c\ne0$, we ask when constancy of Chern HSC alone implies
that the given metric is K\"ahler locally.
Polarization recovers the full curvature tensor of a K\"ahler metric.
For a Hermitian metric it recovers only the part symmetric in the
two holomorphic indices and in the two antiholomorphic indices.
The Chern identities relate the remaining terms to torsion and its
derivatives. We first use them to rule out torsion in dimension three.

\begin{maintheorem}[Local threefold rigidity]\label{p:thm:main}
Let $h$ be a smooth Hermitian metric on a complex threefold with
constant Chern holomorphic sectional curvature $c\ne0$. Then $h$ is
K\"ahler. Neither compactness, completeness nor balancedness is required.
\end{maintheorem}

Earlier threefold results assume compactness. Zhou and Zheng use
vanishing real bisectional curvature \cite{ZhouZheng22}, while Ma and
Nie treat $c\le0$ under normal balancedness \cite{MaNie24}.
In a preprint, Chen and Li remove normality in that range for
compact balanced threefolds \cite{ChenLi26}.

In dimension four we assume either balancedness or a rank bound.
A Hermitian metric is balanced if $d\omega_h^{n-1}=0$.
The rank criteria also apply in higher dimensions and measure how
the image of the torsion derivative grows under successive curvature
endomorphisms. Let
$\nabla$ be the Chern connection of $h$, with torsion $T$ and curvature
$R$. Write $\nabla'$ for its $(1,0)$ part
and set $\mathcal X=\nabla'T$, so that $\mathcal X_p=\nabla_pT$. At each point set
\begin{equation}\label{eq:hull}
\begin{split}
 H_0(x)&=\Span\{\mathcal X_p(u,v):p,u,v\in T_x^{1,0}X\},\\
 H_{j+1}(x)&=H_j(x)+\Span\{R(p,\bar q)a:p,q\in T_x^{1,0}X,\ a\in H_j(x)\}.
\end{split}
\end{equation}
Thus $H_j(x)$ is spanned by compositions of at most $j$ curvature
endomorphisms applied to $\im\mathcal X_x$; no curvature derivatives
or parallel transport occur. If $H_0(x)=0$, all $H_j(x)$ vanish.
If $H_{j+1}(x)=H_j(x)$, their common value is curvature invariant.
\Needspace{6\baselineskip}
\begin{maintheorem}\label{thm:main}\label{ff:thm:support}
Let $(X,h)$ be a connected smooth Hermitian manifold of complex dimension
$n\ge4$, with $H_h\equiv c\ne0$.
Each of the following conditions forces the prescribed metric $h$
to be K\"ahler:
\begin{enumerate}
\item\label{thm:main-fourfold} $n=4$ and $h$ is balanced.
\item\label{thm:main-fourfold-hull} $n=4$ and
\begin{equation}\label{eq:rank-fourfold-hull}
             \dim_{\C}H_1(x)\le3\qquad(x\in X).
\end{equation}
\item\label{thm:main-balanced} $n\ge5$, $h$ is balanced, and
\begin{equation}\label{eq:rank}
             \dim_{\C}H_{n-3}(x)\le n-1\qquad(x\in X).
\end{equation}
\item\label{thm:main-unbalanced} $n\ge5$ and
\begin{equation}\label{eq:rank-unbalanced}
             \dim_{\C}H_{n-2}(x)\le n-1\qquad(x\in X).
\end{equation}
\end{enumerate}
Moreover, under any of these conditions, $(X,h)$ is locally
holomorphically isometric to the complex space form of holomorphic
sectional curvature $c$.
\end{maintheorem}

The three hull bounds can be tested on any subset of real Hausdorff
dimension $>2n-1$ (Lemma~\ref{sup:lem:hausdorff}). Balancedness saves
one iteration for $n\ge5$ and removes the hull condition in dimension
four. These criteria do not settle unrestricted local rigidity
in dimensions at least four.

Wang and Zheng treat balanced fourfolds with Bismut-parallel
torsion \cite{WangZheng2026}; see also
\cite[Theorem~2]{WangZhengBalanced2026}. A preprint of Haohao Wang
proves K\"ahler rigidity for Bismut-torsion-parallel metrics with
nonzero constant Chern HSC, without compactness or completeness
\cite[Corollary~1.3]{HaohaoWangBTP2026}. Theorem~\ref{thm:main}
does not require Bismut-parallel torsion. In a separate preprint,
Chen and Zheng consider HSC and parallel torsion for the
\emph{same} Hermitian connection: nonzero constant HSC then forces
the torsion to vanish \cite[Theorem~2]{ChenZhengParallel2026}.

Chen and Zheng prove in a preprint that compact threefolds with
zero Chern HSC are Chern flat
\cite[Theorem~1]{ChenZhengThreefold2026}; Chern flatness does
not rule out torsion, as the Iwasawa manifold illustrates. In the
complete noncompact setting, Chen, Chen and Nie give examples
with zero Chern HSC but nonzero Chern curvature
\cite[Proposition~1.6 and Example~3.9]{ChenChenNie}.
The next theorem shows that this can occur on compact manifolds,
even when all four Chern--Ricci contractions vanish.

\Needspace{11\baselineskip}
\begin{maintheorem}\label{ce:thm:main}\label{ce:thm:allhigher}
For every integer $n\ge7$ there is a compact connected Hermitian
manifold $(X_n,h_n)$ of complex dimension $n$ such that all four Chern--Ricci contractions vanish
and 
\[
H_{h_n}\equiv0,\qquad R^{h_n}\not\equiv0,\qquad
d\omega_{h_n}^{n-1}=0.
\]
Moreover, $K_{X_n}$ admits
a nowhere-zero Chern-parallel holomorphic volume form.
\end{maintheorem}

The construction uses compact quotients of the complex
seven-dimensional quadric \cite[Theorems~1.1.3 and~4.2.1, $q=7$]{CEKY}.
The examples lie outside Fujiki's class by Theorem~\ref{n:thm:main}(2).
The compact zero-HSC problem in dimensions four through six is not
settled here.

\subsection{Auxiliary forms and K\"ahler flatness}
Theorem~\ref{ce:thm:main} shows that vanishing of the first
Chern--Ricci form does not force flatness at zero HSC.
The next theorem adds a condition on an auxiliary form, not another
curvature hypothesis on $h$. A Stokes identity uses plurinegativity
of the form to control the torsion of $h$.
\Needspace{7\baselineskip}
\begin{maintheorem}
\label{ext:thm:auxiliary}
Let $(X,h)$ be a compact Hermitian manifold with
$H_h\equiv0$ and $\Ric_h^{(1)}=0$. Suppose that $X$ admits a smooth
semipositive real $(1,1)$-form $\kappa$ which is positive definite at
some point and satisfies $i\partial\bar\partial\kappa\le0$ on
decomposable complex two-vectors. Then $h$ is K\"ahler flat. 
\end{maintheorem}

Angella--Guedj--Lu exclude plurinegative Hermitian metrics on compact
complex parallelizable manifolds other than tori
\cite[Proposition~4.1]{extzero:AGL}, using squares of derivatives of
a global holomorphic coframe and Stokes' theorem.
Here we work directly with the torsion of $h$, without assuming
Chern flatness or a global holomorphic coframe. The auxiliary form
may degenerate away from the region where it is positive definite. In dimensions at least three, Theorem~\ref{extzero:thm:mixed-ricci}
gives a mixed integral condition in place of the vanishing of the
first Chern--Ricci form.

\subsection{Outline of the proofs}
\paragraph{Canonical positivity, Theorem~\ref{amp:main}}
Varying a maximizing disc on a fixed log resolution gives trace and
determinant bounds, hence a bounded canonical potential and nefness (Theorem~\ref{cur:can:theorem} and
Lemma~\ref{amp:lem:zero-lelong-nef}).
For ampleness under $\mathrm{(C_2)}$, we seek a smooth plurinegative
metric. If none exists, separation and localization give a nonzero
positive closed current of bidimension $(1,1)$ supported on a polar
set of dimension at most two (Lemma~\ref{cur:lem:localization}).
It pairs to zero with the canonical class and nonpositively with the
regularized background form, but the curve and surface estimates
force the latter pairing to be positive (\S\ref{cur:sec:elimination}).
Thus a smooth plurinegative metric exists and gives ampleness by
Proposition~\ref{amp:prop:smooth-ampleness}.

\paragraph{Fujiki rigidity, Theorem~\ref{n:thm:main}}
When $c<0$, canonical ampleness gives a
negative K\"ahler--Einstein metric by the Aubin--Yau theorem \cite{Aubin}\cite[Theorem~5]{Yau};
an integrated Chern--Lu comparison identifies it with $h$
(\S\ref{n:lem:comparison}). For $c=0$, Tang's theorem gives
$c_1(X)=0$ in real cohomology \cite[Theorem~1.1(2)]{TangNef};
Yau's theorem then gives a Ricci-flat K\"ahler metric \cite[Theorem~2]{Yau}.
The comparison forces the symmetric part of the connection difference to vanish.
Integration by parts shows that $h$ is parallel for the Ricci-flat
connection, hence K\"ahler. Its zero HSC then gives flatness
(\S\ref{z:sec:zero}).

\paragraph{Local rigidity at nonzero Chern HSC, Theorems~\ref{p:thm:main} and~\ref{thm:main}\textup{(\ref{thm:main-fourfold})}}
The threefold and balanced fourfold proofs start from the same Chern
identities (\S\ref{stage:common-source}). In dimension three, on
charts where the ranks of the nonzero torsion and its prime derivative
are constant, they put the derivative image in the torsion image,
which has dimension at most two. The metric
equations exclude this configuration (\S\ref{stage:threefold} and
\S\ref{tloc:bothsignlocaltheorem}). In dimension four, balancedness makes
the torsion trace-free. On a dense regular locus, spectral constraints
and polynomial equations reduce the proof to ideal lines and families
of two-planes. The metric Chern identities exclude these configurations,
including the cone and quadric cases (\S\ref{stage:fourfold-terminals}--\S\ref{stage:fourfold-quadric}).

\paragraph{Curvature-hull criteria, Theorem~\ref{thm:main}\textup{(\ref{thm:main-fourfold-hull})}--\textup{(\ref{thm:main-unbalanced})}}
The hull criteria follow from the common Chern identities without
using the balanced fourfold classification. On a nonzero-torsion chart where the hull
ranks are constant, the first rank estimates exclude small $\dim H_1$
when $H_0\subsetneq H_1$ under the relevant hypotheses. If consecutive
proper hulls agree, their common subbundle is curvature invariant, and
the Chern identities put the torsion image inside it. The support
theorem rules this out (Theorem~\ref{sup:thm:support}). A separate use of real analyticity extends rank bounds from testing sets (\S\ref{sup:sec:finite}).

\paragraph{Zero Chern HSC, Theorems~\ref{ce:thm:main} and~\ref{ext:thm:auxiliary}}
On the seven-dimensional quadric, the octonionic curvature formula
gives a nonzero tensor whose HSC and four Chern--Ricci contractions
vanish. A lattice quotient gives the compact counterexample. Taking products with flat tori yields counterexamples in higher dimensions (see \S\ref{sec:high-dimensional-examples})..

In the auxiliary-form theorem,  the vanishing first
Chern--Ricci form forces balancedness and the vanishing of the other
three Chern--Ricci contractions (Lemma~\ref{extzero:lem:first-ricci}).
In dimension at least three, integration by parts and plurinegativity
make a nonnegative torsion integral nonpositive. The torsion vanishes
where the auxiliary form is positive definite, and real analyticity
extends the vanishing. In dimension two, balancedness gives
K\"ahlerness, which is automatic in dimension one. Zero HSC then
gives flatness
(\S\ref{extzero:sec:zero}).

Table~\ref{tab:main-results} collects the statements and proof locations.
\begin{table}[!htbp]
\centering
\caption{Main results and where they are proved.}\label{tab:main-results}
\small
\begin{tabular}{@{}>{\raggedright\arraybackslash}p{.12\linewidth}
  >{\raggedright\arraybackslash}p{.35\linewidth}
  >{\raggedright\arraybackslash}p{.27\linewidth}
  >{\raggedright\arraybackslash}p{.17\linewidth}@{}}
\toprule
\textbf{Result} & \textbf{Setting} & \textbf{Conclusion} & \textbf{Main proof} \\
\midrule
Thm.~\ref{amp:main} &
Compact Brody hyperbolic, with $\mathrm{(C)}$ or $\mathrm{(C_2)}$ &
$K_X$ nef under $\mathrm{(C)}$; ample and $X$ projective under $\mathrm{(C_2)}$ &
\S\ref{amp:sec:ampleness}--\S\ref{cur:sec:elimination} \\[2pt]
Cor.~\ref{amp:cor:low-dimensional-ampleness} &
Compact Brody hyperbolic, $n\le2$ &
$K_X$ ample; $X$ projective &
\S\ref{amp:sec:ampleness} \\[2pt]
Thm.~\ref{n:thm:main} &
Compact Fujiki class, $n\ge2$, constant Chern HSC $c\le0$ &
$h$ K\"ahler--Einstein if $c<0$; K\"ahler flat if $c=0$ &
\S\ref{stage:fujiki} \\[2pt]
Thm.~\ref{p:thm:main} &
Threefold, constant Chern HSC $c\ne0$; no balancedness assumption &
$h$ K\"ahler &
\S\ref{stage:common-source}, \S\ref{stage:threefold}, \S\ref{tloc:bothsignlocaltheorem} \\[2pt]
Thm.~\ref{thm:main}\textup{(\ref{thm:main-fourfold})} &
Balanced fourfold, constant Chern HSC $c\ne0$ &
$h$ K\"ahler, with no hull hypothesis &
\S\ref{stage:fourfold-terminals}--\S\ref{stage:fourfold-quadric}, \S\ref{sup:sec:finite} \\[2pt]
Thm.~\ref{thm:main}\textup{(\ref{thm:main-fourfold-hull})}--\textup{(\ref{thm:main-unbalanced})} &
Constant $c\ne0$; pointwise proper $H_1$ in dimension four, or
$H_{n-2}$ for $n\ge5$ ($H_{n-3}$ if balanced) &
$h$ K\"ahler &
\S\ref{stage:rankgap-support}--\S\ref{sup:sec:finite} \\[2pt]
Thm.~\ref{ce:thm:main} &
Every dimension $n\ge7$ &
Compact balanced examples with $H_h\equiv0$ and all four Chern--Ricci contractions
zero, but $R^h\not\equiv0$ &
\S\ref{sec:high-dimensional-examples} \\[2pt]
Thm.~\ref{ext:thm:auxiliary} &
Compact, $H_h\equiv0$, $\Ric_h^{(1)}=0$ and a smooth semipositive
plurinegative form positive definite somewhere &
$h$ K\"ahler flat &
\S\ref{extzero:sec:zero} \\
\bottomrule
\end{tabular}
\end{table}

\section{Notation and preliminaries}\label{sec:conventions}
We fix the conventions for metrics and currents, then derive the
Chern identities used for local rigidity.

\subsection{Hermitian metrics, connections, and tensors}
Manifolds are smooth, connected and of positive complex dimension unless
otherwise specified; compactness is an additional hypothesis. Put
$K_X=\Lambda^n(T^{1,0}X)^*$ and, in holomorphic coordinates,
\[
 h^{i\bar j}h_{k\bar j}=\delta^i_k,\qquad
 \omega_h=ih_{j\bar k}dz^j\wedge d\bar z^k,\qquad
 dV_h=\omega_h^n/n!,\qquad
 \Delta_h=h^{j\bar k}\partial_j\partial_{\bar k}.
\]
Repeated paired indices are summed; $\partial_j=\partial/\partial z^j$,
$\partial_{\bar k}=\partial/\partial\bar z^k$ and $g_\R=2\Rea h$.
For the Chern connection $\nabla=\nabla^h$, write
$\nabla'=\nabla^{1,0}$, $\nabla''=\nabla^{0,1}$,
$\nabla_j=\nabla_{\partial_j}$, $\nabla_{\bar k}=\nabla_{\partial_{\bar k}}$ and
$\nabla^j=h^{j\bar k}\nabla_{\bar k}$.
Hermitian pairings, including induced ones, are linear in the first
argument; thus $\langle\zeta,\partial u\rangle_h
=h^{p\bar q}\zeta_p\partial_{\bar q}u$ for real $u$.
For $\alpha=i\alpha_{j\bar k}dz^j\wedge d\bar z^k$, set
$\tr_h\alpha=h^{j\bar k}\alpha_{j\bar k}$, so
\begin{equation}\label{eq:trace-wedge}
 \alpha\wedge\frac{\omega_h^{n-1}}{(n-1)!}
 =\tr_h\alpha\,dV_h.
\end{equation}
For complexified fields,
$R(U,V)=[\nabla_U,\nabla_V]-\nabla_{[U,V]}$ and
$T(U,V)=\nabla_UV-\nabla_VU-[U,V]$; $\mathscr T$ denotes a
$(1,1)$ current, as in Definition~\ref{cur:def:C}.
Write $R(U,\bar V,W,\bar Z)=\langle R(U,\bar V)W,Z\rangle_h$.
In a unitary frame,
$R_{i\bar j k\bar\ell}=\langle R(e_i,\bar e_j)e_k,e_\ell\rangle_h$
and $T(e_i,e_k)=T^\ell{}_{ik}e_\ell$; in coordinates,
\begin{equation}\label{eq:curvature-coordinate}
 R_{i\bar j k\bar\ell}
 =-\partial_i\partial_{\bar j}h_{k\bar\ell}
 +h^{p\bar q}(\partial_i h_{k\bar q})(\partial_{\bar j}h_{p\bar\ell}).
\end{equation}
For $\xi\ne0$, $H_h(\xi)=R(\xi,\bar\xi,\xi,\bar\xi)/|\xi|_h^4$.

Matrix transpose, conjugate transpose and algebraic dual are
$A^{\mathsf T}$, $A^*=\bar A^{\mathsf T}$ and $A^\vee$;
slot transpose is $B^{\mathsf T}(u,v)=B(v,u)$.
The classical adjugate satisfies
$A\operatorname{adj}A=(\operatorname{adj}A)A=(\det A)I$.
Hermitian adjoints use the specified metrics, and differential
adjoints their $L^2$ pairings with $dV_h$ unless otherwise stated.
Distinguish the holomorphic-first array $H_h^{\rm coef}=(h_{i\bar j})$
from the column matrix $G_h^{\rm col}=(H_h^{\rm coef})^{\mathsf T}$:
\[
 h_{\rm bil}(v,\bar w)=\langle v,w\rangle_h
 =v^{\mathsf T}H_h^{\rm coef}\bar w
 =\bar w^{\mathsf T}G_h^{\rm col}v.
\]
In a holomorphic frame, connection matrix products use the column convention
$(G_h^{\rm col})^{-1}\partial G_h^{\rm col}$; its transpose is
$(\partial H_h^{\rm coef})(H_h^{\rm coef})^{-1}$.
We keep form coefficients holomorphic-first; transposition leaves
determinants and positivity unchanged.

\subsection{Curvature forms and disc integrals}
Our normalized curvature operators are
\begin{equation}\label{cur:eq:normalization}
 \dcn=\frac{i}{4\pi}(\bar\partial-\partial),\qquad
 \ddcn=\frac{i}{2\pi}\partial\bar\partial=\frac1{2\pi}\ddc,
 \qquad\rho_h=\frac1{2\pi}\Ric_h^{(1)}.
\end{equation}
Here $\Ric_h^{(1)}=-i\partial\bar\partial\log\det h$ is the first
Chern--Ricci form. Thus $-\rho_h$ represents $c_1(K_X)$. For a holomorphic frame
$e$ with squared norm $e^{-\phi}$,
$\operatorname{curv}(L,h_L)=\ddcn\phi$; for the canonical section
$s_D$ of an effective divisor in a smooth divisor metric,
the Poincar\'e--Lelong formula \cite[Chapter~III, (2.15)]{amp:new:DemaillyCADG}
gives $\ddcn\log|s_D|^2=[D]-\operatorname{curv}\mathcal O(D)$.

\paragraph{Disc integrals and coefficient conversion}
\label{conv:disc-normalization}
On $\mathbb D=\{|z|<1\}$ write $dS=dx\,dy$, $ds$ for circle
arclength, $\Delta=4\partial_z\partial_{\bar z}$ and
$G_0=\log|z|^{-2}$. A target form
$q=iq_{a\bar b}dz^a\wedge d\bar z^b$ has coefficient pairing
$q(v,w)=q_{a\bar b}v^a\bar w^b$; arguments $(v,\bar w)$ denote
the corresponding complex-bilinear pairing.
Its disc array is $[q]_{\rm disc}=2\pi(q_{a\bar b})$,
$q_{\rm disc}(v,\bar w)=v^{\mathsf T}[q]_{\rm disc}\bar w$;
matrix products use its transpose. In particular $h_{\rm disc}=2\pi h$.
For coefficient and current integrals respectively,
\begin{equation}\label{conv:eq:disc-conversion}
\begin{aligned}
 A_q(f)&=\frac1\pi\int_{\mathbb D}q(f',f')\,dS,
 &P_q(f)&=\frac1\pi\int_{\mathbb D}G_0q(f',f')\,dS,\\
 A_q^{\rm cur}(f)&=\int_{\mathbb D}f^*q=2\pi A_q(f),
 &P_q^{\rm cur}(f)&=\int_{\mathbb D}G_0f^*q=2\pi P_q(f),\\
 A_{\rho_k}^{\rm cur}&=A_{\Ric_k^{(1)}},
 &P_{\rho_k}^{\rm cur}&=P_{\Ric_k^{(1)}}.
\end{aligned}
\end{equation}
Hence $A_q^{\rm cur}=\pi^{-1}\int q_{\rm disc}(f',\overline{f'})dS$.
A subscripted bounded disc weight $\chi$ multiplies the integrand, so $A_{\chi q}^{\rm cur}=2\pi A_{\chi q}$;
a further subscript $r$ restricts to $|z|<r$. The weight is a function on the disc and need not come from the
target. For currents we use the measures on the disc defined by the
scalar pullback identities below; an unweighted divisor intersection of
multiplicity one contributes one unit of current area. Forms themselves are not rescaled.
Constant metric scaling preserves the Chern connection, torsion
and normalized Ricci form.

Put $d_k(f)=|f'(0)|_k^2$,
$\langle u\rangle_r=(2\pi)^{-1}\int_0^{2\pi}u(re^{i\vartheta})d\vartheta$
and $\mathcal D(a)=\pi^{-1}\int_{\mathbb D}|a'|^2dS$.
Omit $r$ at the boundary when defined; when Green applies,
$P_{\ddcn v}^{\rm cur}(f)=\langle v\circ f\rangle-v(f(0))$.
Sobolev regularity of maps is read in target charts.
On the disc $W^{1,p}$, $p>2$, is a Banach algebra contained in $C^0$;
for holomorphic $b$ continuous up to the boundary, $\|b\|_{H^2}^2=\langle|b|^2\rangle$.

\subsection{Currents, Lelong numbers, and pairings}
A real $(1,1)$ current is \emph{quasi-positive} if bounded below
by a smooth real $(1,1)$ form, and \emph{quasi-closed} if its exterior
derivative is smooth. The order is weak positivity.
Strongly positive $(q,q)$ forms form the closed cone generated by
$i^{q^2}\xi\wedge\bar\xi$ with decomposable $(q,0)$ forms $\xi$.
A weakly positive $(2,2)$ current pairs nonnegatively with strongly
positive $(n-2,n-2)$ test forms; for smooth forms this tests
decomposable bivectors, and in dimension one the condition is vacuous.
A current of bidimension $(1,1)$ acts on $(1,1)$ test forms and has
bidegree $(n-1,n-1)$ on an $n$-fold. The pairs of indices agree on a
surface. We use this convention for the positive obstruction current
in the localization argument.

By \cite[Lemma~1.8]{cur:BGL25}, a quasi-positive quasi-closed
$(1,1)$ current is $\mathscr T=\omega_0+\ddcn u$.
Indeed, correct a local smooth primitive of $d\mathscr T$ to type
$(1,1)$ by the $\partial$-lemma and conjugation. Subtracting it leaves a closed current with local quasi-psh potentials.
On overlaps, $\ddcn$ of the difference of two potentials is smooth.
Taking the local trace gives a smooth Laplacian, so the difference is
smooth by elliptic regularity. A partition of unity now combines the
potentials into a global quasi-psh $u$, with smooth real remainder
$\omega_0$.
Its Lelong number $\nu(\mathscr T,x)=\nu(u,x)$ is independent of the decomposition and defines
$E_t(\mathscr T)=\{x:\nu(u,x)\ge t\}$, $t>0$;
write $\Pole(u)=\{u=-\infty\}$ and $\nu_G(u)$ for the generic
Lelong number along a prime divisor $G$.

A real Bott--Chern $(1,1)$ class is a smooth real closed $(1,1)$ form modulo
global $\ddcn$ of smooth real functions. Fix a Hermitian form $\gamma$.
A class $\alpha$ is \emph{nef} if for every $\varepsilon>0$ it has a smooth
representative $\alpha_\varepsilon\ge-\varepsilon\gamma$.
Write $[Q]_{BC}$ for the class of a closed real $(1,1)$ current.
If $\eta$ is a $\ddcn$-closed current of bidimension $(1,1)$ and
$\alpha_0$ is a smooth representative of $\alpha$, set
$\eta\cdot\alpha:=\langle\eta,\alpha\rangle:=\langle\eta,\alpha_0\rangle$.
Stokes' theorem for currents \cite[Chapter~I, \S2.B.1]{amp:new:DemaillyCADG}
shows that this pairing is independent of $\alpha_0$.
On a compact surface it is the Aeppli--Bott--Chern pairing. Divisors in this cohomological pairing denote their first Chern
classes. In a current--form pairing, angle brackets denote the
current acting on the smooth test form, regardless of the order of
the entries.

\begin{lemma}[Nef pairing]\label{cur:lem:nef-ddc-pairing}
Let $\alpha$ be a nef real Bott--Chern $(1,1)$ class on a compact complex
manifold. If $R$ is a positive $\ddcn$-closed current of bidimension $(1,1)$,
then $\langle R,\alpha\rangle\ge0$.
\end{lemma}
\begin{proof}
Fix a Hermitian form $\gamma$. For each $\varepsilon>0$, choose a smooth
representative $\alpha_\varepsilon$ of $\alpha$ with
$\alpha_\varepsilon\ge-\varepsilon\gamma$. Since $R$ is positive,
$\langle R,\alpha\rangle=\langle R,\alpha_\varepsilon\rangle
\ge-\varepsilon\langle R,\gamma\rangle$.
The right-hand side tends to zero as $\varepsilon\downarrow0$, so
$\langle R,\alpha\rangle\ge0$.
\end{proof}
For a line bundle $L$ and a curve $C$, $L\cdot C$ denotes the degree
of the pullback of $L$ to the normalization of $C$.

The expression $\mathcal O(\sum a_jE_j)$
for real coefficients denotes a real tensor combination of divisor
line bundles, equipped with the corresponding real combination of
smooth metric weights and curvature forms.

By \cite[Proposition~3.3(iv)]{cur:DP}, a Bott--Chern class with a positive
closed representative $T$ on a compact complex space is nef if its
restrictions to the irreducible components of $E_t(T)$, $t>0$, are nef.
On an integral compact curve, nonnegative degree suffices: the
pullback of $\alpha+\varepsilon\gamma$ to the normalization has positive
integral, hence a smooth potential makes it positive. Add small cutoff
logarithmic poles over the singular points and descend. Shift local smooth
functions making $\alpha+\varepsilon\gamma$ positive after $\ddcn$ addition
to dominate near the poles and lie below on outer collars.
The regularized maximum then gives $\alpha+\ddcn u_\varepsilon\ge
-\varepsilon\gamma$ with smooth $u_\varepsilon$ on the curve.

\subsection{Identities for constant Chern holomorphic sectional curvature}
\subsubsection{Polarization and tensor conventions}
Let $\delta_{ij}$ denote the Kronecker symbol.
The curvature satisfies the Hermitian symmetry
$\overline{R_{i\bar j k\bar\ell}}=R_{j\bar i\ell\bar k}$.
When $H_h\equiv c$, polarization of the quartic identity gives
\begin{equation}\label{p:eq:HSC}\equationalias{n:eq:polarization}
\begin{split}
 R_{i\bar j k\bar\ell}+R_{k\bar j i\bar\ell}
 +R_{i\bar\ell k\bar j}+R_{k\bar\ell i\bar j}
 =2c(\delta_{ij}\delta_{k\ell}+\delta_{i\ell}\delta_{kj}).
\end{split}
\end{equation}
This determines the component symmetric in both index pairs.
The complex space-form tensor in the same convention is
\begin{equation}\label{p:eq:Rsf}
 (R_{\mathrm{sf}})_{i\bar j k\bar\ell}
 =\frac c2(\delta_{ij}\delta_{k\ell}+\delta_{i\ell}\delta_{kj}).
\end{equation}
A comma on a tensor in a unitary frame denotes covariant differentiation;
commas on metric coefficients in holomorphic coordinates denote ordinary
partial derivatives.
For a tensor $U$, set
$\nabla^2_{q,p}U=\nabla_q(\nabla_pU)-\nabla_{\nabla_qp}U$,
using the induced connection on all slots.
The Chern identities needed below are
\begin{align}
 &R_{i\bar j k\bar\ell}-R_{k\bar j i\bar\ell}
 =-T^\ell{}_{ik,\bar j},\label{p:eq:firstBi}\\
 &\nabla_pR_{i\bar j k\bar\ell}-\nabla_iR_{p\bar j k\bar\ell}
 =T^r{}_{ip}R_{r\bar j k\bar\ell},\label{p:eq:diffBi}\\
 &\nabla^2_{q,p}U-\nabla^2_{p,q}U=-T^r{}_{qp}\nabla_rU.
 \label{p:eq:purecomm}
\end{align}
The first identity is \cite[(2.6)]{TangConstant}. Equation~\eqref{p:eq:purecomm} and its conjugate hold for every
tensor $U$, since the Chern curvature has type $(1,1)$. We retain the torsion term and differentiate every bundle factor.

Tensor norms sum over ordered indices, for example
\[
 |T|^2=\sum_{i,k,\ell}|T^\ell{}_{ik}|^2.
\]
No factorial is inserted for alternating or symmetric tensor slots.
Differential forms carry the exterior-power metric, so
$|\partial\omega_h|^2=|T|^2/2$. All unqualified norms are formed with
$h$; the mixed norm for the identity map will be specified separately.
We write $\|U\|_2^2=\int_X|U|_h^2\,dV_h$.
Symmetrization $\Sym$ is the average over the indicated permutations.
In contrast, $u\odot v=u\otimes v+v\otimes u$ is unnormalized.
These conventions apply before epsilon contraction or comparison of
polynomial coefficients with symmetric-tensor entries.

\subsubsection{Ricci contractions and the torsion trace}
In a unitary frame the four Chern--Ricci contraction tensors
$r_h^{(m)}$, $1\le m\le4$, have coefficients
\[
 (r_h^{(1)})_{i\bar j}=\sum_kR_{i\bar j k\bar k},\quad
 (r_h^{(2)})_{i\bar j}=\sum_kR_{k\bar k i\bar j},\quad
 (r_h^{(3)})_{i\bar j}=\sum_kR_{i\bar k k\bar j},\quad
 (r_h^{(4)})_{i\bar j}=\sum_kR_{k\bar j i\bar k}.
\]
In the dual coframe $(\varphi^i)$, write
$\Ric_h^{(m)}=i\sum_{i,j}(r_h^{(m)})_{i\bar j}\varphi^i\wedge\bar\varphi^j$
for the associated $(1,1)$ forms. The first two forms are real.
The third and fourth coefficient tensors are Hermitian adjoints
of one another, so their forms are complex conjugate but need not
be real. The first form is $2\pi\rho_h$ in \eqref{cur:eq:normalization}.
For a $(1,1)$ form, $\tr_h$ denotes the trace of its coefficient tensor.
For $H_h\equiv c$, put $\lambda_c=(n+1)c/2$.

In a unitary coframe set
$\eta=\sum_i(\sum_rT^r_{ri})\varphi^i$,
$s_h=\tr_h r_h^{(1)}$,
$\widehat s_h=\tr_h r_h^{(3)}$ and $\chi=s_h-\widehat s_h$.
The torsion and Ricci identities \cite[(2.5), (2.7)]{TangConstant} are
\begin{align}
 \partial\omega_h^{n-1}&=-\eta\wedge\omega_h^{n-1},
 \label{n:eq:eta-balanced-formula}\\
 \Ric_h^{(4)}-\Ric_h^{(1)}&=i\bar\partial\eta,
 \qquad\Ric_h^{(3)}-\Ric_h^{(1)}=-i\partial\bar\eta.
 \label{n:eq:ricci13-general}
\end{align}
Their trace and the double trace of \eqref{p:eq:HSC} give
\begin{align}
 \chi&=\tr_h(i\partial\bar\eta),\label{n:eq:chi-trace}\\
 s_h+\widehat s_h&=n(n+1)c=2n\lambda_c,\label{n:eq:two-scalars}\\
 s_h&=n\lambda_c+\tfrac12\chi.\label{n:eq:scalar-with-chi}
\end{align}
In particular $\widehat s_h$ and $\chi$ are real.
For $n\ge2$, put $\Theta_h=\omega_h^{n-1}/(n-1)!$;
differentiating \eqref{n:eq:eta-balanced-formula} and tracing yields
\begin{equation}\label{n:eq:ddbar-omega-nminus1}
 i\partial\bar\partial\Theta_h=(|\eta|_h^2-\chi)dV_h.
\end{equation}

\subsubsection{The Chern system for the torsion}\label{sec:foundations}
Unless stated abstractly, $V=T^{1,0}X$ carries the prescribed Chern
connection, torsion and metric; set $\mathcal X=\nabla'T$.
We use `prime' for $(1,0)$ directions and `barred' for $(0,1)$ directions.
Write $q_A:V\to V/A$, $q_A^2=\Sym^2q_A$,
$\operatorname{Ann}A=\{\xi:\xi|_A=0\}$,
$A_a=T(a,\cdot)$ and $D_a=\mathcal X_a(a,\cdot)$.
For an alternating product,
$\operatorname{rad}T=\{v:T(v,V)=0\}$ and
$(T^\vee\xi)(u,v)=\xi(T(u,v))$.
A smooth subbundle $N$ is \emph{preserved in prime directions along $D$}
when $\nabla'_p\Gamma(N)\subset\Gamma(N)$ for every $p\in\Gamma(D)$;
omitting $D$ means $D=V$.

For $H_h\equiv c\ne0$ the normalized curvature defect
$P:\Lambda^2V\to\Sym^2V$ is defined by
\begin{equation}\label{eq:P}\equationalias{hd-P-definition}\equationalias{b:eq:defect}\equationalias{ff:eq:defect}
 \frac12h^{a\bar j}h^{b\bar l}
 (R_{i\bar j k\bar l}+R_{i\bar l k\bar j})
 =\frac c2(\delta_i^a\delta_k^b+\delta_i^b\delta_k^a)+cP_{ik}^{ab}.
\end{equation}
Polarizing $H_h=c$ makes $P$ alternating in its inputs.
A \emph{Chern tensor system} consists of these tensors on an
open complex manifold, with $\mathcal X=\nabla'T$ and
$T(u,v)=\nabla_uv-\nabla_vu-[u,v]$.
An \emph{abstract tensor system with a connection} consists of a smooth
complex bundle $V\to M$ and a connection $\nabla'_p$ indexed by $p\in V$,
acting along a specified bundle map $V\to T_{\mathbb C}M$, with fields
$T:\Lambda^2V\to V$, $P:\Lambda^2V\to\Sym^2V$ and $\mathcal X=\nabla'T$.
There is no prescribed metric or curvature realization of $P$.
Each statement specifies the required \ChernFirst{}, \ChernCyclic{}, \ChernPure{},
flatness and commutator identities (for example \eqref{newff:eq:X-pure-Hessian}).
For an involutive tangent distribution the direction map is inclusion,
and its torsion is given by the bracket formula. A subbundle containing
tensor values need not inherit the connection or torsion.
A \emph{pointwise tensor system} has the same algebraic tensor types
on a complex vector space, with
$\mathcal X:V\otimes\Lambda^2V\to V$ independent of differentiation.
We impose only the stated algebraic equations. Induced derivatives act
on all slots and determinant-line weights, and all tensor spans are complex.
A named equation refers to the full formula unless we specify a restriction
or projection.

Put $uv=(u\odot v)/2$, $u^2=u\otimes u$ and $E(u,v)=uv+P(u,v)$.
Balancedness is $\eta=0$ in \eqref{n:eq:eta-balanced-formula}.
The transpose $P^\vee:\Sym^2V^*\to\Lambda^2V^*$ is evaluated by
$P^\vee(\xi,\zeta)(u,v)=P(u,v)(\xi,\zeta)$.
For alternating vector-valued $S$, set
\[
 \mathfrak b_p(S)(u,v)=\tfrac12\{S(u,p)\otimes v+
 S(u,v)\otimes p+S(p,v)\otimes u\}.
\]
With $\Sym(u\otimes v)=uv$ and $\operatorname{Alt}(u\otimes v)=u\wedge v$,
the three common maps are
\begin{align}
 a_p(S)&=\Sym\mathfrak b_p(S),\label{eq:a}\\
 b_p(S)&=\operatorname{Alt}\mathfrak b_p(S),\label{eq:b}\\
 F_p(S,P)&=a_p(S)+P b_p(S)=E\mathfrak b_p(S).\label{eq:F}
\end{align}

\begin{proposition}[The complete Chern system]\label{prop:source}\label{lem:general-curvature-transport}\label{ff:lem:barred}
Let $p,u,v,q\in V$, and define the Jacobiator of $T$ by
$J_T(p,u,v)=T(T(p,u),v)+T(T(u,v),p)+T(T(v,p),u)$.
The Chern tensors satisfy
\begin{align}
 &\nabla_pP=F_p(T,P),\tag{FIRST}\label{eq:transport}\\
 &\mathcal X_p(u,v)+\mathcal X_u(v,p)+\mathcal X_v(p,u)+J_T(p,u,v)=0,\tag{CYCLIC}\label{eq:cyclic}\\
 &F_q(\mathcal X_p,P)-F_p(\mathcal X_q,P)+F_p(T,P)b_q(T)-F_q(T,P)b_p(T)
                 +F_{T(p,q)}(T,P)=0,\tag{PURE}\label{eq:pure}
\end{align}
For $Y_{\bar p}=\nabla_{\bar p}T$ and a unitary frame $(e_a)$,
\begin{equation*}\tag{BARRED}\label{eq:barred}
 P(u,v)=-\frac1{4c}\sum_a e_a\odot Y_{\bar e_a}(u,v).
\end{equation*}
With $\eta(u)=-\operatorname{tr}T(u,\cdot)$, the trace equations are
\[
 \operatorname{tr}\mathcal X_p(u,\cdot)=-(\nabla_p\eta)(u),\qquad
 \sum_p \mathcal X_p(u,v)^p=(\partial\eta)(u,v).
\]
For a balanced metric both right sides vanish.
\end{proposition}
\ChernFirst{} transports $P$ in prime directions; \ChernCyclic{}
relates $\mathcal X$ to the torsion Jacobiator. \ChernPure{} is the
compatibility of two prime derivatives, with the torsion-direction
term retained. \ChernBarred{} recovers $P$ from barred torsion
derivatives using the prescribed metric.
\begin{proof}
Alternating the inputs in \eqref{eq:P} and using
\eqref{p:eq:firstBi} gives \ChernBarred{}.
For \ChernFirst{}, fix $\bar Y$ and set
$d_{pik}=(\nabla_pR)(e_i,\bar Y,e_k,\bar Y)$.
Differentiated sectional-curvature polarization gives
$d_{pik}=-d_{pki}$. The elementary relation
\[
2d_{pik}=(d_{pik}-d_{ipk})-(d_{ikp}-d_{kip})
                         +(d_{kpi}-d_{pki})
\]
then gives \ChernFirst{} by \eqref{p:eq:diffBi}, barred polarization
and \eqref{eq:P}. The pure torsion structure equation gives \ChernCyclic{}.
Differentiating \ChernFirst{} on every slot gives
$\nabla^2_{p,q}P=F_q(\mathcal X_p,P)+F_p(T,P)b_q(T)$.
Subtracting the two orders by \eqref{p:eq:purecomm} proves complete
\ChernPure{}, including its positive final term.

Differentiating the torsion trace gives the first trace formula.
Contracting \ChernCyclic{} in its first and output indices cancels the two
matrix-product traces and leaves
$(\nabla_u\eta)(v)-(\nabla_v\eta)(u)+\eta(T(u,v))
=(\partial\eta)(u,v)$, proving the second.
\end{proof}
For a balanced metric the contraction
\begin{equation*}
 \mathcal R_i{}^a=\sum_jP_{ij}^{aj},\qquad
 \mathcal R=\frac{(r_h^{(1)})^\#}{c}-\frac{n+1}{2}I
\end{equation*}
is traceless and self-adjoint. Thus for an orthogonal splitting
$V=\C e\perp H$ with $P(H,H)\subset\Sym^2H$ and
$P(e,h)=Q(h)+e(Mh)$, $Q:H\to\Sym^2H$, $M\in\End H$, one has
$\sum_j(\Id\otimes e_j^*)Q(e_j)=0$ in an orthonormal basis of $H$.
Contracting \eqref{eq:P} gives
$\mathcal R=((r_h^{(1)}+r_h^{(3)})^\#)/(2c)-(n+1)I/2$;
balance and \eqref{n:eq:ricci13-general} give
the displayed formula. Input alternation and output
symmetry give zero trace. Under the block hypotheses
$\pi_{\C e}\mathcal R(H)=0$; self-adjointness gives
$\pi_H\mathcal R(e)=0$, which is the asserted contraction of $Q$.
In a unitary point frame, write $Y_a=Y_{\bar e_a}$. Barred Bianchi
and Hermitian curvature reality give
\begin{equation}\label{newunb:r1v3:physical}
(Y_a)^b_{ik}=-R_{i\bar a k\bar b}+R_{k\bar a i\bar b},
\qquad
R_{i\bar a k\bar b}=\overline{R_{a\bar i b\bar k}}.
\end{equation}
For $c\ne0$ put $\widehat Y^a=c^{-1}Y_a$. Double alternation gives
\begin{equation}\label{newunb:r1v3:alternating-reality}
\mathcal A_{ik}^{ab}:=(\widehat Y^a)^b_{ik}-(\widehat Y^b)^a_{ik}
       =\overline{\mathcal A_{ab}^{ik}}.
\end{equation}
These identities hold for the prescribed Chern metric. Abstract systems require them as separate assumptions.

For the Chern tensors of the prescribed metric, or for an abstract
tensor system satisfying the full torsion-corrected pure-flat
commutator, suppose $\mathcal X=\nabla'T$ takes values in a smooth
subbundle $A$ and put $B_p=q_A\nabla'_p|_A$. The identity
\begin{equation}\label{newff:eq:X-pure-Hessian}
 (\nabla'_p\mathcal X)_q-(\nabla'_q\mathcal X)_p+\mathcal X_{T(p,q)}=0
\end{equation}
gives, on all lower inputs,
\begin{equation}\label{r48:eq:output-shape}
 B_p\mathcal X_q=B_q\mathcal X_p.
\end{equation}
Only differentiation of the output survives projection; all input and
torsion-direction corrections take values in $A$. No preservation of
$A$ is assumed.

Write complete \ChernPure{} as $\mathcal C_{pq}+P\mathcal D_{pq}=0$,
separating its $P$-independent and $P$-linear parts. Evaluating it with derivative directions $u,v$ and input pair $(u,v)$ gives
\begin{equation}\label{eq:repeated}\equationalias{rf:repeated}\equationalias{newunb:r1v3:repeated}
 0=\tfrac12(s_u\odot v-s_v\odot u)+\tfrac34t\odot t
                                      +P(s_u\wedge v+u\wedge s_v),
\end{equation}
where $t=T(u,v)$,
$s_u=\mathcal X_u(u,v)+T(u,t)/2$, and
$s_v=\mathcal X_v(u,v)+T(v,t)/2$.
Equivalently, extend $E(x,y)=xy+P(x,y)$ to $V\otimes V$. Then
\begin{equation}\label{newunb:r1v3:graph}
\begin{gathered}
 \mathscr R_u(v):=2s_u\otimes v-2s_v\otimes u+3t\otimes t\in\mathcal W,\\
 \mathcal W=\ker E=\{\operatorname{skw}z-2P(z):z\in\Lambda^2V\}.
\end{gathered}
\end{equation}
Here $\operatorname{skw}(u\wedge v)=u\otimes v-v\otimes u$.
Applying $E$ to $\mathscr R_u(v)$ gives twice the expression on the
right of the repeated identity. Since $E$ is the identity on
$\Sym^2V$, its kernel meets $\Sym^2V$ only at zero.
These are pointwise identities in every dimension.

\phantomsection\label{pure:rankone-contraction}
\emph{Rank-one contraction.}
On a subspace $D$ containing $u\ne0$, suppose $T|_{D^2}=u\otimes\psi$
and, for every pair of inputs $x,y\in D$, the two corrected
vectors of \eqref{eq:repeated} lie in the fixed line $\C u$.
Choose $\ell(u^2)=1$ and put $\xi(v)=\ell(E(u,v))$.
On $\ker(\xi|_D)$, applying $\ell E$ to \eqref{newunb:r1v3:graph}
leaves $3\psi(x,y)^2=0$. Since $\xi(u)=1$, this gives
$\psi=\xi|_D\wedge\psi(u,\cdot)$ on $D$.

Define the complete preimage
\[
 \mathcal K_{pq}=\mathfrak b_q(\mathcal X_p)-\mathfrak b_p(\mathcal X_q)
 +\mathfrak b_p(T)b_q(T)-\mathfrak b_q(T)b_p(T)
 +\mathfrak b_{T(p,q)}(T).
\]
Direct substitution gives
$\mathcal K_{uv}(u,v)=s_u\otimes v-s_v\otimes u+3t\otimes t/2$.
Thus $2\mathcal K_{uv}(u,v)=\mathscr R_u(v)$ proves the preceding
repeated identities without \ChernCyclic{}.
For the \ChernCyclic{} defect
$C_{pqu}=\mathcal X_p(q,u)+\mathcal X_q(u,p)+\mathcal X_u(p,q)+J_T(p,q,u)$,
expanding all five blocks gives
\[
 2\{\mathcal K_{pq}(u,v)-\mathcal K_{uv}(p,q)\}
 =-C_{pqu}\otimes v+C_{pqv}\otimes u+C_{puv}\otimes q-C_{quv}\otimes p.
\]
Twice the cyclic sum in $(q,u,v)$ is the same right side with
$+C_{puv}\otimes q$ replaced by $-C_{puv}\otimes q$.
Hence under \ChernCyclic{} the preimage is pair symmetric and satisfies Bianchi;
its repeated polarization is
\begin{equation*}
 4\mathcal K_{au}(a,v)=\mathscr R_a(u+v)-\mathscr R_a(u)-\mathscr R_a(v).
\end{equation*}
Complete \ChernPure{} is then equivalent to $E\mathscr R_a(b)=0$ for all $a,b$:
For $F(p,q,u,v)=E\mathcal K_{pq}(u,v)$ the displayed polarization gives
$F(a,u,a,v)=0$. Polarizing in $a$ gives alternation in slots 1 and 3.
Together with alternation in 1 and 2 and in 3 and 4, this makes $F$
a four-form. Bianchi then gives $3F=0$.

For any covector $\alpha$ and alternating tensor $T$, use the common
vectorial product and quadratic reference tensor
\begin{align}
 &B_\alpha(u,v)=\alpha(u)v-\alpha(v)u,\qquad \mathfrak H_p(u,v)=-T(p,T(u,v))-T(T(p,u),v)-T(u,T(p,v)).
 \label{unif:reference:H}
\end{align}
Expanding the three nested torsions gives, in every dimension,
\begin{equation}\label{unif:reference:elementary}
 \operatorname{cyc}\mathfrak H=-J_T,\qquad
 \mathfrak H_a(a,b)=-2T(a,T(a,b)).
\end{equation}
The definitions are tensorial. Later encodings and restrictions use
these same tensors, without introducing a new connection or metric.

\subsection{Regular loci, analyticity, and the Fujiki reduction}
\begin{lemma}[Regular loci and closed identities]\label{unif:lem:regular-loci}
For finitely many smooth finite-rank bundle maps, the common locus
where their ranks are locally constant is open and dense. A smooth
tensor's nonzero locus together with the interior of its zero set
is likewise dense. A continuous tensor identity that holds on a dense subset holds everywhere.
The same applies to inclusions between smooth subbundles.
\end{lemma}
\begin{proof}
On any nonempty open set where a map has positive maximal rank, a
nonzero minor at a point of maximal rank gives a smaller open set on
which the rank is constant. If the maximal rank is zero, the rank is
already constant. Repeat for the finitely many maps to obtain density.
The zero-set boundary has empty interior, and the extension assertions
follow from continuity in local frames.
\end{proof}
We differentiate images and kernels only on constant-rank open charts.
An identity can extend by continuity even when a frame or a
normalization by a vanishing scalar cannot.

\subsubsection{Real analyticity of constant-curvature metrics}
\begin{proposition}\label{univ:analyticity}\label{lemma:analyticity}\label{c0:r1:analytic}
A smooth Hermitian metric with constant Chern holomorphic sectional
curvature, in any complex dimension and for any real constant $c$,
is real analytic in holomorphic coordinates.
\end{proposition}
\begin{proof}\label{app:analyticity-calculation}
In holomorphic coordinates write $h=\sum_Au_AH_A$ in a real
orthonormal basis of Hermitian matrices and put $Q_A(v)=H_A(v,v)$.
Integration over the Euclidean unit sphere with normalized measure gives $n^2$ real equations
\begin{equation}\label{eq:global-operator}
 \int Q_A(v)\{R_h(v,\bar v,v,\bar v)-c\,h(v,v)^2\}\,d\mu(v)=0.
\end{equation}
The coordinate curvature formula gives a constant principal part
and lower terms analytic in $(h,Dh)$ on the positive cone.
For a nonzero real covector $\xi$, put
$\zeta_i=(\xi_{x^i}-i\xi_{y^i})/2$. For $Q=\sum_Aa_AH_A\ne0$,
its symbol satisfies
\begin{equation}\label{hd-elliptic-square}
 a^{\mathsf T}\Sigma(\xi)a
 =\int\Bigl|\sum_i\zeta_iv^i\Bigr|^2Q(v,v)^2\,d\mu(v)>0.
\end{equation}
If the integral vanished, $Q(v,v)$ would vanish wherever
$\sum_i\zeta_i v^i\ne0$, hence everywhere on the sphere by continuity.
This would force $Q=0$. Compactness for unit $\xi$ and unit $a$
gives uniform strong ellipticity.
Since each $u_A$ is smooth, it is locally of class $C^{2,\mu}$ for
$0<\mu<1$. To apply Morrey's analyticity theorem
\cite[\S6.7, Theorem~6.7.6]{Morrey} with $N=n^2$, $s_j=0$ and $t_k=2$,
subtract the quadratic Taylor polynomial of each $u_A$ at the point
under consideration. The theorem gives analyticity of the normalized
coefficients and hence of $h$.
\end{proof}

\subsubsection{A K\"ahler criterion for Fujiki manifolds}
\begin{lemma}\label{n:lem:no-rational-kahler}
A compact complex manifold in Fujiki's class $\mathcal C$ with no
rational curves is K\"ahler.
\end{lemma}
\begin{proof}
A smooth complex manifold is strongly $\mathbb Q$-factorial, since
rank-one reflexive sheaves are line bundles. The trivial generalized
pair $(X,0)$ is generalized klt. Hacon--Li--Xie
\cite[Corollary~1.4]{HaconLiXie} show that such a class-$\mathcal C$
space is K\"ahler unless its curve-current cone contains a class
$-[C]$ with $C$ rational. The hypothesis excludes this alternative.
\end{proof}

\part{Canonical positivity and Fujiki-class rigidity}\label{part:ampleness}
We use the same trace and determinant estimates to construct bounded
canonical potentials for smooth forms and for currents. Smooth
ampleness enters both the proof under $\mathrm{(C_2)}$ and the negative
Fujiki case, but only the proof under $\mathrm{(C_2)}$ requires
localization to the polar set.

\section{Smooth canonical positivity}\label{amp:sec:ampleness}
A Poisson envelope on the punctured determinant bundle gives
nefness. For ampleness, we give a uniform lower volume bound
along the Hermitian continuity family.

\begin{proposition}[The smooth plurinegative case]\label{amp:prop:smooth-ampleness}
Let $X$ be a compact connected Brody-hyperbolic complex manifold of
positive dimension. If $X$ admits a smooth plurinegative Hermitian
metric, then $X$ is projective and $K_X$ is ample.
\end{proposition}

In complex dimension one, hyperbolicity excludes genus zero and one
by uniformization, so the canonical bundle has positive degree.
Hence assume $n\ge2$ below.
For the smooth ampleness proof, fix a smooth plurinegative metric $k$.
A metric and its fundamental $(1,1)$-form use the same letter.

We follow \cite[Sections~3--6]{amp:new:OpenAI} in maximizing a
fixed-center holomorphic-disc functional with an area penalty and
using a Poisson envelope on the punctured determinant bundle.
We retain the torsion and
$\ddc k$ terms in the area Hessian, justify their weak boundary
limits and remove the penalty for each fixed Hermitian area metric;
Appendix~\ref{amp:calc:area} gives the calculation.

\subsection{Disc compactness and boundary regularity}\label{amp:app:sec:preliminaries}\label{amp:app:sec:extremal-discs}

\begin{lemma}\label{amp:app:pre:hyperbolic}
If $H_k\le-\kappa<0$, every entire holomorphic map $\C\to X$ is constant.
\end{lemma}
\begin{proof}
For a holomorphic $f:\C\to X$ put $s=f'$ and $e=|s|_k^2$.
The Chern norm identity on $f^*T^{1,0}X$ gives
$e_{z\bar z}=|D_zs|_k^2-R^k(s,\bar s,s,\bar s)$.
Cauchy--Schwarz gives $(\log e)_{z\bar z}\ge-H_k(s)e$ on $e>0$,
the Hermitian disc comparison of Kobayashi
\cite[Theorem~9.2, p.~478]{Kobayashi1967Distances}.
Compare with $v_R=2R^2/\{\kappa(R^2-|z|^2)^2\}$ on $|z|<R$.
Since $f$ is entire, $e$ is bounded on the closed disc of radius
$R$, whereas $v_R$ tends to infinity at its boundary. Thus $e/v_R$
tends to zero there. At any positive maximum of the ratio, the
inequality for $\log e$ and the identity
$(\log v_R)_{z\bar z}=\kappa v_R$ give $e/v_R\le1$ there, and
therefore throughout the disc. For fixed $z$, $v_R(z)\to0$ as
$R\to\infty$, so $e(z)=0$. Thus $f$ is constant.
\end{proof}

Use the coefficient normalization \eqref{conv:eq:disc-conversion}.
On compact $X$, finite area is metric independent and makes $P_q(f)$
finite for smooth $q$. The disc estimates require no curvature assumption on $k$.
\begin{lemma}[Green identities]\label{amp:app:pre:green}
For smooth $u$ on $\ampDisc$ and $r<1$,
\begin{align}
 \langle u\rangle_r-u(0)
 &=\frac2\pi\int_{|z|<r}\log\frac r{|z|}
       \partial_z\partial_{\bar z}u\,\mathrm d S,\label{amp:app:eq:green-poisson}\\
 \frac1\pi\int_{|z|<r}\partial_z\partial_{\bar z}u\,\mathrm d S
 &=\frac r2\langle\partial_r u\rangle_r.\label{amp:app:eq:green-area}
\end{align}
If $u\in C^0(\overline\ampDisc)$ and
$\partial_z\partial_{\bar z}u\in L^a(\ampDisc)$ for some $a>1$,
the first identity passes to $r=1$.
\end{lemma}
The divergence theorem gives the second identity. Integrating it
against $2\,dr/r$ gives the first, since
$\langle u\rangle_r\to u(0)$ as $r\downarrow0$.
Continuity and H\"older's inequality give the boundary limit.

\begin{lemma}[Derivative bound and normality]
\label{amp:app:ex:brody}
Let $X$ be a compact Brody-hyperbolic complex manifold and $k$ a
smooth Hermitian metric on $X$. There is $D>0$, depending only on $(X,k)$, such that every holomorphic
disc satisfies $d_k(f)\le D$ and
\begin{equation}\label{amp:app:eq:brody-derivative}
 |f'(z)|_k\le\frac{\sqrt D}{1-|z|}.
\end{equation}
The family of discs is normal, with convergence of all derivatives
on smaller discs in local coordinates on the target.
\end{lemma}
\begin{proof}
It suffices to bound the derivative at the center of every disc.
If this bound failed, Brody's rescaling \cite{amp:old:Brody1978}
would give a nonconstant entire curve. Indeed, for a sequence with
$|f_j'(0)|_k\to\infty$, maximize
$(1/2-|z|)|f_j'(z)|_k$ on $|z|\le1/2$, at $z_j$, and put
$\lambda_j=|f_j'(z_j)|_k$ and $R_j=(1/2-|z_j|)\lambda_j\to\infty$.
Then $g_j(w)=f_j(z_j+w/\lambda_j)$ has
$|g_j'(0)|_k=1$ and $|g_j'|_k\le2$ on $|w|<R_j/2$.
Compactness and the local Cauchy formula yield a nonconstant entire
limit, a contradiction. Applying the resulting central bound to
$f(z+(1-|z|)w)$ gives the stated estimate; the same compactness
argument gives normality and derivative convergence.
\end{proof}

\begin{lemma}[Boundary regularity from an area tail]\label{amp:lem:area-tail-regularity}
Let $f:\mathbb D\to Y$ be holomorphic, where $(Y,k)$ is compact
Hermitian. If $A_k(f)-A_{k,r}(f)=O_f(1-r)$, then $f$ extends
continuously to the closed disc and belongs to $W^{1,p}$ for every
$p\in(2,4)$. There are $C_f>0$ and $r_f\in(0,1)$ such that
$|f'(z)|_k^2\le C_f/(1-|z|)$ for $r_f<|z|<1$.
\end{lemma}
\begin{proof}
The norm identity in the proof of Lemma~\ref{amp:app:pre:hyperbolic} gives
$\Delta e\ge-Ce^2$ for $e=|f'|_k^2$, by compactness of $Y$.
For $R=(1-|z|)/2$, the tail gives $\int_{B(z,R)}e=O_f(R)$,
below the small-energy threshold of \cite[Proposition~4.6]{Zinger2017} near the boundary.
The resulting bound $e(z)\le 8(\pi R^2)^{-1}\int_{B(z,R)}e$
is $O_f((1-|z|)^{-1})$.
The bound on $|f'(z)|_k$ is integrable along radii, uniformly in the
angle. The radial maps therefore converge uniformly to a continuous
boundary map.
On a dyadic annulus of width $s$,
$\int e^{p/2}\le(\sup e)^{(p-2)/2}\int e\le C_{f,p}s^{2-p/2}$;
summing proves the stated Sobolev regularity in target charts.
\end{proof}

Fix a center $x$, a smooth Hermitian area metric $\beta$, a smooth
closed real $(1,1)$-form $q$, and $\delta>0$. The derivative reward
continues to use the fixed metric $k$. Set
\begin{equation}\label{amp:app:eq:disc-functional}
 J_{\beta,q,\delta}(f)=d_k(f)+P_q(f)-\delta A_\beta(f),
 \qquad f(0)=x.
\end{equation}

\begin{proposition}[Maximizing discs]
\label{amp:app:ex:maximizer}\label{amp:app:ex:dilation}\label{amp:app:ex:regularity}
Suppose $\beta$ is plurinegative and $q=\Ric_\beta^{(1)}$.
The functional~\eqref{amp:app:eq:disc-functional} attains a finite maximum
on finite-area discs. Every maximizer satisfies
$A_k(f)-A_{k,r}(f)=O_f(1-r)$ and hence has the boundary regularity
of Lemma~\ref{amp:lem:area-tail-regularity}.
\end{proposition}
\begin{proof}
This is the case without divisors in
Lemma~\ref{cur:lem:singular-maximizers}. Take $\pi=\Id$, $D=u=0$,
$\mathscr T=\kappa=\beta/(2\pi)$ and $h=k/(2\pi)$. In disc coefficients
$q_D$ is $q$ and the reward and area are $d_k$ and $A_\beta$.
The area tail given by the general lemma transfers to $k$ by metric
comparison.
\end{proof}

\subsection{The Hermitian area Hessian}\label{amp:app:sec:hessian}
Let $\beta$ be a smooth Hermitian area metric and keep $k$ for
the reward. Use a disc $f$ with a frame $B=(b_1,\ldots,b_n)$ that
is $\beta$-unitary on the boundary. Take centered variations
$f_t=F(t,\cdot)$ with vectors $zb_j$ and $C^2$ parameter dependence
in $W^{1,p}$ for one common exponent $2<p<\infty$. Lemmas~\ref{amp:app:var:frame} and~\ref{amp:app:var:spray}
(with $a=1$) supply these data.
Put $x=f(0)$, $N=B^*\beta(f)B$, $M=B(0)^*k(x)B(0)$, and
$\ell_\beta=\log(\beta^n/k^n)$. Thus $N=I$ on the boundary.
Write $\mathscr L_{\!\mathrm{var}}=\left.\sum_j\partial_{t_j}
\partial_{\overline{t_j}}\right|_{t=0}$ and
$\Lambda_\beta=i(\partial\overline{\eta_\beta}-\bar\partial\eta_\beta)$,
where $\eta_\beta(w)=\tr(v\mapsto T^\beta(v,w))$.
The real form $\Lambda_\beta$ has bounded $\beta$-norm.

For any smooth real $(1,1)$ form $\omega_\beta$ with coefficient
matrix $\beta$, use the ordered evaluation
\begin{equation}\label{amp:eq:area-C}
 \mathcal C_\beta(V,w,\bar V,\bar w)
 =(\ddc\omega_\beta)(V,w,\bar V,\bar w).
\end{equation}
This is the Hermitian value on $V\wedge w$; no positivity of
$\beta$ is required. Put
\[
 \mathcal B_\beta(f,B)=\frac1\pi\int_\ampDisc\sum_j
 \mathcal C_\beta(zb_j,f',\overline{zb_j},\overline{f'})\,\mathrm d S.
\]

\begin{proposition}[Boundary determinant and Hessian identities]
\label{amp:app:hes:determinant}\label{amp:app:hes:trace}
\label{amp:app:hes:hermitian}\label{amp:app:hes:maximum}
For the preceding boundary frame and centered variations,
\begin{equation}\label{amp:app:eq:determinant-identity}
 \log\det N(0)=P_{\Ric_\beta^{(1)}}(f),\qquad
 \log\det M=P_{\Ric_\beta^{(1)}}(f)-\ell_\beta(x).
\end{equation}
For every smooth closed real $(1,1)$-form $q$ and $\delta>0$,
\[
 \mathscr L_{\!\mathrm{var}}J_{\beta,q,\delta}(f_t)
 =\tr M+\langle(\tr_\beta q)\circ f\rangle_1
 -\delta\{n-A_{\Ric_\beta^{(1)}}(f)+A_{\Lambda_\beta}(f)+\mathcal B_\beta(f,B)\}.
\]
If $\ddc\beta\le0$ weakly on decomposable complex bivectors, then
$\mathcal B_\beta\le0$. No boundary derivative traces are required.
\end{proposition}
\begin{proof}
Holomorphicity gives
$(\log\det N)_{z\bar z}=-r_\beta^{(1)}(f',\overline{f'})$.
At the center, $\det N(0)=e^{\ell_\beta(x)}\det M$; Green therefore gives
\eqref{amp:app:eq:determinant-identity}. The center-derivative
variations $b_j(0)$ give the reward Hessian.
Apply \eqref{amp:app:eq:area-split} and
Proposition~\ref{an:prop:weighted-boundary} with $b=1,C=I$.
Since $q$ is closed, the Poisson term has no torsion or bulk
contribution. The area fluxes for $\beta$ are
$n-A_{\Ric_\beta^{(1)}}$ and $A_{\Lambda_\beta}$; its
$i\partial\bar\partial\omega_\beta$ term remains.
The common-radius limit and parameter derivatives hold in $L^{p/2}$
against $1,G_0$. The determinant identity and the center and
Poisson Hessian terms are those of the unweighted K\"ahler-area
calculation in \cite[Lemma~5.1 and Proposition~5.2, (5.2)--(5.3)]{amp:new:OpenAI},
which also permits different reward and area metrics.
When $\beta$ is K\"ahler the two Hermitian correction terms vanish,
giving the area formula (5.4) there.
\end{proof}

\subsection{Trace and determinant estimates}\label{amp:app:sec:nefness}
\begin{lemma}[Trace--determinant canonical construction]
\label{can:lem:trace-envelope}
Let $X$ be a complex manifold of dimension $n\ge1$.
Fix a smooth Hermitian metric $h_0$, an open $U\subset X$ and
$W\in C^\infty(U,\mathbb R)$. Let $\mathcal S_x$ be a class of discs
centered at $x$, containing constants and every disc holomorphic
past the closed disc with image in $U$.
Let finite real functionals $d,P,A$ vanish on constants and satisfy
$0\le d\le D$, $A\ge0$ on each $\mathcal S_x$. On the extendible discs
just described assume $P=P^{\mathrm{cur}}_{\rho_{h_0}-\ddcn W}$.
Assume $J_\delta=d+P-\delta A$ attains a finite maximum
$M_\delta(x)$ for every $x\in U$ and $\delta>0$.
For $0<\delta\le1$, require each maximizer to carry $M>0$ with
\[
 \log\det M=P-W(x),\qquad
 \tr M\le C_0+r_\delta+C_1\delta A,
 \qquad C_0>0,\quad C_1\ge0,
\]
where $M$ is $n\times n$, the constants are uniform in $x$, the disc
and $\delta$, and $r_\delta\ge0$ is independent of $x$ and the disc
with $r_\delta\to0$ as $\delta\downarrow0$.
Then $M_0(x):=\sup_{\mathcal S_x}(d+P)$ is finite,
$M_\delta(x)\uparrow M_0(x)$, every choice of maximizers satisfies
$\delta A\to0$, and
\[
 0\le M_0(x)\le W(x)+D+n\log(C_0/n).
\]
There are a global constant $C\ge0$ and a function $\varphi$ on $U$
such that
\[
 P(g)\le W(x)+C\quad (x\in U,\ g\in\mathcal S_x),\qquad
 -C\le\varphi\le W,\qquad -\rho_{h_0}+\ddcn\varphi\ge0.
\]
The latter current is the curvature of a locally bounded metric on $K_X|_U$.
\end{lemma}
\begin{proof}
At a maximizer, comparison with the constant disc and trace--determinant
AM--GM give
\[
 0\le\delta A\le D+P,\qquad
 P\le W(x)+n\log((C_0+r_\delta+C_1\delta A)/n).
\]
Replace $\delta A$ in the second inequality by $D+P$. Since
$r_\delta$ is bounded for small $\delta$, the right side grows at
most logarithmically in $P$, uniformly as $\delta\downarrow0$.
Thus at fixed $x$ the maximizing values of $P$ are bounded above,
and taking suprema in $J_\delta\uparrow d+P$ gives $M_\delta\uparrow M_0<\infty$. For every maximizer,
$\delta A\le M_0-M_\delta\to0$; letting $\delta\downarrow0$ in
$M_\delta\le D+P$ and the displayed bound gives the asserted bound on $M_0$.
Thus $P(g)\le W(g(0))+C$ for $C=\max\{0,D+n\log(C_0/n)\}$.

To construct the metric, work on $L^\times=((\det TX)|_U)^\times$,
with projection $p_L$, and set
$\ell=\log\|\cdot\|_{\det h_0}^2$.
Let $\mathcal E$ be the Poisson envelope of $\ell+W\circ p_L$ over
analytic discs extending past the unit circle in $L^\times$. Green gives
$\langle\ell+W\circ p_L\rangle=\ell(0)+W(g(0))-P(g)$.
The bounds $\ell-C\le\mathcal E\le\ell+W\circ p_L$ ensure the envelope
is finite. It is plurisubharmonic by the Poletsky--Rosay theorem in
the form of \cite[Theorem~1.1]{amp:old:DrinovecForstneric2012}.
For a continuous disc on the closed unit disc, the boundary integrals
of the smaller discs converge to its boundary integral as the radius
tends to one. Fiber scaling gives
$\mathcal E=\ell+\varphi\circ p_L$.
For determinant frames $s_i=g_{ij}s_j$, the psh weights
$a_i+\varphi$, $a_i=\log|s_i|_{h_0}^2$, differ by
$\log|g_{ij}|^2$. They define the stated dual metric with curvature
$-\rho_{h_0}+\ddcn\varphi$ and $-C\le\varphi\le W$.
\end{proof}

\subsubsection{Nefness and the smooth volume estimate}
\label{amp:app:app:smoothing}
\begin{lemma}[Zero Lelong numbers and nefness]
\label{amp:lem:zero-lelong-nef}
Let $M$ be a compact complex manifold, let $\omega$ be a smooth Hermitian form,
and let $\alpha$ be a smooth closed real $(1,1)$-form. If
$\mathscr T=\alpha+\ddc u\ge0$ has bounded potential $u$, then, for each
$\eps>0$, there is $v_\eps\in C^\infty(M,\mathbb R)$ such that
$\alpha+\ddc v_\eps\ge-\eps\omega$. Thus $[\alpha]$ is analytically nef.
\end{lemma}
\begin{proof}
Bounded local potentials give $E_c(\mathscr T)=\varnothing$ for
$c>0$. The restriction criterion in Section~\ref{sec:conventions}
(\cite[Proposition~3.3(iv)]{cur:DP}) therefore gives nefness, whose
definition is the asserted smooth approximation; see also
\cite[Theorem~1.1]{amp:new:Demailly1994} for regularization.
\end{proof}

\begin{corollary}[Smooth canonical potentials]
\label{amp:app:nef:bounded-potentials}\label{amp:prop:entrance}
\label{amp:app:cor:bounded-canonical}\label{amp:app:nef:disc-bound}
For some $G<\infty$ depending only on $(X,k)$, every finite-area disc
satisfies $P_{\Ric_k^{(1)}}(f)\le G$ and there is a global potential
$-G\le u\le0$ with $-\Ric_k^{(1)}+\ddc u\ge0$. This gives a bounded
canonical metric and $K_X$ is nef.
\end{corollary}

\begin{proposition}[Uniform lower volume]\label{amp:prop:uniform-volume}
Every smooth plurinegative Hermitian metric $\beta$ with
$\Ric_\beta^{(1)}\ge-\beta$ satisfies
\begin{equation}\label{amp:eq:uniform-volume}
 \beta^n\ge e^{-D}k^n,
\end{equation}
where $D$ is the reward bound of Lemma~\ref{amp:app:ex:brody}.
\end{proposition}

\begin{proof}[Proof of Corollary~\ref{amp:app:nef:bounded-potentials} and
Proposition~\ref{amp:prop:uniform-volume}]
Fix a smooth plurinegative metric $\beta$. Use
Lemma~\ref{can:lem:trace-envelope} with
\[
 (U,h_0,W,d,P,A)=(X,k,\ell_\beta,d_k,P_{\Ric_\beta^{(1)}},A_\beta).
\]
Lemma~\ref{amp:app:ex:brody} and Proposition~\ref{amp:app:ex:maximizer}
supply the disc hypotheses. The curvature conversion is
$\rho_k-\ddcn\ell_\beta=\rho_\beta$ in the normalization of
\eqref{conv:eq:disc-conversion}. At a maximum, \eqref{amp:app:eq:determinant-identity}
and Proposition~\ref{amp:app:hes:hermitian} give
$\log\det M=P-\ell_\beta$ and $\tr M\le C_0+n\delta+C_\beta\delta A$:
plurinegativity gives $\mathcal B_\beta\le0$, and the smooth
first Chern--Ricci form and torsion terms have bounded $\beta$-norm.
Taking $\beta=k$, with a sufficiently large $C_0$, gives the disc
bound and bounded potential since $\ell_k=0$. Multiply the current
inequality by $2\pi$ and apply Lemma~\ref{amp:lem:zero-lelong-nef}.
For the volume estimate, $\Ric_\beta^{(1)}\ge-\beta$ gives
$-\tr_\beta\Ric_\beta^{(1)}\le n$ and $-A_{\Ric_\beta^{(1)}}\le A_\beta$,
so one can use $C_0=n$ and $C_\beta=1+\sup|\Lambda_\beta|_\beta$.
Then $0\le M_0\le\ell_\beta+D$, the desired volume bound.
The metric $\beta$ is fixed when $\delta\downarrow0$.
\end{proof}

\subsection{The Monge--Amp\`ere family and canonical ampleness}
The volume argument follows the Monge--Amp\`ere approach of
\cite[Sections~7--8]{amp:new:OpenAI}. Here the limiting volume is
obtained using the absolutely continuous part of a positive current
on a possibly non-K\"ahler manifold, so the bigness step uses
Popovici's criterion.
Put $\alpha=-\Ric_k^{(1)}$, representing $2\pi c_1(K_X)$.
Corollary~\ref{amp:prop:entrance} and the definition of nefness give
smooth $a_\eps$ with
$\omega_\eps=\alpha+\eps k+\ddc a_\eps\ge\eps k/2>0$.
The negative-sign Hermitian Monge--Amp\`ere theorem
\cite{amp:new:Cherrier1987}\cite[Theorem~3.3]{amp:new:SW}, with
$\lambda=1$ and $F=a_\eps+\log(k^n/\omega_\eps^n)$, gives
$u_\eps=a_\eps+\phi_\eps$ satisfying
\begin{equation}\label{amp:eq:matched-cma}
 \beta_\eps=\alpha+\eps k+\ddc u_\eps>0,\qquad
 \beta_\eps^n=e^{u_\eps}k^n.
\end{equation}
The factor $e^{u_\eps}$ fixes the additive constant without a separate
volume normalization.
Since $\Ric_{\beta_\eps}^{(1)}=-\beta_\eps+\eps k\ge-\beta_\eps$ and
$\ddc\beta_\eps=\eps\ddc k\le0$, Proposition~\ref{amp:prop:uniform-volume},
with the penalty limit taken for each fixed $\eps$, gives
\begin{equation}\label{amp:eq:cma-lower}
 \beta_\eps^n\ge e^{-D}k^n.
\end{equation}

\subsubsection{A positive determinant and bigness}
Put $b=e^{-D}>0$. At a maximum of $u_\eps$,
$\beta_\eps\le\alpha+\eps k\le Ck$ with fixed $C\ge1$ for $\eps\le1$.
Thus $\log b\le u_\eps\le n\log C$.
The uniform bound and separability of $L^1(X)$ give, along a sequence $\eps\downarrow0$,
$u_\eps\rightharpoonup^*u$ in $L^\infty$ and hence
$\beta_\eps\rightharpoonup\Theta_\infty=\alpha+\ddc u\ge0$.
For positive semidefinite matrices,
\[
 \inf_{H>0,\,\det H=1}\tr(HA)=n\det(A)^{1/n}.
\]
AM--GM gives one inequality; testing with
$H=\det(A+tI)^{1/n}(A+tI)^{-1}$ as $t\downarrow0$ gives the other,
including singular $A$.
For the coefficient matrices $B_\eps,K$ of $\beta_\eps,k$, every
constant $H>0$ with $\det H=1$ gives $\tr(HB_\eps)\ge n(b\det K)^{1/n}$.
This linear inequality passes to the weak measure limit and to its
absolutely continuous density. First impose the inequalities for a countable dense set of $H$;
they then hold on one common set of full measure. By continuity in
$H$, we may take the infimum there: $(\Theta_{\infty,\mathrm{ac}})^n\ge bk^n$.
Since $\Theta_\infty/(2\pi)$ represents $c_1(K_X)$,
the singular-metric criterion \cite[Theorem~1.3]{amp:new:Popovici2008},
valid on any compact complex manifold, gives
\[
 \operatorname{vol}(K_X)\ge\frac b{(2\pi)^n}\int_Xk^n>0.
\]

\subsubsection{From bigness to ampleness}
\begin{proof}[Completion of Proposition~\ref{amp:prop:smooth-ampleness}]
Thus $K_X$ is big and $X$ is Moishezon, hence in class $\mathcal C$.
A rational curve would violate hyperbolicity on restriction to $\C$.
Lemma~\ref{n:lem:no-rational-kahler} and \cite{amp:old:Moishezon1967}
give a K\"ahler metric and projectivity. General type and the absence
of rational curves give ampleness by \cite[Lemma~2.1]{amp:new:DT}.
The ample volume identity now yields
\[
 \int_X\alpha^n=(2\pi)^n\operatorname{vol}(K_X)\ge b\int_Xk^n>0.
\]
\end{proof}

\begin{proof}[Proof of Corollary~\ref{amp:cor:low-dimensional-ampleness}]
In dimension one every Hermitian form is pluriclosed; in dimension
two a Gauduchon metric is pluriclosed
\cite[Theorem~2]{GauduchonStandard1985}. Apply
Proposition~\ref{amp:prop:smooth-ampleness}.
Strictly negative Chern HSC is uniformly negative by compactness,
so Lemma~\ref{amp:app:pre:hyperbolic} gives hyperbolicity.
\end{proof}

\section{Currents and bounded canonical metrics}
We extend the fixed-center disc maximization with an area penalty and
the determinant-bundle Poisson envelope of
\cite[Sections~3--6]{amp:new:OpenAI} to currents satisfying $\mathrm{(C)}$.
On the resolution, the current splits into a smooth form and
divisorial terms. We estimate the area variation on the lifted collar
and account for interior masses by the distributional Green formula.
Lemma~\ref{can:lem:trace-envelope} then gives the envelope.
\subsection{Regularization and the resolved current}\label{cur:sec:complete-bridge}
Under $\mathrm{(C)}$ we prove nefness by constructing bounded canonical
potentials. The ampleness argument uses $\mathrm{(C_2)}$; bounded
potentials alone do not give bigness.

Use the potential $\mathscr T=\omega_0+\ddcn u$ from
Section~\ref{sec:conventions} and the analytic upper level sets
$E_c(\mathscr T)$, $c>0$, defined there.

\begin{proposition}[Reduction to one fixed logarithmic current]
\label{cur:prop:intrinsic-regularization}
Suppose $\mathscr T$ satisfies \eqref{cur:eq:C}, with auxiliary Hermitian form $\omega_h$, and, for some integer $r\geq0$,
$\dim_{\mathbb C}E_c(\mathscr T)\leq r$ for every $c>0$.
There is a replacement current
$\widehat{\mathscr T}=\omega_0+\ddcn \widehat u$ such that
\begin{equation}\label{cur:eq:regularized-C}
 \widehat{\mathscr T}\geq\tfrac{\varepsilon}{2}\omega_h,\qquad
 \ddcn \widehat{\mathscr T}=\ddcn \omega_0=\ddcn \mathscr T\leq0.
\end{equation}
The analytic polar set $\widehat Z=\Pole(\widehat u)$ has dimension at most $r$;
$\widehat u$ is smooth off $\widehat Z$ and admits a projective log
resolution $\pi:Y\to X$, isomorphic over $X\setminus\widehat Z$,
with finite simple normal crossings (SNC) model
\begin{equation}\label{cur:eq:resolved-log-model}
 \widehat u\circ\pi=\sum_j a_j\log|s_{E_j}|^2+v,
 \qquad a_j>0,\qquad v\in C^\infty(Y,\mathbb R).
\end{equation}
The norms use smooth divisor metrics; an empty sum is allowed.
Conversely, such a strictly positive logarithmic current with smooth
background, $\ddcn\omega_0\le0$ and polar dimension at most $r$
satisfies \eqref{cur:eq:C} and the Lelong-level dimension bound.
\end{proposition}
\begin{proof}
Apply \cite[Theorem~3.2]{cur:DP} to
$\ddcn u\ge\varepsilon\omega_h-\omega_0$, without closedness of the
background. The approximants have smooth resolved logarithmic
remainders, lower bound $(\varepsilon-\delta_k)\omega_h$,
$\delta_k\to0$, and poles $Z_k\subset E_{c_k}(\mathscr T)$, $c_k>0$
\cite[proof of Proposition~3.3(iii)]{cur:DP}.
Choose $\delta_k<\varepsilon/2$; $\ddcn\ddcn u_k=0$ preserves
plurinegativity. Conversely $d\mathscr T=d\omega_0$ is smooth and positive
Lelong levels lie in the polar set.
\end{proof}

\subsubsection{Divisorial reduction and a strictly positive resolution}
\label{cur:can:resolution}

\begin{lemma}[A common resolved model]
\label{cur:lem:resolved-model}
Let $\mathscr T$ be a strictly positive plurinegative current on a compact
complex manifold $X$, with a smooth resolved logarithmic model. It admits a replacement
with analytic polar set $Z$ of codimension at least two, smooth potential
off $Z$, and the same plurinegativity. If the initial resolution is
isomorphic off an analytic set $\widehat Z$, then $Z\subset\widehat Z$.
The replacement admits a projective log resolution $\pi:Y\to X$,
isomorphic off $Z$, and an effective real exceptional divisor $A$ with
full exceptional support such that
\[
 \kappa=\pi^*\omega_0+
 \operatorname{curv}\mathcal O_Y(-A)>0,
 \qquad -A\text{ is }\pi\text{-ample},
 \qquad \ddcn\kappa=\pi^*\ddcn\omega_0.
\]
Here $\omega_0$ is the smooth background after removing the
divisorial part, and the curvature is taken with a smooth divisor metric. The replacement is $\pi_*\kappa$
and has a global quasi-psh potential with smooth resolved remainder.
There is also a full-support effective integral exceptional $D$
with $\mathcal O_Y(-D)$ relatively very ample and generated and
$\pi=\operatorname{Bl}_{\mathcal I}X$ for
$\mathcal I=\pi_*\mathcal O_Y(-D)$.
\end{lemma}
\begin{proof}
Write $\mathscr T\ge\varepsilon\omega_h$ and take
an SNC resolution by projective smooth blowups
containing all exceptional primes, isomorphic off $\widehat Z$ when specified.
On the base, remove divisorial masses by replacing $u$ with
$u-\sum a_j\log|s_{D_j}|^2$ and $\omega_0$ with
$\omega_0-\sum a_j\Theta_j$ in smooth divisor metrics.
Apply Siu's decomposition formula \cite[Chapter~III, (8.16)]{amp:new:DemaillyCADG}
to the positive closed current obtained locally by smoothly correcting
$\mathscr T-\varepsilon\omega_h$. The smooth correction carries no mass
on the divisors, so removing the divisorial measures preserves positivity
there; off the divisors the lower bound is unchanged. The residual has
no divisorial mass and its local potentials glue.
The change of background is closed, so $\ddcn\omega_0$ is unchanged.
Cancel the strict-transform logarithms:
\begin{equation}\label{cur:eq:pre-repair-resolution}
 \pi^*\mathscr T=\omega_{\rm res}+\sum_i e_i[E_i],\qquad
 \omega_{\rm res}\ge\varepsilon\pi^*\omega_h,
 \qquad\ddcn\omega_{\rm res}=\pi^*\ddcn\omega_0\le0,
\end{equation}
where $e_i\ge0$ and continuity extends the smooth bound.
The exceptional image $Z$ has codimension at least two and lies in
$\widehat Z$ when specified. The residual is smooth off $Z$.

Successive blowups give an effective integral exceptional divisor of
full support whose negative is relatively ample. Choose a multiple $D$
with $\mathcal O_Y(-D)$ relatively very ample and generated; evaluation
embeds $Y$ as the ideal-graph
closure in $\mathbb P_X(\mathcal I)$, so $Y=\operatorname{Bl}_{\mathcal I}X$.
Patch relative Fubini--Study weights by a base partition. Their curvature $\theta_D$ is positive on
$\ker d\pi$; compactness gives $\theta_D+C\pi^*\omega_h>0$.
For $0<t<\varepsilon/C$,
\begin{equation}\label{cur:can:strict-resolution}
 \kappa=\omega_{\rm res}+t\theta_D>(\varepsilon-Ct)\pi^*\omega_h,
 \qquad\ddcn\kappa=\pi^*\ddcn\omega_0\le0.
\end{equation}
Put $A=\sum e_iE_i+tD$, $U=u\circ\pi+t\log|s_D|^2$;
Poincar\'e--Lelong gives $\pi^*\omega_0+\ddcn U=\kappa+[A]$.
Since $A$ is supported on exceptional divisors whose images have
codimension at least two, $\pi_*[A]=0$. Pushing down gives a quasi-psh
representative $u'$ of $\pi_*U$, equal to $U\circ\pi^{-1}$ off $Z$.
A nonexceptional point of a fiber is isolated. Since the fibers are
connected, a fiber over $Z$ cannot contain such a point. Thus
$\pi^{-1}Z=\operatorname{Exc}\pi$. Each exceptional component has
a logarithmic pole with positive coefficient. By properness, as a
base point approaches $Z$, its lifts approach this exceptional set.
Hence $\Pole(u')=Z$, $u'\circ\pi=U$, and
\begin{equation}\label{cur:can:repaired-current}
 \mathscr T'=\pi_*\kappa=\omega_0+\ddcn u'\ge(\varepsilon-Ct)\omega_h,\quad
 u'=\pi_*U,\qquad\pi^*\mathscr T'=\kappa+[A].
\end{equation}
Absorb the smooth remainder into the divisor metric, including $A=0$.
Finally $\omega_{\rm res}|_{\ker d\pi}\ge0$ makes
$-\sum e_iE_i$ relatively nef, hence $-A$ relatively ample.
\end{proof}

Write $\mathscr T=\omega_0+\ddcn u$ for the repaired current and
$Z=\Pole(u)$. The bound $\dim Z\le2$ is used only for ampleness.

\subsection{Bounded canonical potentials}
\label{cur:can:section}\label{cur:sec:canonical}

\begin{theorem}[Bounded canonical potential]\label{cur:can:theorem}
Let $X$ be compact and Brody hyperbolic, and fix a smooth Hermitian
background $h$. Let $\mathscr T=\omega_0+\ddcn u$ be a strictly positive
plurinegative current in the form of Lemma~\ref{cur:lem:resolved-model},
with repaired resolution $\pi:Y\to X$, smooth positive form $\kappa$,
and polar set $Z$. Then a bounded function $\varphi$ on $X$ satisfies
$-\rho_h+\ddcn\varphi\ge0$.
\end{theorem}

Keep the data of the theorem fixed. In particular
$\pi^*\mathscr T=\kappa+\sum_i b_i[E_i]$ with $b_i>0$ on all
exceptional primes, and $U=X\setminus Z$. The constants in the
Hessian and trace estimates may depend on these data, but not on the
center, maximizing disc, or penalty. The disc metrics $\kappa_{\rm disc},h_{\rm disc}$ use
\eqref{conv:eq:disc-conversion}. Use the integral divisor $D$ and
blowup ideal $\mathcal I$ from Lemma~\ref{cur:lem:resolved-model};
$D$ need not equal the real divisor defining $\kappa$.
Write
\[
 \theta=\operatorname{curv}\mathcal O_Y(-D),\qquad
 \ell_D\circ\pi=\log|s_D|^2,
 \qquad \ddcn \log|s_D|^2=[D]+\theta.
\]
On $U$ set $n=\dim X$, $F_X=\log((\mathscr T|_U)^n/\omega_h^n)$ and
$W=F_X-n\ell_D$, so the wedge power uses only the smooth restriction. On $\pi^{-1}U$, $F=F_X\circ\pi=\log(\kappa^n/\pi^*\omega_h^n)$.
Let $K_{Y/X}$ be the divisor of the Jacobian section of $K_Y\otimes\pi^*K_X^{-1}$.
On compact $Y$, $e^{-F}=\pi^*\omega_h^n/\kappa^n$ and $|s_D|^2$ are
bounded, so $F,W$ are bounded below. Factoring the density by the
squared Jacobian gives the first line by Poincar\'e--Lelong; the second defines $q_D$:
\begin{align}
 &\ddcn F=\pi^*\rho_h
        -\rho_\kappa-[K_{Y/X}],
                                                        \label{cur:can:relative-ricci}\\
 &q_D=\rho_h-\ddcn W.                                    \label{cur:can:modified-ricci}
\end{align}
A disc $g$ centered in $U$ lifts uniquely to $f:\mathbb D\to Y$
by the blowup property because its nonzero pullback ideal is
principal on the nonsingular disc.
Define the pullback measures on $\mathbb D$ by
\begin{align}
 &g^*\mathscr T=f^*\kappa+\sum_i b_i f^*[E_i],\qquad g^*q_D=f^*(\rho_\kappa+n\theta)
             +f^*[K_{Y/X}+nD].                         \label{cur:can:source-ricci}
\end{align}
They agree with the scalar-potential pullbacks.
As the $b_i$ are strictly positive, there is a constant depending only
on the fixed geometric data such that
\begin{equation}\label{cur:can:measure-domination}
 |g^*q_D|\le Cg^*\mathscr T.
\end{equation}
\begin{lemma}[Canonical maximizing discs]\label{cur:can:maxima}\label{cur:lem:singular-maximizers}
For the logarithmic data of Theorem~\ref{cur:can:theorem} and each $x\in U$, $\delta>0$,
\[
 J_\delta(g)=d_{h_{\mathrm{disc}}}(g)+P^{\mathrm{cur}}_{q_D}(g)
             -\delta A^{\mathrm{cur}}_{\mathscr T}(g)
\]
attains a finite maximum among finite-current-area discs centered at $x$.
Every maximizing lift $f$ is continuous on the closed disc and belongs to
$W^{1,p}$ for every $p\in(2,4)$, with $|f'|_{\kappa_{\mathrm{disc}}}^2\le C_f/(1-|z|)$
near the boundary and only finitely many interior exceptional intersections.
\end{lemma}
\begin{proof}
Write $A_Q=A_Q^{\rm cur}$ and $P_Q=P_Q^{\rm cur}$. The derivative
bound \eqref{amp:app:eq:brody-derivative} and $G_0/(1-|z|)\in L^2$
give $\int G_0^2g^*\omega_h\le C$.
Cauchy--Schwarz, $|\rho_h|\le C\omega_h$ and truncated Green give
$P_{q_D}\le C\sqrt{A_{\omega_h}}+W(x)-\inf_UW$.
The constant disc has value zero. Since the reward is bounded and
$\mathscr T\ge\varepsilon'\omega_h$, the upper bound gives a finite
supremum and bounds the current areas of a maximizing sequence $g_j$
by $A_*(x,\delta)$.

Pass to a normal limit $g_j\to g$. The positive disc measures
$g_j^*\mathscr T$ and $g_j^*(C\mathscr T+q_D)$ in
\eqref{cur:can:measure-domination} make $u\circ g_j$ and $(Cu-W)\circ g_j$
subharmonic after common local quadratic corrections and locally
bounded above.
The convergence near the center $x\notin Z$ is smooth, so the potentials
cannot tend to $-\infty$ there. Smooth convergence away from the
discrete set $g^{-1}Z$ identifies their limits. Subharmonic compactness gives
convergence of the potentials in $L^1_{\mathrm{loc}}$, so their
measures converge weakly, including the divisor atoms.
Central smooth convergence and \eqref{cur:can:measure-domination} give errors in cutting off the Green kernel $O(\epsilon^2(1+|\log\epsilon|))$ at zero
and $O(\log(1/r)A_*)$ at the boundary. The reward and Poisson terms converge, while the area is lower
semicontinuous. Its negative sign makes the limit a maximizer.

Compare the maximizing disc with its dilation $g_r(z)=g(rz)$.
Writing the difference of the Poisson integrals as an integral over radii by Fubini gives
\begin{equation}\label{cur:can:common-dilation}
 \delta(A_{\mathscr T}-A_{\mathscr T,r})
 \le(1-r^2)d_{h_{\rm disc}}+2\int_r^1A_{q_D,s}\,\frac{ds}{s}.
\end{equation}
Finite area and \eqref{cur:can:measure-domination} imply
\begin{equation}\label{cur:can:tail}
 A_{\mathscr T}-A_{\mathscr T,r}=O_g(1-r).
\end{equation}
Each interior intersection contributes at least the smallest positive
$b_i$ to the current area in \eqref{cur:can:source-ricci}, since
its multiplicity is at least one. Thus finite area gives finitely many intersections. Choose the inner
radius beyond all of them, so the lift avoids the exceptional divisor
on the collar. Lemma~\ref{amp:lem:area-tail-regularity} applied to the smooth area
tail gives the asserted regularity and derivative bound. The boundary
values need not avoid the divisor.
\end{proof}

\begin{proof}[Proof of Theorem~\ref{cur:can:theorem}]
\setcounter{canproofstep}{0}
Write $A_Q=A_Q^{\rm cur}$ and $P_Q=P_Q^{\rm cur}$ in this proof;
let $x=g(0)$ for the maximizing disc. Estimate constants are geometric; collar widths, parameter neighborhoods and Sobolev norms may depend on the disc.
\medskip\noindent
\refstepcounter{canproofstep}\label{cur:can:admissible-family}\label{cur:lem:damped-variations}
Fix $p\in(2,4)$. For a maximizing disc $g$ with lift $f$, Lemmas
\ref{cur:lem:singular-maximizers} and \ref{amp:app:var:frame} give
$f\in C^0\cap W^{1,p}$, finitely many interior exceptional
intersections, and boundary-unitary closed-disc frames $B$ for
$(f^*TY,\kappa_{\rm disc})$ and $e$ for $f^*\mathcal O(D)$.
Removing the interior zeros by scalar Blaschke factors and finite
matrix inner factors gives
\begin{equation}\label{cur:can:finite-inner}
 d\pi B=\widetilde B C,\qquad a=s_D(f)/e=b b_D,
\end{equation}
where $C$ is matrix inner, $b_D$ removes the interior zeros of $a$,
and $\widetilde B$ is nonsingular inside. For $C$, each right division
by $I-P+(z-\zeta)(1-\bar\zeta z)^{-1}P$, with $P$ a kernel-line
projection, removes one determinant zero. These finite removable
divisions preserve $W^{1,p}$, without boundary nonvanishing.

Put $r_D=|s_D|^2$, $v=r_D\circ f=|a|^2e^{p_D}$ and $p_D=\log|e|^2$.
Let $H_{r_D}$ extend the boundary values of $v$ harmonically.
In Proposition~\ref{an:prop:weighted-boundary} take the scalar
$K=e^{-p_D}$, $L_0=-p_D$ and inner factor $1$. The frame $e$ is
unitary on the boundary, so $p_D=0$ and $v=|a|^2$ there. The flux of
$r^2vK=r^2|a|^2$ is therefore $\langle v\rangle+\mathcal D(a)$, so
\begin{equation}\label{cur:can:scalar-energy}
 4\pi\mathcal D(a)=\int_{\mathbb D}\Delta(r_D\circ f)
                   -\int_{\mathbb D}H_{r_D}\Delta p_D.
\end{equation}
Since $|a|^2\le |b|^2$ in the disc and their boundary values agree,
the outward boundary slope of the circular mean of $|a|^2$ is at
least that of $|b|^2$. Green's area identity gives $\mathcal D(b)\le\mathcal D(a)$.
The Laplacians are bounded by $C|f'|^2_{\kappa_{\mathrm{disc}}}\in L^{p/2}$
(using $\Delta p_D=4\theta_{\mathrm{disc}}(f',\overline{f'})$), and
$|b|^2=r_D\circ f$ on the boundary. Thus
\begin{equation}\label{cur:can:damping-energy}
 \|b\|_\infty\le C,\qquad \mathcal D(b)\le C A^{\mathrm{cur}}_\kappa(f).
\end{equation}
Lemma~\ref{amp:lem:factor-preserving-spray}, using a separate
closed-disc frame, gives $V_j=zb\widetilde B_j$ and
$\Psi-g=bH$, $H=O(t)$ in $C_t^2W^{1,p}$.
On the collar write $\mathcal I(g)=\sigma G(f)$,
$\sigma=s_D(f)$ and $G\ne0$. Inverting only the nonvanishing $e_{\rm loc}b_D$ gives
$b=\sigma/(e_{\rm loc}b_D)$. Since $\Psi-g=bH$, the difference of
each local holomorphic ideal generator at $\Psi$ and $g$ is
divisible by $b$, with quotient $O(t)$ in $C_t^2W^{1,p}$. Thus
$\mathcal I(\Psi)=\sigma\{G(f)+O(t)\}$ in the same space.
A finite chart cover keeps the residual tuple nonzero, giving an atom-free
joint $C_t^2W^{1,p}$ collar lift. The unique lifts on compact inner discs have
finite area and finitely many atoms, proving admissibility.
No inverse of $\widetilde B$ on the boundary or division by $\sigma$ is used.

\medskip\noindent
\refstepcounter{canproofstep}\label{cur:can:weighted-hessian}\label{cur:lem:singular-hessians}
For the admissible family $\Psi_t$, put
$\mathscr L_{\!\mathrm{var}}=\left.\sum_j\partial_{t_j}\partial_{\bar t_j}\right|_{t=0}$. Use the Appendix~\ref{amp:calc:area} variation quantities with coefficients $[\beta]_{\rm disc}$.

Green's formula separates the background and boundary terms before
differentiation, accounting for the moving atoms.
For either $(Q,\beta,\zeta,a)=(\mathscr T,\omega_0,u,1)$ or
$(q_D,\rho_h,-W,G_0)$, distributional Green at each radius $r$
in the fixed exceptional-free collar gives
\begin{equation}\label{an:eq:weighted-current-stokes}
 \int_{|z|<r}a\,\Psi_t^*Q
 =\int_{|z|<r}a\,\Psi_t^*\beta
 +\frac1{4\pi}\int_{|z|=r}
    \{a\partial_r(\zeta\circ\Psi_t)-(\zeta\circ\Psi_t)\partial_ra\}\,ds
 -\mathbf1_{a=G_0}\zeta(x).
\end{equation}
Here $\Delta G_0=-4\pi\delta_0$. Split the integrals at a fixed
radius inside the collar. The distributional formula gives $C^2$
dependence of the inner integral, including the moving atoms.
On the collar the $C_t^2W^{1,p}$ lift gives convergence of the
Hessians as $r\uparrow1$.

Apply \eqref{amp:app:eq:area-split} to the smooth backgrounds.
Since $\mathscr L_{\!\mathrm{var}}(\zeta\circ\Psi)=\mathcal N_{\ddcn\zeta}$,
the scalar term adds exactly the normal flux of $\ddcn\zeta$.
On the collar the combined flux is that of
$\pi^*(\beta+\ddcn\zeta)=\kappa$ or $\widetilde q:=\rho_\kappa+n\theta$.
The area bulk is nonpositive by $\ddcn\omega_0\le0$.
The closed Poisson background has no torsion or bulk, and its center
term vanishes since $V_j(0)=0$. The factor $C^{-1}$ is used only
on the collar.

For the outer variations, put $R=BC^{-1}$ and
$K=e^{-p_D}|b_D|^{-2}RR^{*_{\kappa_{\rm disc}}}$ on the collar.
Their covariance is $r^2vK$, and $K=I$ on the boundary. Its
regular log determinant $L_0=\log\det(B^*\kappa_{\rm disc}B)-np_D$ satisfies
$\Delta L_0=-4[\widetilde q]_{\rm disc}(f',\overline{f'})$.
The inner factor $I_0=\det C\,b_D^n$ has divisor $f^*(K_{Y/X}+nD)$.
Thus $J=L_0-\log|I_0|^2$ equals $\log\det K$ on the collar,
has zero boundary trace and satisfies $\ddcn J=-g^*q_D$.
Proposition~\ref{an:prop:weighted-boundary} thus yields
nonpositive atomic normal flux. The smooth normal and torsion
terms are bounded by $C_*(1+A_\kappa(f))$: derivatives of the
fixed smooth function $r_D$ give $|\Delta v|\le C_*|f'|^2$ and
$|v_{\bar z}|\le C_*|f'|$. The Poisson flux is bounded by
$C_*\|v\|_\infty$. Therefore
\begin{align}
 \mathscr L_{\!\mathrm{var}} A_{\mathscr T}(\Psi_t)&\le C_*(1+A_\kappa(f))\le C_*(1+A_{\mathscr T}(g)),
 \label{cur:can:area-upper}\\
 |\mathscr L_{\!\mathrm{var}} P_{q_D}(\Psi_t)|&\le C_*.
 \label{cur:can:poisson-hessian}
\end{align}

\medskip\noindent
\refstepcounter{canproofstep}\label{cur:can:penalty-limit}
\label{cur:lem:singular-penalty}\label{cur:can:envelope}
Set $M=|b(0)|^2\widetilde B(0)^*h_{\rm disc}(x)\widetilde B(0)>0$.
Its trace is the reward Hessian, so maximality gives
\begin{equation}\label{cur:can:trace-bound}
 \tr M\le C_*+C_*\delta A_{\mathscr T}(g),\qquad0<\delta\le1.
\end{equation}
Green's formula for $J$ and \eqref{cur:can:finite-inner} give
\begin{equation}\label{cur:can:center-determinant}
 \log\det M=J(0)-W(x)=P_{q_D}(g)-W(x).
\end{equation}
Apply the same trace--determinant lemma~\ref{can:lem:trace-envelope} with
$(h_0,d,P,A)=(h_{\rm disc},d_{h_{\rm disc}},P_{q_D},A_{\mathscr T})$.
The derivative reward is bounded, and $J_\delta$ attains its
maximum. For the Poisson integral, smoothness at the center
controls the logarithmic singularity of the kernel, and domination
\eqref{cur:can:measure-domination} controls the integral away
from the center. Discs extending past the boundary with image in
$U$ have finite area and
$q_D=\rho_{h_{\rm disc}}-\ddcn W$. Thus $P_{q_D}(g)\le W(g(0))+C_*$ for every finite-area disc centered in $U$, and $-C_*\le\varphi\le W$ with
$-\rho_h+\ddcn\varphi\ge0$ on $U$.

Since $\operatorname{codim} Z\ge2$, the local plurisubharmonic weights
extend across $Z$ by Shiffman's theorem \cite[Theorem~3(ii)]{cur:Shiffman72}.
Their upper-semicontinuous extensions preserve the transitions and
the lower bound for $\varphi$; compactness gives a global upper bound.
\end{proof}

\section{Localization and the surface obstruction}
With nefness established, ampleness under $\mathrm{(C_2)}$ reduces to
finding a smooth plurinegative metric. Otherwise the Hahn--Banach
separation in Lemma~\ref{cur:lem:localization} gives a positive closed
current supported on the polar set of the adjusted potential on $X$.
Here we prove strict positivity of the pairing for the diffuse surface summands;
Section~\ref{cur:sec:elimination} treats the curves.
\subsection{Resolved positivity and polar localization}\label{cur:sec:localization}
Keep the adjusted model $(\mathscr T,\omega_0,u,Z,\pi,\kappa)$
with $\dim Z\le2$ and fix a normalized curvature $\alpha$ of $K_X$. Let $D_{\rm res}=A$ be the real exceptional divisor with its metric in Lemma~\ref{cur:lem:resolved-model}. Then
\begin{equation}\label{cur:eq:resolved-bridge}\equationalias{cur:as:metric}
 \kappa=\pi^*\omega_0+\operatorname{curv}\mathcal O_Y(-D_{\rm res})>0,\qquad
 -D_{\rm res}\text{ is }\pi\text{-ample},\qquad
 \ddcn\kappa=\pi^*\ddcn\omega_0\le0.
\end{equation}
All subsequent modifications have centers over $Z$. Thus no exceptional image strictly contains an irreducible polar surface.

\begin{lemma}[Separation and polar localization]\label{cur:lem:localization}
Assume $n\geq3$ and the hypotheses of the $\mathrm{(C_2)}$ clause of Theorem~\ref{amp:main}.
Use the nefness of $K_X$ established under $\mathrm{(C)}$.
If there is no smooth strictly positive plurinegative \((1,1)\) form,
there is a nonzero positive closed current $\mathscr R$ of bidimension $(1,1)$ satisfying
\begin{equation}\label{cur:eq:localized-R}
 \Supp \mathscr R\subseteq Z,\qquad
 \langle\alpha,\mathscr R\rangle=0,\qquad
 \langle\omega_0,\mathscr R\rangle\leq0.
\end{equation}
\end{lemma}
\begin{proof}
Let $\mathcal F$ be the weakly closed convex cone generated by
$\ddcn P$ for strongly positive bidimension-$(2,2)$ currents, and
$\mathcal B=\{S\ge0:\langle h,S\rangle=1\}$.
If they are disjoint, strict separation of the weakly compact base
$\mathcal B$ gives a smooth $(1,1)$ form positive on $\mathcal B$
and nonpositive on $\mathcal F$. Testing on positive tangent-line currents gives strict positivity;
testing against $\ddcn P$ gives plurinegativity by duality. This
contradicts the hypothesis.
Choose $S\in\mathcal B\cap\mathcal F$. Stokes and weak closure give
$dS=0$, $\langle\alpha,S\rangle=0$, $\langle\omega_0,S\rangle\le0$.

Skoda--El Mir \cite[Chapter~III, Theorem~2.3]{amp:new:DemaillyCADG}
makes the finite-mass zero extension $S_{\rm out}=\mathbf1_{X\setminus Z}S$
closed; thus $\mathscr R=S-S_{\rm out}\ge0$ is closed and supported on $Z$.
Choose smooth convex $\chi_j$, constant below $-j-1$, equal to the
identity above $-j$, and with $0\le\chi_j'\le1$. Then $\chi_j(u)$
is smooth across $Z$ and
\[
 \omega_0+\ddcn\chi_j(u)\ge
 \varepsilon\chi_j'(u)h+(1-\chi_j'(u))\omega_0.
\]
Stokes and dominated convergence on $S_{\rm out}$ give
$\langle\omega_0,S_{\rm out}\rangle\ge\varepsilon\langle h,S_{\rm out}\rangle$.
Hence $\langle\omega_0,\mathscr R\rangle\le0$. If $\mathscr R=0$, then
$S=S_{\rm out}$ and the preceding bound gives
$\langle\omega_0,S\rangle\ge\varepsilon\langle h,S\rangle=\varepsilon>0$,
contrary to $\langle\omega_0,S\rangle\le0$.
By Lemma~\ref{cur:lem:nef-ddc-pairing}, both $S_{\rm out}$ and
$\mathscr R$ pair nonnegatively with the canonical class. Their
pairings sum to zero, so both vanish.
\end{proof}

\subsection{Strict subadjunction on polar curves and surfaces}\label{cur:sec:subadjunction}

Cao--H\"oring's fibrewise Bergman metric motivates the construction
\cite[Theorem~3.4 and Lemma~3.5]{cur:CH}. For rank-one direct images we
use P\u{a}un--Takayama's positivity and extension
\cite[Theorem~3.3.5 and Remark~2.4.2(3)]{cur:PT18} in place of
Cao--H\"oring's extension and H\"older argument
\cite[Proposition~3.3 and the proof of Lemma~3.5]{cur:CH}. We construct
the twist and then account for the conductor.

In the fixed model of \S\ref{cur:sec:localization}, let $W\subset Z$ be a reduced irreducible compact analytic curve or surface. Write $D_0=\sum_j a_jE_j=D_{\rm res}$ and $K_Y=\pi^*K_X+\sum_j k_jE_j$. With $\omega_0$ fixed, perturb the positive $a_j$ rationally, adjusting $\beta=\operatorname{curv}\mathcal O_Y(-D_0)$ by the same coefficients, in the normalization \eqref{cur:eq:normalization}. Openness of positivity and relative ampleness preserves \eqref{cur:eq:resolved-bridge}. Choose the rational coefficients so that no two of the relevant ratios agree. This gives a unique minimizer $E=E_1$ for $c=\min_{\pi(E_j)\supseteq W}(k_j+1)/a_j$. Thus $d_j=k_j-ca_j$ are rational, $d_1=-1$, and $d_j>-1$ for $j\ne1$ with $\pi(E_j)\supseteq W$; other $d_j$ may be at most $-1$.
The minimum is over a nonempty set because $W$ is irreducible and
$Z=\bigcup_j\pi(E_j)$.

Let $\mu:M\to S=\pi(E)\hookrightarrow X$ be the composition of a projective resolution of the normalization, the normalization map and the inclusion. It is finite over $S$ outside finitely many points, in particular generically along every conductor curve. For $\theta=\mu^*\omega_0$ one has
\begin{equation}\label{cur:as:theta}
 \ddcn\theta=0:
\end{equation}
on a curve this follows by bidegree, and on a surface by Stokes and $\ddcn\theta\le0$.

\begin{proposition}[Strict subadjunction with a nonclosed base form]\label{cur:prop:subadjunction}
For the above data and every Hermitian form $\omega_M$ on $M$,
there are an integral divisor $H=-\Delta+H_{\rm exc}$, a smooth
reference metric on the holomorphic line bundle
$L_M=\mu^*K_X-K_M+H$, a global quasi-psh metric potential $\psi$
and $\delta>0$ such that
\begin{equation}\label{cur:as:conclusion}
 A=\operatorname{curv}(L_M)+c\theta+\ddcn\psi
 \ge\delta\omega_M,\qquad \ddcn A=0.
\end{equation}
Here $\Delta\ge0$ has no $\mu$-exceptional components and includes
all divisorial conductor corrections; $H_{\rm exc}$ is signed and
supported on curves contracted to points of $S$.
If $W\ne S$, then $W$ is a curve, $S$ is a surface normal generically
along $W$, and the strict transform $W'\subset M$ of $W$ is not a component
of $H$ and has generic Lelong number $\nu_{W'}(\psi)<1$.
\end{proposition}
The line-bundle curvature $\operatorname{curv}(L_M)+\ddcn\psi$ is
closed, but $c\theta$ need not be.

\begin{proof}
Resolve $E\dashrightarrow M$ by SNC-compatible projective blowups.
The lift to the normalization is holomorphic. The centers lie over the finite exceptional image of the normalization
resolution, hence do not dominate $W$. Small rational closed
exceptional repairs preserve positivity, relative ampleness, SNC,
the coefficient of $E$ and the gaps over $W$.
The graph of $f$, together with a relative embedding of $\pi|_E$,
gives a relative projective embedding over $M$. Put $P=\mathcal O_Y(-D_0)|_E$, $\beta_E=\beta|_E$,
$B=-\sum_{j\ne1}d_j(E_j\cap E)$ and $R=\lceil-B\rceil$.
Adjunction gives
\begin{equation}\label{cur:as:adjunction}
 K_{E/M}+cP+B=f^*(\mu^*K_X-K_M),\qquad
 \beta_E+f^*\theta=\kappa|_E>0.
\end{equation}
Horizontal coefficients of $B$ are below one.
Put $D_{\rm lc}=K_{Y/X}-cD_0$, $A_Y=E+\lceil D_{\rm lc}\rceil$,
$\Gamma=\lceil D_{\rm lc}\rceil-D_{\rm lc}$. Then
$A_Y|_E=R$, $0\le\Gamma<1$, and
\begin{equation}\label{cur:as:kv-calculation}
 (A_Y-E)-K_Y-\Gamma=-\pi^*K_X-cD_0.
\end{equation}
The divisor $\Gamma$ is an SNC boundary on the smooth $Y$, and
$-\pi^*K_X-cD_0$ is $\pi$-ample. Apply Fujino's relative vanishing
\cite[Theorem~1.1(ii)]{cur:FM} with
maps $\Id_Y,\pi$ and $(p,q)=(1,0)$ to $\mathcal O_Y(A_Y-E)$.
The restriction sequence gives
\begin{equation}\label{cur:as:surjection}
 I:=\pi_*\mathcal O_Y(A_Y)\twoheadrightarrow i_*\mu_*Q,
 \qquad Q=f_*\mathcal O_E(R),\quad i:S\hookrightarrow X.
\end{equation}
Normality makes $I$ an ideal, generically unit along $S,W$ since
negative coefficients have images not containing $W$.
Dense evaluation gives kernel $I\cap I_S$, hence
$\mu_*Q=J:=(I+I_S)/I_S$, preserving $1$, and $\operatorname{rank}Q=1$.
Finite-map module equivalence gives
\begin{equation}\label{cur:as:ideal-off-exceptional}
 Q=J\mathcal O_M\subset\mathcal O_M
\end{equation}
on the finite locus. The normalization algebra preserves $J$, placing
it in the conductor. Over a general point, base change gives
$\dim H^0(E_y,\mathcal O(R)|_{E_y})=\rank Q=1$.
Effectivity of $R|_{E_y}$ supplies the section $1$, so
$H^0(E_y,\mathcal O_{E_y})=\C$. In the Stein factorization
\cite[Chapter~10, \S6.1, p.~212]{GrauertRemmertCAS}, the finite map is birational
by generic connectedness and an isomorphism by normality of $M$.
Thus all fibers of $f$ are connected.
Proper meromorphic graphs identify sections of $Q$ with meromorphic
functions on the base. Thus $Q^{**}=\mathcal O_M(H)$ via the generic section $1$, with signs and supports given by
\eqref{cur:as:ideal-off-exceptional}.
With $L_0=\mu^*K_X-K_M$, $\widetilde B=B+f^*H$, one has
\begin{equation}\label{cur:as:normalized-adjunction}
 K_{E/M}+cP+\widetilde B=f^*L_M,\qquad L_M=L_0+H.
\end{equation}
The integral bundle $\mathcal L=f^*L_0-K_{E/M}+R=cP+B+R$
has a metric $h_{\mathcal L}$ of curvature $c\beta_E+[B+R]$.
The SNC coefficients of $B+R$ lie in $[0,1)$, so its multiplier
ideal is $\mathcal O_E$, and $f_*(K_{E/M}+\mathcal L)=L_0\otimes Q$.

Choose $a>0$ with $\kappa|_E-4af^*\omega_M>0$.
A real quadratic $q$ with $(\ddcn q)_y=\theta_y-2a(\omega_M)_y$
satisfies on a sufficiently small Stein ball, by properness,
\begin{equation}\label{cur:as:nonclosed-twist}
 \beta_E+f^*\ddcn q>0,\qquad\theta-\ddcn q\ge a\omega_M.
\end{equation}
The semipositive twist $h_{\mathcal L}e^{-cq\circ f}$ still has
trivial multiplier ideal. Since $f$ is projective and has connected
fibers, the hypotheses of
\cite[Theorem~3.3.5 and Remark~2.4.2(3)]{cur:PT18} apply, and
its fiber $L^2$ metric on the torsion-free direct image extends to
the double dual $L_M$. The untwisted weights $\varphi_e$ satisfy
$\varphi_e+cq\in\operatorname{PSH}$ and $\ddcn\varphi_e+c\theta\ge ca\omega_M$.
They give $\psi$, $\delta=ca$, and \eqref{cur:as:theta} gives $\ddcn A=0$.

If $W\ne S$, the generic identity $J=\mathcal O_S$ along $W$
makes $S$ normal there and $H=0$ near general $W'$.
There $Q$ is locally free, so a frame $e$ of $L_M$ evaluates to a
holomorphic section $\sigma_e$ of $K_{E/M}+\mathcal L$, with
\begin{equation}\label{cur:bn:I}
 I_e(y)=\int_{E_y}|\sigma_e|^2_{h_{\mathcal L}}=e^{-\varphi_e(y)}.
\end{equation}
Wedging with a base volume frame gives a holomorphic
$\mathcal L$-valued top form. Its squared density is locally
integrable because the multiplier ideal is trivial. Properness and
Fubini give $I_e\in L^1$, which contradicts $\nu_{W'}(\psi)\ge1$
by applying Siu's decomposition theorem \cite{Siu74} to the curvature current of $\varphi_e+cq$.
\end{proof}

\subsection{Diffuse currents and nef lifts}
\label{cur:sec:singular-surface-current-ideal}

Here a diffuse current is a positive closed current of bidimension
$(1,1)$ that gives zero trace mass to every analytic curve.

\begin{lemma}[Exceptional correction on a surface]
\label{cur:lem:exceptional-correction}
Let $E_1,\ldots,E_m$ be the exceptional curves of a resolution
of a compact normal surface. If a real Bott--Chern $(1,1)$ class $a_0$ has
$a_0\cdot E_i\ge0$ for all $i$, there is a unique effective real divisor
$D=\sum_i x_iE_i$ such that $a=a_0+[D]_{BC}$ is orthogonal to each $E_i$.
The class $a$ is nef if $a_0$ is nef. If $a_0=[C']_{BC}$ for
an integral nonexceptional curve $C'$, then $a\cdot C'=a^2$, and $a$
is nef exactly when $a^2\ge0$.
\end{lemma}
\begin{proof}
The matrix $N=(E_i\cdot E_j)$ is negative definite
\cite[p.~367]{cur:Grauert1962}. With $y_i=a_0\cdot E_i$,
$x=(-N)^{-1}y\ge0$. Indeed, split $x$ into its positive part
minus its negative part. Pairing $(-N)x=y\ge0$ with the negative
part gives a contradiction if that part is nonzero: the cross term
is nonpositive because $-N$ has nonpositive off-diagonal entries,
and the remaining term is strictly negative by positive definiteness. For $a_0=[C']_{BC}$ this is Mumford's
numerical total transform \cite[\S II(b), p.~17]{cur:Mumford1961},
and orthogonality gives $a\cdot C'=a^2$.
In both nefness claims $a$ is pseudoeffective. Its degrees on the
other curves are nonnegative by effectivity or by nefness of $a_0$.
The curve criterion and \cite[Proposition~3.3(iv)]{cur:DP} give
sufficiency; necessity is $a\cdot C'\ge0$ for nef $a$.
\end{proof}
\begin{lemma}[Nef numerical lift of a diffuse current]
\label{cur:lem:finite-projection-current-lift}
\label{cur:lem:general-nef-correction}\label{cur:lem:nef-numerical-current-lift}
Let $Z$ be a compact reduced pure surface in a complex manifold, and
$\mathscr R$ a positive closed current of bidimension $(1,1)$ supported on $Z$
with zero trace mass on every analytic curve. Let $\mu:M\to Z$ resolve
its normalization; $M$ may be disconnected. There is a unique positive
closed lift $\widetilde{\mathscr R}$ with zero trace mass on every analytic curve,
agreeing with $\mathscr R$ on the common regular locus. Its class $a_0$ is nef. If $E_i$ are the curves contracted
to points, there are unique $x_i\ge0$ such that
\[
 \mathscr R^{\rm num}=\widetilde{\mathscr R}+\sum_i x_i[E_i],\qquad
 a=[\mathscr R^{\rm num}]_{BC}\text{ is nef},\qquad a\cdot E_i=0.
\]
Both currents push forward to $\mathscr R$ and preserve its pairings with every
smooth ambient $(1,1)$ form, not necessarily closed.
\end{lemma}
\begin{proof}
Choose a finite Weierstrass projection $p:U\to B\subset\C^2$
and a psh potential $v$ of $p_*(\mathscr R|_U)$.
Off branch values and images of singular and exceptional loci,
each lifted sheet is dominated by $\ddcn(v\circ p\circ\mu)$.
Dominance of $p\circ\mu$ ensures that the composed psh potential is
not identically $-\infty$. Its curvature bounds the lifted sheets and
has locally finite mass across the omitted set, so Skoda--El Mir
\cite[Chapter~III, Theorem~2.3]{amp:new:DemaillyCADG} gives unique
gluing by zero extension. The diffuse currents charge no omitted
curves and have no point atoms by the $O(r^2)$ mass bound; hence
$\mu_*\widetilde{\mathscr R}=\mathscr R$.
The positive Lelong upper level sets of the lifted current are
zero-dimensional, so
\cite[Proposition~3.3(iv)]{cur:DP} makes $a_0$ nef.
Lemma~\ref{cur:lem:exceptional-correction} gives the unique nef
orthogonal correction. The exceptional curve currents push forward
to zero, so the pairings agree for smooth ambient forms even when
they are not closed.
\end{proof}

\Needspace{6\baselineskip}
\subsection{Canonical degrees and strict surface pairings}
\begin{lemma}[Canonical degree on a surface]\label{cur:lem:canonical-nef-pairing}
If a compact smooth complex surface $M$ admits a holomorphic map of
generic rank two to a Brody-hyperbolic manifold, then
$c_1(K_M)\cdot \beta\ge0$ for every nef real Bott--Chern $(1,1)$ class $\beta$.
\end{lemma}
\begin{proof}
For a Gauduchon form $\gamma$, choose a smooth representative $\beta_\varepsilon$ of $\beta$ with
$\beta_\varepsilon+\varepsilon\gamma\ge0$. By Stokes,
$0\le\int_M(\beta_\varepsilon+\varepsilon\gamma)^2
=\beta^2+2\varepsilon\langle \beta,\gamma\rangle+\varepsilon^2\int_M\gamma^2$,
so $\beta^2\ge0$.
If $K_M\cdot \beta<0$, nefness excludes every effective pluricanonical
divisor, so $\kappa(M)=-\infty$. The surface classification
\cite[Theorem~2.2]{PetersDefinite2021} leaves rational, ruled or class VII
surfaces, up to point blowups. Covering rational curves exclude the
first two by hyperbolicity and generic rank two. The remaining
de Rham intersection form is negative definite or zero
\cite[Lemma~3.1 and p.~844]{PetersDefinite2021}. Thus $\beta^2\ge0$
forces the de Rham class of $\beta$ to vanish. The intersection number
$K_M\cdot \beta$ is then zero, contrary to its assumed negativity.
\end{proof}

\begin{proposition}[Strict surface pairing]\label{cur:prop:no-nef-obstruction}
In the data of Proposition~\ref{cur:prop:subadjunction}, suppose $M$
is a surface. If $Q\ne0$ is a positive closed current on $M$ with
nef Bott--Chern class $\beta$ orthogonal to the $\mu$-exceptional curves,
then $\mu^*\alpha\cdot \beta+c\langle\theta,Q\rangle>0$.
\end{proposition}
\begin{proof}
Choose a Gauduchon form $\gamma$ and $\delta>0$ with $A\ge\delta\gamma$.
Stokes and exceptional orthogonality give the equality below.
The current $A-\delta\gamma$ is positive and $\ddcn$-closed.
Lemma~\ref{cur:lem:nef-ddc-pairing} gives
$A\cdot\beta\ge\delta\langle\gamma,Q\rangle$.
The other two intersections are nonnegative by
Lemma~\ref{cur:lem:canonical-nef-pairing} and nefness on the effective
divisor $\Delta$. Thus
\[
 \mu^*\alpha\cdot \beta+c\langle\theta,Q\rangle
 =A\cdot \beta+K_M\cdot \beta+\Delta\cdot \beta
 \ge\delta\langle\gamma,Q\rangle>0.
\]
\end{proof}

\section{Polar curves and completion of Theorem~\ref{amp:main}}\label{cur:sec:elimination}
The obstruction may charge curves inside polar surfaces. The sign
of the numerical transform's self-intersection determines the pairing estimate.
\subsection{Strict pairing on polar curves}
\begin{proposition}\label{cur:prop:curve-degree}
For an integral polar curve $C$, the constant $c>0$ from
Proposition~\ref{cur:prop:subadjunction} with $W=C$ satisfies
$K_X\cdot C+c\int_C\omega_0>0$.
\end{proposition}
\begin{proof}
Apply subadjunction with $W=C$, choosing the minimal resolution of
the normalization as the model $M$ when $S$ is a surface.
Hyperbolicity gives genus $g\ge2$ on the normalization of $C$.
If $S=C$, $H=-\Delta$ and integration gives
$K_X\cdot C+c\int_C\omega_0=\int_MA+2g-2+\deg\Delta>0$.
If $S$ is a surface, the strict transform $C'$ in the chosen
model is birational to $C$ and absent from $H$.
Lemma~\ref{cur:lem:exceptional-correction} gives effective
$C_{\rm num}=C'+E$ with $E$ exceptional, orthogonal to all exceptional
curves and $C'^2\le C'\cdot C_{\rm num}=C_{\rm num}^2$.
If $C_{\rm num}^2\ge0$, its class is nef by the exceptional
correction lemma. Apply Proposition~\ref{cur:prop:no-nef-obstruction}
to $Q=[C_{\rm num}]$; its pushforward is $[C]$, so this gives the
desired inequality.

For $C_{\rm num}^2<0$ remove the divisorial masses before testing
against convex truncations of $[C_{\rm num}]$. Put $\nu=\nu_{C'}(\psi)<1$ and
$\mathscr Q=\operatorname{curv}(L_M)+\ddcn\psi\ge
\gamma:=\delta\omega_M-c\theta$.
Subtract $N=\sum_G\nu_G(\psi)G$ over components of $C_{\rm num}$.
Then $R=\mathscr Q-[N]-\gamma\ge0$ has zero trace mass on $|C_{\rm num}|$.
With smooth divisor metrics, write $[C_{\rm num}]=\Theta+\ddcn\ell$,
where $\ell$ is the weighted sum of the logarithmic norms.
For the convex truncations of Lemma~\ref{cur:lem:localization},
the smooth forms $\theta_j=\Theta+\ddcn\chi_j(\ell)$ tend to
$[C_{\rm num}]$. The chain rule gives
$\theta_j\ge-C\mathbf1_{\{\ell<-j\}}\omega_M$.
The resulting lower bound for $\langle R,\theta_j\rangle$ tends to
zero because $\{\ell<-j\}$ decreases to $|C_{\rm num}|$, where $R$ has
zero trace mass. Hence $\liminf\langle R,\theta_j\rangle\ge0$. Stokes for $\mathscr Q-[N]$
and weak convergence against smooth $\gamma$ give
$(L_M-N)\cdot C_{\rm num}\ge\int_{C_{\rm num}}\gamma>-c\int_C\omega_0$.
Exceptional orthogonality, $\Delta\cdot C_{\rm num}\ge0$ and
$N\cdot C_{\rm num}=\nu C_{\rm num}^2$ imply
\begin{equation}\label{cur:eq:curve-degree}
 K_X\cdot C+c\int_C\omega_0>K_M\cdot C_{\rm num}+\nu C_{\rm num}^2.
\end{equation}
Minimality and adjunction give
$K_M\cdot C_{\rm num}\ge K_M\cdot C'\ge2-C'^2$,
since $p_a(C')\ge g\ge2$. Using $C'^2\le C_{\rm num}^2$, the
right-hand side of the displayed inequality is at least
$2+(1-\nu)(-C_{\rm num}^2)>0$.
\end{proof}

\subsection{Completion of the proof}
\begin{proof}[Proof of Theorem~\ref{amp:main}]
Apply Proposition~\ref{cur:prop:intrinsic-regularization} and
Lemma~\ref{cur:lem:resolved-model} to a $\mathrm{(C)}$ witness,
with polar bound $n-1$, or $\min(2,n-1)$ under $\mathrm{(C_2)}$.
Theorem~\ref{cur:can:theorem} and
Lemma~\ref{amp:lem:zero-lelong-nef} give a bounded canonical metric
and nefness.

Under $\mathrm{(C_2)}$, Corollary~\ref{amp:cor:low-dimensional-ampleness}
handles $n\le2$. If $n\ge3$ and no smooth plurinegative metric exists,
Lemma~\ref{cur:lem:localization} gives nonzero $\mathscr R$.
Siu decomposition and Skoda--El Mir on the polar surface regular
loci away from their intersections give
$\mathscr R=\sum_jb_j[C_j]+\sum_S\mathscr R_S$, with $b_j>0$
and positive closed diffuse $\mathscr R_S$. After the Siu curve terms
are separated, the omitted intersections carry no residual mass. Lemma~\ref{cur:lem:nef-ddc-pairing} shows that every summand
pairs to zero with the canonical class.
Proposition~\ref{cur:prop:curve-degree} for curves and
Proposition~\ref{cur:prop:no-nef-obstruction} applied to the lifts
of Lemma~\ref{cur:lem:nef-numerical-current-lift} make every nonzero
summand's $\omega_0$ pairing positive. By compactness,
$-Ch\le\omega_0\le Ch$ for some $C>0$, so these pairings are
absolutely summable, contradicting
\eqref{cur:eq:localized-R}. Proposition~\ref{amp:prop:smooth-ampleness}
now gives ampleness and projectivity.
\end{proof}

\section{Fujiki-class rigidity}\label{stage:fujiki}
\label{sec:negative-geometry}
Lemma~\ref{n:lem:fujiki-kahler} gives a K\"ahler background in both
branches, without using singular currents. The integrated Chern--Lu
identity identifies $h$ with the negative K\"ahler--Einstein metric
when $c<0$. At $c=0$ the identity makes the symmetric part of the
connection difference vanish, and a Bochner argument shows that $h$
is parallel for the Ricci-flat K\"ahler connection.
\subsection{Excluding rational curves and obtaining a K\"ahler background}

\begin{lemma}[Fujiki-to-K\"ahler reduction]\label{n:lem:fujiki-kahler}
Let $X$ be a compact complex manifold in Fujiki's class $\mathcal C$, and let
$h$ be a Hermitian metric satisfying $H_h\le0$.
Then $X$ contains no rational curve.  Consequently $X$ is K\"ahler.
\end{lemma}

\begin{proof}
Suppose $f:\CP^1\to X$ is nonconstant. Demailly's genus--degree inequality
\cite[Theorem~4.1, $\kappa=0$]{Diverio2021Curvature} excludes such a map; the
following comparison gives a direct proof. For the Fubini--Study metric
$g$, the energy density $e=|df|_{g,h}^2=|f'|_h^2/g_{z\bar z}$ attains
a positive maximum. There $f'\ne0$, so the curve inequality in
Lemma~\ref{amp:app:pre:hyperbolic} gives
$(\log|f'|_h^2)_{z\bar z}\ge0$ because $H_h\le0$.
Positive Fubini--Study curvature gives
$(\log g_{z\bar z})_{z\bar z}<0$. Their difference is
$(\log e)_{z\bar z}>0$, contradicting the maximum.
A rational curve would give such a map by normalization, so
Lemma~\ref{n:lem:no-rational-kahler} makes $X$ K\"ahler.
\end{proof}

\subsection{An integrated Chern--Lu identity}
\label{r16:compression}\label{app:icl-calculations}
The classical Chern--Lu comparison differentiates the energy density of
a holomorphic map \cite[\S1]{BroderStanfieldSchwarz}. For the identity map, averaging the usual
and crossed curvature contractions isolates the symmetric connection term
after integration.

\begin{lemma}[Integrated Chern--Lu identity]\label{n:lem:ICL}
Let $X$ be compact of complex dimension $n\ge2$, let $g$ be K\"ahler, and let $h$ be Hermitian with $H_h\equiv c$.
For a real constant $a_0$, suppose $\Ric_g^{(1)}=a_0\omega_g$. Put
\[
 u=\tr_g h,\qquad q=\tr\bigl((g^{-1}h)^2\bigr).
\]
Let $\mathsf S$ be the symmetric part of $\nabla^h-\nabla^g$, with mixed norm
\[
 |\mathsf S|^2=g^{i\bar j}g^{k\bar\ell}h_{p\bar q}
 \mathsf S^p{}_{ik}\overline{\mathsf S^q{}_{j\ell}}.
\]
Then
\begin{equation}\label{n:eq:icl-integrated}
 0=\int_X\left[|\mathsf S|^2+a_0u-\frac c2(u^2+q)\right]dV_g.
\end{equation}
\end{lemma}
\begin{proof}
Put $D=\nabla^g$, $K=\nabla^h-D$. The common barred connection and
$h_{p\bar l}K^p_{ik}=D_i h_{k\bar l}$ give
\[
 R^h_{i\bar j k\bar l}
 =h_{p\bar l}(R^g)_{i\bar j k}{}^p-D_{\bar j}D_i h_{k\bar l}
   +h_{p\bar q}K^p_{ik}\overline{K^q_{jl}}.
\]
In a $g$-unitary frame, average the contractions $(j,l)=(i,k),(k,i)$.
K\"ahler symmetry makes both curvature contractions of $g$ equal to
$a_0u$. Both second-derivative contractions are divergences and
integrate to zero by Stokes. Averaging the two quadratic terms in
$K$ gives $|\mathsf S|^2$.
The left side is $c(u^2+q)/2$ by \eqref{n:eq:polarization}.
\end{proof}

\subsection{Negative curvature: ampleness and metric comparison}\label{n:lem:comparison}
\begin{proof}[Proof of Theorem~\ref{n:thm:main} when $c<0$]
The metric $h$ has strictly negative Chern HSC, so
Lemma~\ref{amp:app:pre:hyperbolic} gives hyperbolicity. The K\"ahler
background supplies the smooth plurinegative metric required by
Proposition~\ref{amp:prop:smooth-ampleness}, which gives projectivity
and ample $K_X$. Put $b=-c/2>0$ and
$\kappa=(n+1)b=-\lambda_c$. By \cite{Aubin}\cite[Theorem~5]{Yau}
there is a unique K\"ahler metric $g$ with
$[\omega_g]=(2\pi/\kappa)c_1(K_X)$ and $\Ric_g^{(1)}=-\kappa\omega_g$.
For $M=g^{-1}h$ put $u=\tr M$, $v=\tr M^{-1}$, $q=\tr M^2$,
$F=-\log\det M$, $w=e^{2F}$. The geometric mean of the eigenvalues
of $M$ is $t=e^{-F/n}$. The arithmetic--geometric mean inequality
gives $u\ge nt$ and $v\ge n/t$, so
\begin{equation}\label{n:eq:metric-gap-decomposition}
 \mathscr G:=e^F(u-n)+e^{2F}(v-n)
 \ge nt^{-2n-1}(t-1)(t^{n+1}-1)\ge0.
\end{equation}
Stokes with \eqref{n:eq:eta-balanced-formula}--\eqref{n:eq:ddbar-omega-nminus1}
gives
\begin{equation}\label{n:eq:stokes-second}\equationalias{n:eq:stokes-first}
 \int\Delta_hw=2\Rea\int w\langle\partial F,\eta\rangle_h
              =\int w(|\eta|_h^2-\chi).
\end{equation}
Integrals use $dV_h$; the first sign follows from
$i\bar\partial w\wedge\partial\Theta_h=i\eta\wedge\bar\partial w\wedge\Theta_h$.
Trace
$i\partial\bar\partial F=\Ric_h^{(1)}-\Ric_g^{(1)}$ to get
\begin{equation}\label{n:eq:lap-F-general}
 \Delta_hF=\kappa(v-n)+\chi/2.
\end{equation}
Using $\Delta_hw=2w\Delta_hF+4w|\partial F|_h^2$, eliminate $\int w\chi$ and complete the square:
\begin{equation}\label{n:eq:t2-square}
 \kappa\int_Xe^{2F}(v-n)dV_h
 =-2\int_Xe^{2F}|\partial F-\eta/2|_h^2dV_h.
\end{equation}
For the eigenvalues $\lambda_i$ of $M$, put $Q=(u-n)^2+\sum_i(\lambda_i-1)^2$.
Then
\begin{equation}\label{n:eq:comparison-icl}\equationalias{n:eq:trace-shift-squares}
 u^2+q-(n+1)u=(n+1)(u-n)+Q.
\end{equation}
Using $dV_g=e^FdV_h$, Lemma~\ref{n:lem:ICL} with $a_0=-\kappa$ and
\eqref{n:eq:t2-square} give
\begin{equation}\label{n:eq:combined-comparison}
 0=\kappa\int_X\mathscr G\,dV_h
   +\int_X(|\mathsf S|^2+bQ)\,dV_g
   +2\int_Xe^{2F}|\partial F-\eta/2|_h^2\,dV_h.
\end{equation}
All integrands are nonnegative, so continuity gives $Q=0$ everywhere.
Each eigenvalue of $M$ then equals one, and $h=g$.
K\"ahler polarization gives the space-form tensor of HSC $c$;
compactness and \cite[Chapter~IX, \S7, Theorem~7.9]{KN69}
identify the universal cover with complex hyperbolic space.
\end{proof}

\subsection{Zero Chern HSC: Ricci-flat comparison}\label{z:sec:zero}\label{z:lem:parallel-endomorphism}
\begin{proof}[Proof of Theorem~\ref{n:thm:main} when $c=0$]
By \cite[Theorem~1.1(2)]{TangNef}, $c_1(X)=0$ in $H^2(X,\R)$.
The common K\"ahler background and \cite[Theorem~2]{Yau} give a
Ricci-flat K\"ahler metric $g$. Lemma~\ref{n:lem:ICL} gives
$\mathsf S=0$. For $D=\nabla^g$, $K=\nabla^h-D$, the common
$(0,1)$ part gives
\begin{equation}\label{z:eq:connection-difference}
 K^p{}_{ij}=h^{p\bar\ell}D_i h_{j\bar\ell},\qquad
 \mathsf S^p{}_{ij}=\tfrac12h^{p\bar\ell}
                    (D_i h_{j\bar\ell}+D_j h_{i\bar\ell}).
\end{equation}
Thus $L=g^{-1}h$ satisfies
\begin{equation}\label{z:eq:symmetric-derivative-zero}
 D_iL^k{}_j+D_jL^k{}_i=0.
\end{equation}
By pure flatness and torsion freeness, $D_rD_iL^k{}_j$ is symmetric
in $r,i$; the differentiated equation makes it skew in $i,j$.
Interchanging the first two indices has sign $+1$, and interchanging
the last two has sign $-1$. Applying these two interchanges alternately
three times returns to the original indices with sign $-1$. Thus $\mathsf J=D'L$ satisfies $D'\mathsf J=0$.
The contracted curvature on each tensor slot vanishes because
$g$ is Ricci flat. Hence in a $g$-normal frame $\sum_jD_jD_{\bar j}\mathsf J
=\sum_jD_{\bar j}D_j\mathsf J=0$. Here the integration and adjoints use $g$.
Integration by parts gives $D''\mathsf J=0$, hence $(D')^*\mathsf J=0$ and
$\|\mathsf J\|_2^2=\langle L,(D')^*\mathsf J\rangle_2=0$.
Since $L$ is self-adjoint for $g$ and $D$ preserves $g$,
$D'L=0$ also gives $D''L=0$. Thus $Dh=0$ and $h$ is K\"ahler;
zero-HSC polarization gives flatness. A parallel coframe on the
simply connected cover integrates to a holomorphic local isometry
to $\C^n$. A complete local isometry is a covering, hence an isometry here
since $\C^n$ is simply connected. \cite{Bieberbach1912} gives the finite \'etale torus cover.
\end{proof}

\part{Local rigidity at nonzero curvature}\label{part:rigidity}
The threefold proof and the balanced fourfold classification share
the same torsion equations. We develop the common identities before
specializing the dimension and metric assumptions.

\section{Common torsion equations and preservation}\label{stage:common-source}
Two steps recur in these proofs: polarizing repeated identities and using
ordered vanishing to find a preserved subbundle. For the first, a trilinear
$D_p(u,v)$ skew in $u,v$
with zero cyclic sum is recovered from $F(p,v)=D_p(p,v)$: polarization
$Q(p,u;v):=F(p+u,v)-F(p,v)-F(u,v)=D_p(u,v)+D_u(p,v)$ gives
\begin{equation}\label{unif:eq:cyclic-recovery}
               3D_p(u,v)=Q(p,u;v)-Q(p,v;u).
\end{equation}

\subsection{Central and abelian outputs}
\label{sec:central-pure}
For the bilinear map $E$ of Section~\ref{sec:foundations},
$E(u,v)+E(v,u)=u\odot v$ and $E(u,u)=u^2$.
The pointwise identity concerns one central direction. We then apply
it inside a preserved distribution to show that a central torsion
image forces the internal torsion to vanish.

\begin{lemma}[The central \ChernPure{} row]\label{unif:lem:central-PURE-row}
Let $(T,\mathcal X,P)$ be pointwise tensors on $V$ as in
Section~\ref{sec:foundations}, satisfying \eqref{eq:pure}.
Suppose $a\in V$ satisfies
$T(a,V)=0$, $\mathcal X_a=0$ and $\mathcal X_p(a,V)=0$ for every $p$.
Then
\begin{equation}\label{unif:eq:central-PURE-row}
 E\bigl(T(T(u,v),p)-J_T(p,u,v)-2\mathcal X_p(u,v),a\bigr)=0.
\end{equation}
\end{lemma}
\begin{proof}
Put $t=T(u,v)$. Centrality gives
$b_a(T)(u,v)=t\wedge a/2$ and
$F_a(T)(z,w)=E(T(z,w),a)/2$.
Multiply the complete \ChernPure{} equation for derivatives $(a,p)$
and inputs $(u,v)$ by four. The derivative term becomes $-2E(\mathcal X_p(u,v),a)$.
The two signed compositions are
\[
 E\bigl(T(T(u,p),v)+T(t,p)+T(T(p,v),u),a\bigr),
 \qquad -E(T(t,p),a),
\]
The torsion-direction term is zero. The resulting identity is
\eqref{unif:eq:central-PURE-row} by the definition of $J_T$.
After any final linear output projection $\Pi$, the same row holds
with $\Pi E$ when $T(a,V)=0$ and the values of $\mathcal X_a$ and
$\mathcal X_p(a,V)$ satisfy $\Pi E(w,z)=0$ for every such value $w$
and every $z\in V$. In the discarded terms these values occur only
in the first slot of $E$, as seen from \eqref{eq:F}.
\end{proof}

\begin{lemma}[Central output on a preserved distribution]
\label{lem:central-pure}\label{cor:relative-central-pure}
Let $T,P$ be smooth tensors with $\mathcal X=\nabla'T$
for a partial connection, satisfying \eqref{eq:cyclic} and
\eqref{eq:pure}. Suppose smooth subbundles $Z\subset D\subset V$ satisfy
\[
 \nabla'_D D\subset D,\qquad \nabla'_D Z\subset Z,\qquad
 T(D,D)\subset Z,\qquad T(D,Z)=0.
\]
Then $\mathcal X_D(D,D)=T(D,D)=0$, with the full ambient symmetric
output retained. The ranks of $Z$ and $D$ are arbitrary, including
$D=V$ and $\rank Z=1$. Only preservation in directions tangent to
$D$ is used.
\end{lemma}
\begin{proof}
Differentiating in directions $p\in D$ gives
$\mathcal X_p(D,D)\subset Z$ and $\mathcal X_p(D,Z)=0$;
\ChernCyclic{} gives $\mathcal X_z(D,D)=0$ for $z\in Z$.
Repeat the calculation of Lemma~\ref{unif:lem:central-PURE-row}
with all inputs in $D$, keeping the full output of $E$.
Every nested torsion vanishes by $T(D,D)\subset Z$ and $T(D,Z)=0$,
so $E(\mathcal X_p(u,v),z)=0$ for $p,u,v\in D$, $z\in Z$.
Taking $z=\mathcal X_p(u,v)$ gives $\mathcal X_p(u,v)^2=0$, hence
$\mathcal X_D(D,D)=0$. The complete repeated row then gives
$3T(u,v)\odot T(u,v)/4=0$, hence $T(D,D)=0$.
\end{proof}

For any alternating $S$, the equivalent full operator is
\begin{equation}\label{newff:eq:E-form-F}
 2F_p(S,P)(u,v)=E(S(u,p),v)+E(S(u,v),p)+E(S(p,v),u).
\end{equation}
The canonical product is parallel, so \ChernFirst{} reads
$\nabla_pE=\nabla_pP=F_p(T,P)$.
\subsubsection{Uniform algebraic input}
\begin{lemma}[The complete repeated matrix row]\label{ab:lem:repeated}\label{rf:matrix}
Let $W$ be a finite-dimensional complex vector space and
$E:W\times W\to\Sym^2W$ be bilinear with $E(x,y)+E(y,x)=2xy$.
If endomorphisms $B,Q$ of $W$ satisfy
\begin{equation}\label{ab:eq:matrixrow}\equationalias{rf:matrix-row}
 E(Qx,y)+E(Qy,x)+3Bx\odot By=0,
\end{equation}
then $Q=-3B^2$ and $\ker B^2=\ker B$.
\end{lemma}
\begin{proof}
Contract the polarized diagonal identity with a symmetric matrix $S$
satisfying $QS=SQ^{\mathsf T}$. The alternating part of $E$ cancels, giving
\begin{equation}\label{ab:eq:BSB}\equationalias{rf:matrix-contraction}
                     BSB^{\mathsf T}=-QS/3.
\end{equation}
For $k\in\ker Q$ the row at $(k,k)$ gives $Bk=0$. If $k=Qx$ as
well, the row at $(x,k)$ gives $k^2=0$. Therefore $\ker Q$ meets
$\im Q$ trivially, and rank--nullity gives $W=\ker Q\oplus\im Q$.
On a nonzero-eigenvalue Jordan block $J$, take its reversal matrix $S$,
extended by zero. Since $QS=SQ^{\mathsf T}$, \eqref{ab:eq:BSB} gives $B(J)=J$
by comparing images. Apply the identity also to $QS$; on $J$, the
right side of the new identity is $Q$ times the first right side, and
the first right side is invertible. Thus $BQB^{-1}=Q$ on $J$.
A map commuting with a single Jordan block is a polynomial in that
block, so $BS=SB^{\mathsf T}$ and the first identity gives $B^2=-Q/3$ on $J$.
Together with $B|_{\ker Q}=0$ this gives $Q=-3B^2$ and $\ker B^2=\ker B$.
\end{proof}

\begin{lemma}[Fixed-output pencil obstruction]\label{sharedE:theorem}
Let $W$ be a finite-dimensional complex vector space, and let
$E:W\times W\to\operatorname{Sym}^2W$ be bilinear with
$E(x,y)+E(y,x)=2xy.$
Suppose a complex linear subspace $\mathcal L\subset\operatorname{End}W$
satisfies, for every $D\in\mathcal L$,
\begin{equation}\label{sharedE:equation}
 E(D^2x,y)+E(D^2y,x)=2(Dx)(Dy)\qquad(x,y\in W).
\end{equation}
Then $\dim_{\mathbb C}\mathcal L\le1$.
\end{lemma}

\begin{proof}
The diagonal gives $\ker D^2=\ker D$, so $\mathcal L$ has no nonzero
nilpotent member. If $\mathcal L\ne0$, choose $S$ of maximal rank
and split $W=I\oplus K$ with $I=\im S$, $K=\ker S$.
Then $S_I=S|_I$ is invertible. For
$D=\left(\begin{smallmatrix}A&B\\ C&F\end{smallmatrix}\right)$,
the linear coefficients of the bordering minors of $S+tD$ are the
entries of $\det(S_I)F$, so $F=0$.

The $S$ identity gives $E(I,K)=0$ and hence $E(k,i)=2ki$ for
$k\in K$, $i\in I$. With $q:W\to I$ the projection along $K$, define
$\widetilde E(x,y)=\Sym^2(S_I^{-1}q)E(S_I^2x,y)$ for $x,y\in I$.
The $S$ identity gives $\widetilde E(x,y)+\widetilde E(y,x)=2xy$.
For $x,y\in I$ the $K$-components of $D^2x,D^2y$ do not contribute
after projection, while the $I$-component of $D^2|_I$ is $A^2+BC$.
Thus
\[
 \widetilde E(Qx,y)+\widetilde E(Qy,x)=2(Rx)(Ry),
 \qquad Q=S_I^{-2}(A^2+BC),\quad R=S_I^{-1}A.
\]
Lemma~\ref{ab:lem:repeated} applied to $(R,-3Q)$ gives $Q=R^2$.
The linear family of $R$ contains $\Id_I$. Apply the identity to
$\Id_I+tR$ and evaluate it on $(x,Rx)$. The terms in $(Rx)^2$ on
the two sides agree at order $t$; the other terms at that order give
$\widetilde E(R^2x,x)=x(R^2x)$. The diagonal $R$ identity now
gives $(Rx)^2=x(R^2x)$. Modulo $\C x$ this forces $Rx\in\C x$
for every $x$, hence $R=c\Id_I$. It follows that $A=cS_I$, $BC=0$
and $(D-cS)^3=0$. As $D-cS\in\mathcal L$, this gives $D=cS$.
\end{proof}

\begin{lemma}[Common symmetrizers with ambient output]\label{ab:lem:symmetrizers}
Let $V$ be a finite-dimensional complex vector space, let $W\subset V$
be a complex subspace, and let the complex-bilinear map
$E:W\times W\to\Sym^2V$ have symmetric part
$E(x,y)+E(y,x)=x\odot y$. Then
\[
 \mathcal L_E=\{L\in\operatorname{End}W:E(Lx,y)=E(Ly,x)\}
\]
has dimension at most one. Moreover $\im L\cap\ker L=0$ for every $L\in\mathcal L_E$.
\end{lemma}
\begin{proof}
Apply the defining identity at $(Lx,y)$ and $(Ly,x)$:
$E(L^2x,y)+E(L^2y,x)=E(Ly,Lx)+E(Lx,Ly)=2(Lx)(Ly)$.
Use a fixed retraction of the output to $\Sym^2W$ and apply Lemma~\ref{sharedE:theorem}
to obtain $\dim\mathcal L_E\le1$.
Its diagonal gives $\ker L^2=\ker L$, hence $\im L\cap\ker L=0$.
\end{proof}

\begin{lemma}[Scalarization on an abelian ideal]
\label{r17:lem:ideal-scalarization}\label{r18:cor:parallel-ideal-splitting}\label{r12:lem:zero-horizontal-scalarization}
Let a smooth tensor system with a connection in the sense of
\S\ref{sec:foundations} satisfy complete \ChernFirst{} and \ChernPure{},
with $\mathcal X=\nabla'T$, $P:\Lambda^2V\to\Sym^2V$ and
$E(u,v)=uv+P(u,v)$.
Suppose $N$ is a smooth prime-preserved torsion ideal, $T(N,N)=0$,
and $\mathcal X_a(q,b)=0$ for $a,b\in N$, $q\in V$.
Put $B_q=T(q,\cdot)|_N$, $D_q=\mathcal X_q(q,\cdot)|_N$ and
$\mu(q)=\tr_NB_q$. Then $D_q=-2B_q^2$, and every nonzero $B_q$
is a nonzero scalar multiple of a projection.
On $\mu\ne0$, locally for a constant positive integer $r$,
\[
 \begin{gathered}
 B_q=\mu(q)S,\quad\tr S=1,\quad S^2=S/r,\quad\nabla'S=0,\\
 \Sym\nabla'\mu=-2\mu^2/r,\qquad
 (\nabla'\mu)(N,V)=(\nabla'\mu)(V,N)=0.
 \end{gathered}
\]
The ambient prime-preserved splitting $N=I\oplus K$, with
$I=\im S$, $K=\ker S$, satisfies
\[
 B_q|_I=\frac{\mu(q)}r\Id_I,\qquad T(V,K)=0,\qquad E(I,K)=0,
\]
where the last equality has the whole original symmetric output.
If $\im T\subset N$, all conclusions hold with the ambient trace
$\tau(q)=\tr_VT(q,\cdot)=\mu(q)$.
If $N$ is a hyperplane, then on every nonzero constant-rank region
$\mathcal X_p=-2\alpha(p)T$, with $\alpha=\mu/r$ and $r=\rank T$.
\end{lemma}
\begin{proof}
Preservation and the ideal identities give
$\mathcal X_p(V,N)\subset N$ and $\mathcal X_p(N,N)=0$.
The assumption $\mathcal X_a(q,b)=0$ makes the second corrected vector
in \eqref{eq:repeated} vanish.
For $B=B_q$, $D=D_q$, $Q=2D+B^2$, polarized complete \ChernPure{} gives
\begin{equation}\label{ab:eq:H4}
 E(Qa,b)+E(Qb,a)+3Ba\odot Bb=0.
\end{equation}
Apply the retraction to the output in $\Sym^2N$ only, then use
Lemma~\ref{rf:matrix}:
\begin{equation}\label{ab:eq:D}
 D=-2B^2,\qquad\ker B^2=\ker B.
\end{equation}
The unprojected identity now gives
\begin{equation}\label{ab:eq:diag}
 E(B^2x,y)+E(B^2y,x)=2(Bx)(By).
\end{equation}
Let $\mathscr D$ differentiate along $q$, including the evaluation slot.
It satisfies $\mathscr DB=-2B^2$. \ChernFirst{} gives
\begin{equation}\label{ab:eq:H1}
 (\mathscr DE)(x,y)=x(By)-\tfrac12\{E(Bx,y)+E(x,By)\}.
\end{equation}
Set $L=E(B^3x,x)$, $M=E(B^2x,Bx)$, $Z_0=(B^2x)(Bx)$.
Differentiating the diagonal row and evaluating its polarization
at $(x,Bx)$ give $Z_0-(L+M)/2-4L=-4Z_0$ and $L+M=2Z_0$.
Thus $L=M=Z_0$, in particular
\begin{equation}\label{ab:eq:Pkey}
 P(B^2x,Bx)=0.
\end{equation}
On $I=\im B$, $B$ is invertible. Polarization gives
$P(Bu,v)=P(u,Bv)$, hence $P(B^2x,x)=0$.
The diagonal row becomes $x(B^2x)=(Bx)^2$; reduction modulo
$\C x$ forces $Bx\in\C x$. Therefore $B|_I=\rho\Id_I$,
with $\rho\ne0$ for $B\ne0$, proving scalar projection.

Every nonzero member of the linear image $q\mapsto B_q$ has nonzero
trace, so $B_q=\mu(q)S$, $\tr S=1$, $S^2=S/r$.
Here $r$ is locally constant because $\tr S^2=1/r$ has finite range.
Write $\Pi=rS$, $\alpha=\mu/r$. Tracing \eqref{ab:eq:D} gives
$\Sym\nabla'\alpha=-2\alpha^2$; tracing horizontal vanishing
and differentiating $\alpha|_N=0$ give both stated annihilations.
Differentiating all the slots of $B_q=\alpha(q)\Pi$ gives
$-2\alpha(q)^2\Pi=-2\alpha(q)^2\Pi+\alpha(q)\nabla'_q\Pi$.
Thus $\nabla'_q\Pi=0$ on the dense set of directions $\alpha(q)\ne0$
at each point, so $\nabla'\Pi=0$.
The image and kernel give the preserved splitting, and
\eqref{ab:eq:diag} on $I\times K$ gives the full $E(I,K)=0$.
If $\im T\subset N$, the quotient action is zero, so $\tau=\mu$.

For a hyperplane, alternation on $V/N$ already gives $\im T\subset N$.
Choose $\nu(e)=1$, $\ker\nu=N$. Horizontal vanishing gives
$\mathcal X_N=0$, while \eqref{ab:eq:D} gives
$\mathcal X_e(e,b)=-2\rho B_eb$, $B_e^2=\rho B_e$.
Thus $\mathcal X_p=-2\rho\nu(p)T$ and $\mu=r\rho\nu$.
\end{proof}

\begin{theorem}[Trace-free preserved output]
\label{unif:thm:balanced-abelian-source}\label{thm:conceptual-lowrank}
Let $V$ carry a smooth partial connection $\nabla'$ and tensors
$T,P$ as in Section~\ref{sec:foundations}. With $\mathcal X=\nabla'T$,
assume \eqref{eq:transport}, \eqref{eq:cyclic}, \eqref{eq:pure},
and $\tr T(q,\cdot)=0$ for every $q$.
Suppose a smooth prime-preserved subbundle $N$ contains $\im T$.
If $T(N,N)=0$ or $\rank N\le2$, then $T=0$.
\end{theorem}
\begin{proof}
The trace condition eliminates the mixed derivative first, then the
transverse torsion operators.
Since $\im T\subset N$, the internal torsion traces equal the zero
ambient traces. In rank two these are the coefficients of $T(e_1,e_2)$;
in smaller rank alternation suffices. Thus $T(N,N)=0$.
Put $H_q(a,b)=\mathcal X_a(q,b)$ and $B_q=T(q,\cdot)|_N$.
Preservation gives $\mathcal X_p(V,N)\subset N$ and
$\mathcal X_p(N,N)=0$; \ChernCyclic{} makes $H_q$ symmetric.
Complete \ChernPure{} on $(a,b;q,d)$, $a,b,d\in N$, has zero quadratic
and torsion-direction terms, hence
$E(H_q(a,d),b)=E(H_q(b,d),a)$.
By Lemma~\ref{ab:lem:symmetrizers}, nonzero $H$ has one operator
generator $L$. Symmetry makes it rank one, and $\ker L^2=\ker L$
normalizes $L=v\otimes\lambda$, $\lambda(v)=1$.
Thus $H_q=\kappa(q)\lambda\otimes\lambda\,v$; the row at $a=d=v$
gives $E(v,b)=\lambda(b)v^2$. The trace
$\tr_N H_q(a,\cdot)=(\nabla'_a\tr_N B)(q)=0$ forces $H=0$.
By Lemma~\ref{r17:lem:ideal-scalarization}, with $H=0$ every nonzero
$B_q$ would be a nonzero multiple of a projection, contradicting its zero trace.
Now $N$ is central, and Lemma~\ref{lem:central-pure} with
$(D,Z)=(V,N)$ gives $T=0$.
\end{proof}

\Needspace{7\baselineskip}
\subsection{Ordered vanishing and parallel torsion ideals}
\label{sec:conceptual-lowrank}

\label{sec:conceptual-fullcut}
Preservation follows from complete \ChernFirst{} in the form
\eqref{newff:eq:E-form-F}.

\begin{proposition}[Preservation from ordered vanishing]
\label{prop:conceptual-fullcut}
For a nonzero proper smooth subbundle $N\subset V$, the actual field identity
$E(N,V)\subset\Sym^2N$ implies $\im T\subset N$ and $\nabla'_VN\subset N$.
\end{proposition}
\begin{proof}
Put $q:V\to W=V/N$ and $\beta_pn=q\nabla'_pn$.
By canonical symmetry, $E(V,N)\subset N\odot V$. Thus the double
quotient defines $\bar E$ with $\bar E(w,w)=w^2$.
Differentiate the cut, retaining both moving inputs. The output
derivative still has an $N$ factor, so \ChernFirst{} gives
\begin{equation}\label{eq:conceptual-fullcut-quotient}
 \bar E(\beta_pn+qT(n,p)/2,qv)
       +\bar E(qT(n,v),qp)/2=0.
\end{equation}
For $v\in N$, take $qp=qT(n,v)$ to get $T(N,N)\subset N$.
Then for $p\in N$, take $qv=\beta_pn+qT(n,p)/2=\beta_pn$
to get $\nabla'_NN\subset N$.
Thus the output and input motions in the derivative along $n$ of
$E(n,v)\in\Sym^2N$ stay in $\Sym^2N$.
Repeated \ChernFirst{}, $(\nabla'_nP)(n,v)=E(T(n,v),n)$,
has mixed quotient $n\otimes qT(n,v)=0$.
Therefore $T(N,V)\subset N$, and the displayed row gives $\beta=0$.
Now differentiate the cut in any direction. In
$2\nabla'_pP(n,v)=E(T(n,p),v)+E(T(n,v),p)+E(T(p,v),n)$,
only the last term can have a mixed quotient, namely
$n\otimes qT(p,v)$. For any nonzero $n\in N$, the equality $n\otimes qT(p,v)=0$
then gives $qT(p,v)=0$ for all $p,v$. Thus $\im T\subset N$.
\end{proof}
\subsection{A common low-rank metric obstruction}
\begin{lemma}[Normal connection blocks under ordered vanishing]
\label{unif:lem:ordered-cut-normal-shape}
Let the original metric have constant Chern HSC $c\ne0$, and let
$I\subset V$ be a nonzero proper smooth subbundle satisfying
\begin{equation}\label{r43:eq:codimension-two-entrance}
 \im T=I,\quad \nabla'_VI\subset I,\quad
 E(I,V)\subset\Sym^2I,\quad (\Sym^2q_I)P=0 .
\end{equation}
Put $Q=I^\perp$, with $D$ the orthogonal block connection. Write
$\nabla'=D'+A$, $A_p:Q\to I$, and $\nabla''=D''+G$,
$G_{\bar p}=-A_p^\dagger$. Then $Q$ is holomorphic and $A_pr=-A_rp$
for $p,r\in Q$. Put
$B_pa=T(p,a)$, $S_pa=A_ap$, $\mathfrak t(p,r)=T(p,r)$ and
$\mathfrak w(p,r)=-A_pr$ for $a\in I$ and $p,r\in Q$. Thus
$B_p,S_p:I\to I$, and subscripts on $\mathfrak t,\mathfrak w$ denote
first-slot contractions. The dagger denotes the Hermitian adjoint. Then
\begin{equation}\label{unif:eq:ordered-cut-mixed-pair}
 B_r^\dagger S_p=S_r^\dagger B_p,\qquad
 \pi_I R(a,\bar r)p=\tfrac12 S_r^\dagger B_pa,
\end{equation}
and
\begin{equation}\label{r16:eq:ranktwo-shape-rows}
 \begin{aligned}
 (D''_{\bar r}B)(p)&=S_r^\dagger B_p
                  +\mathfrak t_p\mathfrak w_r^\dagger-2ch(p,r)\Id,\\
 (D''_{\bar r}S)(p)&=\mathfrak w_p\mathfrak w_r^\dagger
                                      -\tfrac12S_r^\dagger B_p .
 \end{aligned}
\end{equation}
Here $q_I:V\to V/I$ and $D$ differentiates all family and tensor slots.
\end{lemma}
\begin{proof}
Metric compatibility gives the block formulas and makes $Q$ holomorphic.
For $a\in I$ and a unitary frame $(q_j)$ of $Q$, applying
\ChernBarred{} with both output slots projected to $Q$ gives
$\sum_jq_j\odot G_{\bar q_j}a=0$. Taking adjoints yields $A_pr=-A_rp$, so $G_{\bar r}=\mathfrak w_r^\dagger$.
The ordered cut, followed by Hermitian adjunction, gives
\begin{equation}\label{fourh:strong-cut-curvature}
 R(q,\bar a)a'+R(a',\bar a)q=0
 \qquad(a,a'\in I,\ q\in Q):
\end{equation}
the coefficient with the index raised is
$h(R(a,\bar q)v,a')+h(R(a,\bar a')v,q)=0$.
In \ChernBarred{}, $\pi_QY_{\bar z}=G_{\bar z}T$, and the barred-$I$
mixed-output contribution is $-S_r^\dagger T$. Hence
$Y_{\bar r}(p,a)=(G_{\bar r}+S_r^\dagger)B_pa-2ch(p,r)a$.
Alternating reality~\eqref{newunb:r1v3:alternating-reality} compares
$2h(B_pa,S_ra')$ and $2h(S_pa,B_ra')$, proving the first
mixed-pair identity. Bianchi and \eqref{fourh:strong-cut-curvature}
give $\pi_QR(r,\bar a)a'=-G_{\bar a}B_ra'/2$;
adjunction and that identity give the second.
The $I$ component of \ChernBarred{}, including $T(p,G_{\bar r}a)$,
is the derivative of $B$. For $S$, take a $D$-normal frame: zero mixed
torsion gives $[a,\bar r]=-G_{\bar r}a$. Since $\nabla''Q\subset Q$
and $D''_{\bar r}p=0$ at the point, the three curvature terms
project to $0$, $-(D''_{\bar r}S)(p)a$, and
$A_{G_{\bar r}a}p=\mathfrak w_p\mathfrak w_r^\dagger a$.
\end{proof}

\begin{lemma}[The low-rank ordered cut in codimension at least two]
\label{r43:lem:codimension-two-cut}
For the original metric with $c\ne0$, no nonempty open set satisfies
\eqref{r43:eq:codimension-two-entrance} with
$1\le\rank I\le2$ and $\rank I^\perp\ge2$.
\end{lemma}
\begin{proof}
In Lemma~\ref{unif:lem:ordered-cut-normal-shape}'s notation, write
$P|_{Q^2}=W+M\in\Sym^2I\oplus IQ$. By \ChernFirst{},
$(D'_pM)(q,r)=q(A_pr)-(A_pq)r$. Fix orthonormal $q_1,q_2$ and set
$(t,w)=(\mathfrak t,\mathfrak w)(q_1,q_2)$, $b=t+2w$,
$\mu_{\bar r}=(\Id_I\otimes h(\cdot,r))M(q_1,q_2)$.
Then \ChernBarred{} gives $Y_{\bar r}(q_1,q_2)=(G_{\bar r}+S_r^\dagger)t-2c\mu_{\bar r}$.
Alternating reality on $QQ$--mixed and $QQ$--$QQ$ inputs gives
\begin{equation}\label{r16:eq:ranktwo-pair-rows}
 c\mu_{\bar r}=S_r^\dagger t-B_r^\dagger w,\qquad
 \mathfrak w_r^\dagger t=\mathfrak t_r^\dagger w.
\end{equation}
The first compares $2h(t,S_ra)-2ch(\mu_{\bar r},a)$ with
$2h(w,B_ra)$; the second is
$h(\mathfrak t(u,v),\mathfrak w(r,z))=
h(\mathfrak w(u,v),\mathfrak t(r,z))$.
For $p\in\Span(q_1,q_2)$, pure flatness and mixed \ChernFirst{} give
$D'_pw=-S_pb$, $(D'_p\mu)_{\bar r}=-h(p,r)w$;
here $T^D|_{Q^2}=\mathfrak t+2\mathfrak w$.
Differentiate the first pair row on all slots, using
\eqref{r16:eq:ranktwo-shape-rows} and
\eqref{unif:eq:ordered-cut-mixed-pair} and the prime-preserved
$\bar Q$ family. For $v_p=(D'_p\mathfrak t)(q_1,q_2)+B_pb/2$,
the complete result is
\[
 S_r^\dagger v_p+
 \mathfrak w_p(\mathfrak w_r^\dagger t-\mathfrak t_r^\dagger w)
 =-3ch(p,r)w.
\]
The correction vanishes by \eqref{r16:eq:ranktwo-pair-rows}. If
$w\ne0$, evaluate the differentiated equation at $p=q_i,r=q_j$,
$i,j\in\{1,2\}$, to obtain two independent vectors $v_{q_1},v_{q_2}$.
This is impossible for $\rank I=1$; for $\rank I=2$ the two vectors
form a basis. Hence in the latter case $\im S_r^\dagger\subset\C w$
and $\bigcap_r\ker S_r^\dagger=0$. Since the $S_r^\dagger$ are not all zero,
every nonzero $\mathfrak w(u,v)$ spans this same line. Thus
$F=w^\perp=\bigcap_p\ker S_p$ satisfies $\mathfrak w_Q^\dagger F=S_QF=0$
and $B_QF=0$ by \eqref{unif:eq:ordered-cut-mixed-pair}. For $0\ne f\in F$,
differentiation gives
\[
 S_pD''_{\bar r}f=0,\qquad B_pD''_{\bar r}f=2ch(p,r)f.
\]
The differentiated family slots contribute zero on $F$.
The first identity puts $D''_{\bar r}f$ in $F$, but the second
rules this out for $p=r\ne0$. Therefore
\begin{equation}\label{fourh:normal-expansions}
\equationalias{fourh:normal-motions}\equationalias{fourh:normal-AA}
\equationalias{fourh:normal-U}\equationalias{fourh:normal-matrix}
\equationalias{fourh:normal-B}\equationalias{fourh:H-mixed-row}
 G|_{\bar Q}=0.
\end{equation}
Metric compatibility, the mixed curvature commutator with
$\nabla'_V\bar Q\subset\bar Q$, and adjunction give
\begin{equation}\label{fourh:Q-directions}
 \nabla'_QQ\subset Q,\quad\nabla''Q\subset Q,\quad
 \nabla''_{\bar Q}I\subset I,\quad R(V,\bar Q)I\subset I,\quad
 R(Q,\bar V)Q\subset Q.
\end{equation}
By \eqref{fourh:Q-directions}, $Y_{\bar Q}(Q,Q)$ lies in $I$, and Bianchi
puts it in $Q$, so it vanishes. Likewise $Y_{\bar Q}(I,Q)\subset I$ gives
$R(I,\bar Q)Q\subset I$. Adjunction and $R(Q,\bar V)Q\subset Q$ give
$R(Q,\bar I)Q=0$. Bianchi and \ChernBarred{} now imply
\begin{equation}\label{fourh:whole-Q-cut}
 Y_{\bar V}(Q,Q)=0,\qquad P(Q,Q)=0.
\end{equation}
As $M=0$, \ChernFirst{} gives $q(A_pr)=(A_pq)r$. For independent
$q,r\in Q$ the two terms lie in disjoint subspaces of $IQ$,
so $A_p=0$, hence $G_{\bar p}=0$ and
\begin{equation}\label{r6:eq:whole-cut-normal-FIRST}
 I\text{ and }Q\text{ are both Chern-parallel}.
\end{equation}
Since $\nabla''T$ takes values in $I$, \eqref{p:eq:firstBi} gives $R(I,\bar V)Q=0$,
hence $E(I,Q)=0$ by \eqref{eq:P}. \ChernFirst{} at $(p,u,v)=(a,a,q)$ gives
$E(B_qa,a)=0$, so $E(a,a)=a^2$ excludes nonzero eigenvalues of $B_q$.
Now $\tr B=0$ contradicts the trace of \eqref{r16:eq:ranktwo-shape-rows}
at $p=r\ne0$, since $S=\mathfrak w=0$.
\end{proof}

\subsection{Generic kernel spans and ordered vanishing}
We now turn a rank condition on $P$ into the ordered vanishing used
above. For the given tensor $P$, define the linear pencil
\[
 \mathcal B_z\nu=2P^\vee(\nu,z)-\nu\wedge z.
\]
All covector and exterior slots carry the induced connection.
Let $R_P\subset V^*$ be the span of its generic kernels, not presumed
connection-invariant.

\begin{lemma}[Generic kernel spans and ordered vanishing]
\label{lem:kernel-spanning-pencil}\label{lem:kernel-span-alternative}
If the generic rank of $\mathcal B_z$ is less than $n=\dim V\ge3$,
then either $0\ne R_P\subsetneq V^*$, or
\[
                \mathcal B_z(Uz)=0\quad\hbox{for every }z
\]
for an invertible linear map $U:V^*\to V^*$.
In the proper-span case, the annihilator $N=\Ann R_P$ satisfies the full ordered cut
$E(N,V)\subset\Sym^2N$.
On an actual rank-loss region where $R_P$ is everywhere proper, a smaller
nonempty region carries $R_P$ and $N$ as smooth subbundles.
\end{lemma}
\begin{proof}
Put $R=R_P$, $d=\dim R$, $Y=\sum_z\mathcal B_z(R)$, and
$K=\bigcap_z\ker(\mathcal B_z|_R)$.
The quotient pencil on $R/K$ has essential range $Y$, zero common
kernel and spanning generic kernels. If $Y\ne0$, it is semi-primitive:
a hyperplane lowering upper rank would contain every generic kernel.
For $e=\dim(R/K)$ and upper rank $r\le e-1$,
Atkinson's theorem in the form of \cite[Theorem~1]{Pazzis2014Primitive} gives
$\dim Y\le\binom{r+1}{2}\le\binom e2\le\binom d2$.
The bound also holds when $Y=0$. Taking $z,\nu\in R$ in $\mathcal B_z\nu-\mathcal B_\nu z=-2\nu\wedge z$
shows that $Y$ contains every $\nu\wedge z$ with $z,\nu\in R$.
The dimension bound then gives $Y=\Lambda^2R$.

If $R=V^*$, equality forces $K=0$, $r=n-1$.
The equality case in this semi-primitive form identifies
the pencil with the $n$-dimensional wedge space, with an isomorphic
parameter map. Thus
$\mathcal B_z\nu=G(Hz\wedge F\nu)$ for invertible $F,H,G$;
take $U=F^{-1}H$. This is a pointwise choice.
For proper $R$, let $N=\Ann R$.
Since $N\wedge V=\Ann(\Lambda^2R)$, evaluation gives
$2P(n,v)(\nu,z)=-\nu(v)z(n)$ for $\nu\in R$.
The canonical part cancels this term, so all contractions of
$E(n,v)$ with $R$ vanish: $E(N,V)\subset\Sym^2N$.

Fix a maximal nonzero pencil minor. Solve for its pivot variables
and clear the minor from the denominators. This gives polynomial
kernel representatives of bounded degree, with smooth coefficients
whose span is $R_P$.
A maximal nonzero coefficient minor then gives a smaller chart
on which $R_P$ and $N$ are smooth and have constant rank.
\end{proof}

\phantomsection\label{unif:plane-criterion}
An alternating product $S$ preserving every two-plane is vectorial:
$S(u,\cdot)$ is scalar on $V/\C u$ for $u\ne0$, with scalar
$\beta(u)=\tr S(u,\cdot)/(\dim V-1)$.
Skewness gives $S=B_\beta$ in the convention~\eqref{unif:reference:H}.
In dimension at least three, $z\wedge(z\circ S)=0$ for all $z$ suffices.
Evaluate on $(w,h,k)$ with $z(w)=1$ and $h,k\in\ker z$: each hyperplane,
hence each two-plane by intersection, is preserved.

\begin{lemma}\label{lem:invertible-linear-kernel-first}
On an open set satisfying \eqref{eq:transport} where the generic
rank of $\mathcal B_z$ is less than $n\ge3$, suppose that at a point $x$
there is an invertible $U:V_x^*\to V_x^*$ with $\mathcal B_z(Uz)=0$ for
every $z\in V_x^*$. Then the torsion at $x$ is vectorial:
$T(u,v)=\alpha(v)u-\alpha(u)v$ for some $\alpha\in V_x^*$.
If the torsion trace vanishes, then $T=0$ at that point.
\end{lemma}
\begin{proof}
Put $A=U^{-1}$. Alternation of the kernel identity gives
$\mathcal B_z\nu=-M(A\nu\wedge z)$, and canonical symmetry gives
$MA^{[2]}=2I$, where $A^{[2]}(v\wedge w)=Av\wedge w+v\wedge Aw$.
Thus $M=2(A^{[2]})^{-1}$ and $\ker\mathcal B_z=\C Uz$ for $z\ne0$.
The rank loss gives a local constant-rank chart for the kernel.

For a two-form $\tau$, set
$\mathcal A_p\tau(u,v)=\{\tau(T(u,p),v)+\tau(T(u,v),p)+
\tau(T(p,v),u)\}/2$.
Substitution of
$E(u,v)(\nu,z)=\mathcal B_z\nu(u,v)/2+\nu(u)z(v)$ in full \ChernFirst{} gives
\[
 \nabla'_p(\mathcal B_z\nu)=\mathcal A_p\mathcal B_z\nu
 +\mathcal B_z\nabla'_p\nu+\mathcal B_{\nabla'_pz}\nu
 +z\wedge\iota_p(T^\vee\nu)+z(p)T^\vee\nu.
\]
Differentiate an actual kernel section with $\nabla'_pz=0$ at the point,
reduce modulo $\im\mathcal B_z$, apply $M^{-1}$ and wedge with $z$.
For $\theta(z)=T^\vee(Uz)$, polynomial continuation yields
\begin{equation}\label{eval:range}
 z\wedge Az\wedge\iota_p\theta(z)
       +z(p)z\wedge A^{[2]}\theta(z)=0.
\end{equation}
If $A$ is nonscalar, generic $z,Az$ are independent. For $p\in\ker z$,
this makes the alternating form $\theta(z)|_{\ker z}$ have rank at
most one, hence zero. Therefore $z\wedge\theta(z)=0$ for all $z$.
For scalar $A=aI$, the row gives this conclusion directly.
Apply the \hyperref[unif:plane-criterion]{plane criterion} to
$S=U^{\mathsf t}T$, for which $z\circ S=\theta(z)$.
This gives $\theta(z)=z\wedge\gamma$ for a fixed covector $\gamma$.
Substitution in \eqref{eval:range} gives
$2z(p)z\wedge Az\wedge\gamma=0$.

If $\gamma=0$, then $T=0$. Otherwise replace $z$ by $z+t\gamma$:
$A\gamma\in\C\gamma$, and the induced map on $V^*/\C\gamma$
preserves every line, hence is scalar. Thus
$A=aI+\gamma\otimes\lambda$ with $a\ne0$, and
$T^\vee\alpha=A\alpha\wedge\gamma=a\alpha\wedge\gamma$.
Equivalently $T(u,v)=a\{\gamma(v)u-\gamma(u)v\}$.
Its trace $-(n-1)a\gamma$ proves the balanced assertion.
\end{proof}

\begin{lemma}[Ordered vanishing for a rank-deficient pencil]
\label{unif:lem:nonvectorial-rank-loss-cut}
On an open set satisfying \eqref{eq:transport}, with nonvectorial
torsion at each point and generic rank $\mathcal B_z<n$, one has $0\ne R_P\subsetneq V^*$.
After shrinking, $N=\Ann R_P$ is smooth and
$E(N,V)\subset\Sym^2N$.
\end{lemma}
\begin{proof}
A full kernel span would give an invertible linear kernel by
Lemma~\ref{lem:kernel-span-alternative}, hence vectorial torsion by
Lemma~\ref{lem:invertible-linear-kernel-first}, a contradiction.
After shrinking, Lemma~\ref{lem:kernel-span-alternative} gives a
smooth $N$ with $E(N,V)\subset\Sym^2N$.
\end{proof}

\subsection{Complete recurrence in arbitrary dimension}
\label{sec:common-recurrence}
Use the partial connection of zero $(2,0)$ curvature, retaining
the torsion term in the covariant commutator. Put $L_u=T(u,\cdot)$ and use $\mathfrak H$ from
\eqref{unif:reference:H}.
Suppose $\mathcal X=\mathfrak H$ holds throughout an open set. Define
$\mathfrak K_{pq}=d\mathfrak H_q[\mathfrak H_p]
-d\mathfrak H_p[\mathfrak H_q]+\mathfrak H_{T(p,q)}$.
The pure commutator gives $\mathfrak K=0$, including the last term.
For $S_{pq}=[L_p,L_q]+L_{T(p,q)}$, direct operator expansion gives
\begin{equation}\label{u3c:recurrence-commutator}
 \mathfrak K_{pq}(x,\cdot)
 =-[S_{pq},L_x]+L_{S_{pq}x}+2L_{T(p,q)}L_x,
 \qquad \mathfrak K_{uv}(u,v)=3J_T(u,v,T(u,v)).
\end{equation}
The reference derivative of $L_x$ is
$-\{L_p,L_x\}-L_{T(p,x)}$. Substitution twice cancels the
three-operator terms in pairs, except
$-[S_{pq},L_x]+2L_{T(p,q)}L_x$; the differentiated arguments give
$L_{S_{pq}x}$. Evaluation on $(p,q)$ gives the second identity.
For a Lie product $S_{pq}=2L_{T(p,q)}$, so the first two terms cancel
and $\mathfrak K_{pq}(u,v)=2T(T(p,q),T(u,v))$.

In dimension three write $J_T=j\otimes\Omega$. If $j\ne0$, the
repeated identity in \eqref{u3c:recurrence-commutator} forces
$T(u,v)\in\Span\{u,v\}$ for every pair.
The \hyperref[unif:plane-criterion]{plane criterion} gives $T=B_\alpha$,
whose Jacobiator vanishes, a contradiction. Thus $T$ is Lie; its
specialization of \eqref{u3c:recurrence-commutator} gives
\begin{equation}\label{unif:eq:three-recurrent-derived}
 J_T=0,\qquad T(\im T,\im T)=0,\qquad \rank T\le2.
\end{equation}

\begin{lemma}[Nilpotent binary Lie products]
\label{lem:unif-lowdim-binarylie}
Let $T:\Lambda^2V\to V$ be alternating on a finite-dimensional
complex vector space. If every $L_u=T(u,\cdot)$ is nilpotent and
$J_T(u,v,T(u,v))=0$, then $T$ is nilpotent.
\end{lemma}
\begin{proof}
This is the binary Lie identity, with right multiplication $-L_u$.
Finite dimensionality gives the ascending chain condition, so
\cite[p.~772, Theorem]{Kuzmin1967} gives nilpotence:
$V^1=V$, $V^{j+1}=T(V,V^j)$ terminates at zero.
\end{proof}

\begin{theorem}[Vanishing for a trace-free recurrent torsion tensor]
\label{thm:unif-lowdim-recurrent}
Let an involutive complex tangent distribution of finite rank
carry an actual partial connection of zero pure curvature, with torsion
$T$ and $\mathcal X=\nabla'T$. Suppose complete \ChernCyclic{} and \ChernPure{} hold,
the actual torsion trace vanishes, and the full reference recurrence
$\mathcal X=\mathfrak H(T)$ of \eqref{unif:reference:H} holds. Then $T=0$.
\end{theorem}
\begin{proof}
For $L_x=T(x,\cdot)$ the recurrence gives
$(\nabla'_xL)(x)=-2L_x^2$. At a given point choose an evaluation
field with $\nabla'_xx=0$. Then $x\tr L_x^j=-2j\tr L_x^{j+1}$.
Starting from $\tr L_x=0$, all power traces vanish by induction, so
Newton's identities make every $L_x$ nilpotent.
The full commutator~\eqref{u3c:recurrence-commutator} gives
$J_T(u,v,T(u,v))=0$, so Lemma~\ref{lem:unif-lowdim-binarylie}
makes $T$ nilpotent. If $T\ne0$, its last nonzero descending term
is central and spanned by products; choose a nonzero central
$t=T(u,v)$ there. The recurrence gives
$\mathcal X_u(u,v)=-2T(u,t)=0$ and
$\mathcal X_v(u,v)=-2T(v,t)=0$.
Both corrected vectors in complete repeated \ChernPure{} vanish, leaving
$3t^2/2=0$, a contradiction.
\end{proof}

This applies to the actual restrictions of
Lemma~\ref{r37:lem:internal-restriction} after their internal zero trace
and full recurrence are established.

\section{Threefold tensor identities and metric reductions}\label{stage:threefold}
For Theorem~\ref{p:thm:main}, the regular nonvectorial pencil leads
from the scalar law \eqref{u3c:scalar-law} to the full recurrence
\eqref{u3c:actual-recurrence}. The singular nonvectorial pencil
instead gives ordered vanishing. These prime-direction arguments
do not use a metric or barred identities. Theorem~\ref{unif:thm:three-prime-images}
aligns the images, and the metric exclusions finish the proof in
\S\ref{tloc:bothsignlocaltheorem}.

\subsection{Matrix representation in dimension three}\label{b:sec:threefold}
\label{sec:balanced-three-compression}\label{u3c:compression}
We first use the tensor identities of a rank-three partial connection of zero pure curvature,
without balance, a metric or barred equations.
In an actual frame indexed by $0,1,2$, let $\epsilon_{012}=1$ and put
\[
 T^l{}_{ik}=\epsilon_{ika}C^{al},\qquad
 P_{ik}^{bd}=\epsilon_{ika}\mathsf P^{a,bd},\qquad
 \eta_i=\sum_kT^k{}_{ki}.
\]
All encodings retain their determinant-line connection. With
$[v]_\times w=v\times w$ and the complex-bilinear cross product,
\[
 A=\tfrac12(C+C^{\mathsf T}),\qquad C=A-\tfrac12[\eta]_\times.
\]
Thus $A=0$ is exactly vectorial torsion, and balance is $C=A$, not
$\tr A=0$. Under $e'=eM$, $N=M^{-1}$, the tensors transform by
$C'=\det(M)NCN^{\mathsf T}$ and $A'=\det(M)NAN^{\mathsf T}$.

\subsubsection{Torsion derivatives and auxiliary polynomials}
For a matrix $S$ put $\eta_i(S)=\epsilon_{kia}S^{ak}$ and define
\[
 \mathsf F_p^{a,bd}(S)=
 \tfrac12(S^{ab}\delta_p^d+S^{ad}\delta_p^b
                    -\delta_p^a\Sym(S)^{bd})
 +S^{ar}\epsilon_{rpe}\mathsf P^{e,bd}
 +\tfrac12\delta_p^a\eta_e(S)\mathsf P^{e,bd}.
\]
Contracting \ChernFirst{} against $\epsilon_{ika}/2$ gives
$\nabla_p\mathsf P=\mathsf F_p(C)$.
Write $T_p=T(e_p,\cdot)$ and
\begin{equation}\label{u3c:reference}
 \mathsf H_p=T_pC-CT_p^{\mathsf T}+\eta_pC,
 \qquad D_p=\nabla_pC-\mathsf H_p,
 \qquad j=-C\eta=-A\eta,
 \qquad G=A-\eta_a\mathsf P^a.
\end{equation}
The matrix $\mathsf H_p$ represents the reference tensor
$\mathfrak H_p$ from \eqref{unif:reference:H} in the epsilon notation.
This follows by inserting $T^l{}_{ik}=\epsilon_{ika}C^{al}$ in its
three terms. The same contraction gives
\[
 \begin{gathered}
 T_pC+CT_p^{\mathsf T}+\eta_pC=-e_pj^{\mathsf T},\qquad \Sym\mathsf H_p=2\eta_pA+\tfrac12(e_pj^{\mathsf T}+je_p^{\mathsf T}),\\
 \eta_a(\mathsf H_p)=-4(\operatorname{adj}C)_{pa} +2\eta_p\eta_a+3\epsilon_{par}j^r,\qquad \sum_p\mathsf H_p^{pb}=\sum_p\eta_pC^{pb}.
\end{gathered}
\]
\ChernCyclic{} has the last divergence on its derivative side. Consequently
$\sum_pD_p^{pb}=0$.

Put $E_p=\Sym D_p$ and $\beta^b=\sum_pE_p^{pb}$.
There is a unique symmetric matrix $H$ such that
\begin{equation}\label{u3c:split}
 D_p^{ab}=E_p^{ab}
 +\tfrac12(\delta_p^b\beta^a-\delta_p^a\beta^b)
 -\tfrac12\epsilon_{abj}H_{pj}.
\end{equation}
The epsilon map identifies each remaining skew matrix with a vector.
The divergence constraint makes the resulting matrix symmetric. Thus
$\eta_a(D_p)=-H_{pa}+\epsilon_{par}\beta^r$.
Define auxiliary polynomials and a pencil by
\[
 p_a(z)=z^{\mathsf T}\mathsf P^az,\quad q=z^{\mathsf T}Az,\quad
 V_i=z^{\mathsf T}E_i z-2z_i\beta^{\mathsf T}z,\quad
 B=J_p^{\mathsf T}-[z]_\times,\quad \Delta=\det B,\quad J=\operatorname{adj}B.
\]
Here $(J_p)_{ab}=\partial p_a/\partial z_b$ is the auxiliary Jacobian.
If $V_i^{ab}$ denotes the symmetric coefficient matrix of $V_i$, then
$\gamma^b=\sum_iV_i^{ib}=-3\beta^b$ and
\begin{equation}\label{u3c:V-recovery}
 \beta=-\gamma/3,\qquad
 E_i^{ab}=V_i^{ab}-\tfrac13(\delta_i^a\gamma^b+\delta_i^b\gamma^a).
\end{equation}
Thus $V=0$ implies $E=\beta=0$; $H$ is recovered below.

Differentiate the encoded identity \ChernFirst{} and commute the two
prime derivatives. Index the cyclic pairs $(1,2),(2,0),(0,1)$ by
$k=0,1,2$ and write each pair as $(q_0,p_0)$. Substitution of
$\nabla C=\mathsf H+D$ in the full commutator gives
\begin{equation}\label{u3c:shift}
 \mathsf F_{p_0}^{a,bd}(D_{q_0})
 -\mathsf F_{q_0}^{a,bd}(D_{p_0})+\tfrac32A^{ka}G^{bd}=0.
\end{equation}
The linear difference contains all terms in $D$. The remaining
$\mathsf P$ terms sum to
$-3(C^{ka}+C^{ak})\eta_e\mathsf P^{e,bd}/4$, and the terms independent
of $\mathsf P$ to $3(C^{ka}+C^{ak})A^{bd}/4$.
To obtain both sums, use
$\epsilon_{ijk}\epsilon_{uvk}=\delta_{iu}\delta_{jv}-\delta_{iv}\delta_{ju}$;
they give precisely the final term in \eqref{u3c:shift}.

For $\varpi=z\times w$, contract that full equation against
$w_kw_az_bz_d$. Using \eqref{u3c:split}, it becomes
\[
 \mathcal Q(w,z)=-p(z)\cdot V(w)
 +\sum_{i,a,b}z_a\varpi^iE_i^{ab}w_b
 +\tfrac12H(\varpi,\varpi)+\tfrac32q(w)z^{\mathsf T}Gz=0.
\]
Differentiate in the second variable, keeping $w$ fixed, and then set
$w=z$. In the four terms of $-\partial_2\mathcal Q$, this gives
$J_p^{\mathsf T}V$, $-[z]_\times V$, $0$, and $-3qGz$, respectively.
For the second term the difference between $(z^{\mathsf T}E_i z)_i$
and $V$ is a multiple of $z$, which $[z]_\times$ annihilates; the
third term is quadratic in $\varpi$. Thus
$0=-\partial_2\mathcal Q|_{w=z}=BV-3qGz$. Since
$Gz=C^{\mathsf T}z-B\eta/2$, the forward identities are
\begin{equation}\label{u3c:source}
 BV=3qGz,\qquad W=V+\tfrac32q\eta,
 \qquad BW=3qC^{\mathsf T}z.
\end{equation}
If $V=G=0$, \eqref{u3c:V-recovery} gives $E=\beta=0$, so
$\mathcal Q=0$ reduces to $H(z\times w,z\times w)=0$.
Every bivector in dimension three is decomposable, so $z\times w$
runs over the whole coordinate space. Since $H$ is symmetric and
vanishes on its diagonal, polarization gives $H=0$.
Equation \eqref{u3c:split} then gives $D=0$ on every ambient slot.

\subsubsection{Exterior derivative of the torsion trace and determinant transport}
Put $\kappa=2\beta+j$. Use the trace formulas above and
\[
 (\operatorname{adj}C)_{ia}-(\operatorname{adj}C)_{ai}
                       =\epsilon_{iar}j^r
\]
to write the alternating covariant derivative of $\eta$ as
$2\epsilon_{iar}(\beta+j)^r$. The exterior derivative includes the torsion correction
$\eta(T_{ia})=-\epsilon_{iar}j^r$. Therefore
\begin{equation}\label{u3c:trace-curl}
 (\partial\eta)_{ia}=\epsilon_{iar}\kappa^r,
 \qquad \kappa(z)=-\tfrac13\operatorname{div}_zV(z)-(A\eta)^{\mathsf T}z.
\end{equation}
The second equality follows from the divergence of $V$:
$\operatorname{div}_zV=-6\beta^{\mathsf T}z$.
Taking the symmetric part of \eqref{u3c:reference} gives
\begin{equation}\label{u3c:q-derivative}
 \nabla_iq=V_i+2\eta_iq+z_i\kappa(z).
\end{equation}
With $L_i=C[e_i]_\times+e_i\eta^{\mathsf T}/2$ and $r=\tr(JC^{\mathsf T})$, \ChernFirst{} gives
\begin{equation}\label{u3c:transport}
 \begin{gathered}
 \nabla_i p=(Cz)z_i-\tfrac12e_iq+L_ip,\qquad \nabla_iB=BL_i^{\mathsf T}+2z_iC^{\mathsf T}-C^{\mathsf T}ze_i^{\mathsf T},\\
 \nabla_i\Delta=\tfrac32\eta_i\Delta+2z_ir-(JC^{\mathsf T}z)_i.
\end{gathered}
\end{equation}
The first identity follows by contracting the encoded \ChernFirst{}
with $z_bz_d$. Differentiating in $z$ and using
$[z]_\times[e_i]_\times=e_i z^{\mathsf T}-z_iI$ give the second.
The cofactor derivative, $JB=\Delta I$ and $\tr L_i=3\eta_i/2$
give the third, with no invertibility assumption on $B$.
Frame covariance gives
\[
 B'(z')=\det(M)NB(N^{\mathsf T}z')N^{\mathsf T},
 \quad q'(z')=\det(M)q(N^{\mathsf T}z'),
 \quad \Delta'(z')=\det(M)\Delta(N^{\mathsf T}z').
\]
Thus $q$ and $\Delta$ have determinant-line weight one.

For $\eta=0$, one has $C=A$, $j=-C\eta=0$, so $T$ is Lie and
\ChernCyclic{} gives $\sum_p(\nabla_pA)^{pb}=0$.
Since $\Sym\mathsf H_p=0$, this gives $E_p=\nabla_pA$ and
$\beta=\kappa=0$. Equations~\eqref{u3c:source}--\eqref{u3c:transport}
therefore reduce to
\begin{equation}\label{eq:conceptual-three-forward}
 BV=3qAz,\qquad \nabla_iq=V_i,
 \qquad \nabla_i\Delta=2z_i\tr(JA)-(JAz)_i.
\end{equation}

\subsubsection{A singular-pencil normal form}
\begin{lemma}[Two intrinsic pencil identities]
\label{lem:conceptual-three-pencil}\label{u3c:adjugate}
The pencil $B$ has no nonzero constant right kernel and satisfies
\[
 \sum_{i,j}\partial_{z_i}\partial_{z_j}J_{ij}=12.
\]
\end{lemma}
\begin{proof}
For an auxiliary-constant vector $t$,
$Bt=\nabla_z(t\cdot p)+[t]_\times z$. If this is zero, its
symmetric and skew coefficient matrices force $t=0$.
To prove the second identity, write
\[
 B_{mk}=\partial_m p_k-\epsilon_{mrk}z_r,
 \qquad J_{ij}=\tfrac12\epsilon_{ikl}\epsilon_{jmn}B_{mk}B_{nl}.
\]
In $\sum_j\partial_jJ_{ij}$ the Hessian terms cancel by antisymmetry
of $\epsilon_{jmn}$, giving
$\sum_j\partial_jJ_{ij}=2\epsilon_{ikl}B_{kl}$.
A second differentiation gives
\[
 2\epsilon_{ikl}\partial_i\partial_kp_l
       -2\epsilon_{ikl}\epsilon_{kil}=12.
\]
In particular $J$ is not the zero polynomial.
\end{proof}

\begin{lemma}[Ordered vanishing for the singular three-dimensional pencil]
\label{lem:conceptual-three-singular}
For the general three-dimensional source, on an actual region with
$q\ne0$ and $\Delta=0$, a smaller nonempty
region carries a smooth nonzero proper subbundle $N$ of rank one or
two such that $E(N,V)\subset\Sym^2N$. On the locus where the
common left kernel is zero, the construction always gives rank one.
\end{lemma}
\begin{proof}
In epsilon coordinates, $\mathcal B_z\nu$ has coefficient
$B^{\mathsf T}\nu$. The zero determinant means rank loss, while the
nonzero quadratic $q$ means $A\ne0$ and so excludes vectorial torsion.
Lemma~\ref{unif:lem:nonvectorial-rank-loss-cut} gives $N=\Ann R_P$.
With no common left kernel, $\dim R_P=1$ would, by polynomial
continuation, give a common kernel of $B^{\mathsf T}$.
Hence $\dim R_P=2$, $\rank N=1$.
Lemma~\ref{lem:conceptual-three-pencil} gives pencil rank two;
both selection ranks are constant, ensuring smoothness on that locus.
\end{proof}

\subsection{Scalar transport and recurrence in dimension three}
\label{sec:u3c-regular}
Keep the rank-three partial connection and auxiliary $z$, including the
determinant-line weights. Balancedness is not assumed.

\begin{lemma}[Weighted concircular vectors]\label{u3c:concircular}
Assume the partial connection on $V$ has zero $(2,0)$ curvature.
Let a pure-flat line bundle $L$ carry a nonzero section $a\in V\otimes L$
and a nowhere-vanishing section $\gamma\in L$. If
$\nabla_u a=\theta(u)a-\gamma u$ for every $u$, put
$\lambda=\gamma^{-1}\nabla\gamma-\theta$. Then
\[
 T(u,v)=\gamma^{-1}(\partial\theta)(u,v)a
                         -\lambda(u)v+\lambda(v)u.
\]
In dimension three, if $\partial\theta=\gamma\omega_b$, where
$\omega_b(i,j)=\epsilon_{ija}b^a$, this says
\[
 C=ba^{\mathsf T}-[\lambda]_\times,
 \qquad A=\Sym(a\otimes b),\qquad \eta=-a\times b+2\lambda.
\]
\end{lemma}
\begin{proof}
The full pure commutator, including $-\nabla_{T(u,v)}a$, gives
$\gamma T(u,v)=(\partial\theta)(u,v)a-\zeta(u)v+\zeta(v)u$,
where $\zeta=\nabla\gamma-\gamma\theta$.
Use $[a\times b]_\times=ba^{\mathsf T}-ab^{\mathsf T}$ for the encoding.
Pure flatness makes the connection form $\gamma^{-1}\nabla\gamma$
closed in the nonzero frame $\gamma$.
\end{proof}

\subsubsection{A scalar polynomial differential equation}
On $q\Delta\ne0$, multiply $BW=3qC^{\mathsf T}z$ by $J$ and
substitute \eqref{u3c:transport}, \eqref{u3c:q-derivative}. This gives
\[
 \nabla_i\mathcal S=5\eta_i\mathcal S+z_i\mathcal U,
 \qquad\mathcal S=\Delta^3q,\quad
 \mathcal U=\Delta^2(6qr+\Delta\kappa(z)),
 \quad r=\tr(JC^{\mathsf T}).
\]
Wedge its full pure-flat commutator with $\zeta_z=z_ie^i$.
The $\nabla\mathcal U-5\eta\mathcal U$ term disappears;
the torsion-direction term contributes $q\mathcal U$. Thus
$0=5\kappa(z)\mathcal S+q\mathcal U
=6q\Delta^2(qr+\Delta\kappa(z))$, and polynomial cancellation gives
\begin{equation}\label{u3c:scalar-law}
 qr+\Delta\kappa(z)=0,
 \qquad \nabla_i(\Delta^3q)
 =5\eta_i\Delta^3q-5z_i\Delta^3\kappa(z).
\end{equation}

Work on the dense factorization charts: constant-rank congruence makes
$q$ irreducible, a smooth square, or a product of independent smooth
linear factors. Local square roots and line frames retain their
determinant weights. We can also keep the multiplicities of these factors in $\Delta$
constant after refining the charts: the possible orders are finite, and
divisibility is closed in the coefficients.

\paragraph{Factor transport}\phantomsection\label{polynomial:factor-transport}
For $\mathcal P=f^mR$ with $m\ge1$ and $f\nmid R$, Leibniz gives
$(D\mathcal P)/f^{m-1}\equiv m R\,Df\pmod f$ for a covariant
derivation $D$, with all determinant weights retained.

\begin{lemma}[The regular case with zero curl]
\label{u3c:zero-curl}
On a regular chart $q\Delta\ne0$, if $\kappa=0$,
then $V=G=D=0$ and $\mathcal X=\mathfrak H$, the full reference
derivative \eqref{unif:reference:H}.
\end{lemma}
\begin{proof}
Write $\Delta^3q=f^mR$, $f\nmid R$, for each irreducible factor
$f$ of $q$. \hyperref[polynomial:factor-transport]{Factor transport}
in $\nabla(\Delta^3q)=5\eta\Delta^3q$ gives $f\mid\nabla_i f$.
Thus $q\mid\nabla_iq$, and \eqref{u3c:q-derivative} gives
$V=q\theta$ with smooth auxiliary-constant $\theta$.
Substitution of $V=q\theta$ into \eqref{u3c:source} and cancellation of $q$ give $B\theta=3Gz$. Since
$B\theta=\nabla_z(\theta\cdot p)+[\theta]_\times z$ and $G$ is symmetric,
the skew part gives $\theta=0$, then $G=0$; full recovery after
\eqref{u3c:source} gives $D=0$. Continuity extends these identities.
\end{proof}

\subsubsection{Factor residues and scalar incompatibility}
Suppose $\kappa\ne0$. For reducible $q$, write
$q=\prod_j\ell_j^{e_j}$, with distinct $\ell_j=a_j\cdot z$,
$\sum e_j=2$, and put $m_j=q/\ell_j=d_j\cdot z$,
$b_j=\operatorname{ord}_{\ell_j}\Delta$.
Applying \hyperref[polynomial:factor-transport]{factor transport}
to the scalar law \eqref{u3c:scalar-law} gives
\[
 \kappa(z)=\sum_jk_jm_j,\qquad
 \nabla a_j=\theta_j\otimes a_j-\gamma_jI,\qquad
 \gamma_j=\frac{5k_j}{3b_j+e_j}.
\]
For two factors, reduce modulo $\ell_j$ with $z_i$ coprime to $m_j$;
for a square, the lowest residue gives $\ell_j\mid\kappa$ and
the next gives the derivative.
If all $k_j\ne0$, concircularity~\ref{u3c:concircular} gives
\[
 \partial\theta_j=\gamma_j\omega_{d_j},\qquad
 \eta=-a_j\times d_j+2\gamma_j^{-1}\nabla\gamma_j-2\theta_j.
\]
Differentiate the $e_j$-weighted sum, using
$\sum e_j=2$, $\sum e_ja_j\times d_j=0$:
$\kappa=-\sum e_j\gamma_jd_j$.
Independence of the two $d_j$, or the single square factor, gives
$k_j[1+5e_j/(3b_j+e_j)]=0$, impossible.

Only the case of two distinct factors with exactly one nonzero
residue coefficient remains. Absorb
its coefficient into $l=\kappa(z)=a\cdot z$, and put $q=lm$,
$m=d\cdot z$, of determinant weights one and zero respectively.
Scalar compatibility gives $b=\operatorname{ord}_m\Delta\ge1$;
the residue equations and concircularity give
\[
\begin{gathered}
 \rho=-\frac5{3b+1},\quad t=\rho-1,\qquad
 \nabla a=u\otimes a,\quad\nabla d=v\otimes d+\rho I,\\
 C=ad^{\mathsf T}+\rho[s]_\times,\quad s=v/\rho.
\end{gathered}
\]
Thus $\rho,t\ne0$ are constant. By \eqref{u3c:q-derivative},
$V=l(tz+m\theta)$, $\theta=u+v-2\eta$.
For $\xi=\theta+3\eta/2$, \eqref{u3c:trace-curl} gives
$\xi\cdot a=0$, $\xi\cdot d=-3-4t$, hence
\begin{equation}\label{u3c:single-source}
 W=lUz,\quad U=tI+\xi d^{\mathsf T},\quad
 \det U=-3t^2\rho\ne0,\qquad BUz=3mC^{\mathsf T}z.
\end{equation}
Contract $BUz=3mC^{\mathsf T}z$ with $z$. Euler's identity
for the quadratic $p$ and $z^{\mathsf T}[z]_\times=0$ give
$p\cdot Uz=3mq/2$. Differentiate in $z$ and use the
source: $m\mid U^{\mathsf T}p$, since $mq=lm^2$ and
$z\times Uz=mz\times\xi$. Thus $p=mKz$.
Actual \ChernFirst{} modulo $m$ gives $\rho Kz\equiv Cz\pmod m$, so
\begin{equation}\label{u3c:single-p}
 K=[s]_\times+\mathbf b d^{\mathsf T},\qquad
 p=m([s]_\times z+\mathbf b m).
\end{equation}
At a point choose a determinant-one frame $a=e_1,d=e_3$.
Put $h=2t+3$, $\xi=(0,\xi_2,-3-4t)$.
The $z_2z_3,z_3^2$ coefficients in the first two source rows,
then the three coefficients in the last, give
\[
 hs_3=4t+3,\quad s_1=0,\quad \xi_2=-hs_2,\qquad
 t\mathbf b_1=3/2,\quad\mathbf b_2=\mathbf b_3=0.
\]
Thus $h\ne0$. Determinant and cofactor expansion now give
\[
 \Delta=\frac{18\rho}{th}lm^2,\qquad
 r=\frac{9\rho(2t+1)}{th}lm,
\]
contradicting $qr+\Delta\kappa=0$ since its left side is
$9\rho l^2m^2/t\ne0$. Only pointwise coefficients were used.

For irreducible $q$, scalar compatibility $qr+\Delta\kappa=0$ gives
$q\mid\Delta$, since $\kappa$ is nonzero and linear. As $\Delta$ is cubic,
$q$ occurs once; the factor transport gives
$\nabla_iq=\theta_iq-\tfrac54z_i\kappa(z)$.
Set $\zeta=A^{-1}\kappa$. The derivative and trace identities
\eqref{u3c:q-derivative} and \eqref{u3c:trace-curl}, followed by
pure flatness of the nonzero determinant section, give
\[
 2\theta-\eta=6\zeta,\qquad
 0=\partial\tr(A^{-1}\nabla A)=\partial(3\theta-\tfrac54\zeta).
\]
Commute the full identity
$\nabla_iA=\theta_iA-\tfrac58(e_i\kappa^{\mathsf T}+\kappa e_i^{\mathsf T})$.
At a point with $A=I$, take distinct $i,j,a$.
Its $aa$ entry has no differentiated-$\kappa$ term, giving, without summing $a$,
\[
 (\partial\theta)_{ij}=\tfrac54T_{ij}^{a}\kappa^a
                         =\tfrac54(\partial\eta)_{ij}.
\]
The preceding identities give
$0=62\partial\theta+5\partial\eta=\tfrac{165}{2}\partial\eta$,
again impossible. Hence $\kappa=0$, including chart boundaries
by continuity, and Lemma~\ref{u3c:zero-curl} gives the full recurrence
\begin{equation}\label{u3c:actual-recurrence}\equationalias{fourh:eq:prime-recurrence}
 \nabla_pT(u,v)
 =-T(p,T(u,v))-T(T(p,u),v)-T(u,T(p,v)).
\end{equation}
\begin{lemma}[The torsion image of a rank-three partial connection]
\label{fourh:prime-internal-rank-bound}
Let $M$ be the torsion of a partial connection on a rank-three complex tangent distribution.
Assume its connection is pure flat and satisfies complete \ChernFirst{},
\ChernCyclic{}, \ChernPure{} and the torsion-corrected pure covariant commutator.
Then $\operatorname{rank}M\le2$.
The image of $\nabla'M:V\otimes\Lambda^2V\to V$ also has rank at most two. On a regular nonzero-torsion
chart with $q\ne0$, after shrinking to a nonempty subchart, both images
lie in a proper prime-preserved subbundle $N$. For $\Delta\ne0$ take
the abelian image $N=\im M$; for $\Delta=0$ one can take $N$ with
$E(N,V)\subset\Sym^2N$.
All direction, input, output and determinant-line connections are
included.
\end{lemma}
\begin{proof}
On regular charts, Lemma~\ref{unif:lem:regular-loci} applies.
\eqref{u3c:actual-recurrence} and \eqref{unif:eq:three-recurrent-derived}
give an abelian $I=\im M$ of rank at most two and $X_p=-2M_p\circ M$
when $q\Delta\ne0$. Thus $X$ takes values in $I$, which image
differentiation preserves on the constant-rank chart.
On a vectorial field $q=0$,
\eqref{u3c:eq:uniform-vectorial-prime} gives $X_p=-2\alpha(p)M$.
On $q\ne0$, $\Delta=0$, Lemma~\ref{lem:conceptual-three-singular}
and Proposition~\ref{prop:conceptual-fullcut} supply the proper
subbundle $N$ with $\nabla'_VN\subset N$ and $E(N,V)\subset\Sym^2N$;
image differentiation puts
$X$ in $N$ as well.
The minors that express both rank bounds vanish on a dense set,
so they vanish everywhere.
\end{proof}

\subsection{Trace-free torsion on rank-three involutive distributions}
\label{sec:conceptual-closed-three}

\begin{theorem}[Vanishing for trace-free torsion in rank three]
\label{thm:conceptual-closed-three}
Let $V$ be an involutive complex tangent distribution of rank three,
with an actual partial connection along $V$, torsion $T$, and zero pure
curvature. Suppose $T$ is trace free and $T,\mathcal X=\nabla'T,P$ obey the
complete \ChernFirst{}, \ChernCyclic{} and \ChernPure{} equations. All induced tensor and
determinant-line connections are included. Then $T=0$.
\end{theorem}

\begin{proof}
If $T\ne0$, choose a regular nonzero-torsion chart.
Trace freeness makes $C$ symmetric, so $C=A$ is nonzero and its
quadratic $q$ is not identically zero.
Lemma~\ref{fourh:prime-internal-rank-bound} supplies a preserved
subbundle of rank at most two containing $\im T$ on a nonempty
subchart, contradicting Theorem~\ref{thm:conceptual-lowrank}.
\end{proof}

\subsection{Vectorial and recurrent torsion exclusions}\label{u3c:metric-terminals}
\label{u3c:sec:uniform-vectorial-terminal}
The derivative identity for vectorial torsion uses only the $(1,0)$
equations. To exclude nonzero torsion, we also use the
$(0,1)$ equations of the prescribed Chern connection.

\begin{lemma}[Prime vectorial identity]
For a tensor system with a partial connection in all bundle directions
satisfying \ChernCyclic{} and \ChernPure{} in rank $n\ge3$,
if $T(u,v)=\alpha(u)v-\alpha(v)u$, then
\begin{equation}\label{u3c:eq:uniform-vectorial-prime}\equationalias{fourh:eq:prime-vectorial}
 \nabla_p\alpha_i=-2\alpha_p\alpha_i,\qquad
 \mathcal X_p=-2\alpha(p)T.
\end{equation}
\end{lemma}
\begin{proof}
Write $s_p(i)=(\nabla'_p\alpha)(i)$ for vectors $p,i\in V$. Since
$J_T=0$, \ChernCyclic{} gives, for vectors $p,i,j\in V$,
\[
 (s_p(i)-s_i(p))j+(s_i(j)-s_j(i))p+(s_j(p)-s_p(j))i=0.
\]
Taking $p,i,j$ independent makes all three coefficients vanish, so
$s$ is symmetric. Put $S=s+\alpha\otimes\alpha/2$. The corrected
vectors $s_u,s_v$ in \eqref{eq:repeated} are then
$s_u=S(u,u)v-S(u,v)u$ and $s_v=S(u,v)v-S(v,v)u$;
thus $s_u\wedge v+u\wedge s_v=0$. Contract the remaining identity with $\xi^2$, choosing $\xi(u)=0$
and $\xi(v)=1$ for independent $u,v$. This gives
$S(u,u)+3\alpha(u)^2/2=0$ for every $u$. Since $S$ is symmetric,
$S=-3\alpha\otimes\alpha/2$ and hence $s=-2\alpha\otimes\alpha$.
Differentiating $T(u,v)=\alpha(u)v-\alpha(v)u$ gives
$\mathcal X_p=-2\alpha(p)T$.
\end{proof}

\begin{theorem}\label{u3c:vectorial}\label{prop:vectorial}
Let $h$ be a smooth positive Hermitian metric of complex dimension
$n\ge3$ with constant Chern HSC $c\ne0$. If
$T(u,v)=\alpha(u)v-\alpha(v)u$ on an open set $O$, then $T=0$ on $O$.
\end{theorem}
\begin{proof}
The smooth covector $\alpha=\tr T/(n-1)$ satisfies
$d\omega_h=(\alpha+\bar\alpha)\wedge\omega_h$.
The injective Lefschetz map on two-forms, valid for $n\ge3$,
gives $d(\alpha+\bar\alpha)=0$. Locally write
$\alpha+\bar\alpha=d\varphi$, so $g=e^{-\varphi}h$ is K\"ahler.
Put $D=\nabla^g$, $a=\partial\varphi=\alpha$, $F=\partial\bar\partial\varphi$,
$f=ce^\varphi$. By \eqref{u3c:eq:uniform-vectorial-prime},
$D_pa_i=-a_pa_i$, hence $D_pf_i=0$.
The conformal formula $R^h=e^\varphi(R^g-F\otimes g)$ and polarization give
\[
 R^g_{i\bar j k\bar l}=\frac f2(g_{i\bar j}g_{k\bar l}+g_{i\bar l}g_{k\bar j})
 +\frac14(F_{i\bar j}g_{k\bar l}+F_{k\bar j}g_{i\bar l}
          +F_{i\bar l}g_{k\bar j}+F_{k\bar l}g_{i\bar j}).
\]
Contract the K\"ahler identity $D_pR^g_{i\bar j k\bar l}=D_iR^g_{p\bar j k\bar l}$
in $k,\bar l$ and in $i,\bar j$.
For $b_p=D_p\tr_gF$ the two rows are
$b_p=-2(n+1)f_p$ and
$nD_pF_{k\bar l}=-2nf_pg_{k\bar l}+(2f_k+b_k)g_{p\bar l}$; hence
\begin{equation}\label{u3c:eq:conformal-Hessian}
 D_pF_{i\bar j}=-2(f_pg_{i\bar j}+f_ig_{p\bar j}).
\end{equation}
Substitute $D_{\bar j}f_i=f(F_{i\bar j}+a_i\bar a_j)$ in the
commutator $D_pD_{\bar j}f_i=-f_b(R^g)_{p\bar j i}{}^b$:
\[
 5(a_pF_{i\bar j}+a_iF_{p\bar j})
 -6f(a_pg_{i\bar j}+a_ig_{p\bar j})
 +(Fa)_pg_{i\bar j}+(Fa)_ig_{p\bar j}=0,
\]
where $(Fa)_p=\sum_bF_{p\bar b}a_b$ in a unitary frame.
At $a\ne0$, diagonalize $F$ and choose $a_k\ne0$.
With eigenvalues $\lambda_j$, the $(k,k,k)$ and $(k,j,j)$ rows give
$\lambda_k=f$, $5\lambda_j+\lambda_k=6f$, so $F=fg$ on $\{a\ne0\}$.
Its derivative in \eqref{u3c:eq:conformal-Hessian} is
$3f_pg_{i\bar j}+2f_ig_{p\bar j}=0$.
Taking $i=j\ne p$ forces $a=0$, so $\alpha=0$.
\end{proof}

\begin{lemma}[A recurrent projection with a nonzero central complement]
\label{unif:lem:recurrent-projection}
Let $h$ be a smooth positive Hermitian metric of complex dimension $n\ge3$
with constant Chern holomorphic sectional curvature $c\ne0$.
On an actual open region suppose a smooth nowhere-zero one-form $\alpha$
and smooth nonzero subbundles $I,K$ satisfy
\[
\begin{gathered}
 A:=\ker\alpha=I\oplus K,\qquad
 \im T\subset I,\qquad T(A,A)=0,\\
 T(V,K)=0,\qquad T(x,d)=\alpha(x)d\quad(d\in I).
\end{gathered}
\]
and suppose the full original torsion obeys
$\mathcal X_p=\nabla'_pT=-2\alpha(p)T$.
Then this region is empty. 
\end{lemma}
\begin{proof}
Choose $\alpha(e)=1$. At $(u,v)=(e,d+z)$ with $d\in I$, $z\in K$,
we have $t=d$, $s_u=-3d/2$ and $s_v=0$ by recurrence and $T(A,A)=0$.
Applying $E$ to \eqref{newunb:r1v3:graph} gives $-3E(d,z)=0$, hence
\begin{equation}\label{newunb:eq:cross}
 E(d,z)=0,\qquad P(d,z)=-dz\quad(d\in I,z\in K).
\end{equation}
Centrality removes both input terms in $\nabla''T(K,K)$;
\ChernBarred{} gives $P(K,K)=0$.

For $z\in K$ the complete \ChernFirst{} expression on $(x,z)$ is
$E(T(x,p),z)/2=0$. The mixed commutator on $T$ and the recurrence give
\[
 (\nabla'_pY_{\bar q})(x,z)
 =-2\alpha(p)Y_{\bar q}(x,z)-T(x,R(p,\bar q)z).
\]
Since $T(x,z)=0$, the differentiated scalar term and two
curvature terms vanish. Differentiating \ChernBarred{} then gives
\[
 \sum_q e_q\odot T(x,R(p,\bar e_q)z)=8c\alpha(p)P(x,z)
\]
in a unitary frame.

Take $e_1$ unit normal to $A$, $\alpha=ae^1$, $a>0$, and orthonormal
$z_\mu\in K$, $e_j\in K^\perp\cap A$. Write
$d_j=e_j+\sum_\mu g_{\mu j}z_\mu$ for a basis of $I$.
The two $P$ identities give $P(z_\mu,e_j)^{\mu\mu}=g_{\mu j}$,
hence $R_{z_\mu\bar z_\mu z_\mu\bar e_j}=c\overline{g_{\mu j}}$.
At $p=z=z_\mu$, $x=e_1$, the $z_\mu^2$ coefficient of the mixed
row is $2ac\sum_j|g_{\mu j}|^2=0$. Thus $I\perp K$.

Now choose unit $d\in I$, $e_0\in K$ and set $p=e_1$, $x=d$, $z=e_0$.
Since $T(d,R(e_1,\bar e_q)e_0)=-\alpha(R(e_1,\bar e_q)e_0)d$,
the $d\odot e_0$ coefficient of the same row gives
$\alpha(R(e_1,\bar e_0)e_0)=4ca$, that is, $R_{1\bar0 0\bar1}=4c$.
But the whole derivative of $T(x,z)=0$ is $I$-valued, so
$Y_0(1,0)^1=Y_1(1,0)^0=0$ by $I\perp e_0,e_1$.
Barred Bianchi and reality therefore equate
$R_{1\bar0 0\bar1}$, $R_{0\bar0 1\bar1}$,
$R_{1\bar1 0\bar0}$ and $R_{0\bar1 1\bar0}$.
Their sum is $2c$ by polarization in the unitary directions
$e_0,e_1$, so each component is $c/2$. This contradicts
$R_{1\bar0 0\bar1}=4c$ because $c\ne0$.
\end{proof}

\subsection{The local common-kernel reduction}\label{tloc:foundations}
Let the prescribed metric have constant real Chern HSC $c\ne0$.
For $n\ge2$, define the bilinear map on covectors and its kernel by
\begin{equation}\label{ck:definition}\equationalias{hd-kernel-definition}\equationalias{b:eq:common-kernel}
 \mathscr B(v,w)=\mathcal B_wv=2P^\vee(v,w)-v\wedge w,
 \qquad
 \mathscr K=\{v\in V^*: \mathscr B(v,w)=0\text{ for all }w\in V^*\}.
\end{equation}
Canonical symmetry gives
\begin{equation}\label{ck:skew}
 \mathscr B(v,w)-\mathscr B(w,v)=-2v\wedge w.
\end{equation}
Any $v,w\in\mathscr K$ satisfy $v\wedge w=0$, so
$\dim\mathscr K\le1$. On the interior of the dimension-one locus,
$\mathscr K$ is a smooth line subbundle by constant rank.

\begin{lemma}[Conditional common-kernel transport]\label{tloc:conditional-transport}
Let $n\ge2$. On any actual open set where $\mathscr K$ is a line,
let $N$ be its metric-dual line and $E=N^\perp$. Then
\begin{equation}\label{tloc:ck-preservation}
 T(V,V)\subset E,\qquad \nabla'_V E\subset E,
 \qquad \nabla''N\subset N,
 \qquad \operatorname{im}\mathcal X\subset E.
\end{equation}
\end{lemma}
\begin{proof}
For a nonzero local $\nu\in\mathscr K$, put $\mathcal E(u,v)=uv+P(u,v)$.
The kernel identity is $(\nu\otimes\Id)\mathcal E(u,v)=\nu(u)v$,
so $E=\ker\nu$ satisfies $\mathcal E(E,V)\subset\Sym^2E$.
Proposition~\ref{prop:conceptual-fullcut} gives
\begin{equation}\label{ck:kernel-consequences}
 \nu T=0,\qquad\nabla^{1,0}\mathscr K\subset\mathscr K.
\end{equation}
Thus $T$ takes values in $E=\ker\nu$ and $E$ is preserved in prime
directions, so $\im\mathcal X\subset E$. Metric compatibility makes
$N=E^\perp$ preserved in barred directions.
\end{proof}

\subsubsection{Line-bundle-valued tensors and the geodesic equation}
Now assume $\im T=E$ on the common-kernel line region and put
$m=\rank E=n-1$ ($m=2$ for threefolds). Choose a unit section $Y$ of
$N$ and put $\nu=h(\cdot,Y)$. Let $D$ be the induced Chern connections
on $N$, $E\simeq V/N$ and their tensor products. For $x,z\in E$ put
\begin{equation}\label{tloc:ck-conventions}
 \begin{gathered}
 Sx=T(Y,x),\quad Bx=\pi_E\nabla_xY,\quad
 \tau(x,z)=T(x,z),\quad H=B^*S,\\
 A=D_YS+\tfrac12S^2,\quad Z=S+2B,\\
 Z_0=Z-\tfrac1m(\tr Z)I,\\
 P_0(x,z)=P(x,z),\qquad Qx=\pi_{\Sym^2E}P(Y,x),\qquad
 \mathcal E_0(x,z)=xz+P_0(x,z).
 \end{gathered}
\end{equation}
The tensors $S,B,Z,Q$ have one $N^*$ slot, $A$ has two,
and $\tau,P_0,H$ have none. Taking the adjoint also conjugates the normal-line covector;
products contract conjugate weights using the metric on $N$.

The common-kernel identity gives
\begin{equation}\label{tloc:ck-curvature}\equationalias{ck:curvature}
 R(U,\bar Y,V_0,\bar Z')+R(U,\bar Z',V_0,\bar Y)
       =2c\,h(U,Y)h(V_0,Z')
\end{equation}
for all ambient vectors and gives the splitting
\begin{equation}\label{tloc:P-split}
 P(Y,x)=Yx+Qx,\qquad P_0:\Lambda^2E\longrightarrow\Sym^2E,
 \qquad xY+P(x,Y)=-Qx.
\end{equation}
In particular the ordered expression $E(x,Y)$ need not vanish.
The ordered cut and $\Lambda^2N=0$ give $(\Sym^2q_E)P=0$.
Apply Lemma~\ref{unif:lem:ordered-cut-normal-shape} with $I=E$, $Q=N$.
Its alternating normal shape vanishes on $N\times N$, so
$\pi_E\nabla_YY=0$. Its operators $B_Y,S_Y$ are the present $S,B$,
respectively, and its $\mathfrak t,\mathfrak w$ vanish.
Equations \eqref{unif:eq:ordered-cut-mixed-pair} and
\eqref{r16:eq:ranktwo-shape-rows} give
\[
 H=H^*,\qquad \pi_E R(x,\bar Y)Y=Hx/2,\qquad
 D_{\bar Y}B=-H/2,\qquad D_{\bar Y}S=H-2cI .
\]
By \eqref{tloc:ck-curvature}, $\pi_E R(Y,\bar Y)|_E=2cI-H/2$.

The splitting $N\oplus E$ is preserved in both real directions of $N$.
With $D_+=D_Y$ and $D_-=D_{\bar Y}$, the equations are
\begin{equation}\label{tloc:ck-longitudinal}\equationalias{ck:riccati}\equationalias{ck:H-derivatives}
 \begin{gathered}
 H=B^*S=S^*B,\qquad D_-S=H-2cI,\qquad D_-B=-H/2,\\
 D_+B=-B(B+S),\qquad
 D_-H=-S^*H-2cB^*,\qquad D_+H=-HS-2cB.
 \end{gathered}
\end{equation}
Indeed $[x,Y]_E=(B+S)x-D_Yx$, so the $E$ output of the zero pure
curvature $R(x,Y)Y$ is $-(D_YB)x-B(B+S)x$.
Differentiating $H=B^*S$ then gives
$D_-H=-(B^*+S^*)H+B^*(H-2cI)$, and its adjoint gives $D_+H$.
These equations hold for either sign of $c$.

On a constant-rank $B$ chart, put $F=\ker B$ and let $\Pi$ be the
orthogonal projection onto $F^\perp$. Then
\begin{equation}\label{tloc:active-block}\equationalias{ck:active-block}
 \ker B=\ker B^*,\qquad B=\Pi B\Pi,\quad H=\Pi H\Pi,
 \qquad D_+\Pi=D_-\Pi=0.
\end{equation}
If $Bz=0$, then $Hz=0$, and $D_-B=-H/2$ gives
$D_-z\in F$. Differentiating $Hz=0$ with \eqref{tloc:ck-longitudinal}
gives $B^*z=0$. The kernels have equal dimension, so they coincide.
The restriction of $B$ to $F^\perp$ is invertible. Differentiating $\Pi B=B=B\Pi$
then shows that both off-diagonal blocks of $D_-\Pi$ vanish.
Taking the adjoint gives the $D_+$ equation. 
\begin{lemma}\label{tloc:no-zero-B}
No nonempty common-kernel open set with $\im T=E$ has $B=0$ identically.
\end{lemma}
\begin{proof}
If $B=0$, the normal-direction equation $\pi_E\nabla_YY=0$
and \eqref{tloc:ck-preservation} make $N,E$ Chern-parallel. Apply the argument after \eqref{r6:eq:whole-cut-normal-FIRST}
using \eqref{r16:eq:ranktwo-shape-rows} with $I=E$, $Q=N$.
\end{proof}

Metricity gives $D_YB^*=(D_{\bar Y}B)^*=-H/2$.
Differentiate $H=B^*S$ and substitute $D_YH=-HS-2cB$:
\begin{equation}\label{tloc:A-coupling}
 B^*A=-2cB,\qquad A^*B=-2cB^*.
\end{equation}

\subsubsection{The equations in an adapted frame}
For $z\in E$ put
\[
 b_z(v)=h(v,Bz),\qquad C_z=D_{\bar z}B,\qquad
 G(U,\bar V)=R_D^E(U,\bar V),\qquad r(U,\bar V)=R_D^N(U,\bar V).
\]
Since $N$ is holomorphic and geodesic,
$R_D^{2,0}=R_D^{0,2}=0$ and $\nabla=D+\mathsf K$, with blocks
\begin{equation}\label{tloc:adapted-evaluation}
 \mathsf K_xY=Bx,\quad \mathsf K_xv=0,\quad
 \mathsf K_{\bar z}Y=0,\quad \mathsf K_{\bar z}v=-b_z(v)Y,\quad
 \mathsf K_Y=\mathsf K_{\bar Y}=0
 \qquad(x,z,v\in E).
\end{equation}
For $A=T,P$, the complete difference action is
\begin{equation}\label{tloc:adapted-action}
 (\nabla_pA)(u,v)=(D_pA)(u,v)+\mathsf K_p^{\rm out}A(u,v)
                 -A(\mathsf K_pu,v)-A(u,\mathsf K_pv).
\end{equation}
Here $\mathsf K^{\rm out}$ is $\mathsf K$ on $V$ and
$\mathsf K\otimes I+I\otimes\mathsf K$ on $\Sym^2V$.
For higher derivatives this rule also acts on each direction slot.
In particular
$\mathcal X_z(Y,x)=(D_zS)x-\tau(Bz,x)$.
The curvature blocks are
\begin{equation}\label{tloc:curvature-mixed}
 \begin{gathered}
 R(Y,\bar Y)|_N=c,\quad R(Y,\bar Y)|_E=2cI-H/2,\\
 r(Y,\bar z)=0,\quad G(Y,\bar z)=C_z,\quad R(Y,\bar z)Y=0,\\
 \pi_NR(Y,\bar z)x=h(Hx/2,z)Y,\quad
 R(x,\bar Y)Y=Hx/2,\quad G(x,\bar Y)=C_x^*,
 \end{gathered}
\end{equation}
and
\begin{equation}\label{tloc:curvature-horizontal}
 \begin{aligned}
 \pi_ER(x,\bar z)Y&=-C_zx,&
 \pi_NR(x,\bar z)Y&=-h(Hx/2,z)Y,\\
 \pi_NR(x,\bar z)w&=-h(w,C_xz)Y,&
 \pi_ER(x,\bar z)w&=G(x,\bar z)w-Bx\,b_z(w),\\
 r(x,\bar z)&=-h((H/2+B^*B)x,z).
 \end{aligned}
\end{equation}
The products of the off-diagonal connection forms subtract
$Bx\,b_z$ from $G$ and add $h(Bx,Bz)$ to $r$; their covariant
exterior derivatives give the other blocks. Equation
\eqref{tloc:ck-curvature} and metric symmetry identify
$G(Y,\bar z)=C_z$ and
$h(G(x,\bar Y)w,z)=h(w,C_xz)$.

Barred first Bianchi on $(Y,x)$ gives
\begin{equation}\label{tloc:barred-split}
 D_{\bar z}S=-2C_z,\qquad D_{\bar z}Z=0\quad(z\in E).
\end{equation}
Thus the barred Chern identity, in a unitary frame $(e_a)$ of $E$, gives
\begin{equation}\label{tloc:barred-P}\equationalias{tloc:scalar-physical-Q}
 Qx=-\frac1{2c}\sum_a e_a(D_{\bar e_a}S)x
                =\frac1c\sum_a e_a C_{e_a}x.
\end{equation}
The $Y$ component cancels because $H=B^*S$.

Put $\mathcal F(x^2)=C_x^*x$, a linear map on $\Sym^2E$;
then \eqref{tloc:barred-P} gives $Q=c^{-1}\mathcal F^*$.
For the pair $(Y,x)$ the corrected vectors in \eqref{eq:repeated}
are $Ax$ and $L(x)=(D_xS)x+\tfrac12\tau(x,Zx)$.
Here $Z=S+2B$ combines the two torsion corrections
$-\tau(Bx,x)+\tau(x,Sx)/2$, retaining the moving $Y$ input $Bx$.
The ordered split \eqref{tloc:P-split} gives
\begin{equation}\label{tloc:repeated-Y}\equationalias{tloc:scalar-full-PURE}
 \mathcal E_0(Ax,x)+QL(x)+\tfrac32(Sx)^2=0.
\end{equation}
These equations are computed before projection, using
\eqref{tloc:adapted-evaluation}--\eqref{tloc:adapted-action} and
\eqref{tloc:P-split}: \ChernFirst{} on $(x;u,v),(Y;u,v),(Y;Y,x),(x;Y,z)$;
\ChernCyclic{} on $(Y,x,z),E^3$; \ChernBarred{} on $(Y,x),(x,z)$ with raised families
in both $N,E$; and complete \ChernPure{} and barred Bianchi on ambient slots.
Horizontal HSC uses \eqref{eq:P},\eqref{tloc:curvature-horizontal};
pure curvature uses $R^\nabla(x,z)Y=0$ and
$[x,z]_E=D_xz-D_zx-\tau(x,z)$.

\subsubsection{Holomorphic tensors and determinant curvature}
The $(Y,\bar Y)$ curvature on $S\in N^*\otimes\End E$ is
$-[H,S]/2-cS$, including its line slot. For $(Y,\bar z)$ the mixed
torsion and $r(Y,\bar z)$ vanish, while curvature acts by $[C_z,\cdot]$.
Apply the commutator to the longitudinal and barred equations to obtain
\[
 \begin{gathered}
 D_{\bar Y}D_YS=-\tfrac12(HS+SH)-2cB+cS,\\
 D_YC_z=-C_zS,\qquad D_{\bar z}D_YS=C_zS+SC_z.
 \end{gathered}
\]
Substitution in $A=D_YS+S^2/2$ and $Z=S+2B$ proves
\begin{equation}\label{tloc:holomorphic-derivatives}
 D_{\bar z}A=0,\qquad D_{\bar Y}A=-cZ,\qquad
 D_{\bar z}Z=0,\qquad D_{\bar Y}Z=-2cI\quad(z\in E).
\end{equation}
Consequently
\begin{equation}\label{tloc:holomorphic-invariants}
 D''Z_0=0.
\end{equation}
Thus $Z_0$ is a holomorphic section of $N^*\otimes\End_0E$.
Differentiating the adjoint equation in \eqref{tloc:A-coupling} gives
\begin{equation}\label{tloc:horizontal-coupling}
 A^*(D_xB)=-2cC_x^*.
\end{equation}

On an open set where $B$ is invertible, put $m=\rank E$ and define
\begin{equation}\label{tloc:det-definitions}
 U=-\frac{A}{2c}=(B^*)^{-1}B,\qquad
 f=\det A\in(N^*)^{2m},\qquad
 q=\frac{D_{\bar Y}f}{f}=-c\,\tr(A^{-1}Z).
\end{equation}
The weights are $(N^*)^2$ for $U$ and $\bar N^*$ for $q$, opposite
to $B$. Taking determinants in \eqref{tloc:A-coupling} gives
$|f|=(2|c|)^m$. The equation $D_{\bar x}A=0$ gives $D_{\bar x}f=0$;
metric compatibility and the constant norm give $D_xf=0$ and
$D_Yf=-\bar qf$, using $D_{\bar Y}f=qf$.
The direction connection has
$T_D(x,\bar z)=-h(x,Bz)Y+h(Bx,z)\bar Y$.
Both horizontal Hessians of $f$ vanish, so the full commutator is
$0=-2m r(x,\bar z)f-D_{T_D(x,\bar z)}f$.
Using \eqref{tloc:curvature-horizontal} gives
$m(H+2B^*B)=qB+\bar qB^*$; multiplication by $(B^*)^{-1}$ yields
\begin{equation}\label{tloc:det-identities}
 mZ=\bar qI+qU.
\end{equation}
Substituting \eqref{tloc:det-identities} in the definition of $q$ gives
$\bar q\tr U^{-1}=mq$. Since
$\tr U^{-1}=\overline{\tr U}$ for $U=(B^*)^{-1}B$, conjugation
and the trace of \eqref{tloc:det-identities} give, without division,
\begin{equation}\label{tloc:det-traces}
 q\tr U=m\bar q,\qquad \tr Z=2\bar q.
\end{equation}
Differentiate these identities, using $D_{\bar Y}U=Z/2$:
\begin{equation}\label{tloc:q-derivatives}
 D_Yq=-mc,\qquad
 D_{\bar Y}q=\frac q{\bar q}\left(-mc-\frac{|q|^2}{m}\right)
 \quad(q\ne0).
\end{equation}
The first identity rules out an open set where $q=0$, so the locus
$q\ne0$ is dense. We use the second identity only on that locus.

Differentiating \eqref{tloc:det-identities} also gives
\[
 \left(D_{\bar Y}q+\frac{q^2}{2m}\right)U
       =-\left(mc+\frac{|q|^2}{2m}\right)I.
\]
Since $U$ is invertible, either its scalar right-hand coefficient
vanishes or $U$ and hence $Z$ are scalar. Thus, everywhere,
\begin{equation}\label{tloc:det-scalar-obstruction}
 (|q|^2+2cm^2)Z_0=0.
\end{equation}
In particular positive $c$ forces $Z_0=0$ on every invertible-$B$
open set, without asserting scalarity of $S$ or $B$.

Finally take $m=2$ and, on $q\ne0$, put
$t=\bar q/q$ and $\mathcal D=U-tI$. For a two-dimensional operator,
$\tr U^{-1}=\tr U/\det U$; \eqref{tloc:det-traces} therefore gives
$\det U=t^2$. Cayley--Hamilton, not a choice of eigenvectors, yields
\begin{equation}\label{tloc:nilpotent-identities}\equationalias{tloc:negative-nilpotent-data}
 \begin{gathered}
 \mathcal D^2=0,\qquad Z_0=\tfrac q2\mathcal D,\\
 Z_0\ne0\quad\Longrightarrow\quad
 q\ne0,\quad c<0,\quad |q|^2=-8c,\quad
 Z=\bar qI+\tfrac q2\mathcal D.
 \end{gathered}
\end{equation}
The implication uses \eqref{tloc:det-identities} to see that
$q=0$ forces $Z=0$, and then \eqref{tloc:det-scalar-obstruction}.

\subsection{Ordered vanishing along a line}
\label{cut3:section}
Independently of the preceding line-bundle-valued equations, for the prescribed
metric with $c\ne0$ the full line-cut identity
for a smooth nonzero field $u$ and a smooth one-form $\xi$ is
\begin{equation}\label{cut3:cut}
 P(u,x)=-\tfrac12u\odot x+\xi(x)u^2,\qquad \xi(u)=1.
\end{equation}

\begin{lemma}[Torsion and curvature derivatives under line vanishing]
\label{cut3:full-input}
Suppose the actual line cut \eqref{cut3:cut} has
$\im T\subset I=\C u$, $\nabla'_VI\subset I$ and $T\ne0$.
Then
\begin{equation}\label{cut3:full-field}
 \begin{gathered}
 T=u\otimes\beta,\quad \beta=\xi\wedge\eta,
 \quad\xi(u)=1,\quad\eta(u)=0,\quad\eta\ne0,\\
 \im P\subset I\odot V,\qquad
 \nabla_pP(x,y)=\xi(p)\beta(x,y)u^2.
 \end{gathered}
\end{equation}
Here $\eta$ is the ambient negative torsion trace. This implication
holds in every dimension in which the Chern identities below hold.
\end{lemma}
\begin{proof}
Write $T=u\otimes\beta$, $\alpha=\iota_u\beta$.
Preservation gives $\im\mathcal X_p\subset I$.
Both corrected vectors lie in $I$, and \eqref{cut3:cut} gives
$E(u,v)=\xi(v)u^2$. The \hyperref[pure:rankone-contraction]{rank-one contraction}
gives $\beta=\xi\wedge\alpha$, hence $\eta=\alpha\ne0$, $T(I,V)=I$.
Lemma~\ref{r37:lem:image-shape} gives $(\Sym^2q_I)P=0$ by its barred factor.
Complete \ChernFirst{} \eqref{newff:eq:E-form-F} becomes
\begin{equation}\label{cut3:rankone-map}
 2\nabla_pP(x,y)=\{\beta(x,p)\xi(y)+\beta(x,y)\xi(p)+\beta(p,y)\xi(x)\}u^2
 =2\xi(p)\beta(x,y)u^2,
\end{equation}
using $\beta=\xi\wedge\eta$.
\end{proof}

\subsubsection{The prescribed-metric endpoint}
\begin{lemma}[The prescribed-metric full line cut]
\label{unif:lem:prescribed-line-cut}\label{cut3:orthogonal}
In every complex dimension $n\ge3$, an actual nonzero field
\eqref{cut3:full-field} for the prescribed positive Hermitian metric
of constant real Chern HSC $c\ne0$ is impossible.
In particular, it cannot satisfy $T(I^\perp,I^\perp)=0$, where $I=\C u$.
\end{lemma}
\begin{proof}
Set $F(v)=E(u,v)-\xi(v)u^2$. \ChernFirst{} and full-field transport give
$\eta(p)F(v)+\eta(v)F(p)=0$.
Fix $\eta(p)\ne0$. Take $v=p$ first, then vary $v$ to get $F=0$.
Thus \eqref{cut3:cut} holds and Proposition~\ref{prop:conceptual-fullcut}
preserves $I$ in prime directions.
The identities $\im T=I$ and $(\Sym^2q_I)P=0$
contradict Lemma~\ref{r43:lem:codimension-two-cut}.
\end{proof}

\subsection{Weighted equations under ordered vanishing}
\label{tloc:scalar-section}\label{tloc:rankone-section}\label{tloc:negative-section}

\begin{lemma}[Vanishing from the weighted metric equations]
\label{tloc:weighted-terminal}\label{tloc:scalar-exclusion}
\label{tloc:rankone-exclusion}\label{tloc:rankone-barred-exclusion}
\label{tloc:negative-plane}\label{tloc:negative-terminal}
On a common-kernel region with the notation of
\eqref{tloc:ck-conventions}, suppose $\im T=E$, $\rank E=m\ge2$
and $c\ne0$.
The weighted endomorphism $Z=S+2B$ cannot be scalar on a nonempty
open subset, without a rank or invertibility assumption on $B$ or $S$.
If $m=2$, the entire common-kernel region is impossible; in particular
$B$ has neither rank one nor rank two on a nonempty open subset.
\end{lemma}
\begin{proof}
Suppose $Z=zI$. Work where $B\ne0$, as allowed by Lemma~\ref{tloc:no-zero-B}.
The longitudinal equations give $A=aI$, $a=D_Yz+z^2/2$;
coupling and $H=H^*$ give
\begin{equation}\label{tloc:scalar-algebra}
 aB^*=-2cB,\qquad |a|=2|c|,\qquad zB^*=\bar zB.
\end{equation}
Taking adjoints in $aB^*=-2cB$ and substituting back gives
$|a|=2|c|$, since $B\ne0$. Since $D_{\bar Y}z=-2c$,
we may work where $z\ne0$, so $a=-2cz/\bar z$.
Horizontal holomorphicity, constant $|a|$ and metricity imply
$D_{\bar x}a=D_xa=D_{\bar x}z=D_x\bar z=0$; the formula for $a$
then gives $D_xz=0$. Put $\rho=\bar z/2$, $t=\bar\rho/\rho$:
\begin{align}
 &Z=2\bar\rho I,\quad A=-2ctI,\quad B=tB^*,\quad S=tS^*,
 \quad D_x\rho=D_{\bar x}\rho=0,\label{tloc:scalar-rho}\\
 &D_Y\rho=-c,\qquad
 D_{\bar Y}\rho=(\rho/\bar\rho)(-c-|\rho|^2).
 \label{tloc:scalar-rho-long}
\end{align}
Choose the smooth unit gauge $\rho>0$. Then $S,B$ are Hermitian,
$D_xS=-2C_x^*$, and the full repeated row has
$P_0(Ax,x)=\tau(x,Zx)=0$. Using $Q=c^{-1}\mathcal F^*$ gives
\begin{equation}\label{tloc:scalar-reduced-PURE}
\equationalias{tloc:scalar-norm}\equationalias{tloc:scalar-positive-operator}
 \mathcal F^*\mathcal F=\tfrac{3c}{4}\Sym^2(S)-c^2I.
\end{equation}
Evaluate the operator inequality on squares of eigenvectors of $S$.
It gives $c>0$ and bounds each eigenvalue in absolute value by
$a_0=\sqrt{4c/3}$ from below. On the symmetric products of distinct
eigenvectors, it also forces their eigenvalues to have the same sign.
Since $\rank\mathcal F\le m<\dim\Sym^2E$, its kernel is nonzero, so
some eigenvalue product (possibly a square) equals $a_0^2$. Thus
\begin{equation}\label{tloc:scalar-polynomial}
 \mathsf R=S^2-a_0^2I\ge0,\qquad\ker\mathsf R\ne0.
\end{equation}
For the unit connection form $\ell$ of $N$, the weighted equations are
$Y\rho+\ell_Y\rho=-c$ and
$\bar Y\rho+\ell_{\bar Y}\rho=-c-\rho^2$.
Conjugation with $\ell_{\bar Y}=-\overline{\ell_Y}$ yields
\begin{equation}\label{tloc:scalar-real-direction}
 \ell_Y=\rho/2,\quad\ell_L=0,\quad L\rho=-2c-\rho^2<0,
 \qquad L=Y+\bar Y.
\end{equation}
The longitudinal rows and $B=\rho I-S/2$ give
\begin{equation}\label{tloc:scalar-S-evolution}
 D_LS=\rho S-S^2-4cI.
\end{equation}
On a constant-rank chart, $J=\ker\mathsf R$ has
$S|_J=s_0I$, $s_0=\pm a_0$ locally constant.
Differentiate $(S-s_0I)v=0$ and project to $J$.
The moving input contributes zero, and $\ell_L=0$, so
$\rho s_0=s_0^2+4c$, hence $\rho=4s_0$, contrary to $L\rho<0$.

For $m=2$ it remains to exclude constant positive-rank $B$ charts
with $Z_0\ne0$. If $\rank B=1$, \eqref{tloc:active-block} gives
$B=b\Pi$, $F=\ker B$, $\Pi=\pi_{F^\perp}$ and
$D_Y\Pi=D_{\bar Y}\Pi=0$.
Hermitian symmetry gives $\Pi S=s\Pi$, $S|_F=uI$.
Horizontal holomorphicity of $Z$ and
$(2u-\tr Z)^2=(\tr Z)^2-4\det Z$ imply $D_{\bar v}u=0$:
differentiate where $2u-\tr Z\ne0$ and on the interior of its zero
set, then continue. Since $Z_0\ne0$, $F=\ker(Z-uI)$, giving
\begin{equation}\label{tloc:rankone-barred-caller}
 \Pi(D_{\bar v}\Pi)(I-\Pi)=0\qquad(v\in E).
\end{equation}
In the unit gauge $b>0$ and a unitary frame $(x,n)$ of
$F\oplus F^\perp$, coupling and the longitudinal rows become
\[
 S=\begin{pmatrix}u&t\\0&s\end{pmatrix},\quad
 A=\begin{pmatrix}d&j\\0&a\end{pmatrix},\quad a=-2c,
 \quad s\in\mathbb R,
\]
\[
 D_{\bar Y}(b,s,u,a,d)=
 (-bs/2,\ bs-2c,\ -2c,\ -c(s+2b),\ -cu).
\]
The coefficient $a=-2c$ is written in this unit frame. It represents
a tensor with two $N^*$ factors, so its covariant derivative need not vanish.
By \eqref{tloc:rankone-barred-caller}, differentiating $B|_F=0$
in the $\bar v$ direction gives $C_v|_F=0$. Thus $Q|_F=0$
by \eqref{tloc:barred-P}.
The full repeated row on $x$ is
$(d+3u^2/2)x^2+QL(x)=0$.
The scalar $d+3u^2/2$ has covariant derivative $-7cu$ along $\bar Y$. Since $D_{\bar Y}u=-2c$,
that scalar is nonzero on a dense open set, where $\im Q=\C x^2$.
Write $P_0(x,n)=\mu x^2+\nu xn+\kappa n^2$.
The $xn,n^2$ outputs of the complete row on $\xi x+\zeta n$ give
\[
 \kappa(d-a)=0,\quad a+3s^2/2+\kappa j=0,\quad
 d+a+3us+\nu(d-a)=0.
\]
On regular subcharts put $\epsilon=0$ for $d\ne a$ and $\epsilon=1$
for $d=a$ as a field. Respectively $\kappa=0$ or $u=s+2b$.
The equation $2a+3s(s+2\epsilon b)=0$ has weighted derivative
$2c[(2-8\epsilon)b-7s]=0$, giving $(2-8\epsilon)b=7s$.
Substitution into the former and differentiation of the latter give
$(3+9\epsilon)bs=14c$ and $(8-4\epsilon)bs=14c$, impossible for
$\epsilon\in\{0,1\}$ and $c\ne0$.

If $B$ is invertible, put $\xi=\bar q/2$ and
$\mathcal N=(q/\bar q)U-I$. Equations \eqref{tloc:nilpotent-identities},
\eqref{tloc:ck-longitudinal}, \eqref{tloc:A-coupling},
\eqref{tloc:det-identities} and \eqref{tloc:q-derivatives} give
\[
 \xi\ne0,\quad\mathcal N^2=0\ne\mathcal N,\quad
 A=\xi^2(I+\mathcal N),\quad Z=\xi(2I+\mathcal N),
\]
\[
 D_Y\xi=-\xi^2/2,\qquad
 D_Y\mathcal N=[\mathcal N,B]-\tfrac\xi2(\mathcal N+\mathcal N^2).
\]
Full \ChernFirst{} on $(Y;Y,v)$, with \eqref{tloc:adapted-evaluation} and
\eqref{tloc:P-split}, gives $D_YQ=-QS$, so
$\mathscr M=\{(Q/\xi)v:v\text{ a smooth local section of }E\}$ is a $D_Y$-invariant
image module, without a constant-rank assumption on $Q$.
Choose $0\ne x\in\ker\mathcal N$ and put $w=Bx/\xi$.
The kernel is preserved by $D_Y+B$; thus for a smooth $\gamma$,
$D_Yx=\gamma x-\xi w$ and $D_Yw=(\gamma-3\xi/2)w$.
The full repeated row, including $QL(x)$, gives
$p=7x^2-12xw+6w^2\in\mathscr M$.
For $\mathcal D=(D_Y-2\gamma)/\xi$,
$\mathcal Dp=4xw-6w^2$ and $\mathcal D^2p=-6xw+14w^2$.
The coefficient determinant of $p,\mathcal Dp,\mathcal D^2p$ in
$(x^2,xw,w^2)$ is $140$. Inverting this system gives
$x^2,xw,w^2\in\mathscr M$.
Finally $\widehat B=qB$ satisfies
$\widehat B-\widehat B^*=\widehat B^*\mathcal N$.
Since $\mathcal Nx=0$, the skew-Hermitian left side also vanishes
against $x^*$; hence $(\widehat Bx)^*\mathcal N=0$.
Since $w$ is a nonzero multiple of $\widehat Bx$, it is orthogonal
to $\im\mathcal N$. In dimension two, a nonzero square-zero operator
has image equal to its kernel, the line through $x$. Hence $x,w$
are independent, so $x^2,xw,w^2$ are independent in $\im Q$,
contradicting $\dim\im Q\le2$.
\end{proof}

The proof of Theorem~\ref{p:thm:main} in
\S\ref{tloc:bothsignlocaltheorem} combines these exclusions with
image alignment and the low-rank criterion.

\section{Low-rank torsion and center-direction identities}\label{stage:fourfold-terminals}
We exclude low torsion ranks in balanced fourfolds before applying the
polynomial arguments on the remaining constant-rank regions.

For a line $C=\mathbb C e_C$ with $T(V,C)\subset C$, define its
character $\lambda\in V^*$ by $T(v,e_C)=\lambda(v)e_C$ for every $v$.
The character is independent of the nonzero generator $e_C$
and satisfies $\lambda(e_C)=0$.
In Sections~\ref{stage:fourfold-terminals}--\ref{stage:fourfold-quadric},
a \emph{metric fourfold system} means the Chern system of a prescribed
smooth balanced Hermitian fourfold with constant Chern HSC $c\ne0$.
The term imposes these hypotheses only where explicitly invoked;
other tensor systems retain their separately stated assumptions.

We use the conventions of Section~\ref{sec:conventions}. A polynomial
identity valid for generic auxiliary inputs holds for all inputs.
A nonvanishing condition need not.

\subsection{Structure of the fourfold proof}
The low-rank and center-direction arguments precede the generic spectral
analysis. Section~\ref{stage:fourfold-generic} proves \HscGenericEval{},
Euler dominance, and generic separability in that order. Section~\ref{stage:fourfold-factors}
finds ideal lines when the full zero factor has degree zero or has
degree one with a rank-three shift. In degree one with a rank-two shift, it finds a quadric on
which the tensor in \eqref{newunb:r1v3:graph} vanishes. Theorem~\ref{unif:thm:canonical-source-classification}
combines these cases.

\subsection{Preserved output and nilpotent operator algebras}\label{ff:sec:foundations}
\begin{lemma}[Restriction to an internally preserved distribution]
\label{r37:lem:internal-restriction}
Let $\nabla'$ be the Chern connection of the given metric, or
an actual partial connection on an involutive complex tangent distribution
$V$, with zero pure curvature and local prime-parallel frames. In the
latter case require the actual torsion identity
$T(u,v)=\nabla'_u v-\nabla'_v u-[u,v].$
Assume complete \ChernFirst{}, \ChernCyclic{} and \ChernPure{} and the full torsion-corrected
covariant commutator. Suppose
$H$ is a smooth subbundle with $T(H,H)\subset H$ and
$\nabla'_H H\subset H$. Locally there is a projection $\pi:V\to H$
parallel in every $H$ direction. The induced connection on $H$,
$T_H=T|_{\Lambda^2H}$ and
$P_H=(\Sym^2\pi)P|_{\Lambda^2H}$ satisfy the restricted tensor identities.
If $H$ is preserved in every ambient prime direction, the same
projection is ambient-prime-parallel.
\end{lemma}
\begin{proof}
For sections of $H$, the two connection terms and the torsion term
in the bracket formula lie in $H$. Thus $H$ is involutive, and
the induced connection has zero pure curvature by restriction.
A prime-parallel ambient frame exists by assumption, or as the metric-dual
frame of a holomorphic coframe in the Chern case. Choose fixed linear combinations of this frame spanning a
complement $F$ at the point and hence nearby. Both summands are
preserved in the stated directions, so projection along $F$ is parallel. For inputs in $H$, the intermediate
torsions, prime derivatives, and $b$ operators stay in $H$ or
$\Lambda^2H$. Applying $\Sym^2\pi$ to the ambient equations gives the
restricted identities and the full torsion-corrected commutator.
\end{proof}

\begin{corollary}[Low-rank output and the fourfold common kernel]
On a balanced prescribed metric with constant Chern HSC $c\ne0$:
\begin{enumerate}[label=(\arabic*),ref=\thetheorem(\arabic*)]
\item\label{ab:cor:lowrank}
A smooth prime-preserved subbundle $N\supset\im T$ of rank at most
three forces $T=P=0$.
\item\label{cor:local-fourfold-kernel}
In dimension four, $\mathscr K=0$ on an open dense set.
\end{enumerate}
\end{corollary}
\begin{proof}
In (1), apply Theorem~\ref{thm:conceptual-lowrank} for rank at most two.
In rank three, Lemma~\ref{r37:lem:internal-restriction} gives the restricted
identities. The induced action on $V/N$ is zero since every torsion
value lies in $N$. Thus the trace on $N$ is the ambient trace, which
vanishes. Theorems~\ref{thm:conceptual-closed-three}
and~\ref{unif:thm:balanced-abelian-source} give successively
$T(N,N)=0$ and $T=0$. \ChernBarred{} gives $P=0$.
For (2), \eqref{ck:skew} gives $\dim\mathscr K\le1$.
Any open nonzero-kernel set has the preserved rank-three output
$N=\Ann\mathscr K$ containing $\im T$, by Lemma~\ref{tloc:conditional-transport}.
Part (1) makes $P=0$, contradicting
$\mathscr B(v,w)=-v\wedge w$ and $\mathscr K\ne0$.
The zero-kernel locus is therefore dense and is open by maximal minors.
\end{proof}

\begin{lemma}[Common kernel and nilpotence via Engel's theorem]
\label{lem:engel-common}\label{lem:nilpotent-engel}
Let $\mathfrak g\subset\End W$ be a finite-dimensional complex Lie
algebra consisting of nilpotent operators. If $W\ne0$, their common
kernel is nonzero. Consequently, a finite-dimensional complex Lie
algebra $\mathfrak l$ whose adjoint operators are nilpotent is nilpotent.
\end{lemma}
\begin{proof}
The common-kernel assertion is Engel's theorem \cite[\S3.3]{HumphreysLie1972}.
For the second, apply the common-kernel assertion to the adjoint
operators. Their common kernel is the center. Induction on dimension
makes the quotient by the nonzero center nilpotent, and a central
extension of a nilpotent Lie algebra is nilpotent.
\end{proof}
\subsection{The nonzero-center-character obstruction}
\label{sec:conceptual-heisenberg-first}

\begin{proposition}[Center-direction identities for an invertible transverse block]
\label{prop:heisenberg-first}
Let a rank-four bundle with a partial connection have
$\mathcal X=\nabla'T$ and satisfy \eqref{eq:cyclic} and \eqref{eq:pure}.
Assume a smooth splitting
$V=\C e\oplus U\oplus\C z$ with $\rank U=2$, and $\det B\ne0$ in the table
\[
 T(e,z)=z,\quad T(u,v)=\omega(u,v)z,\quad T(z,U)=0,\quad
 T(e,u)=Bu+\tau(u)z,\quad \tr B=-1.
\]
Here $\omega$ is nondegenerate; $B\in\End U$ and $\tau\in U^*$ may vary. For
$D=U\oplus\C z$ the table gives $T(V,D)\subset D$, with the complex
Heisenberg Lie bracket on $D$; Jacobi need not hold on $V$.
For $\chi(e)=1$, $\chi(U\oplus\C z)=0$, put
$\ell(u)=\chi(\nabla'_zu)$. Then
\[
 \ell=0,\quad \mathcal X_z(u,z)=\mathcal X_u(u,z)=0,\quad
 \nabla'_z(U\oplus\C z)\subset U\oplus\C z,\quad\nabla'_zz\in\C z.
\]
\end{proposition}
\begin{proof}
Write $r_p=\chi(\nabla'_pz)$ and $\nabla'_zz=r_ze+w+az$,
$w\in U$. Normal \ChernCyclic{} on $(u,v,z),(e,u,z)$ gives
$r_z=0$, $r_u=\ell(Bu)$; its first $U$ row is
\begin{equation}\label{eq:heisenberg-w}
 \omega(u,v)w=\ell((I-B)u)Bv-\ell((I-B)v)Bu.
\end{equation}
Thus $w=Bv_a$, $a=\ell(I-B)$, $\omega(u,v_a)=a(u)$.
Since $\omega(u,Bv)+\omega(Bu,v)=-\omega(u,v)$,
$\omega(u,w)=-\ell((I-B^2)u)$. The full derivative of $T(u,z)=0$ gives
\[
 \mathcal X_z(u,z)=-\ell(B^2u)z,\qquad
 \mathcal X_u(u,z)=\ell(Bu)Bu+\alpha(u)z.
\]
Put $f=\ell B$, $g=fB$. Repeated \ChernPure{} has zero intermediate
bivectors and torsion direction, hence
$f(u)E(Bu,z)+g(u)E(z,u)+\alpha(u)z^2=0$.
Modulo $\C z^2$, set $L(u)=[E(u,z)]$, $J(u)=[u\odot z]$.
Canonical symmetry gives
$L(f(u)Bu-g(u)u)+g(u)J(u)=0$, whose $L$ input lies in $\ker f$.
If $g\ne0$, then $f\ne0$. For $g(u)\ne0$, divide by $g(u)$ to get
$J(u)\in L(\ker f)$. Such $u$ span $U$, so $J(U)\subset L(\ker f)$,
of dimension at most one. But $J$ is injective on
the two-plane $U$ since $U\odot z$ meets $\C z^2$ trivially.
Thus $g=0$;
invertibility gives $\ell=w=0$, and the unprojected row gives $\alpha=0$.
\end{proof}

\subsection{Tensor identities along the central line}
\label{sec:heisenberg-shared-rows}
\begin{lemma}[Tensor identities before projection along the line]
\label{lem:heisenberg-shared-rows}
In the setting of Proposition~\ref{prop:heisenberg-first}, assume
$\mathcal X=\nabla'T$ and that the displayed splitting is smooth.
Put $D_0=U\oplus Z$, $Z=\C z$, and $\chi(e)=1$, $\chi(D_0)=0$.
After removing $\tau$ by the smooth change below, the actual fields
$r,\gamma,W,L,k,m,\alpha,A,a$ satisfy
\eqref{factorH:basic}, \eqref{factorH:shapes} and
\eqref{factorH:center}. None of these identities assumes a factor-line cut.
\end{lemma}
\begin{proof}
Replace $e$ by $e+a$, where $\omega(a,\cdot)=-\tau$, to remove $\tau$ without changing $U,z,\chi,B$.
Proposition~\ref{prop:heisenberg-first} and normal \ChernCyclic{} on $(e,u,z)$
give $\nabla'_zD_0\subset D_0$, $\nabla'_zz\in Z$,
$\mathcal X_z(u,z)=\mathcal X_u(u,z)=\chi(\nabla'_uz)=0$.
Define
\[
 \begin{gathered}
 \nabla'_z\chi=\gamma\chi,\quad r=\chi(\nabla'_ez),\quad
 W(u)=\pi_U\nabla'_uz,\quad\mathcal X_z(e,z)=\gamma z,\\
 \mathcal X_z(e,u)=Lu+k(u)z,\qquad
 \mathcal X_e(u,z)=rBu+m(u)z.
 \end{gathered}
\]
Here $L$ includes the moving $B$ derivative. The torsion trace is $1+\tr B=0$ in direction $e$ and zero on $D_0$,
so the derivative trace also vanishes. Together with \ChernCyclic{} $(e,u,z)$,
whose remaining $U$ term is $(I-B)Wu$, this gives
\begin{equation}\label{factorH:basic}
 L=(I-B)W-rB,\qquad\tr L=-\gamma.
\end{equation}
Put $\alpha(u,v)=\chi(\nabla'_uv)$,
$A(u,v)=\alpha(u,v)+\omega(v,Wu)$ and
$a(u,v)=\chi\mathcal X_u(e,v)$.
Then $\mathcal X_u(v,z)=-A(u,v)z$, so its zero diagonal makes
$A$ alternating. \ChernCyclic{} gives $\mathcal X_z(u,v)=2A(u,v)z$ and
\begin{equation}\label{factorH:shapes}
 a(u,v)=\alpha(u,Bv),\qquad a(u,v)-a(v,u)=r\omega(u,v).
\end{equation}
The moving coefficients of $Bv$ have $U$ output, killed by this
normal projection. Put $\Lambda(v)=[E(v,z)]\bmod\C z^2$,
$J(u)=[u\odot z]$. Complete \ChernPure{} $(z,e;u,z)$ reads
\[
 -E(\mathcal X_z(e,u),z)/2+\gamma E(z,u)/2
 -E(\mathcal X_e(u,z),z)=0:
\]
its intermediate bivectors are $-(B+I)u\wedge z/2$ and zero,
and $F_z(T)(U,z)=0$ also kills the torsion-direction term.
Canonical symmetry gives
\begin{equation}\label{factorH:center}
 \Lambda M=\gamma J,\qquad
 M=(I-B)W+rB+\gamma I=L+2rB+\gamma I.
\end{equation}
\end{proof}

Suppose $Z=\C z\subset H\subset V$ are smooth subbundles on a chart, $z$ is a local generator, and $E(H,z)\subset Z\odot V$. Let $u$ be a local section of $H$ with $q_Zu\ne0$, $\nabla'_uu\in H$ and $T(u,z)=0$. If
$E(u,z)=t\,u\odot z+\ell z^2$ and $\nabla'_uz\equiv ku\pmod Z$ as fields, then
\begin{equation}\label{sec:heisenberg-first-comparison}
 (1-2t)k=0.
\end{equation}
Indeed \ChernFirst{} gives $(\nabla'_uE)(u,z)=0$. Projecting the derivative to $\Sym^2(V/Z)$ leaves $2tk(q_Zu)^2$ from output motion, zero from the $u$-input correction by the cut, and $k(q_Zu)^2$ from the $z$-input correction by $E(u,u)=u^2$.
\subsection{Shifted images, abelian hyperplanes, and low-rank torsion}
\label{sec:conceptual-w4-covector}
The algebraic implication is dimension independent; we apply it to
metrics in dimension four.

\begin{lemma}[A three-form needs three independent contractions]
\label{lem:threeform-contraction}
Let $S:\Lambda^2V\to N$ be onto, with $\dim N\le2$, and let
$C,K\in V^*\otimes N^*$. If
\[
 \Omega_C(p,i,j):=C_pS(i,j)-C_iS(p,j)+C_jS(p,i)
                     =3(C_p+K_p)S(i,j),
\]
then $C+K=0$ and $\Omega_K=0$.
\end{lemma}
\begin{proof}
The contractions of the alternating three-form $\Omega_C$ lie in
$\im S^*$, of dimension at most two. A nonzero three-form has at least
three independent contractions: restrict to a three-space where it is a
volume form. Hence $\Omega_C=0$, and surjectivity gives $C+K=0$;
then $\Omega_K=-\Omega_C=0$.
\end{proof}

\begin{proposition}[A dimension-free covector syzygy]
\label{prop:conceptual-w4-covector}
Let $V$ be a complex tangent distribution with an actual partial
connection, and let $T$ be a trace-free alternating product.
Put $\mathcal X=\nabla'T$ and, for a smooth covector $\rho$, set
$S=T-B_\rho.$
Suppose $N=\im S$ is smooth of rank two and there are smooth fields
$L\in V^*\otimes\operatorname{Hom}(N,V)$ and $\upsilon\in N^*$ satisfying
\begin{equation}\label{eq:w4-explicit-source-inputs}\equationalias{eq:w4-zero-full-entrances}
 \begin{split}
 \mathcal X_p(i,j)&=\mathfrak H_p(i,j)
 +\tfrac13\{2L_pS(i,j)+L_iS(p,j)-L_jS(p,i)\},\\
 (\nabla'_p\rho)(v)&=-2\rho(p)\rho(v)
                         +(\upsilon-\rho|_N)S(p,v).
 \end{split}
\end{equation}
If $\rho\circ T=0$ as an actual field, then $N$ is an abelian
torsion ideal and $\rho$ descends to $Q=V/N$.
For $a\in Q$, let $A_a=T(a,\cdot)|_N$, independent of the lift.
Then the intrinsic alternating endomorphism-valued form
\[
 \mathcal J(a,b)=[A_a,A_b]-2\rho(a)A_b+2\rho(b)A_a
\]
satisfies $\upsilon\mathcal J=0$. 
\end{proposition}
\begin{proof}
Since $S$ is onto and $\rho B_\rho=\rho T=0$, $\rho|_N=0$.
Thus $T(V,N)\subset N$, and the internal trace of $T(n,\cdot)$
equals its zero ambient trace. On a two-space this forces
$T(N,N)=0$, so the quotient action $A$ is well-defined.
Put $B_p=S(p,\cdot)|_N=A_{qp}-\rho(p)I$.
Differentiating the actual identity $\rho T=0$, using
$\rho\mathfrak H=0$ and the covector law, gives
\[
 0=\upsilon S(p,T(i,j))+
 \{2\rho L_pS(i,j)+\rho L_iS(p,j)-\rho L_jS(p,i)\}/3.
\]
Substitute $S(p,T(i,j))=B_pS(i,j)+\rho(i)S(p,j)-\rho(j)S(p,i)$.
For $C_p=\rho L_p+3\rho(p)\upsilon$ and
$K_p=\upsilon B_p-2\rho(p)\upsilon$ this is
$\Omega_C=3(C+K)S$.
Lemma~\ref{lem:threeform-contraction} gives
$C+K=0$ and $\Omega_K=0$.
For $n\in N$, $B_n=\rho(n)=K_n=0$; evaluating $\Omega_K$ on
lifted $(e,f,n)$ yields
\[
 \upsilon\{[B_e,B_f]-2\rho(e)B_f+2\rho(f)B_e\}n=0.
\]
Replacing $B_a$ by $A_a-\rho(a)I$ proves the assertion.
\end{proof}

\subsubsection{Covector transport and scalar compatibility}

\begin{lemma}[Transport of an alternating covector derivative]
\label{lem:covector-alternating-transport}
For a prime-flat connection, write $T=t B_\rho+S$, with $s,t\in\C$ constant. Suppose $\rho T=0$ and $\nabla'_p\rho=s\rho(p)\rho+\omega(p,\cdot)$, with $\omega$ alternating. The identity below also holds in a system of differential tensor identities with the complete torsion-corrected covector commutator. Then
\begin{equation}\label{eq:covector-alternating-transport}\equationalias{eq:w4-zero-omega-transport}
\begin{split}
 (\nabla'_p\omega)(q,v)={}&-s\rho(v)\omega(p,q)+s\rho(q)\omega(p,v)
 +(2s-t)\rho(p)\omega(q,v)\\
 &+\tfrac12\{-\omega(S(p,q),v)+\omega(S(q,v),p)
                                      +\omega(S(p,v),q)\}.
\end{split}
\end{equation}
\end{lemma}
\begin{proof}
Put $A(p,q,v)=(\nabla'_p\omega)(q,v)$.
Expand the pure-flat covector commutator using
$\nabla'\rho=s\rho\otimes\rho+\omega$, including the torsion term:
\begin{align*}
 A(p,q,v)-A(q,p,v)
 &=-2s\rho(v)\omega(p,q)+(t-s)\rho(q)\omega(p,v)\\
 &\quad +(s-t)\rho(p)\omega(q,v)-\omega(S(p,q),v)=:C(p,q,v).
\end{align*}
Since $A$ is skew in its last two slots,
$2A(p,q,v)=C(p,q,v)-C(q,v,p)+C(v,p,q)$.
Substitution gives the displayed identity in every dimension.
\end{proof}

\begin{lemma}[Dimension-free covector compatibility]
\label{lem:matching-binary-covector}
\label{r48:lem:covector-compatibility}
Let $H$ be an actual involutive complex tangent distribution with a
partial connection of zero pure curvature and torsion $T_H$.
Let $K\subset H$ be a preserved ideal subbundle, possibly zero.
Suppose $\alpha(K)=0$, $\nabla_k\alpha=0$ for $k\in K$, and
$T_H=tB_\alpha\pmod K$ for a constant $t\in\mathbb C$.
If
\[
 \operatorname{Sym}\nabla\alpha=s\alpha^2,
 \qquad s\in\mathbb C,\qquad 3s\ne t,
\]
then $\nabla\alpha=s\alpha^2$, with no restriction on $\rank(H/K)$.
The same assertion holds for a system of differential tensor identities with
the complete covector and two-form commutators; no quotient manifold
or metric is required.
\end{lemma}
\begin{proof}
Put $\omega=\nabla\alpha-s\alpha^2$ and $S=T_H-tB_\alpha$.
Preservation and the stated vanishings give $\omega(K,H)=0$ and
$\alpha T_H=0$. Apply Lemma~\ref{lem:covector-alternating-transport} with $(\rho,\omega)=(\alpha,\omega)$ and the present $s,t$. Its three $S$ terms vanish since $S$ is $K$-valued and $\omega(K,H)=0$.
This gives $\nabla_K\omega=0$ and, for arbitrary $u,v$,
\[
 (\nabla_u\omega)(u,v)=c\alpha(u)\omega(u,v),\qquad
 (\nabla_v\omega)(u,v)=c\alpha(v)\omega(u,v),\quad c=3s-t.
\]
Differentiate these actual repeated identities with the derivative
slot and both input slots retained, then evaluate in a normal frame.
Their difference and the complete pure-flat two-form commutator give
\[
 \begin{split}
 2c\,\omega(u,v)^2
 &=-(\nabla_{T_H(u,v)}\omega)(u,v)\\
 &=-tc\{\alpha(u)\alpha(v)-\alpha(v)\alpha(u)\}\omega(u,v)=0.
 \end{split}
\]
Here $T_H(u,v)=t\{\alpha(u)v-\alpha(v)u\}+S(u,v)$;
its $K$-valued term vanishes by $\nabla_K\omega=0$.
Since $c\ne0$, every $\omega(u,v)$ is zero. The same full tensor
calculation proves the assertion for differential tensors.
\end{proof}

\begin{proposition}[Obstruction for a nondegenerate alternating form]
\label{prop:w4-symplectic-covector}
Work on a metric fourfold system. Suppose the hypotheses of
Proposition~\ref{prop:conceptual-w4-covector} hold, including
$\rho T=0$, on an open region where $\rho$ is nowhere zero.
If $\omega=\upsilon S$ is nondegenerate, the region is empty.
\end{proposition}
\begin{proof}
Proposition~\ref{prop:conceptual-w4-covector} gives the abelian ideal
$N\subset\ker\rho$ and $\upsilon\mathcal J=0$.
Put $q:V\to Q$, $B_p=A_{qp}-\rho(p)I$ and $r_p=\upsilon B_p$.
Balance gives $\tr B_p=-3\rho(p)$. The plane $N$ is isotropic for the nondegenerate $\omega=\upsilon S$
because $S(N,N)=0$. It is therefore Lagrangian. The induced pairing
between $V/N$ and $N$ is nondegenerate, so $r:Q\to N^*$ is invertible.
Consequently the actual field $\omega(c,\cdot)=\rho$ lies in $N$,
with $r_p(c)=-\rho(p)$.
Writing $F(r_p)=B_p$, the syzygy and trace become
\[
 \upsilon F(s)=s,\quad\tr F(s)=3s(c),\quad
 rF(s)-sF(r)=-2r(c)s+2s(c)r.
\]
Set $r=\upsilon$ and trace:
$F(\upsilon)=(1+2\upsilon(c))I-2c\otimes\upsilon$,
$\upsilon(c)=2$. Thus $F(\upsilon)$ has eigenvalues $1,5$.

The syzygy on $c$ makes $\upsilon B_pB_vc$ symmetric in $p,v$.
Lemma~\ref{lem:covector-alternating-transport}, with $(s,t)=(-2,1)$,
gives $(\nabla'_p\omega)(c,v)=-7\rho(p)\rho(v)+\upsilon B_vB_pc$.
Differentiate $\omega(c,\cdot)=\rho$, using
$\omega(B_pc,v)=-\upsilon B_vB_pc$. Nondegeneracy gives
\begin{equation}\label{eq:w4-dilation}
 \nabla'_pc=p+K_p,\qquad K_p=B_pc+5\rho(p)c.
\end{equation}
Here $K(V)\subset N$, $K(N)=0$. For $\beta_pn=q\nabla'_pn$,
this implies $\beta_pc=qp$. The complete pure-flat commutator,
including torsion, has normal output
\[
 \beta_pK_v-\beta_vK_p=-qT(p,v),\qquad
 D_{pv}:=\beta_pB_vc-\beta_vB_pc=4qT(p,v),
\]
because $qT(p,v)=\rho(p)qv-\rho(v)qp$.
Both input corrections have $N$-valued output.
Actual differentiation of $T=B_\rho+S$ gives
\[
 q\mathcal X_p(u,v)=\beta_pS(u,v)
 +(\nabla'_p\rho)(u)qv-(\nabla'_p\rho)(v)qu.
\]
The Jacobiator on $(p,v,c)$ is $N$-valued, since $N$ is an ideal.
Normal \ChernCyclic{} therefore gives
$D_{pv}+\beta_cS(p,v)+2qT(p,v)=0$, hence $\beta_cS=-6qT$.
Choose $r_e=\upsilon$. Since $S(e,n)=F(\upsilon)n$ is invertible
on $N$ and $qT(e,n)=0$, this forces $\beta_c=0$, then $qT=0$.
That contradicts $\rho\ne0$ on the two-dimensional quotient.
\end{proof}

\subsubsection{A rank-two full image from the Chern system}
The rank-two exclusion uses the Chern identities and the normal
factorization of $P$, with no further torsion-derivative condition.

\begin{lemma}[Traceless symmetrizers on a space of rank at most two]
\label{lem:lowrank-symmetrizer}
Let $W\subset V$ have dimension at most two, and let
$E:W\times W\to\Sym^2V$ satisfy $E(x,y)+E(y,x)=2xy$ for $x,y\in W$.
If $L\in\End W$ has zero trace and
$E(Lx,y)=E(Ly,x)$ for all $x,y\in W$, then $L=0$.
\end{lemma}
\begin{proof}
In dimension at most one the trace gives $L=0$. In dimension two,
a traceless matrix satisfies $L^2=-(\det L)I$. The defining identity gives
$(Lx)^2=E(Lx,Lx)=E(L^2x,x)=-(\det L)x^2$. Projecting to the symmetric
square of the quotient by $\C x$ shows that $Lx\in\C x$ for every
nonzero $x$. Thus $L$ is scalar, and its trace gives $L=0$.
\end{proof}

\begin{theorem}[Obstruction for an injective rank-two action]
\label{thm:ranktwo-injective-action}
Let an actual rank-four source have trace-free torsion $T$,
$\mathcal X=\nabla'T$, and satisfy \ChernFirst{}, \ChernCyclic{} and complete \ChernPure{}, with
\[
                  \tr \mathcal X_p(u,\cdot)=0,\qquad
                         \sum_p \mathcal X_p(u,v)^p=0.
\]
Suppose $N=\im T$ is smooth of rank two and the induced action
$A:Q=V/N\to\mathfrak{sl}(N)$ is injective. Write $\pi:V\to Q$.
If there is a bundle map $\Phi:N\to\Sym^2Q$ such that
\begin{equation}\label{eq:ranktwo-original-normal-factor}
                         (\Sym^2\pi)P=\Phi T,
\end{equation}
no such actual region exists.
\end{theorem}
\begin{proof}
For $n\in N$, the ambient trace is the trace of $T(n,\cdot)|_N$,
since the quotient output is zero. In dimension two these traces
are the coefficients of the sole bracket of basis vectors, so
$T(N,N)=0$. The normal
derivative $\beta_pn=\pi\nabla'_pn$ satisfies, by differentiating both inputs,
\begin{equation}\label{eq:ranktwo-original-abelian-derivative}
 \mathcal X_p(n,m)=-A_{\beta_pn}m+A_{\beta_pm}n\in N.
\end{equation}
For a basis $x,y$ of $N$, the direction-output trace gives
$\tr L=0$ for $L(n)=\mathcal X_n(x,y)$, since its ambient output lies in $N$. Repeated \ChernPure{} gives
$E(Lx,y)=E(Ly,x)$, so Lemma~\ref{lem:lowrank-symmetrizer} gives
\begin{equation}\label{eq:ranktwo-original-internal-X}
 \mathcal X_N(N,N)=0.
\end{equation}
Set $C_q(a,m)=\mathcal X_a(q,m)$,
$W_q(a,n)=C_q(a,n)-C_q(n,a)=\mathcal X_q(a,n)$.
For $a,n,m\in N$, complete \ChernPure{} and normal \ChernCyclic{} give
\begin{equation}\label{eq:ranktwo-original-full-R}\equationalias{eq:ranktwo-original-residual}
 E(C_q(a,m),n)-E(C_q(n,m),a)+E(W_q(a,n),m)=0.
\end{equation}
\begin{equation}\label{eq:ranktwo-original-normalrows}
 A_{\beta_am}n=A_{\beta_an}m,\qquad\beta_aA_qn=\beta_nA_qa.
\end{equation}
The quadratic terms pass through the abelian plane $N$. Some $A_r$ is invertible. Otherwise the determinant, equal to
$-\operatorname{tr}(A_r^2)/2$ for a traceless two-by-two matrix,
would vanish on the image of $A$. By polarization this two-dimensional
image would be totally isotropic for the nondegenerate trace pairing
on the three-dimensional $\mathfrak{sl}_2$, which is impossible. Normal projection of the full residual row, using
this $A_r$, and repeated \ChernPure{} modulo $\im\Phi$ give respectively
\begin{align}
 \Phi A_{\beta_an-\beta_na}&=0,\label{eq:ranktwo-original-phirow}\\
 (\pi q)\odot\beta_n(A_qn)&\in\im\Phi.
 \label{eq:ranktwo-original-phirange}
\end{align}
Here $A_q=A_{\pi q}$; quotient bases below use chosen lifts.
Work pointwise. If $\Phi$ is injective, \eqref{eq:ranktwo-original-phirow}
and injectivity of $A$ give $\beta_{an}=\beta_{na}$. Otherwise
$\dim\im\Phi\le1$; a nonzero generator has at most two linear factors.
For general $q$, \eqref{eq:ranktwo-original-phirange} then gives
$\beta_n(A_qn)=0$, hence for all $q,n$. Polarization and the second
normal row give $2\beta_aA_qn=0$; invertibility of $A_r$ gives
$\beta_{N\times N}=0$.

We first isolate the contradiction for a square-zero action. If $A_q\ne0$, $A_q^2=0$,
$C_q=0$ and $D_q=\mathcal X_q(q,\cdot)|_N$ is $N$-valued,
complete repeated \ChernPure{}, retracted only at its final output, gives
\[
 E_N(2D_qn,m)+E_N(2D_qm,n)+3A_qn\odot A_qm=0,
 \quad E_N=(\Sym^2\rho)E|_{N^2},\quad\rho|_N=I.
\]
Lemma~\ref{ab:lem:repeated} gives $\ker A_q^2=\ker A_q$,
contrary to $A_q^2=0\ne A_q$. We use this square-zero obstruction below.

\phantomsection\label{lem:ranktwo-asymmetric-residual}
For an $N$-valued row-traceless $C$ satisfying
\eqref{eq:ranktwo-original-full-R}, write
$W(a,n)=\omega(a,n)w$ and $P(a,n)=\omega(a,n)\kappa$.
If $w=0$, Lemma~\ref{lem:lowrank-symmetrizer} gives $C=0$.
Otherwise choose $x=w$, $\omega(x,y)=1$. The rows $m=x,y$ and
their traces give
\[
 C_{xx}=0,\quad C_{xy}=2x,\quad C_{yx}=x,\quad C_{yy}=tx-y,
 \qquad\kappa=tx^2/2-2xy.
\]
Thus $E(N,N)\subset\Sym^2N$. The unique traceless $J$ satisfying
$P(a,n)=(Ja)n-a(Jn)$ has $Jx=-x$, $Jy=y-tx/2$, so $J^2=I$ and
\begin{equation}\label{eq:ranktwo-original-residual-form}
 C(a,a)=\omega(w,a)(I-2J)a.
\end{equation}
Thus all alternating residual images lie in the same line
$\ker(J+I)$ determined by $P$.

The determinant on $A(Q)$ is a nonzero binary quadratic. A zero
line gives a nonzero square-zero matrix, since $A$ is injective
and its image consists of traceless matrices. First suppose the
determinant has just one zero line, represented by $e\in Q$. Choose $f\in Q$
with $A_f$ semisimple and a basis $(x,y)$ with
$A_ey=x$, $A_ex=0$, $A_fx=x$, $A_fy=-y$. Indeed the double root
gives $\tr(A_eA_f)=0$, so $A_f$ preserves $\ker A_e$.
For symmetric $\beta|_{N^2}$, \eqref{eq:ranktwo-original-normalrows} gives
$\beta_{xx}=\beta_{xy}=0$, $\beta_{yy}=te$. Normal \ChernCyclic{} gives
$\beta_ex=0$. For $a\in N$ one has $\pi C_e(a,n)=\beta_a A_en=0$.
The quotient block of $\mathcal X_a(e,\cdot)$ takes values in $\C e$
and kills $e$, so $C_e$ has zero trace in its second slot.

If $t\ne0$, then $\beta_y(A_qy)=-tq_f e$, where $q_f$ is the
$f$ coefficient. Equation~\eqref{eq:ranktwo-original-phirange} gives
$\im\Phi=eQ$. For the covector $\lambda(e)=0$, $\lambda(f)=1$ on $Q$,
the scalar projection of $E$ is thus $\lambda^2$. In that projection
the quadratic and torsion-direction terms of \ChernPure{} vanish,
and directions $(e,f)$ on inputs $(f,y)$ give
$-\lambda\beta_ey=\lambda\beta_fx/2$. Normal \ChernCyclic{} at $(e,f,y)$
gives $\lambda\beta_fx=-\lambda\beta_ey$, since $\im\beta_y\subset\C e$.
Thus $\lambda\beta_ey=0$ and $W_e(x,y)=0$.
If $t=0$, all $C_q$ are $N$-valued and row-traceless.
Normal \ChernCyclic{} gives $\beta_fx=-\beta_ey$; writing
$\beta_ey=pe+qf$ gives $W_e(x,y)=qx$ and $W_f(x,y)\equiv-qy\bmod\C x$.
The common-line conclusion again gives $W_e=0$. In either case
$C_e=0$ and $D_e$ is $N$-valued because $\beta_ex=0$; the square-zero
obstruction applies.

If the determinant has two distinct zero lines $e,d$, their kernels
are distinct. Choose bases with $A_ez=x$, $A_dx=z$ and
$A_ex=A_dz=0$. The second normal row gives $\beta_{xx}=\beta_{zz}=0$.
Using symmetry, the first makes $A_{\beta_{xz}}$ kill $x,z$, hence $\beta_{N\times N}=0$.
All $C_q$ are then $N$-valued and row-traceless. Normal \ChernCyclic{}
gives $\beta_ez=\beta_dx=0$. Differentiated abelianity puts
$W_e(x,z)=-A_{\beta_ex}z\in\C x$ and $W_d(x,z)=A_{\beta_dz}x\in\C z$.
By the common-line conclusion one vanishes; exchange the two
lines so $W_e=0$. Thus $C_e=0$ and $\beta_ex=pd$.
If $p=0$, $D_e$ is $N$-valued and the square-zero obstruction applies.

Otherwise apply \ChernPure{} in directions $(e,d)$ on inputs $(n,r_0)$
modulo $\im\Phi$. The quadratic and torsion-direction terms vanish
there; normal \ChernCyclic{} cancels the $\pi r_0$ term. Varying $n$
with $A_{r_0}$ invertible gives
$M(n):=\beta_e(n)\odot d-\beta_d(n)\odot e\in\im\Phi$.
At $n=x$ this gives $pd^2\in\im\Phi$. Repeated \ChernPure{} at $(e,x+z)$
has $s_v=0$, $\pi s_u=pd$ in \eqref{eq:repeated}; its normal projection
is $p\Phi(z)=0$. Since $p\ne0$ and $x,z$ span $N$, the image of
$\Phi$ has dimension at most one. It contains $pd^2\ne0$, so
$\im\Phi=\C d^2$.
Since $\beta_e(N)\subset\C d$, the identity for $M$ puts
$e\odot\beta_d(n)$ in $\C d^2$, so $\beta_d=0$.
Then $W_d=C_d=0$, $D_d$ is $N$-valued, and the square-zero
obstruction applies to $A_d$. No normal form was differentiated.
\end{proof}

\subsubsection{A central image line and a larger kernel}
\begin{lemma}[Preservation of a central line]
\label{lem:central-line-kernel}
Let an actual source satisfy \ChernCyclic{} and complete \ChernPure{}. Suppose its
full torsion image is a smooth line $N$, with $T(N,V)=0$, and there
is a smooth bundle $K$ strictly containing $N$ such that $T(K,V)=0$.
Write $\pi:V\to V/N$. If
$(\Sym^2\pi)P=\Phi T$
for a bundle map $\Phi:N\to\Sym^2(V/N)$, the region is empty.
\end{lemma}
\begin{proof}
Put $\beta_pn=\pi\nabla'_pn$. Image differentiation gives
$\pi\mathcal X_p=\beta_pT$. Normal \ChernCyclic{} on $(z,i,j)$,
$z\in K$, gives $\beta_zT=0$, all other terms vanishing by centrality.
Thus $\beta_z=0$ and $\mathcal X_z$ is $N$-valued. Differentiation of $T(V,K)=0$ also makes
$\mathcal X_p(i,z),\mathcal X_p(z,j)$ $N$-valued.
Since $(\Sym^2\pi)E(N,V)=0$, the projected central row of
Lemma~\ref{unif:lem:central-PURE-row} gives
$0=-(\beta_pT(i,j))\odot\pi z$: both nested torsions are $N$-valued,
and the $P$ part of $E(\mathcal X_p(i,j),z)$ projects to zero by centrality.
Choose $z\in K\setminus N$. Multiplication by $\pi z$ is injective,
so $\beta_pT=0$. Since $T:\Lambda^2V\to N$ is onto, $\beta_p=0$
for every $p$. Thus $N$ is a preserved central torsion image,
which Lemma~\ref{lem:central-pure} excludes.
\end{proof}

\subsubsection{Tensor identities for an abelian hyperplane}
\begin{lemma}[Identities for an abelian hyperplane]
\label{lem:abelian-hyperplane-source-rows}
Let an actual smooth rank-four source satisfy \ChernCyclic{} and complete \ChernPure{},
with $\mathcal X=\nabla'T$, $E(u,v)=uv+P(u,v)$ and
$\tr T(v,\cdot)=\tr\mathcal X_p(v,\cdot)=0$.
Let $V=\mathbb Ce\oplus H$, where $H$ is a smooth rank-three bundle
with $T(H,H)=0$ and $T(V,H)\subset H$. Write $T(e,h)=Ah$ and let
$\beta_p(h)e$ be the normal component of $\nabla'_ph$.
Then $\tr A=0$, and the complete family derivative is
\begin{equation}\label{ah:family}\equationalias{aa2:family}\equationalias{aa3new:family}
 \mathcal X_p(h,k)=-\beta_p(h)Ak+\beta_p(k)Ah,\qquad
 \mathcal X_p(e,h)=B_ph+\beta_p(Ah)e,\qquad\tr B_p=0.
\end{equation}
Put $b(p,h)=\beta_p(h)$ on $H\times H$ and
$\omega=b-b^{\mathsf T}$. The horizontal \ChernCyclic{}, normal \ChernCyclic{} and
repeated \ChernPure{} rows are, respectively,
\begin{align}
 &\omega(p,k)Ah-\omega(p,h)Ak-\omega(h,k)Ap=0,
                                                        \label{aa2:omega}\equationalias{aa3new:omega}\\
 &b(p,Ah)=b(h,Ap),                                      \label{aa2:bA}\equationalias{aa3new:bA}\\
 &-b(u,u)E(Av,v)+b(u,v)E(Au,v)
 +b(v,u)E(Av,u)-b(v,v)E(Au,u)=0.                       \label{aa2:brow}\equationalias{aa3new:brow}
\end{align}
If $\beta_H=0$, put $\gamma=\beta_e$.
For $p,h\in H$ define
\[
 C(p,h)=B_ph+\gamma(p)Ah,\qquad \rho=\gamma A,
 \qquad D=2B_e+A^2.
\]
Then $C=C^{\mathsf T}$, $\tr C_h=0$, $\tr D=\tr A^2$, and
\begin{equation}\label{ah:Crow}\equationalias{aa2:Crow}\equationalias{aa3new:Crow}
\begin{split}
 E(C(p,h),q)-E(C(q,h),p)
 ={}&\gamma(p)\{E(Aq,h)+E(Ah,q)\}\\
 &-\gamma(q)\{E(Ap,h)+E(Ah,p)\},
\end{split}
\end{equation}
\begin{equation}\label{ah:transverse}\equationalias{aa2:transverse}\equationalias{aa3new:transverse}
\begin{split}
 0={}&E(Dp,h)+E(Dh,p)+3Ap\odot Ah +2\rho(p)E(e,h)+2\rho(h)E(e,p)\\
 &+2\gamma(p)E(Ah,e)+2\gamma(h)E(Ap,e)-4E(C(p,h),e).
\end{split}
\end{equation}
These rows retain full ambient output and assume neither a metric
nor a rank condition on $A$.

For the prescribed balanced Hermitian metric with constant Chern HSC
$c\ne0$, choose an actual adapted unitary frame $(e_0=e,e_1,e_2,e_3)$.
Denote the family derivative above by $S_{\beta,B}$, and write
\[
 \nabla^a=h^{a\bar b}\nabla_{\bar e_b},\qquad
 Y^a=c^{-1}\nabla^aT=S_{\beta^a,N^a},\quad\tr N^a=0,\qquad
 \tfrac12\sum_a\beta^a(h)e_a=f(h)e+F_{\mathrm{bar}}h.
\]
The barred Chern identity and curvature reality give
\begin{equation}\label{aa3new:barredcomponents}
 P_{ij}^{ab}=-\tfrac14\{(Y^a)^b_{ij}+(Y^b)^a_{ij}\},\qquad
 J_{ij}^{ab}=-\tfrac14\{(Y^a)^b_{ij}-(Y^b)^a_{ij}\}
                                      =\overline{J_{ab}^{ij}}.
\end{equation}
The full reconstruction is
\begin{equation}\label{aa2:barred}\equationalias{aa3new:barred}
\begin{split}
 P(h,k)&=(f(h)e+F_{\mathrm{bar}}h)(Ak)-(Ah)(f(k)e+F_{\mathrm{bar}}k),\\
 P(e,h)&=Q(h)+e(Mh)-\alpha(h)e^2,\qquad \alpha=fA,\\
 Q(h)&\in\Sym^2H,\qquad M=-N^0/2-F_{\mathrm{bar}}A.
\end{split}
\end{equation}
If also $\beta_H=0$, put
$d(p,h)=f(C(p,h))$. Then
\begin{equation}\label{aa2:wedge}
 f\wedge\rho+\gamma\wedge\alpha=0.
\end{equation}
\begin{equation}\label{aa2:frow}\equationalias{aa3new:frow}
\begin{aligned}
 &d(p,h)Aq-f(q)AC(p,h)-d(q,h)Ap+f(p)AC(q,h)\\
 &\quad=\gamma(p)\{\alpha(q)Ah+\alpha(h)Aq
                  -f(h)A^2q-f(q)A^2h\}\\
 &\qquad-\gamma(q)\{\alpha(p)Ah+\alpha(h)Ap
                  -f(h)A^2p-f(p)A^2h\}.
\end{aligned}
\end{equation}
\begin{equation}\label{aa2:ftrans}\equationalias{aa3new:ftrans}
2\alpha C(p,h)={}\gamma(p)\alpha(Ah)+\gamma(h)\alpha(Ap) -\rho(p)\alpha(h)-\rho(h)\alpha(p).
\end{equation}
\end{lemma}
\begin{proof}
Differentiating the actual torsion family with both moving inputs gives
\eqref{ah:family}; trace freeness gives $\tr A=\tr B_p=0$.
The horizontal and normal parts of \ChernCyclic{} and the repeated \ChernPure{} row
on $u,v\in H$ give \eqref{aa2:omega}--\eqref{aa2:brow}.
When $\beta_H=0$, \ChernCyclic{} is
$B_ph-B_hp=-\gamma(p)Ah+\gamma(h)Ap$, so $C$ is symmetric and
$\tr C_h=0$. The horizontal and transverse \ChernPure{} substitutions give
\eqref{ah:Crow} and \eqref{ah:transverse}, using the full operator
\eqref{newff:eq:E-form-F}.
The full transverse quadratic part is
$E(A^2p,h)+E(A^2h,p)+3Ap\odot Ah$, and $\tr D=\tr A^2$.

The alternating curvature-reality identity
\eqref{newunb:r1v3:alternating-reality}, obtained by double alternation
of barred Bianchi and Hermitian curvature reality, gives
$J_{ij}^{ab}=\overline{J_{ab}^{ij}}$.
Substitute the full raised barred family into Proposition~\ref{prop:source}
to obtain the symmetric reconstruction \eqref{aa2:barred}.
Here \eqref{aa2:wedge} is the $e^2$ output of \ChernFirst{} for $P(h,k)$
in direction $e$. Its output derivatives give $f\wedge\rho$;
the moving-input corrections
$-\gamma(h)P(e,k)+\gamma(k)P(e,h)$ give $\gamma\wedge\alpha$.
Equations \eqref{aa2:frow} and~\eqref{aa2:ftrans} are respectively
the $eH$ output of \eqref{aa2:Crow} and the $e^2$ output of
\eqref{aa2:transverse}, using \eqref{aa2:barred}.

\end{proof}
\subsubsection{A common flag from the Chern system on the hyperplane}
\begin{lemma}[The common hyperplane flag]
\label{lem:abelian-hyperplane-common-flag}
Work on a metric fourfold system. Let $H\subset V$ be a smooth hyperplane with $T(H,H)=0$ and
$T(V,H)\subset H$, and choose a local unitary splitting $V=\mathbb Ce\perp H$.
Put $T(e,h)=Ah$, and let $\gamma(h)e$ be the normal part of
$\nabla'_eh$. Suppose
\[
 \nabla'_H H\subset H,\qquad \nabla''_{\bar e}H\subset H,
                         \qquad\gamma\ne0,\quad\rho:=\gamma A\ne0
\]
on the actual region. The reconstruction of $P$ from barred derivatives has the form
\begin{equation}\label{ah:barred-zero}
 P(h,k)=(F_{\mathrm{bar}}h)(Ak)-(Ah)(F_{\mathrm{bar}}k),\qquad P(e,h)=Q(h)+e(Mh),
                       \quad Q(h)\in\Sym^2H.
\end{equation}
There is a smooth nowhere-zero scalar $a$ such that, for $K=\ker\gamma$,
\begin{equation}\label{ah:common-conclusion}
 \rho=a\gamma,\qquad A(K),F_{\mathrm{bar}}(K),M(K)\subset K,
            \qquad A|_K=-\tfrac a2 I_K,\qquad M|_K=5I_K.
\end{equation}
In particular $A$ is invertible.
\end{lemma}
\begin{proof}
Use the complete rows of Lemma~\ref{lem:abelian-hyperplane-source-rows}.
The barred derivative has the same linear family form as
\eqref{ah:family}, and its $\bar e$ normal coefficient is zero by
$\nabla''_{\bar e}H\subset H$.
The barred Chern identity gives \eqref{ah:barred-zero}, retaining the full
tangential remainder $Q$ and remaining raised normal map $F_{\rm bar}$.

Put $U=I-M$. The complete $eH$ output of \eqref{ah:transverse} is
\begin{equation}\label{ah:Urow}
 2UC(p,h)=\rho(p)(I+M)h+\rho(h)(I+M)p
                       +\gamma(p)UAh+\gamma(h)UAp.
\end{equation}
Here $E(e,h)_{eH}=e(I+M)h$, $E(h,e)_{eH}=eUh$; the first three
terms of \eqref{ah:transverse} lie in $\Sym^2H$.
If $\ell U=0$, then $\ell(I+M)=2\ell$. Applying $\ell$ to the displayed
equation gives $\rho\otimes\ell+\ell\otimes\rho=0$, which implies
$\ell=0$ since $\rho\ne0$.
Thus $U$ is invertible. Solving with $Z=U^{-1}(I+M)$ and tracing gives
\begin{equation}\label{ah:Z}
 \begin{split}
 C(p,h)&=\{\rho(p)Zh+\rho(h)Zp+\gamma(p)Ah+\gamma(h)Ap\}/2,\\
 \rho Z&=-(\tr Z+1)\rho,\qquad Z+I=2U^{-1}.
 \end{split}
\end{equation}
Metricity gives $\nabla'_ee\in\mathbb Ce$.
In \ChernFirst{} for $P(e,h)=Q(h)+e(Mh)$, differentiating $e(Mh)$
contributes $\gamma(Mh)e^2$. The derivative of $Q$, the moving inputs,
and $E(Ah,e)$ have no $e^2$ output. Therefore
\begin{equation}\label{ah:gammaM}
                           \gamma M=0,\qquad\gamma Z=\gamma.
\end{equation}
Consequently $K=\ker\gamma$ is $M,U,Z$-invariant.
Put $r=\rho|_K$, $Z_0=Z|_K$. If $r\ne0$, take $0\ne p\in\ker r$.
Equation~\eqref{ah:Z} gives $Z_0p\in\C p$ and
$C(p,k)=r(k)Z_0p/2$ on $K$. For $r(k)=1$,
\eqref{ah:Crow} gives $E(Z_0p,p)=0$, so $Z_0p=0$.
Thus $Z_0=v\otimes r$, and $t=r(v)$ satisfies $t=-(2+t)$ by the trace row,
so $t=-1$, contradicting invertibility of $Z+I$. Thus
\begin{equation}\label{ah:flag}
                 \rho=a\gamma,\quad a\ne0,\qquad A(K)\subset K,
\end{equation}
with smooth $a$ because $\gamma$ is nowhere zero.

Choose $\gamma(s)=1$. Formula~\eqref{ah:Z} gives
$C(K,K)=0$, $C(s,k)=(aZ+A)k/2$.
The $(s,k;k)$ row of \eqref{ah:Crow} is
$E((aZ-3A)k,k)=0$.
Polarization and the algebraic implication \eqref{eq:anti},
with $R=(aZ-3A)|_K$, give $Z|_K=3A|_K/a$.
Thus $C(s,k)=2Ak$, and the $(p,q;s)$ row gives
\begin{equation}\label{ah:symA}
                         E(Ap,q)=E(Aq,p)\qquad(p,q\in K).
\end{equation}

Put $\vartheta=\gamma F_{\rm bar}$.
The entire $eH$ output of \ChernFirst{} for $P(h,k)$ is
\begin{equation}\label{ah:finalfirst}
 0=\vartheta(h)Ak-\vartheta(k)Ah+\rho(k)F_{\rm bar}h-\rho(h)F_{\rm bar}k
                                  -\gamma(h)Mk+\gamma(k)Mh.
\end{equation}
The first four terms come from the normal motion of the output
factors; the last two are the input corrections. Derivatives of
$A,F_{\rm bar},Q$ give no other $eH$ output. Applying $\gamma$ gives $\vartheta=b\gamma$ and
$F_{\rm bar}(K)\subset K$.
Write $B=A|_K$, $G=F_{\rm bar}|_K$, $N_0=M|_K$, $L=B/a$.
Equations \eqref{ah:finalfirst} on $(s,k)$ and \eqref{ah:Z} give
\[
 N_0=bB-aG=(3L-I)(3L+I)^{-1},\qquad\tr L=-1.
\]
The trace uses the quotient eigenvalue $a$ and $\tr A=0$;
$3L+I$ is invertible because $Z+I$ is.
By alternation on $K$, $P(Lp,q)+P(p,Lq)=(\tr L)P(p,q)=-P(p,q)$.
Thus \eqref{ah:symA} and \eqref{ah:barred-zero} give
\[
 (Lp)q-p(Lq)=P(p,q)=(Lp)(N_0q)-(N_0p)(Lq).
\]
Subtract and put $p=(3L+I)x$, $q=(3L+I)y$.
Since $N_0-I=-2(3L+I)^{-1}$, this gives $(Rx)y=x(Ry)$,
$R=L(3L+I)$. Reducing modulo $\C x$ gives $Rx\in\C x$, so $R$ is scalar.
By Cayley--Hamilton $R=-2L-3(\det L)I$, hence $L=-I/2$ and $N_0=5I$.
Together with the nonzero quotient eigenvalue $a$, these prove
\eqref{ah:common-conclusion} and invertibility of $A$.
\end{proof}
\subsubsection{A common hyperplane obstruction for two torsion ranks}
\label{sec:ranktwo-abelian-hyperplane}
\label{sec:rankthree-abelian-hyperplane}
\begin{lemma}[Symmetric horizontal normal motion]\label{fifth:symmetric-b}
Let $H$ have dimension three, and let $E$ have its canonical symmetric
part and full ambient symmetric output. Suppose $b$ is symmetric,
$b(p,Ah)=b(h,Ap)$, and the complete repeated row
\eqref{aa2:brow} holds. If $\operatorname{rank}A\ge2$, then $b=0$.
\end{lemma}
\begin{proof}
Suppose $b\ne0$. Its radical $R$ is $A$-invariant. Choose $u$ with
$b(u,u)\ne0$. The repeated row on $(u,r)$ gives $E(Ar,r)=0$ for
$r\in R$. Polarization and \eqref{eq:anti} give $A|_R=0$.
Thus $\rank b=1$ would
force $\rank A\le1$.
For $\rank b\ge2$, choose $b$-orthonormal $u,v$ and $r\in R$.
The coefficient of $s$ in the row on $(u+sr,v)$ gives $E(Au,r)=0$.
Density of nonisotropic $u$ yields $E(\im A,R)=0$, hence
$\im A\cap R=0$, since $E(w,w)=w^2$.
Choose a complement $W\supset\im A$ of $R$ and a $b$-orthonormal
basis $u_1,\ldots,u_m$ of $W$, $m\ge2$.
The repeated rows give $E(Au_i,u_i)+E(Au_j,u_j)=0$ for $i\ne j$,
hence $\sum_iE(Au_i,u_i)=0$. Since $bA$ is symmetric, all alternating
$P$ terms cancel in this sum, leaving the symmetric matrix of $A|_W$.
Thus $A|_W=0$, contradicting $\rank A\ge2$.
\end{proof}

\begin{lemma}[Common balanced normal-motion closure]\label{fifth:balanced-normal}
Use the full hyperplane identities of Lemma~\ref{lem:abelian-hyperplane-source-rows}.
Suppose $\beta_H=0$, $\gamma\ne0$, $\operatorname{rank}A\ge2$,
$\operatorname{tr}A=0$ and $\operatorname{tr}C_h=0$.
Then $\rho=\gamma A\ne0$ and $f=0$.
\end{lemma}
\begin{proof}
First suppose $\rho=a\gamma$ and $\gamma C=t\gamma\otimes\gamma$ at a point.
Put $K=\ker\gamma$ and choose $\gamma(u)=1$.
Every $C_h$ preserves $K$; for $h\in K$ its restriction has trace zero.
Equation~\eqref{aa2:Crow} on $K$ and
Lemma~\ref{lem:lowrank-symmetrizer} give $C(K,K)=0$.
For $L=C_u|_K$, the same whole row is
\[
                 E(Lk,v)=E(Av,k)+E(Ak,v)\qquad(k,v\in K).
\]
The diagonal gives $E((L-2A|_K)k,k)=0$; polarization and
\eqref{eq:anti} give $L=2A|_K$. The off-diagonal part gives
$E(Ak,v)=E(Av,k)$. Traces give $-t=\tr L=2\tr A|_K=-2a$.
In particular, if $a=0$, Lemma~\ref{lem:lowrank-symmetrizer}
gives $A|_K=0$, incompatible with $\rank A\ge2$.

If $\rho=0$, rank at least two gives $\im A=\ker\gamma$.
Extend $\gamma(e)=0$. The curvature reconstruction \eqref{aa2:barred} gives
$\gamma^2E(p,q)=\gamma(p)\gamma(q)$ on $H$.
Contracting \eqref{aa2:Crow} and using symmetry therefore gives
$\gamma C=t\gamma\otimes\gamma$.
The preceding calculation with $a=0$ excludes this. Hence
\begin{equation}\label{aa2:rhononzero}
                              \rho=\gamma A\ne0.
\end{equation}

Put $d(p,h)=f(C(p,h))$ and $g_h(p)=\gamma(C(p,h))$.
If $f=s\gamma\ne0$, the complete wedge row \eqref{aa2:wedge}
gives $\rho=a\gamma$, $a\ne0$, and $\alpha=sa\gamma$.
Then \eqref{aa2:ftrans} gives $\gamma C=0$.
The same calculation would give $0=t=2a$, a contradiction.

It remains to exclude independent $f,\gamma$. Choose $u,v,z$ with
$\gamma(u)=f(v)=1$, $\gamma(v)=f(u)=\gamma(z)=f(z)=0$, $z\ne0$.
The wedge row and the full trace give
\[
 \rho=a\gamma+bf,\qquad\alpha=c\gamma+af,\qquad Az=-2az.
\]
Here $Au\ne0$: otherwise $a=c=0$ and $\rank A\le1$.
Write $g(p,h)=g_h(p)$, $d_h(p)=d(p,h)$, and let $t_h$ be the
$z$ coefficient of $C(z,h)$; thus $g_h(u)+d_h(v)+t_h=0$.
The complete \eqref{aa2:frow} on $(z,u;h)$ and $(z,v;h)$ gives
\[
 d_h(z)Au=2a\{\alpha(h)+2af(h)-d_h(u)\}z,\qquad
 g_h(z)Au=-2ag_h(u)z.
\]
If $a\ne0$, the $u$ component of $Au$ is nonzero. Taking $h=v$
therefore gives $g(u,v)=0$, $d(u,v)=3a$.
But \eqref{aa2:ftrans}, namely $cg+ad=ac\gamma^2-abf^2$, gives
$3a^2=0$ on $(u,v)$. Hence $a=0$.
Now $Au\ne0$ forces $C(z,h)=t(h)z$.
Using $A^2h=c\gamma(h)Av+bf(h)Au$, the $(u,v;h)$ row gives
\[
                   \{g_h(u)+d_h(v)\}Au=2bf(h)Au.
\]
Thus $t=-2bf$. Since $t(z)=0$ and $Az=0$, the whole
\eqref{aa2:Crow} on $(h,z;z)$ gives $t(h)z^2=0$.
Hence $b=0$, contrary to $\rho\ne0$.
Thus
\begin{equation}\label{aa3new:d}
                               f=\alpha=d=0 .
\end{equation}
\end{proof}

\begin{theorem}[An actual abelian torsion hyperplane]
\label{thm:fourfold-abelian-hyperplane}
\label{thm:ranktwo-abelian-hyperplane}\label{thm:rankthree-abelian-hyperplane}
Work on a metric fourfold system. If an actual smooth hyperplane $H\subset V$ satisfies
$T(H,H)=0$ and $T(V,H)\subset H$, then $T=P=0$.
In particular, neither a rank-two torsion region with an abelian ideal
hyperplane nor a rank-three torsion region with abelian full image exists;
in the latter case take $H=\im T$. 
\end{theorem}
\begin{proof}
Every torsion output lies in $H$. In an actual adapted unitary splitting,
write $T(e,h)=Ah$ and use the complete tensor identities, with $\operatorname{tr}A=0$.
Work first on the open locus $\operatorname{rank}A\ge2$.
If $A$ is invertible, horizontal \ChernCyclic{} makes $b$ symmetric, so
Lemma~\ref{fifth:symmetric-b} gives $b=0$.
For rank two put $\ker A=\C z$. Equations \eqref{aa2:omega} and
\eqref{aa2:bA} give $b(z,\cdot)=b(\cdot,z)$ and $b(z,\im A)=0$.
If $A|_{\im A}$ is invertible, choose eigenvectors $x,y$ of
eigenvalues $\lambda,-\lambda$, respectively. The $(x,z)$ row of
\eqref{aa2:brow} gives $b(z,z)=0$, so $b$ kills $z$ in both slots.
Since $b(y,x)=-b(x,y)$ by \eqref{aa2:bA}, the $(x,y)$ row divided by
$\lambda$ is $b(x,x)y^2+2b(x,y)xy-b(y,y)x^2=0$, giving $b=0$.

Otherwise $A$ is nilpotent. Choose a chain $Ax=y$, $Ay=z$, $Az=0$.
The rows of \eqref{aa2:brow} on $(x,z)$ and $(x+y,z)$ differ by
$b(x,z)z^2$, so again $b$ kills $z$ in both slots.
Equation~\eqref{aa2:bA} gives $b(y,y)=0$. Put $r=b(x,x)$, $s=b(x,y)$.
Comparing the $(x,y)$ row with its shift $x\mapsto x+z$ gives
$b(y,x)z^2=0$; the row and its shift $y\mapsto y+z$ then give
$s y^2=rE(z,y)$ and $sE(y,z)=r z^2$. By canonical symmetry,
$(s y-r z)^2=0$, hence $b=0$.
Thus, on the rank-at-least-two locus,
\begin{equation}\label{aa2:horizontal}\equationalias{aa3new:horizontal}
 \beta_p(h)=0\qquad(p,h\in H).
\end{equation}

Where $\gamma\ne0$, Lemma~\ref{fifth:balanced-normal} gives
$f=0$, $\rho\ne0$; the common flag
(Lemma~\ref{lem:abelian-hyperplane-common-flag}) gives
$A|_K=-aI/2$, $M|_K=5I$, $a\ne0$, $K=\ker\gamma$.
Equation \eqref{ah:symA} gives $P(K,K)=0$.
Choose $\gamma(s)=1$. For $F_{\rm bar}s=\delta s\pmod K$,
\eqref{ah:finalfirst} gives
$\delta A|_K-aF_{\rm bar}|_K=M|_K=5I$; hence \eqref{ah:barred-zero} gives
\[
P(s,k)=5sk\pmod{\Sym^2K},\qquad
\ker(P|_{\Lambda^2H})=\Lambda^2K.
\]
Since $\beta_H=0$ and $\gamma|_K=0$, $\nabla'_VK\subset H$.
Write $T(q,k)=\xi(q)k$, $\xi(e)=-a/2$, $\xi(H)=0$.
For $k,l\in K$, \ChernFirst{} gives $2(\nabla'_pP)(k,l)
=\xi(p)(E(l,k)-E(k,l))=0$. Differentiating $P(k\wedge l)=0$ therefore
gives $\nabla'_p(k\wedge l)\in\ker(P|_{\Lambda^2H})=\Lambda^2K$.
For a local basis $k,l$ of $K$, the two terms in $(H/K)\otimes K$
have independent $K$ factors. Thus both normal components vanish
and $K$ is prime preserved. With $\Xi=\nabla'\xi$, ideal differentiation gives
$\mathcal X_p(q,k)=\Xi(p,q)k$.
\ChernCyclic{} $(k,l,q)$ gives $\Xi(k,q)l=\Xi(l,q)k$, hence
$\Xi(K,V)=0$; preservation gives $\Xi(V,K)=0$.
Complete repeated \ChernPure{} is the symmetric row
$(\Xi(q,q)+2\xi(q)^2)k^2=0$, so $\Sym\Xi=-2\xi^2$.
The torsion modulo $K$ is $-2B_\xi$; Lemma~\ref{lem:matching-binary-covector}
with $(s,t)=(-2,-2)$ gives $\nabla'\xi=-2\xi^2$.
Thus $H=\ker\xi$ is prime-preserved, contrary to $\gamma\ne0$.
Therefore $\gamma=0$ on the rank-at-least-two locus, where $H$ is a
prime-preserved abelian subbundle containing $\im T$. The trace-free
output theorem (Theorem~\ref{unif:thm:balanced-abelian-source}) gives $T=0$ there.
The remaining field has $\rank T\le1$; Theorem~\ref{newunb:r1v3:rankone}
gives $T=0$, and original \ChernBarred{} gives $P=0$.
\end{proof}

\subsubsection{Full torsion rank at most two in dimension four}
\begin{theorem}[Low-rank torsion in the prescribed fourfold metric]
\label{thm:fourfold-lowrank-torsion}
\label{cor:ranktwo-original-metric}\label{cor:fourfold-rankone-original}
Work on a metric fourfold system. On any open set where $\rank T\le2$, one has $T=P=0$.
In particular the prescribed metric is K\"ahler there. 
\end{theorem}
\begin{proof}
On a rank-two region the smooth image $N=\im T$ is abelian by
balance, as in Theorem~\ref{thm:ranktwo-injective-action}.
The induced traceless action $A:V/N\to\mathfrak{sl}(N)$ is nonzero
at every point: otherwise $T$ factors through the one-dimensional
$\Lambda^2(V/N)$. The locus where this action has rank two is excluded by
Theorem~\ref{thm:ranktwo-injective-action}, whose derivative traces
and normal factor are supplied by the Chern trace rows and
\eqref{r51:eq:barred-image-factor}. Hence $A$ has constant rank one.
The inverse image $H$ of $\ker A$ is a smooth ideal hyperplane because
all torsion values lie in $N\subset H$. Since $T(H,N)=0$ and $H/N$
is a line, alternation gives $T(H,H)=0$.
Theorem~\ref{thm:ranktwo-abelian-hyperplane} excludes it.
Thus no rank-two region exists, leaving $\rank T\le1$. Theorem~\ref{newunb:r1v3:rankone}(3) gives $T=P=0$.
\end{proof}

\section{Generic relations and spectral structure}\label{stage:fourfold-generic}
We establish the generic properties needed for the polynomial reduction
in Section~\ref{stage:fourfold-factors}.
Write $X_T=\{T\ne0\}$ and $V=T^{1,0}_xX$ in fiberwise arguments.
Base-generic refers to an open dense subset of the stated manifold
locus. Auxiliary-generic then refers to a nonempty Zariski open set
in each specified fiber: $\exists U\subset X_T$ open dense, $\forall x\in U$,
$\exists O_x$ nonempty Zariski open. No smooth auxiliary eigenframe is required.

For $E(u,v)=uv+P(u,v)$, set $\mathcal W=\ker E$.
Since $E$ is the identity on $\Sym^2V$, its kernel is the graph
of the correction to the alternating part. The map
$J_P(\sigma)=\operatorname{skw}\sigma-2P(\sigma)$ identifies
$\Lambda^2V$ with that kernel. Hence $\dim\mathcal W=6$. Right evaluation
$\operatorname{ev}_\xi=(\operatorname{id}\otimes\xi)|_{\mathcal W}$
has kernel $\mathcal W\cap(V\otimes\ker\xi)$.
We need two properties:
\begin{align}
 \dim(\mathcal W\cap(V\otimes\ker\xi))&=2
 &&\text{for generic }[\xi]\in\PP(V^*),\label{sep:eval}\\
 \mathcal W\cap(V\otimes\C a)&=0
 &&\text{for generic }[a]\in\PP(V).\label{sep:right-line}
\end{align}
The first is \HscGenericEval{} (Proposition~\ref{prop:fourfold-G1-short}); the second
follows fiberwise from Lemma~\ref{lem:right-injection-diagonal} and
$E(v,v)=v^2$.

Put $A_a=T(a,\cdot)$ and $D_a=\mathcal X_a(a,\cdot)$.
Since $A_aa=0$, let $\bar A_a$ denote its induced endomorphism on $V/\mathbb Ca$.
The complete repeated \ChernPure{} equation gives
\begin{equation}\label{sep:full-source}
 \mathscr R_a(b)={}(2D_a+A_a^2)b\otimes b +3A_ab\otimes A_ab -\{2\mathcal X_b(a,b)+T(b,A_ab)\}\otimes a\ \in\mathcal W.
\end{equation}
This is \eqref{newunb:r1v3:graph} with $(u,v)=(a,b)$ and
$2s_a=(2D_a+A_a^2)b$.
Even if the first two terms cancel, the term with right factor $a$
must still vanish for $\mathscr R_a(b)=0$.
For $\Phi_\sigma(v)=2\mathcal X_v(\sigma)+T(v,T(\sigma))$, the quadratic map
$\mathscr S_T(\sigma)=(\Phi_\sigma\otimes\operatorname{id})
\operatorname{skw}\sigma+3T(\sigma)^{\otimes2}$ satisfies
$\mathscr S_T(a\wedge b)=\mathscr R_a(b)$.

Set $\mathfrak Z=\mathcal X-\mathfrak H$,
$\mathfrak F_a=D_a+2A_a^2$. \ChernCyclic{} and
\eqref{unif:reference:elementary} give cyclic $\mathfrak Z$ with
$\mathfrak Z_a(a,b)=\mathfrak F_ab$.
On an invariant plane $L$, write $T|_L=B_\beta$ for its unique covector
$\beta$. Then $\mathfrak H_p(u,v)=-2\beta(p)T(u,v)$ on $L$, whose
contribution to \eqref{sep:full-source} vanishes. The two independent
right factors in
$\mathscr R_a(b)=2\mathfrak Z_a(a,b)\otimes b-2\mathfrak Z_b(a,b)\otimes a$
therefore give
\begin{equation}\label{sep:plane-zero}
 \mathscr R_a(b)=0\ \Longleftrightarrow\
 \mathfrak Z|_{L\otimes\Lambda^2L}=0\ \Longleftrightarrow\
 \mathfrak F_a|_L=0\quad(a\in L).
\end{equation}
On an invariant plane, this vanishing is basis-independent since
$a\wedge b$ spans $\Lambda^2L$.

For an ideal line $C=\C c$ of character $\lambda$, its \emph{star}
consists of the planes containing $C$. For $a\notin C$ on a simple
quotient-spectrum chart, $c\bmod\C a$ is the ideal-root eigenline of root $\lambda(a)$;
the other two are \emph{nonstar}. Their specified full lifts span $E_a$,
including an allowed zero root. For two nonzero roots this is the canonical
full eigenspace sum. The full weak-pair condition (WP) is
\phantomsection\label{wp:definition}%
$D_aE_a\subset E_a$ for auxiliary-generic $a$; \HscGenericEval{} alone does not imply it.

\Needspace{7\baselineskip}
\subsection{A torsion normal form and its exclusion}\label{r25:source-entrances}
\phantomsection\label{mid:heisenberg-normalization}
Locally, let a smooth trace-free four-dimensional product have a full ideal
line $C$ of character $\lambda$. If $S=T-B_\lambda$ has rank three,
radical $C$, and image $H\supset C$ with $\lambda|_H=0$, a smooth
local frame gives the table in Lemma~\ref{r25:weak-center-terminal}.
Indeed $H$ is an ideal and $C$ is central in $H$. For a direction in
$H$, the actions on $C$ and $V/H$ vanish, so trace freeness of $T$ gives
trace freeness on $H/C$. An alternating product on a two-dimensional
space is determined by its trace covector, so $H/C$ is abelian.
The isomorphism $S:\Lambda^2(V/C)\to H$ makes the resulting area
form $\omega$ nonzero. For any smooth $H=U\oplus C$ and transverse
$e$, write $T(e,c)=rc$, $T(e,u)=Mu+f(u)c$. The trace in direction
$e$ is zero, giving $\tr M=-r$. Replacing $e$ by $e+v_0$, with
$\omega(v_0,\cdot)=-f$, removes $f$. If $\lambda\ne0$, rescale $e$
first to arrange $r=1$. 

\subsubsection{An obstruction to preserving the center}
\begin{lemma}\label{r25:weak-center-terminal}
On a rank-four bundle with a partial connection, assume $X=\nabla'T$,
equations \eqref{eq:cyclic} and \eqref{eq:pure}, and the two balanced
derivative traces. Suppose a smooth splitting
$V=\C e\oplus U\oplus C$, $C=\C c$, has the full table
\[
 T(e,c)=rc,\quad T(e,u)=Mu+f(u)c,\quad
 T(u,v)=\omega(u,v)c,\quad T(U,c)=0,\qquad \tr M=-r,
\]
where $\dim U=2$ and $\omega$ is nondegenerate.
Put $H=U\oplus C$. On an actual region $r\ne0$, the two conditions
\[
 \nabla'_H C\subset C,\qquad \nabla'_V C\subset H
\]
are impossible. 
\end{lemma}
\begin{proof}
Rescale $e$ so $r=1$, put $\zeta(e)=1$, $\ker\zeta=H$, and write
$\nabla'_pc=\sigma(p)c+\zeta(p)w$, $w\in U$.
For $\beta_p(h)=\zeta(\nabla'_ph)$, the full derivative of $T(u,c)=0$ is
\[
 X_p(u,c)=-\{\beta_p(u)+\zeta(p)\omega(u,w)\}c\qquad(u\in H).
\]
For this input pair all outputs lie in $C$, so the direction-output
trace keeps only $X_c(u,c)$. It follows that $\beta_c=0$.
Repeated \ChernPure{} on $(u,c)$ then gives $\beta_u(u)=0$.
Thus $\beta$ is alternating on $H$ and vanishes on any input in $C$.
On the two-dimensional quotient it is therefore $b\omega$.
As $\im T\subset H$, normal \ChernCyclic{} on $(e,u,v)$ gives
$\beta_u(Mv)=\beta_v(Mu)$, or $b\tr M\,\omega=-b\omega=0$.
Hence $\nabla'_HH\subset H$; relative central \ChernPure{}
(Lemma~\ref{cor:relative-central-pure}, $D=H,Z=C$)
contradicts $T(U,U)=\omega c\ne0$.
\end{proof}

\phantomsection\label{unif:central-derivative-comparison}
\emph{Common central-derivative calculation.} Use \eqref{factorH:basic}--\eqref{factorH:center} on $r=1,f=0$,
with $(B,z)=(M,c)$, $\zeta(e)=1$, $\ker\zeta=H$, $q=\zeta(\nabla'_ec)$,
$\nabla'_c\zeta=\gamma\zeta$ and $A=h\omega$.
Suppose $\det M\det(I-M)\ne0$ and $L=\gamma M$. Then
\[
 W=(q+\gamma)(I-M)^{-1}M,\quad q=-\tr((I+M)W)-h,\quad
 \Lambda\mathcal C=\gamma J,\quad
 \mathcal C=(\gamma+2q)M+\gamma I.
\]
For $\omega(x,y)=1$, \ChernPure{} $(c,u;x,y)$ has zero quadratic blocks (the only nonzero
intermediate bivectors lie in $C\wedge U$, killed by $F_c(T,P)$),
and linear row $E(X_u(x,y),c)=3hE(c,u)$.
Since $\pi_UX_u(x,y)=((I+M)W+hM)u$, this gives
\[
 \Lambda\mathcal B=3hJ,\qquad\mathcal B=(I+M)W+hM+3hI.
\]
If $\gamma\ne0$, the equation $\Lambda\mathcal C=\gamma J$
makes $\mathcal C$ injective because $J$ is injective. If $\gamma=0$
but $q\ne0$, then $\mathcal C=2qM$ is invertible. In both cases
the image of $\Lambda$ lies in the image of $J$, giving
$E(H,c)\subset C\odot V$.
For an $M$-eigenvector $u$, $Wu\in\C u$ makes $\alpha(u,u)=0$,
so $X_u(e,u)\in H$. \ChernPure{} $(c,u;e,u)$ has zero quadratic blocks;
modulo $C\odot V$ it is $(\gamma-q)(Mu)u/2=0$. As $M$ is
invertible, $q=\gamma\ne0$.
Now $\Lambda=J(3M+I)^{-1}$, so $(I+M)W=8hM$.
Its trace and the shape row give $q=7h$.
Substituting $W$ and tracing $2\gamma(I+M)=8h(I-M)$ gives
$\gamma=12h$, a contradiction. Therefore
\[
                         q=\gamma=0,\qquad W=0.
\]

\subsubsection{Exclusion under the weak-pair condition}
\begin{proposition}\label{r25:heisenberg-WP}
For a metric fourfold system, the torsion table of Lemma~\ref{r25:weak-center-terminal} is impossible
under generic quotient separability and full \HscWeakPair{} for the specified
two nonstar full eigenlifts.
The independent low-rank metric theorem is retained for an allowed zero root.
\end{proposition}
\begin{proof}
On $\det M\ne0$, normalize $r\ne0$ to $r=1$, leaving $r=0$
as the remainder field. Remove $f$ and choose smooth eigenvectors
$x,y$ with $T(x,y)=c$:
\[
 T(e,c)=rc,\quad T(e,x)=mx,\quad T(e,y)=ny,\quad T(x,y)=c,\quad m+n=-r.
\]
Put $a=r-m$, $b=r-n$. Simplicity makes $m,n,a,b,m-n$ nonzero,
and $a+b=3r$. For $A=te+ux+vy+sc$ the full nonstar plane is
$E_A=\langle x+vc/(ta),y-uc/(tb)\rangle$.
Put $\zeta(e)=1$, $\ker\zeta=H$, $\beta_p(h)=\zeta(\nabla'_ph)$.
Actual image differentiation gives $\zeta X_p=\beta_pT$;
normal \HscWeakPair{} gives $\beta_A(E_A)=0$ because both roots are nonzero.
Its coefficients are
$\beta_e(c)=K$, $\beta_y(x)=-K/a$, $\beta_x(y)=K/b$, all others zero.
Normal \ChernCyclic{} on $(e,x,y)$ is
\begin{equation}\label{r38:eq:heisenberg-normal}
 K(ab-r^2)=0.
\end{equation}
On $K\ne0$, $r^2=ab\ne0$; normalization then makes $m,n$ locally constant.
In the shared center identities, $q=K$ and alternation of $A$ makes
$W$ diagonal. The moving $U$-connection contributes
$[\Gamma_c,M]$ to $L-\gamma M$; its diagonal is zero, while the
basic row makes $L$ diagonal. Thus $L=\gamma M$ and the
\hyperref[unif:central-derivative-comparison]{common central-derivative calculation}
gives $q=K=0$, a contradiction.
With $K=0$ all normal coefficients vanish, so $\nabla'_VH\subset H$.
Corollary~\ref{ab:cor:lowrank} excludes this rank-three preserved output.
On the remainder $\det M=0$, simplicity excludes $r=0$;
then $\im T=\im M\oplus C$ has rank two, excluded by
Theorem~\ref{thm:fourfold-lowrank-torsion}.
\end{proof}

\subsubsection{Two complete plane-vanishing hypotheses}\label{r25:zero-plane-entrances}
Assume $\mathcal X=\nabla'T$, the complete \ChernCyclic{} and \ChernPure{} identities,
and the balanced derivative traces. In a smooth frame $(e,x,y,c)$ write
\[
 T(e,c)=c,\quad T(e,x)=\mu x+\delta y,\quad T(e,y)=\lambda y,
 \quad T(x,y)=c,\qquad \mu+\lambda=-1.
\]
All other independent products vanish.

\phantomsection\label{r25:resonant-zero-plane}
In the fixed-resonance case take $(\mu,\lambda,\delta)=(1,-2,0)$ and assume $\mathscr S_T=0$
on $\langle te+ux+wc,\,3ty-uc\rangle$ for general $t\ne0$.

\phantomsection\label{r25:Jordan-zero-plane}
For an incident repeated Jordan root take $(\mu,\lambda,\delta)=(-1/2,-1/2,1)$ and assume $\mathscr S_T=0$
on $\langle a,3a_0y-2a_1c\rangle$ for general
$a=a_0e+a_1x+a_2y+a_3c$ with $a_0\ne0$.

Neither torsion table implies its plane-vanishing hypothesis or \HscGenericEval{}.
Put $U=\langle x,y\rangle$, $C=\C c$, $H=U\oplus C$ and $d=1-\lambda$.
For $a=te+ux+wc$ and $b=dty-uc$, the table gives $T(a,b)=\lambda tb$.
Polynomial continuation of the plane-vanishing identities gives
\[
 Z=X_a(a,b)+2\lambda^2t^2b=0,\qquad Y=X_b(a,b)=0.
\]
Use Lemma~\ref{lem:heisenberg-shared-rows} with $B=M$, $z=c$,
$Mx=\mu x+\delta y$, $My=\lambda y$ and $\omega(x,y)=1$.
For $v\in U$ the full product rule gives
\[
 \begin{gathered}
 \pi_U X_v(e,c)=(I-M)Wv,\qquad \pi_U X_e(v,c)=rMv,\\
 \pi_U X_v(x,y)=Wv-\alpha(v,x)My+\alpha(v,y)Mx.
 \end{gathered}
\]
Both moving-input torsion terms lie in $H$, so
$\chi X_p(a,b)=\chi(\nabla'_p(T(a,b)))$.
The $t^2u$ coefficient of $\chi Z$ and the $t^3$ coefficient of
$\chi Y$ give $\alpha(x,y)=r/d$ and $\alpha(y,y)=0$.
Then $A(y,y)=0$ gives $Wy\in\C y$, and the $t^2w$ coefficient of
$\pi_UZ$ gives $Ly=r\lambda y$. By \eqref{factorH:basic},
$(I-M)Wy=2r\lambda y$, hence $Wy=2r\lambda y/d$.
The $t^2u$ coefficient of $\pi_U Y$ gives
\[
 d\{Wy-\lambda\alpha(y,x)y+\alpha(y,y)Mx\}=3r\lambda y,
\]
so $\alpha(y,x)=-r/d$. By \eqref{factorH:shapes},
$r=\lambda\alpha(x,y)-\mu\alpha(y,x)-\delta\alpha(y,y)=-r/d$.
Since $d+1\ne0$, $r=Wy=0$. Alternation of $A$ now gives
$Wx=hy$ and $\alpha(x,x)=-h$. The $tu^2$ coefficient of $\pi_UZ$ is
\[
 d\{Wx-\alpha(x,x)My\}-(I-M)Wx=d\lambda hy=0.
\]
Thus $W=0$ as $d\lambda\ne0$. The common rows give
$\nabla'_HC\subset C$ and $\nabla'_VC\subset H$.
The center character is $1$, so
Lemma~\ref{r25:weak-center-terminal} excludes both fields.

\subsubsection{The scalar repeated-root field}
\begin{lemma}\label{r25:scalar-Heisenberg}
The torsion table with scalar transverse block $T(e,c)=c$,
$T(e,u)=-u/2$, $T(u,v)=\omega(u,v)c$, $T(U,c)=0$,
cannot occur in a complete \ChernFirst{}/\ChernCyclic{}/\ChernPure{} source with balanced traces
and an actual pure-flat connection having torsion $T$.
\end{lemma}
\begin{proof}
The moving connection terms of the scalar $M=-I/2$ cancel, so
$L=\gamma M$ in the shared center identities. The
\hyperref[unif:central-derivative-comparison]{common central-derivative calculation}
gives $q=W=0$. Proposition~\ref{prop:heisenberg-first} then gives
$\nabla'_HC\subset C$ and $\nabla'_VC\subset H$, contrary to
Lemma~\ref{r25:weak-center-terminal}.
\end{proof}

\subsection{Full generic relation evaluation without a common-kernel premise}
\label{sec:full-evaluation-short}

The right evaluation and the pencil satisfy
\[
 \nu(\operatorname{ev}_z J_P(\sigma))
 =(\nu\wedge z)(\sigma)-2P^\vee(\nu,z)(\sigma)
 =-(\mathcal B_z\nu)(\sigma),
 \qquad \sigma\in\Lambda^2V,\quad \nu,z\in V^*.
\]
Thus $\operatorname{ev}_z$ and $\mathcal B_z$ have the same rank.

At a fixed point, choose a nonzero volume form $\Omega\in\Lambda^4V^*$.
Define $q_T:V^*\to V$ by
$\iota_{q_T(\xi)}\Omega=\xi\wedge T^\vee\xi$, using the algebraic
dual from Section~\ref{sec:conventions}. The rational map
$r_T:\mathbb P(V^*)\dashrightarrow\mathbb P(V)$ sends $[\xi]$ to
$[q_T(\xi)]$ where $q_T(\xi)\ne0$ and is independent of the scale of
$\Omega$.

\begin{proposition}[Generic rank, separability, and dominance]\label{prop:common-generic-source-locus}
For a metric fourfold system, there is a dense open subset $U\subset X_T$
in the manifold topology. At each $x\in U$:

\begin{enumerate}[label=\textup{(\roman*)},ref=\thetheorem\textup{(\roman*)}]
\item\label{prop:fourfold-G1-short}
$\operatorname{rank}\operatorname{ev}_\xi=4$ on a nonempty Zariski-open
subset of $\mathbb P(V_x^*)$ \textup{(\HscGenericEval{})}.
\item\label{prop:generic-separability-short}\label{prop:generic-cyclicity-short}
$\bar A_a$ has three distinct quotient roots on a nonempty Zariski-open
subset of $\mathbb P(V_x)$, so its minimal polynomial has degree three there.
\item\label{prop:generic-root-dominance-uniform}
The fiberwise map $r_T$ is dominant.
\end{enumerate}

We prove \HscGenericEval{}, then dominance, then separability.
Separability then gives cyclicity.
\end{proposition}
\begin{proof}[Proof of \HscGenericEval{} in Proposition~\ref{prop:fourfold-G1-short}]
On an open nonzero-torsion region where the generic rank of
$\mathcal B_z$ is at most three, balance makes $T$ nonvectorial.
Lemma~\ref{unif:lem:nonvectorial-rank-loss-cut} gives a smooth full cut
$N=\Ann R_P$ of rank $1$--$3$; Proposition~\ref{prop:conceptual-fullcut}
makes it a prime-preserved output bundle. Corollary~\ref{ab:cor:lowrank}
then contradicts $T\ne0$.
The coefficient minors test the generic rank of the pencil. If the
deficient locus had interior, Lemma~\ref{unif:lem:regular-loci}
would give the constant-rank chart needed for this contradiction.
Thus \HscGenericEval{} holds on an open dense set.
\end{proof}

\subsection{Quadratic Euler geometry}
\label{r11:sec:quadratic-euler}
The following fiberwise algebra uses no cyclicity or separability.
By $\iota_{q_T(\xi)}\Omega=\xi\wedge T^\vee\xi$,
\begin{equation}\label{sep:admissibility}
 [q_T(\xi)]=[a]\ \Longleftrightarrow\
 \xi(a)=0,\quad \xi A_a=\mu\xi\text{ for some }\mu,
 \quad q_T(\xi)\ne0.
\end{equation}
Indeed wedging with $\xi$ gives $\xi(q_T(\xi))=0$, and contraction
by $q_T(\xi)$ gives the left-eigenvector condition. Conversely these
conditions make $a$ annihilate the nonzero three-form
$\xi\wedge T^\vee\xi$, whose contraction kernel is its volume-dual line.

Normalize $\beta^2=2\operatorname{pf}(\beta)\Omega$.
If $\iota_a\Omega=\xi\wedge\beta$, contraction gives
$(\iota_a\beta-\operatorname{pf}(\beta)\xi)\wedge\beta=0$.
For nondegenerate $\beta$, wedge multiplication is invertible; polynomial
continuation gives $\iota_a\beta=\operatorname{pf}(\beta)\xi$ for every
$\beta$. Specializing to $\beta=T^\vee\xi$ gives
\begin{equation}\label{r11:eq:euler-pfaffian}
 \xi A_{q_T(\xi)}=\operatorname{pf}(T^\vee\xi)\xi.
\end{equation}

\begin{lemma}[A linear family of decomposable two-forms]
\label{lem:entrance-two-form-ruling}
Let $L\subset\Lambda^2V^*$, $\dim V=4$, satisfy $L\wedge L=0$.
Then either $L\subset\upsilon\wedge V^*$ for some $\upsilon\ne0$, or
$L=\Lambda^2H$ for a three-dimensional $H\subset V^*$.
In the latter alternative $\dim L=3$.
\end{lemma}
\begin{proof}
Every nonzero member is decomposable by skew normal form, and their
supporting planes pairwise intersect. Unless all share a line, choose
$x\wedge y,x\wedge z$ and a member whose plane $C$ omits $x$.
Its intersections with $\langle x,y\rangle,\langle x,z\rangle$
are distinct, so $C\subset H=\langle x,y,z\rangle$.
A supporting plane outside $H$ would meet both initial planes along
$\C x$ and miss $C$, impossible. Thus $L\subset\Lambda^2H$;
the three chosen members are independent, giving equality.
The cases $\dim L\le1$ already have a common factor.
\end{proof}

\paragraph{Pfaffian shift extraction}\phantomsection\label{pfaffian:shift-extraction}
For $S=T-B_\rho$, $q_S=q_T$ because $B_\rho^\vee\xi=\rho\wedge\xi$.
The identity $\rho\wedge\xi\wedge T^\vee\xi=\rho(q_T(\xi))\Omega$
gives $\operatorname{pf}(S^\vee\xi)=\operatorname{pf}(T^\vee\xi)-\rho(q_T(\xi))$.
If $T$ is trace free, $r_T$ dominant and this Pfaffian zero for all $\xi$,
polarization gives $(\im S^\vee)\wedge(\im S^\vee)=0$.
A common factor $\upsilon$ would give
$\upsilon\wedge\xi\wedge T^\vee\xi=0$, hence $\upsilon(q_T)=0$,
contrary to dominance. Lemma~\ref{lem:entrance-two-form-ruling}
gives $\im S^\vee=\Lambda^2\Ann C$, $\rank S=3$, $\rad S=C$.
Here $C$ is a unique line; for $c\in C$,
$0=\tr T(c,\cdot)=3\rho(c)$, so $T(v,c)=\rho(v)c$.
On smooth charts an invertible minor of $\rho\mapsto\rho(q_T)$ recovers
$\rho$ smoothly, and the constant-rank radical gives smooth $C$.

\begin{lemma}[Dominance of a nondegenerate quadratic Euler system]
\label{r9:lem:quadratic-euler-dominance}
Let $q=(q_0,q_1,q_2,q_3)$ be homogeneous quadratic polynomials over
$\mathbb C$, satisfying $\sum_i x_iq_i=0$. If the $q_i$ are linearly
independent, then $[q]:\mathbb P^3\dashrightarrow(\mathbb P^3)^*$
is dominant.
\end{lemma}
\begin{proof}
Let $Y$ be the projective image closure and $Z$ the base locus.
If the quadratics have a common factor, it must be linear by their independence. Write $q(x)=\ell(x)Ax$. The Euler identity makes $A$ skew-symmetric. Independence of the
quadratics makes $A$ invertible, so the map is dominant.
Otherwise $\dim Z\le1$, while $Y$ lies in no plane.

The fibre over a general $y\in Y$ lies in $H_y=\PP(\ker y)$.
If $\dim Y=1$, its closure is $H_y$. Fixing $H_{y_0}$, every
$H_y\cap H_{y_0}$ is a line in $Z$; there are finitely many such lines.
Irreducibility therefore puts the general $H_y$ in one pencil,
so $Y$ is a line, a contradiction. Dimension zero is also excluded.

Suppose $\dim Y=2$. Choose dual coordinates with
$y=[0:0:0:1]$, $H_y=\{x_3=0\}$. An irreducible equation of a fibre-curve component not contained in $Z$
divides the restrictions of $q_0,q_1,q_2$, which are not all zero
(otherwise $H_y$ is a fibre). A quadratic factor would leave constant quotients. After
canceling that factor, the Euler identity on $H_y$ would be a
linear relation among the three independent coordinate functions,
so all those constants would be zero. Hence this component is a line
$m=0$, and $(q_0,q_1,q_2)|_{H_y}=mBx$ with $B$ nonzero skew of rank two.
A change of basis in the domain and the dual change in the codomain put the restriction in the
form $[m u:m v:0:r]$, where $r=q_3|_{H_y}$. Extend the
independent linear forms $u,v$ to linear coordinates $(u,v,w)$
on the vector space underlying $H_y$.

The restricted image lies in the intersection of a plane with $Y$.
Because that plane does not contain $Y$, this intersection has dimension
at most one.
On $u=1$ its affine coordinates are $(v,r/m)$. The first coordinate
is independent of $w$, so the Jacobian has rank at most one only if
$\partial_w(r/m)=0$. Clearing denominators gives the homogeneous
identity $m\partial_w r=r\partial_w m$.
The fibre line $m=0$ is not in $Z$, so $m\nmid r$ and
$m\mid\partial_w m$. Since $m$ is linear, $\partial_wm=\partial_wr=0$.
For a nonzero lift $a$ of $u=v=0$ in $H_y$,
we have $q(a)=0$ and
$dq(a)|_{\ker y}=0$.
If $dq(a)=0$, quadratic polarization gives $q(x+ta)=q(x)$;
Euler then gives $a\cdot q=0$, contradicting independence.
Hence $dq(a)$ has rank one and kernel $\ker y$.
The map $[a]\mapsto[\ker dq(a)]$ from the rank-one locus of $Z$
therefore covers general $Y$, impossible since $\dim Z\le1$.
\end{proof}

\subsection{Full tensor identities and invariant planes}
\label{sec:separability-relations}
Work on a metric fourfold system.

\subsubsection{Generic injectivity from a nonvanishing diagonal}
\begin{lemma}\label{lem:right-injection-diagonal}
Let $V$ be a complex vector space of dimension $n\ge2$, let $Q$ be
finite dimensional, and let $E:V\times V\to Q$ be bilinear with
$E(v,v)\ne0$ whenever $v\ne0$.
Then the projective incidence
\[
 Z=\{([v],[a])\in\PP(V)\times\PP(V):E(v,a)=0\}
\]
has dimension at most $n-2$. Consequently $v\mapsto E(v,a)$ is
injective for every $[a]$ outside a proper closed algebraic subset
of $\PP(V)$.
\end{lemma}
\begin{proof}
The projective zero locus $\widehat Z=\{[v:a]:E(v,a)=0\}$
in $\PP(V\oplus V)$ misses the linear space $\{[v:v]\}$.
The diagonal has dimension $n-1$, so the projective dimension theorem
gives $\dim\widehat Z\le n-1$. Away from the two coordinate spaces,
the map $[v:a]\mapsto([v],[a])$ onto $Z$ has $\C^*$ fibers, given by
the nonzero relative scale. Thus $\dim Z\le n-2$.
The second projection is closed by projectivity and is therefore proper.
\end{proof}

\subsubsection{Ordinary invariant planes and their full components}
\begin{lemma}\label{lem:sep-plane-components}
Let $V$ be a complex four-dimensional vector space and
$T:\Lambda^2V\to V$ an alternating product. For $a\ne0$, let
$\bar A_a$ be the endomorphism of $V/\mathbb Ca$ induced by
$v\mapsto T(a,v)$. Suppose $\bar A_a$ is cyclic for general
$[a]\in\mathbb P(V)$.
Every irreducible component of the variety of invariant two-planes
has dimension at most two.
If a linear covector $\rho$ is a quotient root with a unique ordinary
quotient eigenline for general $a$, those eigenplanes fill an
irreducible surface $\Sigma$ whose entire pointed incidence dominates
$\PP(V)$. With $S=T-B_\rho$ and
$W_\Sigma=\operatorname{span}\{\Lambda^2L:L\in\Sigma\}$, one has
\[
 W_\Sigma\subset\ker S,\qquad 2\le\rank S\le3,
 \qquad 3\le\dim W_\Sigma\le4.
\]
If $\dim W_\Sigma=4$, then $\rank S=2$ and $W_\Sigma=\ker S$.
If $\dim W_\Sigma=3$, then $W_\Sigma=C\wedge V$ for a line $C$
and $S(C,V)=0$. If, in addition, $\operatorname{tr}T(v,\cdot)=0$
for all $v\in V$, then $\rho(C)=0$ and
\[
 T(e_C,v)=-\rho(v)e_C\qquad(e_C\in C,\ v\in V).
\]
\end{lemma}
\begin{proof}
Over a $d$-dimensional component of invariant planes, the pointed
incidence has dimension $d+1$: each plane contributes its projective line. If its image closure $Y$ has dimension $r\le2$, directions
through a general smooth point lie in $\PP(TY)$, so
$d+1\le2r-1\le3$. If it dominates $\PP(V)$, each fiber over a general $[a]$ is
finite: a cyclic endomorphism has only one ordinary eigenline for
each distinct root, even if the root is repeated. Thus $d+1=3$.

For the specified root, minors recover its unique ordinary eigenline
rationally. Its graph closure has dimension three; the preceding bound
makes its plane image an irreducible surface $\Sigma$ with entire
pointed component. Writing $T|_L=B_{\beta_L}$, continuation of
$\beta_L(a)=\rho(a)$ on that component gives $\beta_L=\rho|_L$.
Hence $W_\Sigma\subset\ker S$ and $\dim W_\Sigma\ge3$.
Rank at most one for $S$ would give at least a two-dimensional kernel
of $\bar A_a-\rho(a)I$, contradicting cyclicity. Thus
$2\le\rank S\le3$, $\dim W_\Sigma\le4$.
Dimension four gives $W_\Sigma=\ker S$, $\rank S=2$.
Dimension three makes $\Sigma=\PP(W_\Sigma)$ decomposable.
Lemma~\ref{lem:entrance-two-form-ruling}, applied to $V^*$, gives
$C\wedge V$ or $\Lambda^2H$; the latter contradicts pointed dominance.
Finally $S(C,V)=0$, so for a first input in $C$ the torsion is the
vectorial part. Its zero trace gives $3\rho(C)=0$, proving the ideal formula.
\end{proof}

Applications use the rational minor charts and constant-rank bundles
of $S$, its radical and image.

\subsection{The common mixed-plane obstruction}\label{unif:sec:mixed-zero}
Separability and the split-quadric argument in
Section~\ref{stage:fourfold-quadric} use the same mixed-plane calculation.
It assumes neither generic separability nor right evaluation.

\begin{lemma}[The complete mixed-block calculation]
\label{r20:lem:split-prime-blocks}
Let $V=U\oplus W$ be an actual smooth splitting into two-planes of
an involutive rank-four complex tangent distribution. Its actual
partial connection $\nabla'$ has zero pure curvature and trace-free
torsion $T$, with $\mathcal X=\nabla'T$ satisfying complete \ChernCyclic{}.
Suppose smooth area forms $\omega,\nu$, covectors $\beta,\alpha$,
and extension fields $z\in W$, $y\in U$ give the whole product
\begin{equation}\label{twosplit:B1}
 \begin{gathered}
 T(u,v)=-2\{\beta(u)v-\beta(v)u\}+\omega(u,v)z,\quad z\in W,\\
 T(w,w')=2\{\alpha(w)w'-\alpha(w')w\}+\nu(w,w')y,\quad y\in U,\qquad T(u,w)=\alpha(w)u+\beta(u)w.
\end{gathered}
\end{equation}
and every mixed plane satisfies $\mathscr S_T=0$.
Assume the auxiliary polynomial
$\Delta(u,w)=9\beta(u)\alpha(w)+\omega(u,y)\nu(w,z)$
is not identically zero in any coefficient fiber.
Then the actual moving blocks satisfy
\eqref{twosplit:B4}, \eqref{twosplit:B5}, \eqref{twosplit:F1},
and \eqref{r10:eq:split-leaf-curvature} below.
\end{lemma}
\begin{proof}
The mixed-plane vanishing alone gives the following row.
For $t=T(u,w)=\alpha(w)u+\beta(u)w$, repeated zero reads
\[
 (2\mathcal X_u(u,w)+4\beta(u)t)\otimes w
 -(2\mathcal X_w(u,w)-4\alpha(w)t)\otimes u=0.
\]
The independent right factors give
\begin{equation}\label{twosplit:B2}\equationalias{sq:row-2}
 \mathcal X_u(u,w)=-2\beta(u)t,\qquad
 \mathcal X_w(u,w)=2\alpha(w)t.
\end{equation}
Write $\nabla'_pu=D^U_pu+B_pu$, $\nabla'_pw=C_pw+D^W_pw$,
$g=D^U\beta$ and $h=D^W\alpha$. The full product derivative is
\begin{equation}\label{twosplit:B3}
\begin{split}
 \pi_U\mathcal X_p(u,w)&=\{h_p(w)-2\beta(C_pw)\}u
                   +3\beta(u)C_pw-\nu(B_pu,w)y,\\
 \pi_W\mathcal X_p(u,w)&=\{g_p(u)-2\alpha(B_pu)\}w
                   +3\alpha(w)B_pu-\omega(u,C_pw)z.
\end{split}
\end{equation}
For $p=u,w$, wedge the $U$ component with $u$ and the $W$ component
with $w$. The preceding repeated identity makes both wedges zero. The two coefficient systems
on $(\omega(u,C_pw),\nu(B_pu,w))^{\mathsf T}$ have matrix
\[
 \begin{pmatrix}3\beta(u)&-\omega(u,y)\\
                 \nu(w,z)&3\alpha(w)\end{pmatrix}.
\]
For each $p=u,w$, the determinant $\Delta$ forces both unknowns to
vanish on $\Delta\ne0$, hence everywhere by polynomial continuation.
Radiality and the remaining coefficients give smooth $\chi,\upsilon$ with
\begin{equation}\label{twosplit:B4}
\begin{gathered}
 B_uu=C_ww=0,\quad C_uw=\chi(w)u,\quad B_wu=\upsilon(u)w,\\
 h_u(w)=-\beta(u)\{\chi(w)+2\alpha(w)\},\quad
 g_w(u)=\alpha(w)\{2\beta(u)-\upsilon(u)\},\\
 g_u(u)=-2\beta(u)^2,\qquad h_w(w)=2\alpha(w)^2.
\end{gathered}
\end{equation}
Set $B_uv=\omega(u,v)b$ and $C_ww'=\nu(w,w')c$. For
$\xi\in U^*$ and $\zeta\in W^*$, define $b_\xi\in U$ and
$a_\zeta\in W$ by $\omega(u,v)b_\xi=\xi(u)v-\xi(v)u$ and
$\nu(w,w')a_\zeta=\zeta(w)w'-\zeta(w')w$, respectively.
The $W$ component of \ChernCyclic{} $(u,w,w')$ and its exchanged $U$ component are
\begin{equation}\label{twosplit:B5}
\begin{split}
0={}&-3\beta(u)a_\chi+(6\upsilon(u)-3\beta(u))a_\alpha
       +\omega(u,y)(b-z)+2\omega(u,c)z,\\
0={}&3\alpha(w)b_\upsilon-(3\alpha(w)+6\chi(w))b_\beta
       +\nu(w,z)(c-y)+2\nu(w,b)y.
\end{split}
\end{equation}
In the first row the internal, cross and Jacobi terms contribute
$(-4\beta,6\upsilon-2\beta,3\beta)$ to $B_\alpha$ and
$\nu(w,w')(\omega(u,y)b,2\omega(u,c)z,-\omega(u,y)z)$ to extension
outputs. Contracting both rows by $\nu(\cdot,w),\omega(u,\cdot)$ and adding gives
\begin{equation}\label{twosplit:F1}
3\beta(u)\chi(w)+3\upsilon(u)\alpha(w)
 +\omega(u,y)\nu(b,w)+\omega(u,c)\nu(w,z)=0.
\end{equation}
Thus the direction-output trace follows from \ChernCyclic{}.
For $D=D^U\oplus D^W$, $K=\nabla'-D$, pure flatness gives
\begin{equation}\label{twosplit:S3}
(D_pB)_q-(D_qB)_p+B_{T^D(p,q)}=0,
\qquad T^D(p,q)=T(p,q)-K_pq+K_qp,
\end{equation}
since the block curvature and $[K_p,K_q]$ are diagonal.
On $(u,w;u)$ its derivative terms are $(D^U_u\upsilon)(u)w$ and zero;
$T^D(u,w)=(\alpha(w)-\chi(w))u+(\beta(u)+\upsilon(u))w$ gives
\begin{equation}\label{r10:eq:split-leaf-curvature}
(D^U_u\upsilon)(u)+\upsilon(u)^2+\beta(u)\upsilon(u)=0.
\end{equation}
\end{proof}

\begin{lemma}[Vanishing of the two extension characters]
\label{unif:lem:mixed-zero-character}
Retain the actual pure-flat torsion connection, complete \ChernCyclic{},
trace-free table \eqref{twosplit:B1}, and vanishing of the complete tensor on mixed planes as in Lemma~\ref{r20:lem:split-prime-blocks}, but impose no
condition on $\Delta$. Suppose complete \ChernPure{} holds for the canonical
map $\Theta(u,v)=uv+P(u,v)$, where $P$ is alternating.
Then
\[
                         \alpha(z)=\beta(y)=0.
\]
\end{lemma}
\begin{proof}
By the exchange $(U,W,\beta,\alpha)\leftrightarrow(W,U,-\alpha,-\beta)$,
it suffices to exclude an actual $\alpha(z)\ne0$ region.
Where $\Delta\not\equiv0$, use Lemma~\ref{r20:lem:split-prime-blocks}
and set $\eta=\omega(\cdot,y)$, $\xi=\omega(\cdot,c)$,
$\zeta=\nu(\cdot,z)$, $\psi=\nu(\cdot,b)$.
Contract \eqref{twosplit:B5} by $\nu(\cdot,w)$ and
$\omega(u,\cdot)$ and subtract \eqref{twosplit:F1}:
\begin{equation}\label{r42:eq:split-factored-rows}
\begin{gathered}
3(3\upsilon-\beta)\otimes\alpha=(\eta-3\xi)\otimes\zeta,\\
3\beta\otimes(\alpha+3\chi)=\eta\otimes(3\psi-\zeta).
\end{gathered}
\end{equation}
Evaluate at $z,y$ and substitute back:
\begin{equation}\label{twosplit:F2}
\alpha(z)(3c-y)=0,\qquad\beta(y)(3b-z)=0,
\end{equation}
and, with $\delta=\psi-\zeta/3$,
\begin{equation}\label{twosplit:S2}
\upsilon=\beta/3,\quad
\beta\otimes(3\chi+\alpha)=\eta\otimes\delta,\quad c=y/3.
\end{equation}
Equations \eqref{twosplit:B4}, \eqref{r10:eq:split-leaf-curvature}
give $0=(-2/3+1/9+1/3)\beta(u)^2$, hence $\beta=\upsilon=0$.
Now $\Delta=\eta\otimes\zeta\ne0$, so $\delta=0$ and $b=z/3$.
Since $B|_W=0$ as a field, \eqref{twosplit:S3} on $w,w'\in W$
has zero derivative terms and
$T^D(w,w')=2B_\alpha(w,w')+\nu(w,w')y/3$. Thus
\begin{equation}\label{r10:eq:split-final-normal}
0=B_{T^D(w,w')}u=-\tfrac13\nu(w,w')\eta(u)b,
\end{equation}
contradicting $b=z/3\ne0$ and $\eta\ne0$.

On $\Delta=0$, evaluation at $(u,z)$ gives $\beta=0$, hence $y=0$.
Put $r=-\alpha$, $d=z/r(z)$, rescale $\omega$, and choose
$0\ne c\in\ker r$. Thus \eqref{twosplit:B1} becomes
\[
T(n,u)=r(n)u,\quad T(u,v)=\omega(u,v)d,\quad
T(n,n')=-2B_r(n,n'),\quad r(d)=1.
\]
Put $h_p=D^W_pr$, $a_p=rB_p$. The $\Delta$-independent rows
\eqref{twosplit:B2}--\eqref{twosplit:B3} give
$h_u=0$, $h_n(n)=-2r(n)^2$. The remaining full input derivative is
\[
\pi_U\mathcal X_p(u,v)=\omega(u,v)C_pd-a_p(u)v+a_p(v)u.
\]
The mixed $W$ rows at $n=d,c$, with polarization, give
\[
a_d(u)=-\omega(u,C_dd),\quad a_c=0,\quad C_cc=0,
\quad C_dc+C_cd=0.
\]
The $U$ \ChernCyclic{} rows on $(d,u,v),(c,u,v)$, whose Jacobiators are
$W$-valued, give $2\omega(u,v)C_dd=\omega(u,v)C_cd=0$.
Hence $C|_W=0$: $W$ is internally preserved and involutive, with
pure-flat induced connection, torsion $B_\lambda$, $\lambda=-2r$, and
$\Sym\nabla'_W\lambda=\lambda^2$.
Lemma~\ref{lem:matching-binary-covector} ($K=0$, $s=t=1$) gives
$\nabla'_W\lambda=\lambda^2$.
For $u,v,p\in W$, actual differentiation now gives
$\mathcal X_p(u,v)=T(p,T(u,v))=\lambda(p)T(u,v)$.
The complete repeated preimage is $6T(u,v)^2\in\ker\Theta$.
Since $\ker\Theta\cap\Sym^2V=0$, this gives $T(W,W)=0$,
contrary to $\lambda(d)=-2$.
\end{proof}

\subsection{Generic auxiliary separability from the Chern system}
\label{sec:generic-separability-short}
Section~\ref{sec:generic-root-dominance-uniform} proves the Euler
dominance used here from \HscGenericEval{}, independently of separability.

\begin{lemma}[Ideal-line quotient identities]\label{lem:ideal-line-quotient-algebra}
Let $T:\Lambda^2V\to V$ be trace free on a complex four-space and
$C=\C c$ an ideal line, with $T(v,c)=\lambda(v)c$. In bilinear volume
coordinates on $V/C$, write the quotient products of $T$ and
$T-B_\lambda$ as $N(u\times v)$ and $D(u\times v)$.
Write $[x]v=x\times v$ and use the dot product to identify covectors
with coordinate vectors. With $A=\Sym N$,
\begin{equation}\label{mid:ideal-volume}
 N-N^{\mathsf T}=[\lambda],\qquad
 D=N+[\lambda]=2N-N^{\mathsf T},\qquad D-D^{\mathsf T}=3[\lambda].
\end{equation}
For a symmetric three-matrix $A$,
\begin{equation}\label{mid:quotient-cofactors}
 \det(A+s[x])=\det A+s^2x^{\mathsf T}Ax,\qquad
 \Sym\adj(A+s[x])=\adj A+s^2xx^{\mathsf T}.
\end{equation}
The quotient products satisfy
\begin{equation}\label{mid:quotient-polynomials}
 \begin{aligned}
 \det(tI-N[u])&=t\{t^2+\lambda(u)t+u^{\mathsf T}(\adj N)u\},\\
 \det(N[u]-\lambda(u)I)&=-\lambda(u)u^{\mathsf T}(\adj D)u.
 \end{aligned}
\end{equation}
If a three-matrix $M$ has rank at least two, $u^{\mathsf T}(\adj M)u$
is not identically zero.
\end{lemma}
\begin{proof}
The $C$ trace gives $\tr(N[u])=-\lambda(u)$, hence the skew part
of $N$; $[\lambda](u\times v)=-B_\lambda(u,v)$ gives $D$.
Expanding cofactors gives
$\adj(A+s[x])=\adj A-s[Ax]+s^2xx^{\mathsf T}$.
Symmetrizing and tracing the product with $A+s[x]$ give
\eqref{mid:quotient-cofactors}, since $\tr([x][y])=-2x^{\mathsf T}y$.
$N[u]u=0$ and $\adj(N[u])=uu^{\mathsf T}\adj N$ give the first
polynomial. Since $u^{\mathsf T}(\adj D-\adj N)u=2\lambda(u)^2$,
evaluation at $t=\lambda(u)$ with the odd-dimensional sign gives the second.
Finally $\adj M$ has rank one or three, so it cannot be alternating.
\end{proof}

\begin{proof}[Proof of separability in Proposition~\ref{prop:generic-separability-short}]
Suppose the coefficient-closed nonseparable set has interior; intersect
with the independent \HscGenericEval{}, dominance and generic right-line-injectivity loci.
Put $k(a)=\tr A_a^2$, $d(a)=\tr A_a^3/3$. The zero discriminant gives
$k^3=54d^2$, hence a linear covector $\rho$ with
\[
 \chi_a(t)=t^3-kt/2-d=(t-\rho(a))^2(t+2\rho(a)).
\]
Indeed $\rho=0$ for $k=0$; otherwise $\rho=-3d/k$ has constant reduced
denominator since $\rho^2=k/6$. The Euler--Pfaffian identity
\eqref{r11:eq:euler-pfaffian} gives
$(\operatorname{pf}-\rho q_T)^2(\operatorname{pf}+2\rho q_T)=0$.
Thus $\operatorname{pf}(T^\vee\xi)=\lambda(q_T(\xi))$ with
$\lambda=\rho$ or $-2\rho$.

The \hyperref[pfaffian:shift-extraction]{Pfaffian shift extraction}
gives the rank-three shift $S=T-B_\lambda$ and its smooth radical
ideal line $C$.
Put $H=\im S$. By dominance, for general $\xi$ we have
$a=q_T(\xi)\notin H\cup C$; we also take $\xi(c)\ne0$. The isomorphism $\Lambda^2(V/C)\xrightarrow{S}H$
and $a\notin H$ give $\ker(\bar A_a-\lambda(a)I)=\C[c]$ on $V/\C a$.
By \eqref{r11:eq:euler-pfaffian} its image is the kernel of the
induced covector $\bar\xi$, complementary to $\C[c]$.
By dominance $\lambda$ is generically simple, hence
$\lambda=-2\rho\ne0$ and $\rho$ is smooth.

Suppose the double root is cyclic for general $a=q_T(\xi)$, with $\xi$
general and $r=\rho(a)\ne0$. Choose a chain
$A_aa=0$, $A_ae=re$, $A_af=rf+e$, $A_ac=-2rc$.
The left eigenvector $\xi$ at $-2r$ annihilates $a,e,f$, so
$J=\langle a,e,f\rangle=\ker\xi$ and \HscGenericEval{} applies directly.
At $\nabla'_aa=0$, differentiation of $(A_a-rI)e=0$ gives
$D_ae\in\im(A_a-rI)=\langle a,e,c\rangle$, since $e=(A_a-rI)f$.
Thus $U=2\mathfrak F_ae$ has $[f]U=0$; put $V_0=2\mathfrak F_af$.
Write $R_0$ and $R_2$ for the complete repeated \ChernPure{} tensors
at $(a,e)$ and $(a,f)$, and $R_1$ for their polarization. All three
lie in the two-dimensional $\mathcal W\cap(V\otimes J)$.
Modulo $V\otimes\C a$, their rows in the order $R_0,R_1,R_2$ are
\[
 U\otimes e,\qquad U\otimes f+V_0\otimes e,\qquad
 (V_0-3re)\otimes f+3(rf+e)\otimes e .
\]
If $R_0\ne0$, right-line injectivity gives $U\ne0$, so the first two
rows are independent. Hence $R_2=bR_0+dR_1$ in the right-support subspace of $\mathcal W$.
The $f$ column gives $V_0-3re=dU$, hence $[f]V_0=0$;
the $e$ column then gives $3r=0$, a contradiction.
Hence $R_0=0$: the repeated tensor $\mathscr S_T$ vanishes on
every such eigenplane, without assuming weak pairs or separability.

For the quotient matrices $N,D$ of $T,S$, put $G=\Sym N$.
The double root and \eqref{mid:quotient-polynomials} give
$u^{\mathsf T}(\adj N)u=\rho(u)^2$; \eqref{mid:quotient-cofactors}
then gives $\adj G=0$ and $\rank G\le1$.
On an actual incident field $C\subset H$, $\det D=0$ gives
$\lambda^{\mathsf T}G\lambda=0$, hence $G\lambda=0$ and $\lambda|_H=0$.
The \hyperref[mid:heisenberg-normalization]{incident normalization}
gives the table of Lemma~\ref{r25:weak-center-terminal} with $r=1,f=0$ and
$N_0=M+I/2$, $N_0^2=0$. Where $N_0\ne0$, choose
$y=N_0x$ and normalize $\omega(x,y)=1$ by a smooth local square root.
For $a_0\ne0$, the ordinary repeated full eigenline is
$\C(3a_0y-2a_1c)$. The preceding argument gives vanishing in \eqref{sep:plane-zero};
the coefficient test for the Jordan block in \S\ref{r25:zero-plane-entrances}
excludes the field. The remainder $N_0=0$ is excluded by
Lemma~\ref{r25:scalar-Heisenberg}.

On the transverse open locus, $D$ is invertible, $\rank G=1$ and
$\lambda^{\mathsf T}G\lambda\ne0$.
Choose $h\in\im G$ with $\lambda(h)=1$, put $E=H\cap\ker\lambda$,
and normalize $T(u,v)=h$ for a smooth basis of $E$. The quotient table is
\[
 T(h,c)=c,\quad T(h,u)=-u/2,\quad T(h,v)=-v/2,\quad
 T(u,v)=h,\quad T(c,E)=0.
\]
For $a=xh+tc+e$ with $x\ne0$, $0\ne e\in E$, this gives
$\ker(A_a+xI/2)=\C e$. Thus the double root is cyclic, and
$\langle xh+tc,e\rangle$ satisfies $\mathscr S_T=0$.
By \eqref{sep:plane-zero} and polynomial continuation, every plane spanned by a vector of $E$ and a vector of $\langle h,c\rangle$
satisfies $\mathscr S_T=0$.
This is \eqref{twosplit:B1} with $\beta=0$, $\alpha(h)=1/2$,
$\alpha(c)=0$, $y=0$, $z=h$, contradicting
Lemma~\ref{unif:lem:mixed-zero-character} because $\alpha(z)\ne0$.
The transverse locus and the interior of the incident locus are both
excluded. Thus the nonseparable set has empty interior.
\end{proof}

\subsection{Tensor identities for a vectorial hyperplane}\label{sec:vectorial-hyperplane-short}
A smooth rank-four tensor system with a connection is understood in the sense
of Section~\ref{sec:foundations}, with trace-free torsion $T$,
$\mathcal X=\nabla'T$, and $P:\Lambda^2V\to\Sym^2V$.
Use $E(u,v)=uv+P(u,v)$ and $\mathcal W=\ker E$ from
\eqref{newunb:r1v3:graph}.
Use the same tensor preimage from \eqref{eq:a}--\eqref{eq:F} to write
\[
 \mathcal K_{pq}=\mathfrak b_q(\mathcal X_p)-\mathfrak b_p(\mathcal X_q)
 +\mathfrak b_p(T)b_q(T)-\mathfrak b_q(T)b_p(T)+\mathfrak b_{T(p,q)}(T).
\]
The tensors satisfy \ChernCyclic{} and complete \ChernPure{}, the latter expressed
as $E\mathcal K_{pq}=0$; we do not assume $\mathcal K_{pq}=0$.
When imposed, \ChernFirst{} means $\nabla'_pE=F_p(T,P)$ with
the full operator \eqref{newff:eq:E-form-F}.
Unless stated otherwise, the alternating tensor and connection are
independent. A Chern connection means that of the prescribed
metric fourfold system in Section~\ref{stage:fourfold-terminals}.
Condition \textup{(\HscGenericEval{})} is the two-dimensional kernel condition
\eqref{sep:eval} for general $[\xi]$, not for a prescribed hyperplane.
Let $H$ be a smooth hyperplane on which $T|_{\Lambda^2H}=B_\gamma$.
Where $\gamma\ne0$, the following normalization uses only smoothness
and trace freeness.

\begin{lemma}[A smooth splitting and its complete moving inputs]
\label{lem:vectorial-normalization}
Suppose $H$ is an actual smooth hyperplane and
$T|_{\Lambda^2H}=B_\gamma$, with $\gamma\ne0$.
Locally choose $s\in H$ with $\gamma(s)=1$.
There is a smooth complement $e$ for which, with $I=\ker\gamma$,
\begin{equation}\label{vh:normalization}
 T(h,e)=Mh-2\gamma(h)e,
 \qquad \tr M=0,\qquad Ms=0.
\end{equation}
Put $n(e)=1$, $n|_H=0$, and let $\nabla^H$ be the projected
connection of this splitting, in all derivative directions.
For $b_p(h)=n(\nabla'_p h)$, $g_p=\nabla^H_p\gamma$ and
$\nabla'_p e=v_p+a_pe$, the full derivatives are
\begin{align}
 &\mathcal X_p(h,k)=g_p(h)k-g_p(k)h
   +b_p(h)(Mk-3\gamma(k)e)-b_p(k)(Mh-3\gamma(h)e),
                                                        \label{vh:motion}\\
 &\mathcal X_p(h,e)=C_ph+\{b_p(Mh)-2g_p(h)\}e,\quad\tr C_p=0,\label{vh:cross}\\
 &C_p=\nabla^H_pM-a_pM-3v_p\otimes\gamma+\gamma(v_p)I_H.
                                                        \label{vh:C}
\end{align}
\end{lemma}
\begin{proof}
The traces in directions $h$ and $e$ give the coefficient $-2\gamma(h)$
and $\tr M=0$. Replacing $e$ by $e+w$ changes
$M$ to $M+3w\otimes\gamma-\gamma(w)I_H$.
If $Ms=v+as$, $v\in I$, choose $w=-v/3-as/2$; then $M's=0$.
These smooth choices require no spectral or rank division.
The full covariant product rule gives \eqref{vh:motion}--\eqref{vh:C};
the complementary scalar terms cancel in the normal part of
\eqref{vh:cross}. Taking the trace in \eqref{vh:C} gives
$\tr C_p=-3\gamma(v_p)+3\gamma(v_p)=0$.
\end{proof}

\begin{lemma}[Full comparison of two nonzero endpoints]
\label{lem:full-pair-with-dependent-endpoints}
Let $a,e,f$ be independent, $A_ae=\lambda e$, $A_af=\mu f$,
$\lambda\ne\mu$. Suppose
$\mathscr R_a(e)\ne0$ and $\mathscr R_a(f)\ne0$, with
$\mathscr R$ as in \eqref{sep:full-source}, and
\[
 \dim\bigl(\mathcal W\cap(V\otimes\langle a,e,f\rangle)\bigr)=2.
\]
Then $D_a\langle e,f\rangle\subset\langle e,f\rangle$.
If $(a,e_1,e_2,e_3)$ is a full eigenbasis, $\tr A_a=\tr D_a=0$,
and both pairs $(e_i,e_3)$, $i=1,2$, meet these hypotheses,
then $Ue_3=0$ for $U=2D_a+4A_a^2$, and $U=0$ if $\lambda_3\ne0$.
Here $A_ae_i=\lambda_i e_i$; $\lambda_1=\lambda_2$ is allowed.
\end{lemma}
\begin{proof}
Put $U_e=2D_ae+4\lambda^2e$, $U_f=2D_af+4\mu^2f$,
$\delta=\lambda-\mu$, $H_e=2\mathcal X_e(a,e)$,
$H_f=2\mathcal X_f(a,f)$ and
$H_{ef}=2\mathcal X_e(a,f)+2\mathcal X_f(a,e)-\delta T(e,f)$.
The endpoint tensors and their polarization are
\[
\begin{gathered}
 R_e=U_e\otimes e-H_e\otimes a,\qquad
 R_f=U_f\otimes f-H_f\otimes a,\\
 R_{ef}=(U_f+3\mu\delta f)\otimes e
             +(U_e-3\lambda\delta e)\otimes f-H_{ef}\otimes a.
\end{gathered}
\]
If the two nonzero endpoints are dependent, their distinct right
columns give $U_e=U_f=0$. Otherwise they form a basis of the
stipulated kernel, so $R_{ef}=cR_e+dR_f$ gives
\begin{equation}\label{entrance:full-pair-comparison}
 U_f+3\mu\delta f=cU_e,\qquad
 U_e-3\lambda\delta e=dU_f,\qquad H_{ef}=cH_e+dH_f.
\end{equation}
If $cd=1$, substitution between the first two equations gives
$\lambda e=d\mu f$. Independence of $e,f$ and $d\ne0$ then give
$\lambda=\mu=0$, contrary to $\lambda\ne\mu$. Thus the coefficient
matrix is invertible. Its right-hand sides lie in $\langle e,f\rangle$,
so $U_e,U_f$ do also, proving the inclusion.

For the trace assertion, apply the comparison to both pairs.
The two invariant planes give $Ue_3=ue_3$, $Ue_i=p_ie_i+q_ie_3$.
If $u\ne0$, both pairs are independent; the two columns give
$p_i=3\lambda_i\delta_i$, $c_ip_i=0$, $c_iq_i=u+3\lambda_3\delta_i$,
with $\delta_i=\lambda_i-\lambda_3$. Put $K=\sum\lambda_i^2$.
Since $Ua=0$ and $\sum\lambda_i=0$, the trace gives $u=4K-(p_1+p_2)=K$.
Since $p_i=3\lambda_i\delta_i$ with $\delta_i\ne0$,
$c_ip_i=0$ gives $\lambda_i c_i=0$. Multiplying the last column
by $\lambda_i$ and summing gives $2\lambda_3K=0$. Since $K=u\ne0$,
$\lambda_3=0$ and $\lambda_i\ne0$ for $i=1,2$. Then $c_i=0$, so the
last column gives $u=0$, a contradiction. Thus $u=0$. If $\lambda_3\ne0$,
an independent pair would give $q_i=0$ and $3\lambda_3\delta_i=0$
by the two columns. Hence both pairs are dependent and $U=0$.
\end{proof}

\begin{theorem}[Exclusion of a vectorial hyperplane]
\label{thm:vectorial-source-alternatives}
For a balanced Hermitian fourfold with $H_h\equiv c\ne0$, no open
set with nonzero torsion and \HscGenericEval{} admits a smooth vectorial hyperplane.
\end{theorem}
\begin{proof}
Theorem~\ref{thm:fourfold-abelian-hyperplane} excludes $\gamma=0$
and $M=0$: balance makes $H$ an abelian ideal in the first case,
and the normalized table makes $I+\C e$ one in the second.
Use \eqref{vh:normalization} on a chart with $M\ne0$,
and put $r=\gamma M$, $Fa=Ma-3\gamma(a)e$. Let $\Gamma$ be the
closure of the planes $\langle a,Fa\rangle$. For $\gamma(a)=1$ these
are $\langle a,f(a)\rangle$, where $f(a)=e-Ma/3-r(a)a/6$.

\emph{Pointwise alternative.} We first rule out nonzero complete
sources on both plane families at the same point. We then show that neither family can have identically zero plane
tensor on an open chart.
Suppose both $\operatorname{Gr}(2,H)$ and $\Gamma$ contain a plane
with $\mathscr S_T\ne0$.
Write $M|_I=A+s\otimes r$, $\tr A=0$. For
$\xi|_I=\ell$, $\xi(s)=z$, $\xi(e)=w$, the polynomial
$\ell\wedge(\ell A+zr)$ is nonzero: otherwise $r=0$ and $A$ is
scalar, hence $M=0$. For general $\xi$, solve
$\ell(v)=-z$, $(\ell A+zr)(v)=3w$. Then $a=s+v$ and
$[u]=\PP(\ker\ell)$ give $\ker\xi=\langle a,u,f(a)\rangle$.
The covector map and both plane projections are dominant, so \HscGenericEval{} and
both nonzero endpoints hold on a nonempty Zariski open set of pairs $(a,[u])$.
Fix general $a$ and choose independent $u_1,u_2$ in its open fiber; put $f=f(a)$.
The full eigenbasis $(a,u_1,u_2,f)$ has roots $(0,1,1,-2)$.
By balance $\tr A_a=\tr D_a=0$, so the trace clause of
Lemma~\ref{lem:full-pair-with-dependent-endpoints} gives
$2D_a+4A_a^2=0$, whereas its trace is $24$.

For the moving-input identities alone, assume just the polynomial
inclusion from the pair comparison:
\begin{equation}\label{vhtwo:weak}
 D_au\in\langle u,f(a)\rangle,\qquad D_a=\mathcal X_a(a,\cdot).
\end{equation}
On $\gamma(a)=1$, the $e$ coefficient in \eqref{vh:motion} is $3b_a(u)$.
Thus $u\mapsto D_au-3b_a(u)f(a)$ preserves every line of $I$,
so it equals $\lambda(a)\Id_I$. The $H$ part gives
\begin{equation}\label{vhtwo:motion}
 b_a(a)Mu+(g_a(a)-\lambda(a))u+
 \{-g_a(u)+b_a(u)r(a)/2\}a=0.
\end{equation}
If $b_a(a)\not\equiv0$, varying $a$ gives $Mu\in\C u$ for all $u\in I$,
hence $M=0$. Thus $b_H$ is alternating. The normal \ChernCyclic{} row is
$\gamma\wedge b_H=0$; \eqref{vhtwo:motion} then gives the second identity
below, by homogeneity. Tangential \ChernCyclic{} on $(s,u,v)$ gives the last two:
\begin{gather}
 b_p=\gamma(p)\kappa-\kappa(p)\gamma,\quad\kappa(s)=0,
 \qquad g_p(u)=r(p)\kappa(u)/2 \quad(p\in H,u\in I),\label{vhtwo:gI}\\
 \kappa\wedge r=0,\qquad g_u(s)=-2\kappa(Au),\label{vhtwo:cyclic}
\end{gather}
using $\kappa(u)Av-\kappa(v)Au=\kappa(Av)u-\kappa(Au)v$ for trace-free $A$ on $I$.

\emph{Vanishing of $\mathscr R$ on $\Gamma$.} On $r\ne0$, choose $r(x)=1$
with $x\in I$, put $q=-r(Mx)$ and replace $e$ by $e'=e+qs$.
Then $M'=M+3qs\otimes\gamma-qI$, $F'_a=F_a-qa$, and
$h=M'x$ satisfies $\gamma(h)=1$, $\gamma(M'h)=0$.
Since $\gamma(x)=0$, one has $F'_x=M'x=h$, so $\langle x,h\rangle$
is a member of $\Gamma$. Vanishing there gives
$\mathcal X_h(x,h)=2x$, hence $b_h(x)=0$, $g_h(h)=-2$.
For $f=e'-M'h/3$, \eqref{vh:cross} minus one third of
\eqref{vh:motion} gives $n\mathcal X_h(h,f)=-2g_h(h)=4$:
the two $b_h(M'h)$ terms cancel. Plane zero on $\langle h,f\rangle$
instead gives $\mathcal X_h(h,f)=-8f$. Thus $r=0$.

Set $A=M|_I$, $N=I+\C e=\im T$, extend $\gamma(e)=0$ and
put $\beta_pn=\gamma(\nabla'_pn)$. Actual image differentiation gives
$\gamma\mathcal X_p(u,v)=\beta_pT(u,v)$. For
$a=ts+v$, $b=3te-Av$, the full zero-plane identities give
$\beta_a(b)=\beta_b(b)=0$ polynomially. If $A$ is invertible, $b$
ranges over $N$, so $\beta|_{N^2}$ is alternating. Normal \ChernCyclic{}
on $(s,u,v)$ makes $\beta|_{I^2}$ symmetric, hence zero. The first
polynomial and normal \ChernCyclic{} on $(s,u,e)$ give
\[
 \beta_s(e)=0,\qquad \beta_s(Au)=3\beta_u(e),\qquad
 \beta_s(Au)+2\beta_u(e)+\beta_e(u)=0.
\]
Thus $4\beta_u(e)=0$ and $\beta=0$, contradicting
Corollary~\ref{ab:cor:lowrank} for the actual preserved image $N$.
On the remaining field $A^2=0$, choose smooth $z=Ax\ne0$.
The frame $(s,x,e,z)$ has
\[
 T(s,z)=z,\quad T(s,x)=x,\quad T(s,e)=-2e,\quad
 T(x,e)=z,\quad T(x,z)=T(e,z)=0.
\]
For $a=ts+ux+wz$, $b=3te-uz$, plane zero supplies the full
identities $\mathcal X_a(a,b)=-8t^2b$, $\mathcal X_b(a,b)=0$.
The fixed-resonance calculation in \S\ref{r25:zero-plane-entrances}
excludes this chart. Thus $\mathscr S_T$ cannot vanish on every plane of $\Gamma$ on an open set.

\emph{Vanishing of $\mathscr R$ on $\operatorname{Gr}(2,H)$.}
Plane zero and cyclic recovery, with $J_T|_{H^3}=0$, give
$\mathcal X_h(i,j)=-2\gamma(h)T(i,j)$ on $H$.
The normal, $I^2$ and $(s,u)$ rows of \eqref{vh:motion} give
$b_h=\kappa(h)\gamma$, $g_h+2\gamma(h)\gamma=t(h)\gamma$,
$(t(h)I+\kappa(h)M)|_I=0$. A nonzero $\kappa(h)$ would again
make $M|_I$ scalar. Consequently
\begin{align}
 \nabla'_HH&\subset H,\quad
 \mathcal X_h(i,j)=-2\gamma(h)T(i,j),\label{vhzero:earned}\\
 \mathcal X_h(k,e)&=C_hk+4\gamma(h)\gamma(k)e,\quad\tr C_h=0,
 \label{vhzero:crossderivative}\\
 b_p(h)&=b\,n(p)\gamma(h),\quad
 \mathcal X_p(H,H)\subset H,\label{vhzero:remainingnormal}
\end{align}
where the last row is normal \ChernCyclic{} on $(e,h,k)$.
On a constant-rank chart put $\mathcal L=E(H\otimes H)$,
$\mathcal Q=\Sym^2V/\mathcal L$. The full five-term \ChernPure{} preimage is
\begin{equation}\label{vhquot:strips}
 \mathcal K_{es}(s,u)\equiv-3e\otimes u/2+9u\otimes e/4
 \pmod{H\otimes H},\qquad u\in I.
\end{equation}
Indeed its right and left normal contractions are respectively the
sums of $(0,1,-1/2,3/4,1)u$ and $(0,-2,1,3/2,-2)u$.
Writing $J(u)=[u\odot e]$, $\mathcal A(u)=[E(u,e)]$ gives
\begin{equation}\label{vhquot:ratio}
 \mathcal A|_I=2J|_I/5.
\end{equation}
Since $H$ and $I$ are preserved in $H$ directions, \ChernFirst{}
gives preservation of $\mathcal L$ in those directions. For its quotient connection $D$ and
$\nabla'_he=v_h+a_he$, the full product rule gives
\begin{align}
 D_hJ&=a_hJ,\label{vhquot:DJ}\\
 D_h\mathcal A&=a_h\mathcal A+\gamma(h)(\mathcal A-2J)/2.
 \label{vhquot:DA}
\end{align}
Differentiating \eqref{vhquot:ratio} along $s$ cancels $a_s$ and
gives $J(I)=0$. Hence $\mathcal L\supset\Sym^2H+I\odot e$ and
$K=\mathcal W\cap(H\otimes H)$ has dimension at most one.

For $u,v\in I$, put $A_{uv}=\mathcal K_{ue}(u,v)$,
$B_u=\mathcal K_{us}(u,e)$ and $L_u=C_uu/2-r(u)u/4$.
The whole tensors satisfy $A_{uv}\in K\cap(H\otimes I)$, $B_u\in K$.
Modulo the right factor $\C u$ their complete rows are
\begin{align}
 A_{uv}&\longmapsto L_u\otimes[v],\label{vhqone:first}\\
 B_u&\longmapsto L_u\otimes[s]-3u\otimes[Mu]/4.
 \label{vhqone:second}
\end{align}
The derivative terms contribute $C_uu\otimes[v]/2$ and
$C_uu\otimes[s]/2$. The total of both compositions and the positive
torsion-direction term, for $A_{uv}$ and $B_u$ respectively, is
$-r(u)u\otimes[v]/4$ and
$-r(u)u\otimes[s]/4-3u\otimes[Mu]/4$; normal outputs cancel
before projection. In $(H/\C u)^{\otimes2}$ their images are
$[L_u]\otimes[v]$ and $[L_u]\otimes[s]$. For
$v\in I\setminus\C u$, independence of $[v],[s]$ and $\dim K\le1$
force $L_u\in\C u$. Thus both tensors lie in
$S_u=H\otimes\C u+\C u\otimes H$.
If both tensor families vanish identically, the right projections
give $L_u=0$ and $Mu\in\C u$ for every $u\in I$, hence $M=0$.
Otherwise a nonzero family spans $K$ for general $u$;
$S_x\cap S_y\cap S_{x+y}=\Lambda^2I$ and polynomial continuation
give the field identity $K=\Lambda^2I$.
Since $b_e|_I=0$, $\nabla'_eI\subset H$. Differentiate the actual
kernel field by \ChernFirst{}, retaining both inputs:
\[
 \nabla'_e(u\otimes v-v\otimes u)+Mu\otimes v-Mv\otimes u
 \in K=\Lambda^2I.
\]
The derivative of the skew tensor remains skew. Thus taking the
symmetric part gives $(Mu)v=(Mv)u$ on $I$, which makes $M|_I$
scalar and again contradicts $M\ne0$.
Lemma~\ref{unif:lem:regular-loci} supplies a constant-rank chart in
every proposed horizontal-zero region. Thus no open region has vanishing of $\mathscr S_T$ on every two-plane in $H$. The locus
where some such plane has nonzero $\mathscr S_T$ is open; the
pointwise alternative forces vanishing on $\Gamma$ there, also impossible.
\end{proof}

\begin{corollary}\label{cor:vectorial-hyperplane-no-extra-G1}
No actual nonempty nonzero-torsion region of a given balanced
Hermitian fourfold with $c\ne0$ has a smooth vectorial hyperplane.
\end{corollary}
\begin{proof}
By Proposition~\ref{prop:fourfold-G1-short}, every such proposed region
meets the open dense \HscGenericEval{} locus in a nonempty actual open subset.
Theorem~\ref{thm:vectorial-source-alternatives} excludes that subset.
\end{proof}
\subsection{Generic Euler root dominance from the Chern system}
\label{sec:generic-root-dominance-uniform}

Retain the Euler vector and its exact domain from
\eqref{sep:admissibility}, and put
\[
 \mathcal R_T=\{\alpha\in V^*: \alpha(q_T(\xi))=0
                                     \text{ for every auxiliary }\xi\}.
\]
A constant relation is independent of $\xi$. As a field on the base,
it need not be parallel.

\begin{lemma}[Euler admissibility and the relation dictionary]
\label{lem:root-uniform-dictionary}
For $0\ne\alpha\in V^*$ and $H=\ker\alpha$,
$\alpha\in\mathcal R_T$ is equivalent to
\[
             T(h,k)=\gamma(h)k-\gamma(k)h\qquad(h,k\in H)
\]
for a covector $\gamma\in H^*$.
\end{lemma}
\begin{proof}
We first turn the relation on $q_T$ into a condition on the planes
in $H$. By the volume identity, $\alpha\in\mathcal R_T$ is equivalent to
$\alpha\wedge\xi\wedge T^\vee\xi=0$ for all $\xi$.
For $\xi\notin\C\alpha$, this is equivalent to $\xi(T(h,k))=0$ on
$h,k\in H\cap\ker\xi$. For independent $h,k\in H$, such $\xi$ are
dense in $\Ann\langle h,k\rangle$, so $T(h,k)\in\langle h,k\rangle$.
Conversely, preservation of these planes gives the wedge identity.
The \hyperref[unif:plane-criterion]{plane criterion} gives the required
product, with $\gamma(h)=\tr_H T(h,\cdot)/2$ smooth for smooth $H$.
\end{proof}

\begin{proof}[Proof of root dominance in Proposition~\ref{prop:generic-root-dominance-uniform}]
Choose a smooth local volume and apply Lemma~\ref{unif:lem:regular-loci}
to the finite coefficient map
\[
 V^*\longrightarrow\operatorname{Sym}^2V,\qquad
 \alpha\longmapsto\bigl(\xi\longmapsto\alpha(q_T(\xi))\bigr).
\]
Any nonempty interior of its rank-deficient locus contains a chart
with a smooth nonzero kernel section $\alpha$. Lemma~\ref{lem:root-uniform-dictionary} makes
$\ker\alpha$ a vectorial hyperplane, contradicting
Corollary~\ref{cor:vectorial-hyperplane-no-extra-G1}.
Thus $\mathcal R_T=0$ on an open dense set. The Euler identity and
Lemma~\ref{r9:lem:quadratic-euler-dominance} give dominance there,
independently of cyclicity and separability. Intersect the three open dense sets to obtain $U$.
\end{proof}

The vectorial-hyperplane incidence does not require dominance, and
\HscGenericEval{} does not depend on the Euler map. Dominance asserts
only that the image closure is $\PP(V)$; the coordinate quadrics
of $q_T$ may still vanish simultaneously.

\section{Polynomial reduction and full zero factors}\label{stage:fourfold-factors}
On the generic simple-spectrum region, the degree of the full zero
factor counts quotient eigenlines on which the repeated defect
vanishes. Proposition~\ref{prop:polynomial-full-annihilator} rules out
degrees two and three. The remaining cases give an ideal line or a
rank-two shifted image.

\subsection{A transverse ideal line and the full nonstar pair}
\label{sec:transverse-weak-pair-uniform}

We work on a metric fourfold system with $\mathcal X=\nabla'T$ and assume
weak-pair inclusion. The first transport result controls the ideal
line and its transverse shifted image.
Write $\Theta=E$ and $W=\mathcal W=\ker\Theta$.

\subsubsection{The transverse coefficient identities}

\begin{lemma}\label{lem:transverse-ideal-coefficient-entrance}
Use the trace-free ideal-line notation $T,C,\lambda,N,D,A$ of
Lemma~\ref{lem:ideal-line-quotient-algebra}. Suppose $S=T-B_\lambda$
has radical $C$, rank three, and image $H$ transverse to $C$.
Then $S:\Lambda^2H\to H$ is an isomorphism and
$\nu=T|_{\Lambda^2H}$ takes values in $H$; its volume matrices are
$N=A+[\lambda]/2$, $D=A+3[\lambda]/2$, with $\det D\ne0$.
For $a=u+tc$ in $H\oplus C$,
\[
 \chi_{\bar A_a}(z)=(z-\lambda(u))
 \bigl\{(z+\lambda(u)/2)^2+u^{\mathsf T}(\operatorname{adj}A)u\bigr\}.
\]
Generic separability implies $\rank A\ge2$ and makes the two
nonideal roots nonzero, distinct, and different from $\lambda(a)$
for general $a$.
\end{lemma}
\begin{proof}
Ideality gives $\lambda(c)=0$ and $S(c,V)=0$, so
$\Lambda^2(V/C)\simeq\Lambda^2H\xrightarrow{S}H$.
Since $B_\lambda(H,H)\subset H$, also $\nu(\Lambda^2H)\subset H$.
Lemma~\ref{lem:ideal-line-quotient-algebra} gives $N,D$. The operator
on $H\oplus C$ has diagonal blocks $N[u],\lambda(u)$; removing the
null factor belonging to $a$ in \eqref{mid:quotient-polynomials}
and using \eqref{mid:quotient-cofactors} gives the quotient polynomial.
If $\rank A\le1$, the quadratic factor has a repeated root.
Since $D=2N-N^{\mathsf T}$ is invertible, $\rank N\ge2$, and the
same lemma makes the nonideal-root product nonzero as a polynomial.
Generic separability supplies the remaining root assertions.
\end{proof}

\begin{proposition}[Ideal-line transport from the nonstar eigenspace pair]
\label{r38:lem:ideal-translation}
\label{prop:ideal-transverse-weak-partial}
\label{prop:ideal-incident-self-transport}
\label{sec:ideal-incident-self-transport}
Let an actual smooth alternating tensor on a complex rank-four bundle
have zero trace and a smooth full ideal line $C=\mathbb Cc$,
$T(a,c)=\lambda(a)c$. Put $S=T-B_\lambda$ and suppose
$\operatorname{rad}S=C$ and $\operatorname{rank}S=3$.
Let $\mathcal X=\nabla'T$ for the actual connection. Suppose that,
for general auxiliary $a$, the two nonstar roots are distinct,
nonzero and different from $\lambda(a)$, and their full eigenspace
sum $E_a$ satisfies
\begin{equation}\label{wp:weak-pair}\equationalias{iwp:weak-pair}
 D_aE_a\subset E_a,\qquad D_a=\mathcal X_a(a,\cdot).
\end{equation}
Then $\nabla'_cc\in C$ on either of the following actual charts:
\begin{enumerate}
\item $C\cap\operatorname{im}S=0$;
\item $C\subset\operatorname{im}S$ and the quotient product
      $\nu:\Lambda^2(V/C)\to V/C$ is invertible.
\end{enumerate}
This self-direction conclusion uses only the differentiated actual ideal
and the full weak pair. In chart (1), writing $H=\operatorname{im}S$,
complete \ChernCyclic{} additionally gives
\begin{equation}\label{wp:partial-transport}
 \nabla'_HH\subset H,\qquad\nabla'_cc\in C.
\end{equation}
In chart (1), the character may vanish identically.
\end{proposition}
\begin{proof}
On $U=V/C$ put $L_u=N[u]$, $B_u=L_u-\lambda(u)I$, so
$D[u]=B_u+u\lambda^{\mathsf T}$ by \eqref{mid:ideal-volume}.
The ranks of $D$ are three and two in charts (1) and (2).
Lemma~\ref{lem:ideal-line-quotient-algebra} gives
$q_D(u)=u^{\mathsf T}(\adj D)u\not\equiv0$. Put
$k_u=u^{\mathsf T}(\adj D)/q_D(u)$. The adjugate identity gives
$k_u(u)=1$, $k_uD[u]=0$ and $k_uB_u=-\lambda$.
We compare the two eigenspace lifts for $a$ and $a+tc$. The full
torsion product gives, for a general lift $a$ of $u$,
\[
 E_{a+tc}=\{v+tk_u(qv)c:v\in E_a\}.
\]
Indeed its added central coefficient is
$t\{\lambda(u)k_u(qv)-\lambda(qv)\}=tk_u(L_uqv)$ on nonstar
eigenvectors. Their nonzero roots give
$\bar E_u=qE_a=\im L_u$, $U=\C u\oplus\bar E_u$.

Write $U_p=q\nabla'_pc$, $w=U_c$; ideal differentiation gives
$q\mathcal X_p(c,v)=B_{qv}U_p$. For $\tilde v\in E_a$, $v=q\tilde v$, put
$Y(t)=\mathcal X_{a+tc}(a+tc,\tilde v+tk_u(v)c)$, of degree at most two.
For general $t$, \HscWeakPair{} gives $qY(t)\in\bar E_u$ and
$n(Y(t))=tk_u(qY(t))$, where $n(c)=1$ and
$n|_{\C a\oplus E_a}=0$.
The quadratic coefficient of $qY(t)$ is $Z_u(v)=B_vw-k_u(v)B_uw$.
Comparing cubic coefficients in the normal equation gives $k_uZ_u=0$.
Since $qY(t)$ lies in $\bar E_u$, its quadratic coefficient $Z_u(v)$
lies there for every $v\in\bar E_u$.
Since $Z_u(u)=0$, this gives $Z_u(U)\subset\bar E_u$.
In $k_uZ_u(v)$ the terms containing $\lambda(w)$ cancel, giving
$(\det D/q_D(u))\operatorname{vol}(u,v,w)$. In (1) this forces
$w\in\C u$ for general $u$, hence $w=0$.

In (2), the rank-one adjugate gives $k_u=m^{\mathsf T}/m(u)$
for $0\ne m\in\ker D^{\mathsf T}$, and $\xi_u=u^{\mathsf T}N^{-1}$
kills $\bar E_u$.
For $\tau(u)=\xi_uw$, $v_0=[\lambda]m$, the row
$m(u)\xi_uZ_u=0$ reads
$u\times\{m(u)w+\tau(u)v_0\}=0$.
Thus $wm^{\mathsf T}+v_0w^{\mathsf T}N^{-\mathsf T}$ is scalar;
its rank is at most two, so it is zero.
Since $D^{\mathsf T}m=0$ and $N-N^{\mathsf T}=[\lambda]$, we have
$N^{\mathsf T}m=v_0$ and $Nm=2v_0\ne0$. Multiplying the zero matrix
above on the right by $N^{\mathsf T}$ gives
$2wv_0^{\mathsf T}+v_0w^{\mathsf T}=0$. Subtracting twice the transpose
gives $-3v_0w^{\mathsf T}=0$, so $w=0$.

When $w=0$, put $K(qp)=U_p$. Under \ChernCyclic{} on $(c,a,\tilde v)$
(its Jacobiator lies in $C$), the linear coefficient $Z_1(v)=2B_vKu-B_uKv-k_u(v)B_uKu$
of $qY(t)$ lies in $\bar E_u$ on $U$, since $Z_1(u)=0$.

For internal preservation in (1), set $n(c)=1$, $n|_H=0$,
$\beta_p(h)=n(\nabla'_ph)$. Full product differentiation gives
\[
 n\mathcal X_p(h,k)=\beta_pS(h,k)
 =\beta_pT(h,k)+\lambda(k)\beta_p(h)-\lambda(h)\beta_p(k).
\]
Since $E_u\subset H$ complements $\C u$, \HscWeakPair{} gives
$\beta_uS(u,H)=0$. Normal \ChernCyclic{} gives zero cyclic sum for
$\beta_uS(v,w)$, since $T(H,H)\subset H$.
The repeated vanishing extends to all $u$ by polynomial continuation.
Together with the cyclic sum, \eqref{unif:eq:cyclic-recovery} gives
$\beta_uS(v,w)=0$ for all $u,v,w\in H$. Since $S$ maps onto $H$,
$\beta_H=0$.
\end{proof}

\Needspace{5\baselineskip}
\subsection{Compatibility identities and the case \texorpdfstring{$\rho\circ T=0$}{vanishing contraction}}
\subsubsection{Zero contraction without a pencil hypothesis}
\label{sec:w4-zero-contraction-without-pencil}

\begin{theorem}[Exclusion of $\rho\circ T=0$ for a rank-two shift]
\label{thm:w4-zero-contraction-without-pencil}
On a metric fourfold system, suppose there are smooth fields $\rho$, $L$, and
$\upsilon$ such that
\[
 S=T-B_\rho,\qquad
 N=\im S,\qquad \rank N=2,
\]
with $L\in V^*\otimes\operatorname{Hom}(N,V)$ and $\upsilon\in N^*$
satisfying the full actual identities \eqref{eq:w4-explicit-source-inputs}.
If $\rho\circ T=0$ on the region, that region is empty.
\end{theorem}
\begin{proof}
On $\rho\ne0$, Proposition~\ref{prop:conceptual-w4-covector} gives
the abelian ideal $N\subset\ker\rho$. On $Q=V/N$ put
$B_p=T(p,\cdot)|_N-\rho(p)I$, $r_p=\upsilon B_p$, $\omega=\upsilon S$.
Balance and the covector syzygy give
\begin{equation}\label{eq:w4-zero-r-source}
 \begin{gathered}
 \tr B_p=-3\rho(p),\qquad \nabla'_p\rho=-2\rho(p)\rho+\omega(p,\cdot),\\
 r_pB_q-r_qB_p-2\rho(p)r_q+2\rho(q)r_p=0.
 \end{gathered}
\end{equation}
If $\rank r=2$, the cross pairing makes $\omega$ nondegenerate,
contrary to Proposition~\ref{prop:w4-symplectic-covector}.
As $\rank r\le1$, if $f\in\ker\rho$ has $r_f\ne0$ we can
choose $\rho(e)=1$, $r_e=0$ by adding a multiple of $f$.
Then \eqref{eq:w4-zero-r-source} gives nonzero left eigenvectors
$\upsilon,r_f$ of $B_e$ with eigenvalues $0,-2$, contrary to
$\tr B_e=-3$. Thus $r|_{\ker\rho}=0$.

Set $H=\ker\rho\supset N+\im T$.
Since $N$ is abelian, $H/N$ is a line and the cross pairing is zero,
$\omega(H,H)=0$. The covector law gives $\nabla'_HH\subset H$.
The induced action on $V/H$ is zero, so internal torsion has zero trace.
Lemma~\ref{r37:lem:internal-restriction} and
Theorem~\ref{thm:conceptual-closed-three} give $T(H,H)=0$.
Theorem~\ref{thm:fourfold-abelian-hyperplane} gives $T=0$,
impossible because $-B_\rho$ has rank three.
Thus $\rho=0$ as a field; then $S=T$ has rank two, excluded by
Theorem~\ref{thm:fourfold-lowrank-torsion}.
\end{proof}

The exclusion assumes the identities \eqref{eq:w4-zero-full-entrances}
and $\rho T=0$, but imposes no condition on the pencil.

\Needspace{6\baselineskip}
\subsubsection{A shifted rank-two image annihilated by the repeated defect}
\begin{proposition}\label{research:shifted-annihilator}
Work on a metric fourfold system. Put $A_a=T(a,\cdot)$, $D_a=\mathcal X_a(a,\cdot)$ and
$\mathfrak F_a=D_a+2A_a^2$. Suppose a smooth covector $\rho$ gives an actual
smooth rank-two image
\[
 S=T-B_\rho,
 \qquad N=\im S,
\]
and $\mathfrak F_a|_N=0$ for every auxiliary vector $a$. No such region exists.
\end{proposition}
\begin{proof}
For $n\in N$, $v\in V$, the identities $\mathfrak F_aa=0$ and
$\mathfrak F_aN=0$ give
$t\mathfrak F_{n+tv}v=\mathfrak F_{n+tv}(n+tv)=0$.
For $t\ne0$, divide by $t$ and then pass to $t=0$. This gives
$\mathfrak F_nv=0$ for every $v$, or $\mathfrak F_n=0$.
Initially $N$ is a torsion subalgebra. The quotient action of $A_n$
is $\rho(n)I$; balance gives internal trace $-2\rho(n)$, hence
$T|_{N^2}=-2B_{\rho|_N}$ in dimension two.
Thus $\tr A_n^2=6\rho(n)^2$. Since $D_n=-2A_n^2$ and
$\tr D_n=0$, $\rho|_N=0$, $T(N,N)=0$ and $T(V,N)\subset N$.
Now $A_n^2=D_n=0$, so $\mathcal X_n(u,n)=0$; no preservation was used.

Put $B_u=A_u|_N$. The corrected slots in~\eqref{eq:repeated} on
$(u,n)$ are $-3B_u^2n/2$ and $0$. Retract the final output to $\Sym^2N$
and polarize in $n$: the resulting row is~\eqref{sharedE:equation}.
Lemma~\ref{sharedE:theorem} gives $\dim\{B_u:u\in V\}\le1$.
The quotient action has trace $\rho(u)$, so balance gives
$\tr B_u=-\rho(u)$. Where $\rho\ne0$, this forces $B_{\ker\rho}=0$.
Thus $H=\ker\rho$ is abelian ($H/N$ is a line) and an ideal ($\rho T=0$).
Theorem~\ref{thm:fourfold-abelian-hyperplane} gives $T=0$,
contrary to $\rank(-B_\rho)=3$. Hence $\rho=0$ throughout,
so $\rank T=2$, contrary to Theorem~\ref{thm:fourfold-lowrank-torsion}.
\end{proof}

\subsubsection{The degenerate shifted-pencil obstruction}
\begin{proposition}\label{research:common-factor-boundary}
On a metric fourfold system, let $\rho$ be smooth and put $S=T-B_\rho$.
On an actual constant-rank chart with $\rank S\le2$, if every
member of $\im S^*$ is decomposable, then $S=T=\rho=0$.
Consequently every actual rank-two shifted-image chart has an open
dense locus where its output pencil contains a nondegenerate form.
\end{proposition}
\begin{proof}
If $T\ne0$ anywhere, choose a point of its nonzero locus where
$r_T$ is dominant by Proposition~\ref{prop:generic-root-dominance-uniform}.
Decomposability makes $\operatorname{pf}(S^\vee\xi)=0$ for all $\xi$,
so the \hyperref[pfaffian:shift-extraction]{Pfaffian shift extraction}
gives $\rank S=3$, a contradiction. Hence $T=0$;
$S=-B_\rho$ has rank zero or three, so $\rho=S=0$.
Pencils containing a nondegenerate form constitute an open locus;
the preceding argument excludes interior of its complement in rank two.
\end{proof}

\begin{lemma}[A simple quotient root and the repeated defect]
\label{unif:lem:spectral-riccati}
Let $T$ be a smooth alternating tensor on $V=T^{1,0}X$, with
$\nabla'$ the $(1,0)$ part of a smooth connection on $V$.
For $a\ne0$, put $A_a=T(a,\cdot)$, $D_a=(\nabla'_aT)(a,\cdot)$ and
$\mathfrak F_a=D_a+2A_a^2$; bars denote the induced maps on $V/\C a$.
Suppose a smooth covector $\rho$ gives a simple root $\rho(a)$ of
$\bar A_a$ for general $a$, and $\bar{\mathfrak F}_a$ kills that
eigenline. Then $\Sym\nabla'\rho=-2\rho^2$.
\end{lemma}
\begin{proof}
Extend general $a$ normally at the point and use a complementary
quotient frame, so $\nabla'_a\bar A_a=\bar D_a$. For the simple root
$\lambda=\rho(a)$ choose a smooth right eigenvector $b$ and, at the point,
a left eigenvector $\ell$ with $\ell(b)=1$. Pairing the differentiated
eigenvector equation with $\ell$ cancels the derivative of $b$ and gives
$\nabla'_a\lambda=\ell(\bar D_ab)=-2\lambda^2$ by the defect hypothesis.
This includes a simple zero root. Polarization gives the claim.
\end{proof}

\Needspace{8\baselineskip}
\subsubsection{Full rank-two shifted-image recovery}
\begin{proposition}\label{prop:full-shifted-image-recovery}
On a smooth rank-four tangent distribution, assume $\mathcal X=\nabla'T$
and \eqref{eq:cyclic}. For a smooth covector $\rho$, set
\[
 S=T-B_\rho,\quad
 N=\operatorname{im}S,\quad \operatorname{rank}N=2,\quad q:V\to V/N.
\]
Assume the two-form pencil $\operatorname{im}S^*$ has a nondegenerate
member and
\[
 \operatorname{Sym}\nabla'\rho=-2\rho\otimes\rho,\qquad
 \mathfrak F_a|_{\ker S_a}=0\quad\hbox{for generic auxiliary }a,
 \quad \mathfrak F_a=\mathcal X_a(a,\cdot)+2T_a^2.
\]
Then there are unique actual smooth fields
$L\in V^*\otimes\operatorname{Hom}(N,V)$ and $\upsilon\in N^*$ such that
\begin{equation}\label{recovery:full-identities}\equationalias{fullcoverage:4}
\begin{gathered}
 \mathfrak F_a=L_aS_a,\qquad \mathcal X_p(u,v)=\mathfrak H_p(u,v) +\tfrac13\{2L_pS(u,v)+L_uS(p,v)-L_vS(p,u)\},\\
 (\nabla'_p\rho)(v)=-2\rho(p)\rho(v) +(\upsilon-\rho|_N)S(p,v),\qquad q\nabla'_pn=qL_pn+qp\{\rho(n)+\upsilon(n)\}.
\end{gathered}
\end{equation}
\end{proposition}
\begin{proof}
The map $S_a$ fails to be onto precisely when some nonzero pencil
member annihilates $a$ in one slot. Such a member is degenerate, so
it occurs at a projective Pfaffian root. The resulting union of
two-dimensional radicals has codimension two, even at a repeated root. The kernel condition gives the unique
rational factor $\mathfrak F_a=L_aS_a$. Minor right inverses make
$L_a$ regular off that locus. Any nonconstant reduced denominator
would define a hypersurface of poles, impossible in codimension
two. Thus $L_a$ is polynomial, and homogeneity makes it linear.
An invertible minor of the injective coefficient map
$L\mapsto(a\mapsto L_aS_a)$ gives smooth coefficients.
The cyclic tensor $\mathcal X-\mathfrak H$ has repeated value
$\mathfrak F$ by \eqref{unif:reference:elementary}, so
\eqref{unif:eq:cyclic-recovery} gives its full correction.

Put $\beta_pn=q\nabla'_pn$, $G_p=qL_p$, $\sigma=\rho T=\rho|_NS$,
$\kappa=\nabla'\rho+2\rho^2\in\Lambda^2V^*$.
Full image differentiation gives
$q(\mathcal X_p-B_{\nabla'_p\rho})=\beta_pS$.
Using $qA_a^2b=\rho(a)^2qb-
\{\rho(a)\rho(b)+\sigma(a,b)\}qa$ and Riccati, its repeated row is
\begin{equation}\label{r22:eq:normal-repeated}
 (\beta_a-G_a)S_a=(qa)\otimes(\kappa+2\sigma)(a,\cdot).
\end{equation}
For $n\in N$, set $R_n(a)=(\beta_a-G_a)n$. Generic surjectivity gives
$(qa)\wedge R_n(a)=0$ identically. Polarizing with $\ker q$ gives
$R_n(\ker q)=0$. The induced map preserves every line, so $R_n=\ell(n)q$,
with $\ell$ smooth by evaluation on a local lift. Cancel $qa\ne0$ in \eqref{r22:eq:normal-repeated} and
extend by continuity to get $\kappa+2\sigma=\ell S$.
Taking $\upsilon=\ell-\rho|_N$ gives the last two identities;
surjectivity of $S$ gives uniqueness.
\end{proof}

\subsection{Canonical polynomial full-output factors}
\label{sec:polynomial-full-annihilator}

\begin{proposition}
\label{prop:polynomial-full-annihilator}
\label{cor:defect-factor-at-most-one}\label{cor:original-zero-factor-transport}
Let $V=T^{1,0}X$ on a balanced Hermitian fourfold region with constant
Chern HSC $c\ne0$. For $a\in V\setminus\{0\}$ set
\[
 A_a=T(a,\cdot),\quad D_a=\mathcal X_a(a,\cdot),\quad
 \mathfrak F_a=D_a+2A_a^2,\quad
 q_a:V\to V/\C a,\quad \widetilde{\mathfrak F}_a q_a=\mathfrak F_a.
\]
Write $\bar A_a$ for the endomorphism induced by $A_a$ on $V/\C a$,
and $\chi_{\bar A_a}(t)=\det(tI-\bar A_a)$. Assume that at every
point $\bar A_a$ has three distinct eigenvalues for generic $a$.
In rational quotient frames define the monic factors
\[
 f=\gcd\bigl(\chi_{\bar A},
 \{(\widetilde{\mathfrak F}\operatorname{adj}(tI-\bar A))_{ij}\}\bigr),
 \qquad g=\chi_{\bar A}/f,\qquad d=\deg_t f.
\]
They are polynomials in $a,t$ with the corresponding homogeneous
coefficient degrees. One has the full-output identity
$\widetilde{\mathfrak F}_a g(a,\bar A_a)=0$; a root belongs to $f$
exactly when the full defect kills its right eigenline.
Open constant-degree charts are dense, their factor coefficients are
smooth, and on every such chart $d\le1$. The degree-zero locus is open.
On a degree-one chart, $f=t-\rho(a)$ for a smooth covector satisfying
$\operatorname{Sym}\nabla'\rho=-2\rho\otimes\rho$.
\end{proposition}
\begin{proof}
In a complementary frame on $a_0\ne0$, the quotient matrices
$A=\bar A$, $F=\widetilde{\mathfrak F}$ have degrees one and two
over $K=\C(a)$, and $\chi=\det(tI-T(a,\cdot))/t=\det(tI-A)$ is polynomial.
At a simple root the adjugate has image its right eigenline, so
\[
 Fp(A)=0\quad\Longleftrightarrow\quad g\mid p\qquad(p\in K[t]).
\]
This proves the root criterion and frame independence. The monic
factors of $\chi$ have polynomial coefficients: those coefficients
are integral over $\C[a]$, which is integrally closed. Scaling $a,t$
gives the stated degrees.

The column matrix $(\operatorname{vec}F,\operatorname{vec}FA,\operatorname{vec}FA^2)$
has rank $k=3-d$ by the annihilator criterion. After clearing denominators, a minor
is nonzero for general $a$ precisely when at least one of its
coefficients is nonzero. These coefficient tests give dense
constant-rank charts and an open rank-three locus.
For $k>0$, coefficient comparison in $FA^k+\sum_{j<k}g_jFA^j=0$
has full column rank since $1,t,\ldots,t^{k-1}$ are independent modulo $g$.
A full-rank minor gives smooth $g$, and monic division gives $f$.

If $d=3$, $D_a=-2A_a^2$ and $\tr D_a=0$ give $\tr A_a^2=0$.
Its derivative along $a$ with $\nabla'_a a=0$ is $-4\tr A_a^3=0$.
Balance then gives $\chi_{\bar A_a}=t^3$, contrary to simplicity.
If $d=2$, write $g=t-\rho(a)$ with $\rho$ smooth linear.
The defining gcd makes $\mathfrak F$ nonzero at each base point.
Then $\mathfrak F_aA_a=\rho(a)\mathfrak F_a$ and
$\mathfrak F_aS_a=0$ for $S=T-B_\rho$.
Since $g\mid\chi$, Lemma~\ref{lem:sep-plane-components} applies
to this simple linear root and gives $2\le\rank S\le3$.
On the open rank-three locus its root planes form a star with smooth
radical $C=\C z$ and
\begin{equation}\label{eigdef:star}
 S(C,V)=0,\qquad\rho(C)=0,\qquad T(v,z)=\rho(v)z.
\end{equation}
On the nonstar full eigenlifts of Lemma~\ref{lem:entrance-full-eigenlift},
the annihilator gives $\mathfrak F_ae_i=0$ since their roots differ
from $\rho(a)$. Thus $D_a|_{E_a}=-2A_a^2|_{E_a}$ is full \HscWeakPair{}, and
Theorem~\ref{r8:thm:incident-weak-empty} excludes this locus.
For rank-two $N=\im S$, the simple zero root of $\bar S_a$ gives $\im S_a=N$ generically.
Thus $\mathfrak F_aN=0$ polynomially, contrary to Proposition~\ref{research:shifted-annihilator}.
Finally $d=1$ means $f=t-\rho(a)$ with smooth linear $\rho$;
the root characterization and Lemma~\ref{unif:lem:spectral-riccati}
give $\Sym\nabla'\rho=-2\rho^2$.
\end{proof}

\subsection{Spectral identities and the degree-one Euler map}
\label{sec:full-spectral-exterior-entrance}
The degree of the Euler root map is its generic fiber degree; it is
not $d=\deg_t f$.

Work on a nonzero-torsion region of a metric fourfold system. Put $A_a=T(a,\cdot)$, $D_a=\mathcal X_a(a,\cdot)$ and
$\mathfrak F_a=D_a+2A_a^2$, where $\mathcal X=\nabla'T$.
The canonical bilinear map is $\Theta(u,v)=uv+P(u,v)$,
where $uv=(u\odot v)/2$ as in the Chern system,
and $W=\ker\Theta\subset V\otimes V$.
The Euler-map arguments invoke the complete Chern identities and each
generic condition where specified. \HscGenericEval{} is \eqref{sep:eval} in this notation,
not an assumption at prescribed hyperplanes.

\begin{lemma}[Full lifts of a simple quotient spectrum]
\label{lem:entrance-full-eigenlift}
If $\bar A_a$ has three distinct eigenvalues for general auxiliary $a$,
then $A_a$ has a full eigenbasis $a,e_1,e_2,e_3$ for general $a$.
\end{lemma}
\begin{proof}
The quotient characteristic polynomial is $\chi_a(t)=\det(tI-A_a)/t$.
If its constant term is not identically zero, all four eigenvalues
of $A_a$ are generically distinct. Otherwise write $\chi_a=tR_a$ and fix a
general $a$ with simple quotient roots. If $T(a,b)=a$, then
$A_{a+sb}a=-sa$ gives $R_{a+sb}(-s)=0$ for $s\ne0$, since
$a,b$ are independent. Thus $R_a(0)=0$ by continuity, contradicting
the simplicity of the zero root. Hence $a\notin\im A_a$.
The image under $A_a$ of a zero-root lift lies in $\C a$ and must
therefore be zero. For a nonzero root, adding a suitable multiple
of $a$ to the lift removes the error in that line.
\end{proof}

\begin{lemma}[Nonvanishing defects force root-map degree one]
\label{lem:entrance-two-good-trace}\label{lem:entrance-nozero-birational}
At a fiber where generic separability, \HscGenericEval{}, generic right-line
injectivity and root dominance hold, suppose $\mathfrak F_a$ is nonzero on
all three quotient eigenlines for general $a$. Then the Euler root
map $r_T:[\xi]\mapsto[q_T(\xi)]$ is birational.
\end{lemma}
\begin{proof}
Euler admissibility and dominance make the generic fiber a finite set
of at most three simple left quotient eigenlines; its size is the
function-field degree. Suppose it is at least two. The image of the
\HscGenericEval{} exceptional set has dimension at most two. A general
target therefore has at least two \HscGenericEval{} preimages, and can be chosen
where full eigenlifts, right-line injectivity and defect nonvanishing hold. For two of them the hyperplanes are
$\langle a,e_1,e_3\rangle$, $\langle a,e_2,e_3\rangle$ after relabeling.
The right-$e_i$ columns $2\mathfrak F_ae_i$ make all endpoints
nonzero. The common trace calculation in
Lemma~\ref{lem:full-pair-with-dependent-endpoints} would give
$\mathfrak F_ae_3=0$, a contradiction. The generic degree is therefore one.
\end{proof}

\begin{proposition}[The shift determined by a birational Euler map]
\label{prop:entrance-degree-one-exterior}
At a fiber where the Euler root map is dominant and birational,
there is a unique covector $\rho$ such that
\[
 \operatorname{pf}(T^\vee\xi)=\rho(q_T(\xi)),\qquad
 \beta\wedge\beta=2\operatorname{pf}(\beta)\Omega.
\]
The tensor $S=T-B_\rho$ has rank three and a unique radical line $C$,
with $\rho(C)=0$ and $T(v,c)=\rho(v)c$ for $c\in C$.
On a region where these fiberwise hypotheses hold, $\rho$ and $C$ are smooth.
\end{proposition}
\begin{proof}
By \eqref{r11:eq:euler-pfaffian}, the rational inverse of $r_T$
selects a rational quotient root $\lambda(a)$, homogeneous of degree
one. It satisfies the monic characteristic equation over $\C[a]$.
Since this ring is integrally closed, $\lambda$ is a polynomial;
its degree-one homogeneity makes it a linear covector $\rho$.
The Pfaffian identity extends from the dense inverse domain, and
dominance gives uniqueness.
The \hyperref[pfaffian:shift-extraction]{Pfaffian shift extraction}
gives $S,C$ and smoothness.
\end{proof}

\begin{proposition}[\HscGenericEval{} for the hyperplane determined by the quotient eigenline]
\label{prop:entrance-specific-good-pair}
Suppose an actual smooth full ideal line $C=\mathbb Cc$ has character
$\rho$. For the first conclusion assume only Euler root dominance,
simplicity of the ideal quotient root $\rho(a)$ for general $a$,
and \HscGenericEval{} for general covectors at each coefficient fiber.
Its left quotient eigenline $[\xi_a]$ is rationally inverse to $r_T$,
and \HscGenericEval{} holds for $J_a=\ker\xi_a$ at general $a$.
For distinct nonstar roots, let $E_a$ span their specified full
eigenvector lifts. Then $J_a=\langle a,E_a\rangle$ omits $C$.
On these charts, if both complete endpoints are generically nonzero,
$D_aE_a\subset E_a$.
\end{proposition}
\begin{proof}
For general $\xi$ with $\xi(c)\ne0$, dominance makes $a=q_T(\xi)$
general. Admissibility gives $\xi(a)=0$ and $\xi A_a=\lambda\xi$;
evaluation on $c$ forces $\lambda=\rho(a)$.
The simple ideal root determines a unique left quotient line, recovered
by a nonzero minor. Its rational map composed with $r_T$ is therefore
the identity on a dense domain, so it is the rational inverse.
Dominance of the inverse gives \HscGenericEval{} on $J_a$ for general $a$.

When the other roots are distinct, Lemma~\ref{lem:entrance-full-eigenlift}
supplies their full eigenvectors, including a possible zero root.
The left ideal-root covector kills $a,E_a$ and pairs nontrivially
with $c$, proving the stated hyperplane identity.
On the joint open set where both complete endpoints are nonzero,
Lemma~\ref{lem:full-pair-with-dependent-endpoints} gives
$D_aE_a\subset E_a$. Its full comparison includes the right-$a$
columns and covers dependent endpoints; no right-line injectivity
hypothesis is added.
\end{proof}

\subsection{An incident full ideal line and its nonstar pair}
\label{sec:incident-weak-pair-uniform}
We assume full \HscWeakPair{} on a metric fourfold system. Generic
\HscGenericEval{} alone does not give the dimension of the evaluation
kernel at a prescribed conormal. The quotient products are algebraic and use no quotient
metric or HSC theorem.

\begin{lemma}[The incident form with invertible quotient]
\label{r8:lem:incident-coefficients}
Let a trace-free alternating product on a complex four-space have
an ideal line $C=\mathbb Cc$, with $T(a,c)=\lambda(a)c$.
Suppose $S=T-B_\lambda$ has rank three, radical $C$, and image
$H\supset C$, and the quotient product on $V/C$ is invertible.
Then a coefficient-point frame gives
\begin{equation}\label{r8:eq:incident-I}
 S(e_2,e_3)=c,\qquad S(e_3,e_1)=e_2,\qquad
 S(e_1,e_2)=e_1,\qquad \lambda=e_2^*/3.
\end{equation}
Here $S(C,V)=0$ and $H=\langle c,e_1,e_2\rangle$.
The frame is pointwise algebraic; its derivatives are not prescribed.
\end{lemma}
\begin{proof}
Write $N=(2D+D^{\mathsf T})/3$ by \eqref{mid:ideal-volume}.
Since $N$ is invertible, the left and right radical generators
$l,r$ of the rank-two $D$ are independent, and $D(r,l)\ne0$:
otherwise $Nr,Nl$ both pair to zero with $\langle r,l\rangle$.
Choose $0\ne v\in\ker D(r,\cdot)\cap\ker D(\cdot,l)$.
Then $(r,v,l)$ is a basis and rank two gives $D(v,v)\ne0$.
Scale to $D(r,l)=D(v,v)=1$. In this basis
\[
 D=\begin{pmatrix}0&0&1\\0&1&0\\0&0&0\end{pmatrix}.
\]
This ordinary congruence $G$ is realized in the volume encoding
by the frame change $(\det G)G^{-1}$.
The remaining central row $m:\Lambda^2U\to C$ changes to $m-fD$
when the lift changes by $f$. Rank three of $S$ makes $m$
nonzero on $\ker D$; kill its other entries and scale $c$ to
normalize this one. Reading the ordered columns
$(e_2\wedge e_3,e_3\wedge e_1,e_1\wedge e_2)$ gives the table,
and $(D-D^{\mathsf T})/3=[\lambda]$ gives its character.
\end{proof}

\subsubsection{The full star cut from an invertible incident quotient}
\label{sec:ideal-incident-invertible-star}

\begin{proposition}\label{prop:ideal-incident-invertible-star}
Under case (2) of Proposition~\ref{prop:ideal-incident-self-transport},
the full weak-pair inclusion, complete \ChernCyclic{}, and complete \ChernPure{} imply
$\nabla'_V C\subset C$ and $\mathscr S_T=0$ on every two-plane
containing $C$.
\end{proposition}

\begin{proof}
Use the pointwise table \eqref{r8:eq:incident-I}. On $U=V/C$ put
\[
 N=\begin{pmatrix}0&0&2/3\\0&1&0\\1/3&0&0\end{pmatrix},\quad
 \lambda=e_2^*/3,\quad\xi_u=u^{\mathsf T}N^{-1}.
\]
With $K(qp)=q\nabla'_pc$, the linear coefficient $Z_1$ computed in the
proof of Proposition~\ref{prop:ideal-incident-self-transport}
satisfies $\xi_uZ_1=0$. For $u=(x,y,z)$ and
$\tau(u)=u^{\mathsf T}N^{-1}Ku$, this is
\begin{equation}\label{r8:eq:incident-K-row}
 6z(u\times Ku)=yz(N^{-1}K)^{\mathsf T}u-2z\tau(u)e_2+y\tau(u)e_3.
\end{equation}
On $y=0$, cancel $z$ and compare components to obtain $(Ku)_2=0$
and $9z(Ku)_1=0$. On $z=0$, $y\tau=0$ gives $K_{31}=K_{32}=K_{22}=0$.
Thus only $b=K_{12}$ and $j=K_{33}$ can remain. At $e_2+e_3$ the row
is $(6j,15b/2,-15b/2)=0$. Hence $K=0$ and $\nabla'_V C\subset C$.

Preservation of $C$ gives $\mathcal X_p(a,c)=(\nabla'_p\lambda)(a)c$.
In the trace of \ChernCyclic{} on $(c,a,v)$, the Jacobiator and
$\mathcal X_a(\cdot,c)$ take values in $C$ and kill $c$;
$\mathcal X_c(a,\cdot)$ has trace zero. Thus $(\nabla'_c\lambda)(a)=0$.
Complete repeated \ChernPure{} is now
\[
 \{2(\nabla'_a\lambda)(a)+4\lambda(a)^2\}c^2\in W.
\]
It vanishes by symmetry, proving the Riccati equation and both
complete zero-star columns. No normal form was differentiated.
\end{proof}

\subsubsection{A primitive left-root row preserves the full rank-three image}
\label{sec:ideal-rankthree-image-preservation}

\begin{proposition}\label{prop:ideal-rankthree-image-preservation}
On a chart of the balanced fourfold, let $C$ be a full ideal
line of character $\lambda$, and put
\[
 S=T-B_\lambda,\qquad \rank S=3,\quad\operatorname{rad}S=C,
 \qquad H=\im S,\quad K=\im T,\quad \dim K=3,\quad K\ne H.
\]
Suppose the general quotient spectrum is separable and full weak-pair
inclusion holds for the two nonzero nonstar eigendirections.
Then $\nabla'_VK\subset K$.
The incident case $C\subset H$, $\lambda|_H\ne0$ is included:
if $K=H=\ker n$, then $nT=nS=0$ would give
$\lambda\wedge n=0$, contrary to $\lambda|_H\ne0$.
Only the actual full-image derivative and full weak-pair inclusion
are used for the source step; no \HscGenericEval{} identity is asserted at an
exceptional auxiliary vector. The nonzero-root hypothesis is explicit.
\end{proposition}
\begin{proof}
Since $S\ne T$, $\lambda\ne0$ and $C\subset K$.
Put $q:V\to U=V/C$, choose a volume, and encode
$S(a,b)=M(qa\times qb)$ with $M:\C^3\to H$ an isomorphism.
Choose $H=\ker n$ and linear $\psi_u\in V^*$ with
$M^{\mathsf T}\psi_u=u$. The quadratic covector
\begin{equation}\label{r39:eq:common-ideal-row}
 \theta_a=n(a)\psi_{qa}-Q(a)n,\qquad Q(a)=\psi_{qa}(a)
\end{equation}
kills $a,S(a,V)$, so $\theta_aT(a,\cdot)=\lambda(a)\theta_a$.

We show that the entries of $\theta_a|_K$ have no nonconstant common
factor. Take
$n(k_1)=1$, $k_2,k_3\in K\cap H$. The last two entries are
$n(a)\ell_j(qa)$, with independent $\ell_j(u)=\psi_u(k_j)$.
A common divisor of the whole row would therefore force $n\mid Q$.
For $D=qM$, $Q(Mz)=z^{\mathsf T}Dz$ would make $D$ skew.
Equation~\eqref{mid:ideal-volume} then gives
$N=[\lambda]_\times/2$, $D=3[\lambda]_\times/2$, hence
$qH=qK$ of rank two. Both three-planes contain $C$, contradicting
$H\ne K$.

The two nonzero nonstar eigenvectors lie in $K\cap\ker\theta_a$.
Thus, for general $a$, $E_a=\ker(\theta_a|_K)$ and $T(a,E_a)=E_a$.
\label{lem:primitive-row-image-preservation}
Put $\beta_pk=\pi_K\nabla'_pk$. Full image differentiation gives
$\pi_K\mathcal X_p(u,v)=\beta_pT(u,v)$, including both moving inputs.
\HscWeakPair{} therefore gives $\beta_a|_{E_a}=0$.
Pairing with a quotient covector makes $\beta_a$ a homogeneous linear
row rationally proportional to the primitive quadratic row $\theta_a|_K$.
The denominator of that multiplier in reduced form divides every entry
of $\theta_a|_K$, so the multiplier is polynomial. A nonzero polynomial
multiple of a homogeneous quadratic row cannot have degree one.
Thus $\beta=0$.
\end{proof}

\phantomsection\label{conv:incident-zero-character}
For a smooth full ideal line $C$ of character $\lambda$, assume
$S=T-B_\lambda$ has radical $C$, rank three and image $H=N_\lambda$
with $C\subset H$ and $\lambda|_H=0$. The
\hyperref[mid:heisenberg-normalization]{incident normalization}
gives the table of Lemma~\ref{r25:weak-center-terminal} with $f=0$.
With generic quotient separability and full \HscWeakPair{} for the specified
nonstar full eigenlifts, Proposition~\ref{r25:heisenberg-WP} excludes this table.

\subsubsection{One weak-pair classification for full ideal lines}
\begin{theorem}[The complete ideal-line weak-pair classification]
\label{r8:thm:incident-weak-empty}
Suppose a metric fourfold system has a smooth full ideal line $C=\mathbb Cc$ of character
$\lambda$. Put
\[
 S=T-B_\lambda,\qquad \operatorname{rad}S=C,\qquad
 \operatorname{rank}S=3,\qquad H=\operatorname{im}S.
\]
Assume generic quotient separability and the full weak-pair inclusion
$D_aE_a\subset E_a$ for the specified two nonstar full eigendirections.
\textup{(i)}\label{thm:transverse-weak-pair-star}
The transverse locus $C\cap H=0$ is empty.
\textup{(ii)} The region is empty.
\end{theorem}
\begin{proof}
Let $N$ be the quotient torsion matrix. If $\rank N\le1$,
write $N=uv^{\mathsf T}$; by \eqref{mid:ideal-volume}, $\lambda$
annihilates $u,v$ and hence $qH=\im D\subset\langle u,v\rangle$.
Thus $\rank N\ge2$ whenever $C\cap H=0$ or $\lambda|_H\ne0$.
Lemma~\ref{lem:ideal-line-quotient-algebra} then makes the nonstar roots
generically nonzero on these loci.

If $\rank N=2$ and $\lambda\ne0$, $K=\im T$ contains $C$
and has rank three. Equality $K=H=\ker\zeta$ would give
$\zeta T=\zeta S=0$ and $\lambda\wedge\zeta=0$, hence $\lambda|_H=0$.
So $K\ne H$ in the transverse or nonzero-character case, and
Proposition~\ref{prop:ideal-rankthree-image-preservation} with
Corollary~\ref{ab:cor:lowrank} excludes that rank-two locus.

On $C\cap H=0$, Proposition~\ref{prop:ideal-transverse-weak-partial}
applies by Lemma~\ref{lem:transverse-ideal-coefficient-entrance}
and gives $\nabla'_HH\subset H$. In its volume encoding,
$\nu=N\times$, $S|_{\Lambda^2H}=D\times$ and
$D=2N-N^{\mathsf T}=N+[\lambda]$ is invertible.
The restricted equations and their rank bound
(Lemmas~\ref{r37:lem:internal-restriction},
\ref{fourh:prime-internal-rank-bound}) give $\det N=0$.
Invertibility of $D$ excludes $\rank N\le1$ and $\lambda=0$, so
\begin{equation}\label{wp:singular-determinants}
\rank N=2,\qquad\lambda\ne0.
\end{equation}
The preceding rank-two exclusion proves (i). On the remaining
$C\subset H$ chart, suppose first that $\lambda|_H\ne0$.
The same exclusion forces $N$ to be invertible.
Propositions~\ref{prop:ideal-incident-self-transport} and
\ref{prop:ideal-incident-invertible-star}, with the preceding
nonzero-root observation, give
\[
\nabla'_VC\subset C,\quad\mathfrak F_ac=0,\quad
\Sym\nabla'\lambda=-2\lambda^2,\quad\nabla'_c\lambda=0.
\]
Use \eqref{r39:eq:common-ideal-row} with $H=\ker\zeta$.
In the pointwise volume frame \eqref{r8:eq:incident-I},
$M(qe_2)=e_2$ gives $Q(e_2)=1$, and
$S(e_1,V)=\langle e_1,e_2\rangle$. Hence $Q|_H\ne0$ and
$S(a,V)+C=H$ for general $a\in H$. No such frame is differentiated.
Full \HscWeakPair{} and $\mathfrak F_aa=\mathfrak F_ac=0$ give
$\theta_a\mathfrak F_a=0$ on the general full eigenbasis,
hence polynomially. Restriction to $H$ and cancellation of $Q$
give $\zeta\mathfrak F_a=0$ for $a\in H$.

For $\beta_p(h)=\zeta(\nabla'_ph)$, complete differentiation gives
$\zeta\mathcal X_p=(\nabla'_p\lambda)\wedge\zeta+\beta_pS$.
Since $\zeta A_a^2=\lambda(a)^2\zeta$ on $H$, Riccati yields
$\beta_aS(a,V)=0$ there, while preservation of $C$ gives $\beta_a(c)=0$.
Thus $\nabla'_HH\subset H$ by linearity in $a$.
The coefficient form \eqref{r8:eq:incident-I} gives internal
torsion with abelian image $I$ of rank two and, for $\lambda(h)=1$,
action eigenvalues $1,-2$. Internal restriction and
Theorem~\ref{unif:thm:three-prime-images} identify its two images
and preserve $I$ on constant-rank charts.
Rank two rules out the nonzero horizontal alternative of
Lemma~\ref{r17:lem:abelian-ideal-horizontal}; scalarization
(Lemma~\ref{r17:lem:ideal-scalarization}) contradicts the two eigenvalues.
Only the restricted tensor equations are used, not a quotient HSC.

On the remaining field $\lambda|_H=0$, the
\hyperref[conv:incident-zero-character]{zero-character boundary argument}
applies with $N_\lambda=H$, retaining incident rank three, separability
and full nonstar \HscWeakPair{}. This proves (ii).
\end{proof}

\section{Shifted-image recovery and the regular cone}\label{stage:fourfold-star}
When the repeated tensor $\mathscr S_T$ vanishes on every zero plane
of a regular rank-two shifted pencil,
Corollary~\ref{cor:full-zero-pays-shifted-recovery} gives the
differentiated identities used to exclude the cone. The split
quadric is treated in Section~\ref{stage:fourfold-quadric}.

\subsection{Full zero planes imply shifted-image recovery}
\label{sec:full-shifted-image-recovery}

\begin{corollary}[Full zero planes imply the recovery hypotheses]
\label{cor:full-zero-pays-shifted-recovery}
On a balanced Hermitian fourfold of constant
nonzero Chern HSC, suppose the
smooth $\rho$ has rank-two shifted image with regular pencil, and
$\mathscr S_T(\sigma)=0$ for every $\sigma\in\Lambda^2 L$ on each plane $L$ with
$S(\Lambda^2L)=0$. On the established dense generic-separability locus,
$V=\ker S_a\oplus\operatorname{im}S$ for general auxiliary $a$.
All four identities \eqref{recovery:full-identities} hold there.
\end{corollary}
\begin{proof}
Since $B_\rho$ has rank zero or three, $\rank S=2$ forces $T\ne0$.
Proposition~\ref{prop:generic-separability-short} supplies the dense locus.
For general $a$, $\im S_a\subset N=\im S$ has dimension at most two.
Thus $\bar S_a=\bar A_a-\rho(a)I$ has a zero
eigenvalue, which is simple by the quotient spectrum.
It follows that $\rank S_a=\rank\bar S_a=2$, so $q_a$ is
injective on $\im S_a=N$. The decomposition into the kernel
and image of $\bar S_a$ lifts to $V=\ker S_a\oplus N$.
By \eqref{sep:plane-zero}, vanishing on the plane $\ker S_a$ gives
$\mathfrak F_a|_{\ker S_a}=0$; the simple-root calculation of
Lemma~\ref{unif:lem:spectral-riccati} then gives
$\Sym\nabla'\rho=-2\rho^2$.
Proposition~\ref{prop:full-shifted-image-recovery} gives the four identities.
\end{proof}

\subsection{The regular-cone normal form and exclusion}
\label{r5:sec:regular-cone}
Let the prescribed balanced Hermitian metric have constant Chern HSC
$\kappa\ne0$. Suppose $S=T-B_\rho$ is smooth of rank two, with a
regular two-form pencil and a quadric-cone kernel section, and
$\mathscr S_T=0$ on each plane of the cone. Only the rank of $S$, not that of $T$ or $\rho T$,
is prescribed here.

The wedge pairing on the pencil orthogonal to $\ker S$ has rank one
because the orthogonal section defines a quadric cone. Choose a smooth
generator $\nu$ of the radical line of the pencil and a nondegenerate
member $\Omega$.
Since $\nu^2=\nu\wedge\Omega=0$, $J=\rad\nu$ is $\Omega$-Lagrangian.
Choose a smooth $\Omega$-Lagrangian complement $H$, an area form
$\omega$ on $J$, and $A:J\to H$ defined by $\Omega(j,Ak)=\omega(j,k)$.
Rescale $\nu$ so that $\nu(Aj,Ak)=\omega(j,k)$.
A common zero plane other than $J$ projects to a line in $H$,
since $\nu$ descends to a nondegenerate area form on $H$. Write that
line as $\C Aj$. The plane meets $J$ in $\C j$: $\Omega$-isotropy
puts the intersection in the kernel of $\omega(\cdot,j)$. Thus these
planes are
\begin{equation}\label{fullcone:C1}
 \langle j,Aj+s\rangle,\qquad 0\ne j\in J,\quad s\in J.
\end{equation}

Put $b=\rho|_J$ and $a=-\rho\circ A$; write
$B_\xi(j,k)=\omega(j,k)d_\xi$, where $\omega(j,d_\xi)=\xi(j)$.
Since $\Omega,\nu$ span the annihilator of $\ker S$, write $S=n\Omega+n'\nu$.
Write $n=x+Ay$, $n'=m_0+Az$ with $x,y,m_0,z\in J$, and set
$\upsilon=\omega(\cdot,x)$, $\chi=a-\upsilon$, $l=2\upsilon+3\chi$.
The two zero traces are $3b(j)+\omega(j,y)=0$ and
$\upsilon(j)-3a(j)+\omega(j,z)=0$ for $T(j,\cdot)$ and $T(Aj,\cdot)$,
respectively. Thus $n=d_\upsilon-3Ad_b$, $n'=m_0+Ad_l$, and
$T=B_\rho+S$ gives
\begin{equation}\label{fullcone:C2}
 \begin{gathered}
 T(j,k)=B_b(j,k),\qquad T(j,Ak)=\upsilon(j)k+\chi(k)j +A\{-2b(j)k+3b(k)j\},\\
 T(Aj,Ak)=\omega(j,k)m_0+A B_{\upsilon+2\chi}(j,k),\qquad m_0\in J.
\end{gathered}
\end{equation}
and, in components,
\begin{equation}\label{fullcone:C3}
 S(j,k)=0,\qquad S(j,Ak)=\omega(j,k)n,
                   \qquad S(Aj,Ak)=\omega(j,k)n'.
\end{equation}
In this paired splitting, the actual connection has the form
\begin{equation}\label{fullcone:C4}
 \nabla'_p j=D_pj+A B_pj,\qquad
 \nabla'_p(Aj)=C_pj+A(D_pj+Q_pj).
\end{equation}
Here $D$ is the connection on $J$, $B,C$ are normal blocks, and $Q$
includes the derivative of $A$.

If pure curvature is zero and $B|_J=0$ as a field, put $M_s=B_{As}$ and
$F(j,s)=A^{-1}\pi_HT(j,As)$. Applying the general identity
\eqref{twosplit:S3} with normal block $AB$ gives
\begin{equation}\label{fullcone:N4}
 (D_jM)_s+Q_jM_s-M_{Q_js}+M_{F(j,s)+M_sj}=0.
\end{equation}
The derivative with the two directions exchanged vanishes because
$B|_J=0$ and $D$ preserves $J$. The torsion
direction has components $\pi_JT(j,As)-C_js$ and $A\{F(j,s)+M_sj\}$.
The former is killed by $B|_J$; transporting the remaining derivative
through the actual $A$ gives $(D_jM)_s+Q_jM_s-M_{Q_js}$,
retaining every output, input and direction slot.

\begin{theorem}[Exclusion of the regular-cone configuration]
\label{r4:thm:regular-cone-terminal}
Work on a metric fourfold system. Suppose an actual smooth
covector $\rho$ gives $S=T-B_\rho$ of rank two with regular two-form
pencil, whose kernel section in the Grassmannian is a quadric cone.
If $\mathscr S_T=0$ on every plane of this cone, the
actual open region is empty. 

\label{thm:rankfour-full-zero-cone}
In particular this excludes $\rank T=4$, $\rho T\ne0$ with the
same rank-two shifted image and whole cone on which $\mathscr S_T=0$; no ideal
line or quotient-simplicity premise is required.
\end{theorem}
\begin{proof}
In the normal form \eqref{fullcone:C1}--\eqref{fullcone:C4},
Corollary~\ref{cor:full-zero-pays-shifted-recovery} gives recovery
on a nonempty open set, with
$\zeta:=\nabla'\rho+2\rho^2\in\langle\Omega,\nu\rangle$.

\label{r4:lem:dual-cone-motion}
Put $M_j=B_{Aj}$, $\tau(j)=\tr M_j$, $Nj=0$, $N(Aj)=j$ and
$K_p=\nabla'_pN$. The component identities give $S(Nx,y)=S(x,Ny)$.
Differentiate $S(v,Nv)=0$ for $v=s+Aj$, $j\ne0$; the two
moving-$v$ terms cancel. Vanishing on $L_v=\langle v,Nv\rangle$ gives
\[
S(v,K_pv)=\zeta(p,v)Nv-\zeta(p,Nv)v=0\qquad(p\in L_v).
\]
Thus $K_pv\in\ker S_v=L_v$ generically, and the wedge identities
extend polynomially. In blocks,
$K_p=\left(\begin{smallmatrix}-B_p&-Q_p\\0&B_p\end{smallmatrix}\right)$.
The normal component of $K_vv\in L_v$ gives $B_s=t(s)I$ and
\begin{equation}\label{fullcone:G0}
M_jj=\lambda(j)j.
\end{equation}
Subtract $(t(s)+\lambda(j))v$. The tangential quadratic term in
$s$ modulo $\C j$ is $-2t(s)s$, so $t=0$; the linear term gives
$Q_sj+M_js+\lambda(j)s\in\C j$. Radiality and trace therefore give
\[
M_s=(\tau(s)-\lambda(s))I+s\otimes(2\lambda-\tau),
\qquad Q_s=\ell(s)I-s\otimes\tau.
\]
This calculation leaves the block $C_p$ arbitrary.

Since the normal shape on $J$ and $S(J,J)$ vanish, the recovered
covector identity \eqref{recovery:full-identities} gives
$D_jb=-2b(j)b$. On $b\ne0$, set $F_j=-2b(j)I+3j\otimes b$. The full normal \ChernCyclic{} row is
\[
0=\{(D_jF)_k-(D_kF)_j+[Q_j,F_k]-[Q_k,F_j]\}l
    +3\tau(l)B_b(j,k)+12b(l)B_b(j,k).
\]
The last terms are respectively the internal derivative output
and normal Jacobiator; other normal terms vanish by $B_J=0$.
Substitution yields
\begin{equation}\label{unif:eq:paired-cyclic-trace}
6(b(l)+\tau(l))B_b(j,k)+3b(l)B_\tau(j,k)=0.
\end{equation}
For independent $j,k$, take $l\in\ker b$ first and then let $l$ vary:
$\tau=tb$ and $2+3t=0$. The full trace of \eqref{fullcone:N4},
with cancelling $\ell(j)\tau(s)$ and $\tau(M_sj)$, is
\begin{equation}\label{unif:eq:paired-flat-trace}
(D_j\tau)(s)+\tau(j)\tau(s)-2b(j)\tau(s)+3b(s)\tau(j)=0.
\end{equation}
For the actual fields $\tau=-2b/3$, $D_jb=-2b(j)b$, it becomes
$10b(j)b(s)/9=0$. Hence $b=0$ as a field.

\label{r4:lem:cone-module-trace}
Where $a=\upsilon+\chi\ne0$, the mixed $\rho$ equation gives
$a(M_sj)=-\zeta_\Omega\omega(j,s)$. At $s=j$ this says
$\lambda(j)a(j)=0$, so $\lambda=0$ and $M_s=\tau(s)I-s\otimes\tau$.
Now $F=0$ in \eqref{fullcone:N4}. Its trace gives
$D_j\tau=-\tau(j)\tau$, and the whole row then gives
$2(\tau(j)s-\tau(s)j)\otimes\tau=0$. Thus $\tau=M=0$,
so the abelian torsion ideal $J$ is preserved in every $(1,0)$ direction.
Since $T(j,Aj)=a(j)j$, its image contains $J$.
Lemma~\ref{r17:lem:abelian-ideal-horizontal} gives
$\mathcal X_j(q,k)=0$ for $j,k\in J$.
As $b=0$ and $\rank S=2$, \eqref{fullcone:C3} gives $\upsilon\ne0$.
Lemma~\ref{r17:lem:ideal-scalarization} bounds the action pencil
by one. But $L_s=T(As,\cdot)|_J=-\chi(s)I-s\otimes\upsilon$ is injective:
if $L_s=0$, testing on $\ker\upsilon$ gives $\chi(s)=0$, then $s=0$.
Thus $a=0$ as a field.
Now $\rho=0$ and $T=S$ has rank two, contradicting the independent
Theorem~\ref{thm:fourfold-lowrank-torsion}.
\end{proof}

\Needspace{10\baselineskip}
\section{Quadric exclusion and fourfold completion}\label{stage:fourfold-quadric}
If the regular shifted pencil has two distinct Pfaffian roots, its zero
planes form a split quadric. After excluding this quadric, we combine
the result with the cone and ideal-line exclusions in
Theorem~\ref{unif:thm:canonical-source-classification}.

\subsection{Vanishing of the plane tensor on a split quadric}
\label{sec:two-extension-uniform-terminal}
\begin{theorem}[A complete split quadric forces zero torsion]
\label{thm:two-extension-split-quadric-terminal}
On a metric fourfold system, suppose a smooth $\rho$ makes
$S=T-B_\rho$ rank two. If a smooth splitting $V=U\oplus W$ into
two-planes has $S(U,W)=0$ and $\mathscr S_T=0$ on every mixed
plane, the region is empty. 
\end{theorem}
\begin{proof}
Put $\beta=\rho|_U$, $\alpha=-\rho|_W$, and choose smooth
nonvanishing area forms $\omega,\nu$ on $U,W$.
Then $S(U,W)=0$ and ambient trace freeness give the complete table
\eqref{twosplit:B1}, with both extensions retained.
The internal $S$ outputs are independent by rank two, so the regular pencil
$\langle\omega,\nu\rangle$ has precisely the mixed zero planes.
Lemma~\ref{unif:lem:mixed-zero-character} gives
$\rho T=-\alpha(z)\omega+\beta(y)\nu=0$ on the full region.
On the dense locus of Corollary~\ref{cor:full-zero-pays-shifted-recovery},
its recovery identities let Theorem~\ref{thm:w4-zero-contraction-without-pencil}
exclude the region.
\end{proof}

\subsection{Exclusion of rank-two shifts with vanishing plane tensors}
\label{sec:rankthree-full-zero-cone}
\begin{theorem}\label{r22:thm:ranktwo-shift-zero-terminal}
Let the given balanced Hermitian fourfold have constant Chern HSC
$c\ne0$. Suppose a smooth covector $\rho$ gives
$S=T-B_\rho$ of constant rank two, and $\mathscr S_T(\sigma)=0$ for every $\sigma\in\Lambda^2 L$ on each plane $L$ with $S(\Lambda^2L)=0$.
Then the actual region is empty. In particular,
\label{thm:rankthree-full-zero-cone}
this excludes $\rank T=3$, $\rho T\ne0$ and regular $\im S^*$.
\end{theorem}
\begin{proof}
Proposition~\ref{research:common-factor-boundary} gives a nonempty
regular-pencil neighborhood.

The pencil has a nonzero binary Pfaffian. On its distinct-root open
set, choose smooth generators $\omega,\nu$ of the two root lines.
They are decomposable and $\omega\wedge\nu\ne0$, so their radicals
$U,W$ are complementary two-planes. Every mixed plane is killed by
$S$, and Theorem~\ref{thm:two-extension-split-quadric-terminal}
excludes this open locus.

With the discriminant zero, the wedge pairing on the pencil has
rank one. Thus its volume orthogonal $\ker S=(\im S^*)^\perp$ has
a one-dimensional radical and its Pl\"ucker quadric has rank three. It is precisely the regular
cone of \eqref{fullcone:C1}, excluded by
Theorem~\ref{r4:thm:regular-cone-terminal}.
\end{proof}

\subsection{Ideal lines and torsion classification}

\begin{theorem}[Ideal lines and torsion classification]
\label{unif:thm:canonical-source-classification}
Work on a metric fourfold system on the dense locus of
Proposition~\ref{prop:common-generic-source-locus} where \textup{(\HscGenericEval{})}, simple
quotient spectrum, right-line injectivity, and Euler dominance hold.
Use the graph $W=\ker E$ of \eqref{newunb:r1v3:graph}.
\begin{enumerate}[label=\textup{(\Roman*)}]
\item No nonempty open set has a smooth torsion ideal line
$C=\mathbb C e_C$ with $T(a,e_C)=\rho(a)e_C$ for which, for general
$[a]\notin\mathbb P(C)$, both complete endpoints on the two nonstar
quotient eigenlines are nonzero.
\item Every constant-degree chart for the canonical factors $f,g$
of Proposition~\ref{prop:polynomial-full-annihilator} is empty.
In particular, the nonzero-torsion locus is empty.
\end{enumerate}
\end{theorem}
\begin{proof}
For \textup{(I)}, Proposition~\ref{prop:entrance-specific-good-pair}
gives birationality and full \HscWeakPair{}. Its inverse selects the ideal root
$\rho(a)$, so the character in Proposition~\ref{prop:entrance-degree-one-exterior}
is $\rho$. Thus $S=T-B_\rho$ has rank three and radical $C$,
and Theorem~\ref{r8:thm:incident-weak-empty} excludes the region.

For \textup{(II)}, Proposition~\ref{prop:polynomial-full-annihilator}
supplies dense smooth charts with $d\le1$.
For every full eigenlift from Lemma~\ref{lem:entrance-full-eigenlift},
including a zero root, \eqref{sep:full-source} becomes
\[
 \mathscr R_a(b)=2\mathfrak F_ab\otimes b
                  -2\mathcal X_b(a,b)\otimes a\in W.
\]
Independent right factors and \eqref{sep:right-line} show that it
vanishes exactly when $\mathfrak F_ab=0$. Thus $f$ selects precisely
the quotient eigenlines for which the full repeated tensor vanishes.
If $d=0$, Lemma~\ref{lem:entrance-nozero-birational} and
Proposition~\ref{prop:entrance-degree-one-exterior} give a smooth ideal line.
If $d=1$, write $f=t-\rho(a)$ and $S=T-B_\rho$. Its selected surface
$\Sigma$ has $\mathscr S_T=0$ by \eqref{sep:plane-zero}, and
$\dim W_\Sigma\in\{3,4\}$ by Lemma~\ref{lem:sep-plane-components}.
If $\dim W_\Sigma=4$ at a point, four selected bivectors stay
independent on the eigenline minor charts. Lemma~\ref{lem:sep-plane-components}
gives $W_\Sigma=\ker S$ and $\rank S=2$ there.
The Pl\"ucker quadric in $\PP(W_\Sigma)$ is nonzero by the isotropic
dimension bound. It contains the irreducible surface $\Sigma$, so
$\Sigma$ is one of its components. A reducible or doubled quadric
would put $\Sigma$ in a plane, contrary to its spanning
$\PP(W_\Sigma)$. Thus the quadric equals $\Sigma$, contrary to
Theorem~\ref{r22:thm:ranktwo-shift-zero-terminal}.
On the remaining degree-one charts, $W_\Sigma=C\wedge V$ by the same lemma and
$C=\operatorname{rad}S$ is smooth: a two-dimensional radical would force
$\rank S\le1$. This ideal line has character $\rho$.

For general $a$, no endpoint vanishes when $d=0$, and only the
endpoint for the ideal root vanishes when $d=1$. Thus in either
case both nonstar endpoints are nonzero, and \textup{(I)}
excludes the chart.
\end{proof}

\subsubsection{Completion of local balanced fourfold rigidity}
\label{ff:sec:global-proof}
The following theorem establishes part~\textup{(\ref{thm:main-fourfold})}
of Theorem~\ref{thm:main}.
\begin{theorem}[Local balanced fourfold rigidity]
\label{ff:thm:main}\label{thm:local-balanced-fourfold}
Let $(X,h)$ be a smooth connected Hermitian fourfold. If
$d\omega_h^3=0$ and $H_h\equiv c\ne0$, then the prescribed metric
$h$ is K\"ahler and locally a complex space-form metric of holomorphic
sectional curvature $c$. 
\end{theorem}
\begin{proof}
Any nonzero-torsion open set meets the smooth canonical charts
on the dense generic locus. Theorem~\ref{unif:thm:canonical-source-classification}
excludes that intersection. Thus $T=0$ everywhere.
K\"ahler polarization in \eqref{p:eq:HSC} gives \eqref{p:eq:Rsf},
so the prescribed metric is locally a complex space form of curvature $c$.
\end{proof}
\part{Curvature-hull criteria}\label{part:hull}
The hull criteria follow from the common torsion identities without
the balanced fourfold classification. The first rank increase depends
on the initial rank. The subsequent growth and propagation arguments
are the same.

\section{Torsion derivatives and curvature-generated subspaces}\label{stage:rankgap-support}
For balanced metrics with nonzero constant Chern HSC, nonzero
torsion forces $\dim H_0\ge2$ (Theorem~\ref{rf:rank-gap}).
At a point where $\dim H_0=2$, one has $\dim H_1\ge4$
(Corollary~\ref{h2:cor:growth}). Without assuming balancedness,
Theorem~\ref{sup:thm:support} rules out a proper positive-rank
curvature-invariant subbundle containing the torsion image.

\subsection{A pointwise rank gap for the full torsion derivative}\label{sec:pointwise-rank-gap}
Work over a finite-dimensional complex vector space $V$ with algebraic
 tensors $T:\Lambda^2V\to V$,
$\mathcal X:V\otimes\Lambda^2V\to V$ and
$P:\Lambda^2V\to\Sym^2V$. Use the common products and operators of
\S\ref{sec:foundations}; for fixed $P$, abbreviate
$F_p(S)=F_p(S,P)$. Assume the complete pointwise identities \ChernCyclic{} and \ChernPure{},
\eqref{eq:cyclic}--\eqref{eq:pure}. We impose no metric, trace or
connection-preservation hypotheses unless stated.
Their full repeated row is \eqref{rf:repeated}; we use the proved
matrix lemma~\ref{rf:matrix} and nilpotent-adjoint theorem~\ref{lem:nilpotent-engel}.

For a bilinear map with canonical symmetric part and
$L\in\operatorname{End}W$, one also has
\begin{equation}\label{eq:anti}
 E(Lx,y)+E(Ly,x)=0\ (x,y\in W)\quad\Longrightarrow\quad L=0.
\end{equation}
The diagonal makes $L$ nilpotent; at $Lx\ne0=L^2x$ the pair $(x,Lx)$ gives $(Lx)^2=0$.

\begin{lemma}[Vanishing when the derivative takes values in the center]\label{h2:lem:central-general}
\label{rf:zero}\label{lem:zero}
Assume complete \ChernCyclic{} and \ChernPure{}. Let a subspace $A$ satisfy
\[
 \im \mathcal X\subset A,\quad T(A,V)=0,\quad \mathcal X_a=0\ (a\in A),
 \quad \mathcal X_p(A,V)=0\ (p\in V).
\]
Then $T=\mathcal X=0$. The case $A=0$ states that $\mathcal X=0$
forces $T=0$. 
\label{rf:central}
In particular this applies to $A=\C e$ and $\mathcal X=e\chi$
when $T(e,V)=0$, $\chi_e=0$ and $\chi_p(e,V)=0$.
\end{lemma}
\begin{proof}
We first show that the induced torsion product on $V/A$ vanishes,
then use centrality to eliminate the remaining $A$-valued terms.
Put $t=T(u,v)$. Lemma~\ref{unif:lem:central-PURE-row} and
$-J_T=\operatorname{cyc}\mathcal X$ give the full row
\begin{equation}\label{h2:eq:centralrow}
 E\bigl(T(t,p)+(\operatorname{cyc}\mathcal X)(p,u,v)-2\mathcal X_p(u,v),a\bigr)=0
 \qquad(a\in A).
\end{equation}
Centrality and \ChernCyclic{} make $V/A$ Lie. A nonzero adjoint eigenvalue
has lifts $t=\lambda v+a_0$, $a_0\in A$, $\widehat v\ne0$.
Set $w=\lambda t+2\mathcal X_u(u,v)=\lambda^2v+b$, $b\in A$.
The rows $p=u,v$ give $E(w,A)=E(\mathcal X_v(u,v),A)=0$;
self-pairing kills $\mathcal X_v(u,v)\in A$. Repeated \ChernPure{} is then
$E(w,v)+3t^2=0$. Since $E(w,b)=0$, $E(w,v)=w^2/\lambda^2$;
its final symmetric quotient is $4\lambda^2\widehat v^2=0$.
Thus Engel's lemma~\ref{lem:nilpotent-engel} makes $V/A$ nilpotent.
If it is nonabelian, choose a nonzero bracket $\widehat t$ in its
last lower-central term. Then $T(u,t),T(v,t)\in A$, and
\eqref{h2:eq:centralrow} makes $T(p,t)+2\mathcal X_p(u,v)$,
$p=u,v$, pair to zero with $A$. Self-pairing kills both, so repeated
\ChernPure{} gives $t^2=0$, impossible. Hence $\im T\subset A$.
Now $J_T=\operatorname{cyc}\mathcal X=0$; the central row and
self-pairing give $\mathcal X=0$, then repeated \ChernPure{} gives $T=0$.
All quotients above are of completed outputs; $P$ is not descended.
\end{proof}

\begin{theorem}[The pointwise output-rank gap]\label{rf:rank-gap}
Suppose \eqref{eq:cyclic}--\eqref{eq:pure} hold and, for all $p,u,v\in V$,
\[
 \operatorname{tr}[w\mapsto\mathcal X_p(u,w)]=0,\qquad
 \operatorname{tr}[w\mapsto\mathcal X_w(u,v)]=0.
\]
If $\dim\operatorname{span}\{\mathcal X_p(u,v):p,u,v\in V\}\le1$, then
$T=\mathcal X=0$. In particular, for the Chern tensors of a balanced
Hermitian metric with constant nonzero Chern holomorphic sectional
curvature, a nonzero $\mathcal X=\nabla'T$ has output rank at least two
at every point.
\end{theorem}
\begin{proof}
The case $\mathcal X=0$ is Lemma~\ref{rf:zero}. Otherwise
$A=\operatorname{im}\mathcal X$ is a line, so $T(A,A)=0$.
Apply the rank-one case of Theorem~\ref{h2:thm:abelian-plane},
proved below using only the two $\mathcal X$ traces.
\end{proof}

This pointwise gap requires neither bundle preservation nor descent of
$P$; the matrix coefficient $3$ comes from complete \ChernPure{}.

\subsection{Two-dimensional derivative-output growth}
\label{h2:sec:plane-growth}
Continue to assume the pointwise identities \ChernCyclic{} and
\ChernPure{}. We use the repeated identity \eqref{eq:repeated}. For a subspace $W\subset V$, let
$q_W:V\to V/W$ be the canonical quotient. For the Chern tensors of a balanced metric,
Proposition~\ref{prop:source} supplies all three traces
\begin{equation}\label{h2:eq:traces}
 \tr_VT(u,\cdot)=0,\qquad
 \tr_V\mathcal X_p(u,\cdot)=0,\qquad
 \sum_i \mathcal X_{e_i}(u,v)^i=0.
\end{equation}
The last trace contracts the direction and output slots in any basis.
Each lemma states the traces it needs.

The repeated row can also be written as
\begin{equation}\label{h2:eq:repeatedE}
 E(s_u,v)-E(s_v,u)+\frac32t^2=0.
\end{equation}
The coefficient $3/4$ in~\eqref{eq:repeated}, or $3/2$ in
\eqref{h2:eq:repeatedE}, includes the positive final term in \ChernPure{}.

\phantomsection\label{conv:full-source-restriction}
If a subspace $W$ is genuinely closed under $T$ and contains all restricted
$\mathcal X$ outputs, restrict all directions and inputs to $W$ and replace $P$ by
$\Sym^2r\circ P|_{\Lambda^2W}$ for any linear retraction $r:V\to W$.
This preserves complete \ChernCyclic{} and \ChernPure{}: every intermediate torsion and
bivector remains in $W$, and only the completed symmetric output is
projected. If all $\mathcal X$ outputs lie in a subspace $A\subset W$, both $\mathcal X$
traces restrict, because every diagonal contribution indexed outside $W$
is zero. The torsion trace restricts by this argument only when all torsion
outputs lie in $W$, or when its internal vanishing is separately established.

The traceless two-plane symmetrizer is Lemma~\ref{lem:lowrank-symmetrizer}.

\subsubsection{Plane traces and the abelian case}
\begin{lemma}[Low-rank derivative-output traces]\label{h2:lem:trace-plane}
Let $1\le\dim A\le2$, $\im \mathcal X\subset A$, and assume the two contracted $\mathcal X$
traces in~\eqref{h2:eq:traces}. Then
\begin{equation}\label{h2:eq:plane-trace}
 \mathcal X_p(A,A)=0,\qquad \mathcal X_a(b,v)=\mathcal X_b(a,v)\quad(a,b\in A).
\end{equation}
Consequently \ChernCyclic{} gives $J_T(A,A,V)=0$.
\end{lemma}
\begin{proof}
The traces of all $S(a,\cdot)$ determine $S:\Lambda^2A\to A$ in
dimensions at most two. Apply this to $\mathcal X_p|_{\Lambda^2A}$
and to $S(a,b)=\mathcal X_a(b,v)-\mathcal X_b(a,v)$. Their traces
are respectively the ordinary derivative trace and minus the sum of
both derivative traces; output in $A$ makes all exterior diagonal
entries zero. \ChernCyclic{} then gives $J_T(A,A,V)=0$.
\end{proof}

\begin{lemma}[Vanishing when the derivative takes values in a central line or plane]\label{h2:lem:central-plane}
Under complete \ChernCyclic{}/\ChernPure{} and both $\mathcal X$ traces, suppose $1\le\dim A\le2$,
$\im \mathcal X\subset A$, and $T(A,V)=0$. Then $T=\mathcal X=0$.
\end{lemma}
\begin{proof}
By Lemma~\ref{h2:lem:trace-plane}, complete \ChernPure{} $(a,b;u,v)$,
$a,b\in A$, reduces to
$E(\mathcal X_a(u,v),b)=E(\mathcal X_b(u,v),a)$:
centrality kills both compositions, and the mixed derivative terms
cancel. The direction-output trace and
Lemma~\ref{lem:lowrank-symmetrizer} give $\mathcal X_A=0$.
Complete \ChernPure{} $(a,p;b,v)$ now reduces to
$E(\mathcal X_p(b,v),a)+E(\mathcal X_p(a,v),b)=0$.
Equation~\eqref{eq:anti} gives $\mathcal X_p(A,V)=0$.
Lemma~\ref{h2:lem:central-general} applies.
\end{proof}

\Needspace{8\baselineskip}
\begin{lemma}[Vanishing with a one- or two-dimensional abelian image subspace]\label{h2:lem:small}
Let pointwise tensors $(T,\mathcal X,P)$ on $W$ satisfy
\eqref{eq:cyclic} and \eqref{eq:pure}. Suppose a subspace $A$ of
dimension one or two satisfies
\[
 T(W,W)\subset A,\quad T(A,A)=0,\quad \im \mathcal X\subset A,
 \quad \mathcal X_p(A,A)=0.
\]
Assume the two derivative traces in \eqref{h2:eq:traces} and
$\tr_W T(q,\cdot)=0$. Then $T(W,A)=0$ and consequently $T=\mathcal X=0$.
\end{lemma}
\begin{proof}
For $a,b\in A$, \ChernCyclic{} makes $C_q(a,b)=\mathcal X_a(q,b)$ symmetric.
Apply \ChernPure{} with inputs $(a,b;q,d)$, where $a,b,d\in A$.
All quadratic terms vanish. The remaining equation
$E(C_q(a,d),b)=E(C_q(b,d),a)$, the direction-output trace and
Lemma~\ref{lem:lowrank-symmetrizer} give $C_q=0$.
Put $B=T(q,\cdot)|_A$, $D=\mathcal X_q(q,\cdot)|_A$.
Repeated \ChernPure{} on $(q,a)$ has $s_a=0$; polarize and retract the
completed output to $\Sym^2A$. Lemma~\ref{ab:lem:repeated} applies
to \eqref{ab:eq:matrixrow} with $Q=2D+B^2$, giving
$D=-2B^2$, $\ker B^2=\ker B$.
Output containment restricts the traces to
$\tr_A B=\tr_A B^2=0$. In dimension at most two, $B^2=0$,
hence $B=0$. Thus $T(W,A)=0$, and
Lemma~\ref{h2:lem:central-plane} gives $T=\mathcal X=0$.
\end{proof}

\begin{theorem}[Torsion-closed derivative output of rank at most two]
\label{h2:thm:abelian-plane}\label{h2:cor:closed-plane}
Assume complete \ChernCyclic{}/\ChernPure{} and both $\mathcal X$ traces in~\eqref{h2:eq:traces}.
Let $A=\im \mathcal X$ have dimension one or two, with $T(A,A)\subset A$.
If $\dim A=2$, assume also $\tr_VT(u,\cdot)=0$.
Then $T=\mathcal X=0$.
\end{theorem}
\begin{proof}
Lemma~\ref{h2:lem:trace-plane} gives
$\mathcal X_p(A,A)=J_T(A,A,V)=0$; restricting to $A$ and using
Lemma~\ref{rf:zero} gives $T(A,A)=0$. Thus $L_a=T(a,\cdot)$
commute and kill $A$. Put $K=\bigcap_a\ker L_a$ and
$K_2=\{w:L_Aw\subset K\}$.
We shall use
\[
 L_Az\subset A,\quad\tr_A(a\mapsto L_az)=0\quad\Longrightarrow\quad z\in K.
\]
Indeed $A+\C z$ is torsion-closed with output in $A$ and zero
internal torsion trace; both derivative traces restrict by
\ref{conv:full-source-restriction}, so Lemma~\ref{h2:lem:small} applies.
For $k,l\in K$, \ChernCyclic{} gives
\begin{equation}\label{h2:eq:Ktrace}
 L_aT(k,l)=\mathcal X_a(k,l)-\mathcal X_k(a,l)+\mathcal X_l(a,k).
\end{equation}
Its $A$-valued right side is traceless in $a$, proving $T(K,K)\subset K$.
Since $A$ is central in $K$, Lemma~\ref{h2:lem:central-plane} gives
\begin{equation}\label{h2:eq:Kabelian}\equationalias{rf:kernel-abelian}
 T(K,K)=0,\qquad\mathcal X_K(K,K)=0.
\end{equation}
For $w\in K_2$ and $k\in K$, \eqref{h2:eq:Ktrace} holds with
$(w,k)$ in place of $(k,l)$. The trace criterion gives $T(w,K)\subset K$.
In complete \ChernPure{} $(a,b;w,k)$, $a,b\in A$, both intermediate
bivectors lie in $\Lambda^2K$, where the subsequent \ChernFirst{} maps
vanish. Its remaining traceless symmetrizer gives
$\mathcal X_A(w,K)=0$. With $a\in A$, $b,k\in K$ the row is
$E(\mathcal X_b(w,a),k)+E(\mathcal X_b(w,k),a)=0$.
For $U(k)=\mathcal X_b(w,k)$, \eqref{eq:anti} gives $U|_A=0$;
then $E(U(k),A)=0$ and self-pairing give
\begin{equation}\label{rf:horizontal-zero}
 \mathcal X_K(w,K)=0.
\end{equation}
On $K$ put $B_w=T(w,\cdot)$, $D_w=\mathcal X_w(w,\cdot)$,
$C_w=B_w|_A$. Repeated \ChernPure{} has $s_k=0$; polarize and retract
only its final output to $\Sym^2K$.
Lemma~\ref{ab:lem:repeated}, with $Q=2D_w+B_w^2$, gives
\begin{align}
 D_w=-2B_w^2,\qquad\ker B_w^2&=\ker B_w,\label{h2:eq:Bw}\\
 \im B_w=\im B_w^2&\subset A,\label{h2:eq:lowering}\\
 \tr_A C_w^2&=0,\label{h2:eq:Ctrace}
\end{align}
the last identity from the derivative trace.

There is no nonzero common eigencharacter of the $L_a$.
For $A=\C e$, the traces give
\begin{equation}\label{rf:trace-annihilation}
 \mathcal X_p(e,V)=0,\qquad\mathcal X_e=0;
\end{equation}
an eigenpair $T(e,v)=\lambda v$, $\lambda\ne0$, contradicts
repeated \ChernPure{}, which gives $\lambda^2v\odot v=0$.
For $\dim A=2$, normalize $T(e,v)=v$, $T(f,v)=0$.
On $A+\C v$, set $C(a,b)=\mathcal X_a(b,v)$ and $Ma=\mathcal X_v(a,v)$.
The derivative traces restrict; $C$ is symmetric and its maps,
like $M$, are traceless. \ChernPure{} $(e,f;f,v)$ and the symmetrizer
lemma give $C(A,f)=0$, hence $C(e,e)=\beta f$.
\ChernPure{} $(e,f;e,v)$ is $\beta f^2/2+E(v,f)/4=0$:
the two compositions are $-E(v,f)/4,-E(v,f)/2$, the latter subtracted.
Repeated \ChernPure{} for $a=xe+yf$ therefore gives
\begin{equation}\label{h2:eq:char1}\equationalias{h2:eq:char2}
 E(Ma,a)=x^2R,\qquad R=2v^2+2\beta fv+2\beta^2f^2.
\end{equation}
Retract the completed output to $\Sym^2A$ along $R$;
$\C R\cap\Sym^2A=0$ preserves the canonical symmetric part.
By \eqref{eq:anti}, $M=0$, contradicting $R\ne0$.

Hence each $L_a$ is nilpotent: the commuting family is simultaneously
triangularizable, and its diagonal entries are the common characters
just excluded. In a common triangular basis the maps strictly lower
the induced flag. Thus the filtration
$K_1=K$, $K_{j+1}=\{v:L_Av\subset K_j\}$ exhausts $V$.
A stationary step would make every later step equal to it. For $w\in K_2\setminus K$,
\eqref{h2:eq:Bw}--\eqref{h2:eq:lowering} make $C_w$ nonzero and
invertible on $\im B_w$. The trace of $C_w^2$ excludes a line image;
thus $\dim A=2$, $C_w$ is invertible, and
$(\tr C_w)^2=2\det C_w\ne0$.
The linear functional $w\mapsto\tr C_w$ on $K_2$ therefore has
kernel $K$. If $K\ne V$, then $K_2=K\oplus\C w$.
For $v\in K_3$, write $L_av=\alpha(a)w+k_a$.
Commutativity gives $\alpha(a)C_wb=\alpha(b)C_wa$;
invertibility forces $\alpha=0$. Thus $K_3=K_2=V$.
Now $B_w(K)\subset A$ and $T(w,w)=0$ make ambient balance read
$\tr C_w=\tr_VT(w,\cdot)=0$, a contradiction.
Hence $K=V$, so $T(V,A)=0$. Lemma~\ref{h2:lem:central-plane}
gives $T=\mathcal X=0$.
\end{proof}

\subsubsection{Full cuts and the two-plane growth gap}
Let $\im\mathcal X\subset A\subset C$ and
$(\Sym^2q_C)P(A,V)=0$. Set
$E_C=(\Sym^2q_C)\circ E$, retaining inputs in $V$. No tensor identities on $V/C$ are assumed. Since $E_C(A,V)=E_C(V,A)=0$,
the complete repeated row gives
\begin{equation}\label{h2:eq:projected-repeated}
 E_C(T(a,T(a,v)),v)+3(q_CT(a,v))^2=0
 \qquad(a\in A,\ v\in V).
\end{equation}
Indeed, all derivative outputs lie in $A$, and the term with second
input $a$ also vanishes. Taking $v=b\in A$ gives $T(A,A)\subset C$.
The coefficient of $z$ at $v=b+zx$ is then
\begin{equation}\label{h2:eq:projected-cross}
 E_C(T(a,T(a,b)),x)=0\qquad(a,b\in A,\ x\in V).
\end{equation}
Putting $x=T(a,T(a,b))$ shows $T(a,T(a,b))\in C$.
The canonical part of \eqref{h2:eq:projected-cross} then vanishes,
so $(\Sym^2q_C)P(T(a,T(a,b)),V)=0$.

\begin{lemma}[Pointwise containment from a full double cut]
\label{lem:containment}
Let $C\subset V$ contain $\im \mathcal X$ and suppose
\begin{equation}\label{h2:eq:fullcut}
 (\Sym^2q_C)P(C,V)=0.
\end{equation}
Complete \ChernCyclic{} and \ChernPure{} imply $\im T\subset C$. 
\end{lemma}
\begin{proof}
The cut descends $P$ to $\widehat P$ on $V/C$.
Equation~\eqref{h2:eq:projected-repeated} first gives $T(C,C)\subset C$.
For $u\in C$ it then reads
$\widehat E(D^2x,x)+3(Dx)^2=0$ with $Dq_Cv=q_CT(u,v)$ and
$\widehat E(x,y)=xy+\widehat P(x,y)$. Polarize and apply
Lemma~\ref{rf:matrix} with $(B,Q)=(D,D^2)$: $4D^2=0$ and
$\ker D^2=\ker D$, so $T(C,V)\subset C$.
The quotient torsion $\widehat T$ is Lie by \ChernCyclic{}; derivative terms
in projected \ChernPure{} vanish because every $\mathcal X$ output lies in $C$.
Moreover
\[
 (\Lambda^2q_C)b_p(T)=b_{q_Cp}(\widehat T)\Lambda^2q_C,
 \quad
 (\Sym^2q_C)F_p(T)=F_{q_Cp}(\widehat T,\widehat P)\Lambda^2q_C.
\]
Hence both compositions and the positive torsion term descend, so
Lemma~\ref{rf:zero} gives $\widehat T=0$.
\end{proof}

\begin{theorem}[Weak-cut obstruction in a three-dimensional support]
\label{h2:thm:weakgap}
Assume complete \ChernCyclic{}/\ChernPure{} and all traces~\eqref{h2:eq:traces}. Suppose
$A=\im \mathcal X$ has dimension two. There is no subspace $C\supset A$ with
$\dim C\leq3$ and
\begin{equation}\label{h2:eq:weakcut}
 (\Sym^2q_C)P(A,V)=0.
\end{equation}
\end{theorem}
\begin{proof}
By \eqref{h2:eq:projected-repeated} and
Theorem~\ref{h2:cor:closed-plane}, for a basis $e,f$ of $A$,
$t=T(e,f)\notin A$ and $C=A\oplus\C t$.
The two rows \eqref{h2:eq:projected-cross}, with $(a,b)=(e,f),(f,e)$,
give
\begin{equation}\label{h2:eq:Cclosed}
 T(A,t)\subset C,\qquad T(C,C)\subset C,
\end{equation}
\begin{equation}\label{h2:eq:newcut}
 (\Sym^2q_C)P(T(A,t),V)=0.
\end{equation}
If $T(A,t)\not\subset A$, then $A+T(A,t)=C$; these cuts and
\eqref{h2:eq:weakcut} give the full $C$ cut. Lemma~\ref{lem:containment}
puts $\im T\subset C$, so balance at $a\in A$ equals the $t$
coefficient of $T(a,t)$, forcing $T(A,t)\subset A$, a contradiction.
Set $L=T(t,\cdot)|_A$. Lemma~\ref{h2:lem:trace-plane} gives
$J_T(e,f,V)=0$, so $T(t,\cdot)|_C$ is the commutator of the
$e,f$ actions; hence $\tr_A L=0$.
The complete repeated row, using \eqref{h2:eq:plane-trace}, is
$-E(Le,f)+E(Lf,e)+3t^2=0$.
The alternating terms sum to $-(\tr_A L)P(e,f)=0$, so projection
to $\Sym^2(V/A)$ gives $3(q_At)^2=0$.
\end{proof}

\Needspace{6\baselineskip}
\subsection{Curvature-induced vanishing and balanced hull growth}
\label{sec:support-reduction}
For the pointwise hulls in \eqref{eq:hull}, the curvature formula for $P$
\eqref{eq:P} gives
\begin{equation}\label{h2:eq:physicalcut}
                (\Sym^2q_{H_1})P(H_0,V)=0.
\end{equation}
Indeed, put the curvature action input in $H_0$ and project both
raised output slots modulo $H_1$. The curvature outputs lie in
$H_1$, and one factor of the canonical product lies in $H_0\subset H_1$.
Since $P$ is alternating, the same vanishing holds with the inputs exchanged. This cut needs neither balance nor
any smoothness or preservation of the pointwise hulls.

\begin{corollary}[Pointwise growth gap]\label{h2:cor:growth}
For a balanced metric with constant nonzero Chern HSC,
\[
          \dim H_0=2\quad\Longrightarrow\quad\dim H_1\ge4.
\]
\end{corollary}
\begin{proof}
If $\dim H_1\le3$, apply Theorem~\ref{h2:thm:weakgap} with
$A=H_0$, $C=H_1$ and \eqref{h2:eq:physicalcut}.
\end{proof}

\subsection{Proper subbundles invariant under Chern curvature}
\label{sup:sec:terminal}
The geometric arguments use the prescribed metric and
Proposition~\ref{prop:source}. In the abstract results we impose only the
stated hypotheses and pass to a quotient only with tensors and connections
that descend. When differentiating a smooth subbundle, we keep all
moving slots and the torsion-direction term before projecting the output.

\begin{lemma}[Image, kernel, and normal-shape transport]
\label{r37:lem:image-shape}
Let $\mathcal X=\nabla'T$ have output in a smooth subbundle $A$.
Here $q_S:V\to V/S$ denotes the quotient map.
On a constant-rank chart of $T$, put $I=\im T$.
Then $\nabla'_VI\subset I+A$; in particular $A\subset I$ makes
$I$ prime preserved. If $I\cap A=0$, there are smooth maps
$\xi_p:I\to A$ such that
\[
 \mathcal X_p=\xi_p\circ T,
 \qquad q_I\nabla'_p i=q_I\xi_p(i)\quad(i\in I).
\]
In particular, if $I\subset\operatorname{rad}T$ and
$A\cap\operatorname{rad}T=0$, then $\mathcal X=0$.
Independently of the rank of $T$, under the full Hessian identity
\eqref{newff:eq:X-pure-Hessian} the output-shape law
\eqref{r48:eq:output-shape} applies with $B_p=\beta_p=q_A\nabla'_p|_A$.
The image statements also hold for $T|_{\Lambda^2D}$ when $D$ is
prime preserved and the restricted derivative has output in $A$.
Under complete \ChernCyclic{}, if $D$ is a prime-preserved torsion ideal
and its internal image $I_D=\im(T|_{\Lambda^2D})$ is smooth with
$\im\mathcal X\subset I_D$, then $I_D$ is an ambient-prime-preserved
torsion ideal.
By \eqref{eq:barred} for the prescribed metric with $c\ne0$,
every smooth $C\supset I$ satisfies
\begin{equation}\label{r51:eq:barred-image-factor}
 (\Sym^2q_C)P=\mathscr B_C T,\qquad
 \mathscr B_C(i)=-\frac1{4c}\sum_jq_C(e_j)\odot q_C\nabla_{\bar e_j}i,
 \quad i\in I,
\end{equation}
where $(e_j)$ is unitary and $q_C\nabla_{\bar e_j}i$ is the class modulo
$C$ of the barred derivative of a local section of $I$.
Consequently $E(I,V)\subset\Sym^2I$ and $T(I,V)=I$ on a dense set
imply $(\Sym^2q_I)P=0$ everywhere.
Independently of the Hessian assumption, let $D$ be a smooth
hyperplane with $T(D,D)=0$, $\im T\subset D$, and let $J\subset D$
contain $\mathcal X_p(D,D)$ for every $p$. If
$q_JT(n,\cdot)|_D$ has rank at least two for $n\notin D$, then
$D$ is prime preserved.
\end{lemma}
\begin{proof}
For $C\supset\im T$ and any direction $Z$, the full product rule is
\[
 q_C(\nabla_ZT)=(q_C\nabla_Z|_C)T,
\]
since both moving-input torsions lie in $C$. With $C=I$ this gives
image transport and, if $I\cap A=0$, the factorization
$\mathcal X_p=\xi_pT$. Under that transversality and $A\cap\rad T=0$, if $I$ is central, differentiation of
$T(i,v)=0$ gives $T(\nabla'_pi,v)=0$; subtracting its $I$ component
puts $\xi_p(i)$ in $A\cap\rad T=0$.
The same restriction proves the image rules on preserved $D$.
For its internal image $I_D$, the same rule gives preservation;
\ChernCyclic{} $(p,u,v)$, $u,v\in D$, gives ideality, since all derivative
terms and the two other nested products lie in $I_D$.
The barred product rule and \eqref{eq:barred} give
\eqref{r51:eq:barred-image-factor}. The ordered vanishing condition then gives zero for
$\mathscr B_IT(I,V)$; dense spanning and smoothness give
$\mathscr B_I=0$.
Finally, write $q_D\nabla'_pu=\xi_p(u)q_Dn$ and
$B=T(n,\cdot)|_D$. Differentiating $T(D,D)=0$ gives
\[
 \mathcal X_p(u,v)=-\xi_p(u)Bv+\xi_p(v)Bu.
\]
Modulo $J$, a nonzero $\xi_p$ would force $\rank(q_JB)\le1$.
Thus the stipulated rank bound preserves $D$.
\end{proof}

In this subsection $h$ is any smooth Hermitian metric with
$H_h\equiv c\ne0$; balancedness is not assumed. We use the complete
system of Proposition~\ref{prop:source}, its full repeated row
\eqref{eq:repeated}, and the already proved algebraic lemmas
\ref{rf:matrix} and~\ref{sharedE:theorem}. Write
$E(u,v)=uv+P(u,v)$ for $P$ defined by \eqref{eq:P}. Then
\begin{equation}\label{sup:eq:canonical}
                   E(u,v)+E(v,u)=2uv,
\end{equation}
and \ChernFirst{} has the form \eqref{newff:eq:E-form-F}.
Metric compatibility gives curvature adjunction, and the mixed torsion
Bianchi identity of Proposition~\ref{prop:source} gives the torsion formula:
\begin{equation}\label{sup:eq:bianchi}
 R(p,\bar q)^*=R(q,\bar p),\qquad
 Y_{\bar q}(u,v)=R(v,\bar q)u-R(u,\bar q)v.
\end{equation}
All projections below are those of the prescribed metric.

\begin{theorem}[No proper invariant subbundle containing torsion]\label{sup:thm:support}
No nonempty open set admits a smooth complex subbundle
$A\subset V$ of rank $1\le\rank A<n$ satisfying
\begin{equation}\label{sup:eq:support}
                   \im T\subset A,\qquad R(V,\bar V)A\subset A.
\end{equation}
\end{theorem}

\subsubsection{Projected curvature and transverse torsion}
Assume \eqref{sup:eq:support}, set $B=A^\perp$, and write $\pi_A,\pi_B$
for the projections. Curvature adjunction and \eqref{eq:P} give
\begin{equation}\label{sup:eq:cuts}
 P_B(A,V)=0,\qquad P_A(B,V)=0,\qquad
 P_A=(\Sym^2\pi_A)P,\quad P_B=(\Sym^2\pi_B)P .
\end{equation}
The $AA$ projection of repeated \ChernPure{} on $B^2$ leaves only the
torsion square, so
\begin{equation}\label{sup:eq:TBB}
 T(B,B)=0.
\end{equation}
Differentiating $T(B,B)=0$ puts $Y_{\bar q}(B,B)$ in $A$, since both
moving-input terms take values there. Equation~\eqref{sup:eq:bianchi}
puts it in $B$.
Thus \ChernBarred{} gives
\begin{equation}\label{sup:eq:YBBPBB}
 Y_{\bar q}(B,B)=0\quad(q\in V),\qquad P(B,B)=0.
\end{equation}
We next control the family of torsion actions in directions transverse
to $A$. For $v\in B$, put $D_va=T(v,a)$, $C_va=\pi_A\mathcal X_v(v,a)$
and $E_A=(\Sym^2\pi_A)E$. The same $AA$ projection on $(v,a)$,
polarized in $a$, is \eqref{ab:eq:matrixrow} for
$(E,B,Q)=(E_A,D_v,2C_v+D_v^2)$. Lemma~\ref{rf:matrix} gives
\begin{equation}\label{sup:eq:matrix}
 C_v=-2D_v^2,\quad \ker D_v^2=\ker D_v,\quad
 E_A(D_v^2a,b)+E_A(D_v^2b,a)=2(D_va)(D_vb).
\end{equation}
Since $E_A$ is independent of $v$, Lemma~\ref{sharedE:theorem} gives
\begin{equation}\label{sup:eq:pencilrank}
 \rank(D:B\longrightarrow\End A)\le1 .
\end{equation}
Write $E_B=(\Sym^2\pi_B)E$, $\Psi_{\bar q}a=\pi_B\nabla_{\bar q}a$.
The normal product rule on $T(v,a)$ and Bianchi yield
\begin{equation}\label{sup:eq:normalBianchi}
 \Psi_{\bar q}D_va=R_B(a,\bar q)v .
\end{equation}

\subsubsection{The transverse torsion kernel}
\begin{proposition}\label{sup:prop:kernel}
The bundle map $D:B\to\End A$ cannot have a nonzero smooth kernel
bundle on any open region satisfying \eqref{sup:eq:support}.
\end{proposition}
\begin{proof}
Let $0\ne K=\ker D\subset B$. Equation~\eqref{sup:eq:TBB} gives
$T(K,V)=0$; \eqref{sup:eq:normalBianchi} gives $R_B(A,\bar V)K=0$.
The curvature formula for $P$ and \eqref{sup:eq:cuts} therefore give
\begin{equation}\label{sup:eq:EK}
 E(A,K)=0,\qquad E(K,A)=2KA.
\end{equation}
Apply \ChernFirst{} to local sections $a$ of $A$ and $k$ of $K$,
with $k\ne0$:
\begin{equation}\label{sup:eq:kernelmove}
 E(\nabla_pa,k)+E(a,\nabla_pk)=0.
\end{equation}
Its $BB$ output is $(\pi_B\nabla_pa)k=0$, using
$P(B,B)=E_B(A,V)=0$; symmetric multiplication is injective, so
\begin{equation}\label{sup:eq:primeA}
 \pi_B\nabla_pA=0\quad(p\in V).
\end{equation}
The first term of \eqref{sup:eq:kernelmove} vanishes. Its $AA$
output, with $a=\pi_A\nabla_pk$, gives
\begin{equation}\label{sup:eq:primeK}
 \pi_A\nabla_pK=0\quad(p\in V).
\end{equation}
For $w\in B$, \ChernFirst{} gives $(\nabla_pP)(k,w)=E(T(p,w),k)/2=0$.
Differentiating $P(k,w)=0$ leaves only
$P(k,\pi_A\nabla_pw)=k(\pi_A\nabla_pw)=0$.
Thus $B$ is also prime preserved; metricity makes both $A$ and $B$ parallel.
Normal Bianchi now gives $R_B(A,\bar V)B=0$, hence
\begin{equation}\label{sup:eq:EABzero}
 E(A,B)=0.
\end{equation}
For $a,b\in A$ and $v\in B$, differentiate $E(a,v)=0$ along $b$. The input derivatives stay in $A,B$,
so the displayed identity makes their contributions zero;
\ChernFirst{} gives $E_A(D_va,b)+E_A(D_vb,a)=0$. By \eqref{eq:anti}, $D=0$.
Thus $T(B,V)=0$, $Y_{\bar q}(B,A)=0$, and Bianchi gives
$R_A(B,\bar V)A=0$.
These block vanishings and curvature invariance make the left side of
\begin{equation}\label{sup:eq:mixedHSC}
 h(R(a,\bar a)v,v)+h(R(v,\bar v)a,a)
 +h(R(a,\bar v)v,a)+h(R(v,\bar a)a,v)=2c|a|^2|v|^2\ne0
\end{equation}
zero for nonzero $a\in A,v\in B$, a contradiction.
The argument from \eqref{sup:eq:EABzero} uses only a parallel splitting
and the support hypotheses, so also applies in the hyperplane case.
\end{proof}

If $\rank B\ge2$, \eqref{sup:eq:pencilrank} gives a positive-rank kernel
on a constant-rank chart, contradicting Proposition~\ref{sup:prop:kernel}.
\subsubsection{The invariant-hyperplane case}\begin{proof}[Completion of Theorem~\ref{sup:thm:support}]
The remaining case is $B=A^\perp=\C e$, with $|e|=1$.
The projected curvature identities and adjunction give
\begin{equation}\label{sup:eq:U}
 \begin{gathered}
 E(e,a)=e(Ua),\quad E(a,e)=e(2I-U)a,\quad U^*=U,\\
 R_A(e,\bar e)=cU,\quad R_B(e,\bar e)e=ce,\quad
 R_B(a,\bar b)e=c\,h((2I-U)a,b)e.
 \end{gathered}
\end{equation}
Three-$B$ polarization gives $R_B(A,\bar e)=0$.
For $D=T(e,\cdot)|_A$ and $w=\pi_A\nabla'_ee$,
normal Bianchi and metricity give
\begin{equation}\label{sup:eq:wkernel}
 h(Da,w)=0,\qquad U|_{\ker D}=2I.
\end{equation}
The $AA$ output of complete \ChernFirst{} on $E(e,a)=e(Ua)$ is
\begin{equation}\label{sup:eq:wEA}
 E_A(w,a)=w(Ua).
\end{equation}
If $w\ne0$, set $a=w$. The left side is $w^2$, so injectivity of
multiplication by $w$ gives $Uw=w$. By \eqref{sup:eq:matrix},
$A=\ker D\oplus\im D$; Hermiticity and \eqref{sup:eq:wkernel}
make $w$ orthogonal to both summands, impossible. Thus
\begin{equation}\label{sup:eq:wzero}
 \pi_A\nabla'_ee=0.
\end{equation}
Write $\nabla=\mathcal D+\Theta$ for the projected block connection.
Set
$G_p=\pi_A\nabla'_pe$, $\Phi_pa=\pi_B\nabla'_pa$, $G(a)=G_a$;
then $G_e=0$. For the block-diagonal curvature form $\Omega$, the second Bianchi
identity $d^\nabla\Omega=0$ has off-diagonal part $[\Theta\wedge\Omega]=0$.
Its $A\leftarrow B$ block on $(e,a,\bar e)$ gives
\begin{equation}\label{sup:eq:UG}
 UG=G.
\end{equation}
Normal Bianchi gives
$-h(Da,Gb)=c\,h((2I-U)a,b)$, hence
\begin{equation}\label{sup:eq:Gadjoint}
 G^*D=-c(2I-U).
\end{equation}
The mixed \ChernFirst{} output on $E(e,a)=e(Ua)$ along $e$ is the first
identity below. Along arbitrary $p$, the $BB$ derivative of its
right side is $e\Phi_p(Ua)$; the moving inputs on the left give
$e\Phi_pa$ and the torsion terms have zero $BB$ projection. Thus
\begin{equation}\label{sup:eq:firstU}
 \mathcal D'_eU=(2I-U)D,\qquad \Phi_p(U-I)=0\quad(p\in V).
\end{equation}
Consequently $\rank(2I-U)\le\rank G\le\dim\ker(U-I)$.
Since $U$ is Hermitian, the rank counts all eigenspaces except at
$2$, while the upper bound counts only the eigenspace at $1$.
Thus $U$ has no other eigenvalues and $\Pi=2I-U$ is an orthogonal
projection. Differentiate $\Pi^2=\Pi$
using $\mathcal D'_e\Pi=-\Pi D$ and multiply by $\Pi$ on both sides:
\begin{equation}\label{sup:eq:Griccati}
 \Pi D\Pi=0.
\end{equation}
Since $G^*=G^*\Pi$, one has $-c\Pi=G^*D\Pi=0$.
Thus $U=2I$, and $UG=G$ and the $BB$ identity for all $p$ give
$G=\Phi=0$. Metricity makes $A\oplus B$ parallel.
The argument \eqref{sup:eq:EABzero}--\eqref{sup:eq:mixedHSC}
now contradicts HSC. Together with Proposition~\ref{sup:prop:kernel},
this proves the theorem for either sign of $c$.
\end{proof}

\section{Torsion output without balancedness}\label{newunb:section}
Without the balanced trace identities, we instead use restrictions
on the torsion image. The two metric theorems apply from dimension three onward: one treats torsion
contained in a prime-preserved abelian subbundle, the other torsion of rank at most one (Theorems~\ref{newunb:thm:bare-output} and~\ref{newunb:r1v3:rankone}).
For the metric applications in this section, $h$ has constant Chern HSC
$c\ne0$ and $n\ge4$, except where a lower dimension is stated.
The tensor lemmas use only their stated hypotheses. Write
$\tau(u)=\tr T(u,\cdot)=-\eta(u)$ and $F_p(S)=F_p(S,P)$.
We use scalarization in
Lemma~\ref{r17:lem:ideal-scalarization} and the metric vectorial exclusion in Theorem~\ref{u3c:vectorial}.

\subsection{Tensor identities for preserved abelian ideals}

\begin{lemma}[The horizontal alternative for an abelian ideal]
\label{r17:lem:abelian-ideal-horizontal}
Let a smooth tensor system with a connection in the sense of
\S\ref{sec:foundations} satisfy complete \ChernFirst{}, \ChernCyclic{} and \ChernPure{},
with $\mathcal X=\nabla'T$, $P:\Lambda^2V\to\Sym^2V$ and
$E(u,v)=uv+P(u,v)$.
Suppose $N$ is a smooth prime-preserved torsion ideal and $T(N,N)=0$.
Put $H_q(a,b)=\mathcal X_a(q,b)$ for $a,b\in N$.
On any open region where $H$ is nonzero, locally there is a smooth
prime-preserved line $J\subset N$ such that
\[
                  E(J,V)\subset J^2,\qquad\im T\subset J.
\]

If $N$ is a hyperplane, then $\im T=\im\mathcal X$ on charts
where both images have constant rank.
\end{lemma}
\begin{proof}
Differentiating the ideal and using \ChernCyclic{} and complete \ChernPure{} gives the
symmetric $H_q$ and the row
\begin{equation}\label{newff:eq:transverse-cross}
 E(H_q(a,d),b)=E(H_q(b,d),a).
\end{equation}
We use the calculation in Theorem~\ref{unif:thm:balanced-abelian-source}
up to, but not including, the use of the balanced traces. On a smaller
smooth chart it gives
\[
 N=\C v\oplus K,\quad K=\ker\lambda,\quad\lambda(v)=1,
 \quad H_q(a,b)=\kappa(q)\lambda(a)\lambda(b)v,
 \quad E(v,a)=\lambda(a)v^2,
\]
with $\kappa\ne0$, $\kappa|_N=0$.
Put $B_q=T(q,\cdot)|_N$. For $k'\in K$, $w=\nabla'_qv$,
$z=\nabla'_qk'$, full \ChernFirst{} on $E(v,k')=0$ is
\[
 E(w-B_qv/2,k')+v(B_qk')+
 \{\lambda(z)-\lambda(B_qk')/2\}v^2=0.
\]
Project to $\Sym^2K$ and take $k'$ to be the $K$ component of $w-B_qv/2$. This gives $w-B_qv/2\in\C v$; the unprojected identity gives
$B_qk'=b_q(k')v$ and $(\nabla'_q\lambda)(k')=b_q(k')/2$.
The polarized repeated \ChernPure{} equation, with $Q_q=2\mathcal X_q(q,\cdot)|_N+B_q^2$, is
\[
 E(Q_qa,b)+E(Q_qb,a)+3B_qa\odot B_qb=0\quad(a,b\in K).
\]
Projecting to $\Sym^2K$ and using \eqref{eq:anti} gives $Q_qK\subset\C v$;
the same equation then gives $b_q=0$. Hence
\begin{equation}\label{r17:eq:central-kernel}
 T(V,K)=0,\quad\nabla'_VK\subset K,\quad
 \nabla'_qv-B_qv/2\in\C v.
\end{equation}
Let $\mu=\tr_NB$, $w_q=T(q,v)$, $D_q=\mathcal X_q(q,v)$.
Internal traces give $\mu(q)=\lambda(w_q)$,
$(\nabla'_a\mu)(q)=\lambda(a)\kappa(q)$.
Centrality, preservation and \ChernCyclic{} give
$\mathcal X_p(K,V)=\mathcal X_K=0$.
Since $B_qw_q=\mu(q)w_q$, the central and repeated \ChernPure{} rows are
\begin{equation}\label{r17:eq:whole-central-row}
 E(2D_q+\mu(q)w_q,k')=0,\qquad
 \kappa(q)E(v,q)=E(D_q+\mu(q)w_q/2,v)+\tfrac32w_q^2.
\end{equation}
For $q\in\ker\kappa$, projecting the second identity to $\Sym^2K$ gives
$w_q=\mu(q)v$; hence $w_q=\mu(q)v+\kappa(q)k$, $k\in K$.
Self-pairing in the first row kills its $K$ component, so the second becomes
\begin{equation}\label{r17:eq:whole-horizontal-cut}
 E(v,q)=a(q)v^2+3\mu(q)vk+\tfrac32\kappa(q)k^2,
 \quad a(v)=1,\quad a(K)=0.
\end{equation}
Put $\nabla'_vv=lv$, $j=\nabla'_vk$, $m=\nabla'_v\kappa$.
Full differentiation of $w_q=\mu(q)v+\kappa(q)k$ gives
$(m-l\kappa)k+\kappa j=0$. Using it in the full derivative of
\eqref{r17:eq:whole-horizontal-cut}, its $K^2$ part is
$3\kappa(q)kj/2=0$ by \ChernFirst{}. On $k\ne0$, this gives $j=0$;
the $vK$ parts are then $3\kappa(q)vk$ and $-2\kappa(q)vk$,
since $\nabla'_v\mu=\kappa$. Thus $k=0$.
Since $k=0$, the central-kernel identity preserves $J=\C v$ and
the cut gives $E(J,V)\subset J^2$;
Proposition~\ref{prop:conceptual-fullcut} gives $\im T\subset J$.
For a hyperplane $N$, image differentiation and Lemma~\ref{rf:zero}
give $\im T=\im\mathcal X=J$ on $H\ne0$.
On the interior $H=0$, Lemma~\ref{r17:lem:ideal-scalarization}
gives $\mathcal X=-2\beta\otimes T$, $\beta\ne0$ wherever $T\ne0$.
At constant ranks, the two images are smooth subbundles. Their
equality on the dense open strata extends by continuity; on an open
set where the torsion vanishes, its derivative also vanishes.
\end{proof}

\begin{corollary}[Equality of torsion and derivative images]
\label{r48:cor:abelian-output-alignment}
Let an actual tensor system satisfy \ChernFirst{}, \ChernCyclic{}
and \ChernPure{} with a pure-flat partial connection. Suppose it has
a smooth prime-preserved abelian subbundle
$N$ containing $\im T$. Then $\im T=\im\mathcal X$ on charts
where both images have constant rank, in any dimension or codimension.
Put $H_q(a,b)=\mathcal X_a(q,b)$ for $a,b\in N$.
On a field $H=0$, every nonzero constant-rank torsion chart has
$r=\rank T$, $\alpha=\tau/r\ne0$, and
\[
 \begin{gathered}
 \nabla'\alpha=-2\alpha^2,\qquad
 \mathcal X=-2\alpha\otimes T,\\
 A_0=\ker\alpha\supset N,\quad \nabla'_VA_0\subset A_0,
 \quad T(A_0,A_0)=0,\\
 A_0=I\oplus K,\quad I=\im T,\quad K=\operatorname{rad}T,
 \quad T(x,d)=\alpha(x)d\quad(d\in I).
 \end{gathered}
\]
Here $\tau(q)=\tr T(q,\cdot)$. Only the complete pure-flat covector
and two-form commutators are added to the preceding hypotheses.
\end{corollary}
\begin{proof}
On $H\ne0$, Lemma~\ref{r17:lem:abelian-ideal-horizontal}, image
differentiation and Lemma~\ref{rf:zero} give image equality.
On the interior $H=0$, scalarization~\ref{r17:lem:ideal-scalarization}
has internal trace $\tau$. A zero-$\tau$ field makes $N$ central,
so Lemma~\ref{lem:central-pure} gives $T=0$.
On $\tau\ne0$ it gives $B_q=\alpha(q)\Pi$, $\Pi^2=\Pi$,
$r=\rank\Pi>0$, $\alpha=\tau/r$,
$\Sym\nabla'\alpha=-2\alpha^2$, with both slots annihilating $N$.
Lemma~\ref{lem:matching-binary-covector}, at $(H,K,s,t)=(V,N,-2,0)$,
gives $\nabla'\alpha=-2\alpha^2$.
Thus $A_0=\ker\alpha\supset N$ is preserved and $T(A_0,N)=0$;
Lemma~\ref{lem:central-pure} gives $T(A_0,A_0)=0$.
The horizontal tensor on $A_0$ vanishes. If it were nonzero on a
subchart, the internal trace identity in Lemma~\ref{r17:lem:abelian-ideal-horizontal} would give
$(\nabla'_a\tau)(q)=\lambda(a)\kappa(q)$ with
$\lambda(v)=1$, $v\in A_0$, $\kappa\ne0$, contradicting
$\nabla'_v\tau=0$.
Hyperplane scalarization now gives
$\mathcal X=-2(\tau/r_0)\otimes T$, $r_0=\rank T$.
Comparing symmetric trace laws gives $r_0=r$; its splitting is
$A_0=\im T\oplus\rad T$, since the action is a nonzero scalar
projection. On fixed nonzero torsion-rank charts, the recurrence
extends over $\tau=0$; Lemma~\ref{rf:zero} excludes such zeros.
Finally Lemma~\ref{unif:lem:regular-loci} extends image equality
across the $H$ strata.
\end{proof}

\begin{lemma}[A preserved abelian ideal of codimension two]
\label{r47:lem:abelian-carriers}
Let an actual tensor system with $\mathcal X=\nabla'T$ satisfy
\ChernFirst{}, \ChernCyclic{} and \ChernPure{} and the complete
pure-flat covector commutators.
If a smooth
prime-preserved abelian torsion ideal $N$ satisfies
$\operatorname{codim}N=2$ and $\im\mathcal X\subset N$, then
$\im T\subset N$. 
\end{lemma}
\begin{proof}
For $N=0$ use Lemma~\ref{rf:zero}. Otherwise suppose the quotient
product on $Q=V/N$ is nonzero on an open set. Ideality and preservation
descend its connection, and $\nabla'\bar T=q_N\mathcal X=0$.
With $\lambda=-\tr_Q\bar T\ne0$,
\[
 \bar T(\bar u,\bar v)=\lambda(v)\bar u-\lambda(u)\bar v,
 \quad\nabla'\lambda=0,\quad\lambda|_N=0,\quad\lambda T=0.
\]
The horizontal tensor vanishes on this open set: otherwise the
horizontal alternative gives $\im T\subset N$, contradicting the
nonzero quotient product. Scalarization applies to $B_q=T(q,\cdot)|_N$.
On $\mu=\tr_NB\ne0$, put $\alpha=\mu/r$, with projection rank $r$:
\begin{equation}\label{fourh:eq:ideal-trace-first}
 B_q=\alpha(q)\Pi,\quad\Pi^2=\Pi,\qquad
 \Sym\nabla'\alpha=-2\alpha^2,\qquad
 (\nabla'\alpha)(N,V)=(\nabla'\alpha)(V,N)=0.
\end{equation}
The quotient bracket gives $\alpha T=\alpha\wedge\lambda$.
If $\alpha\wedge\lambda\ne0$, write
$\nabla'\alpha=-2\alpha^2+k\alpha\wedge\lambda$.
Set $a=-\alpha$, $\Omega=a\wedge\lambda$.\label{r16:lem:flat-covectors}
The full exterior identities are
$\partial a=(2k+1)\Omega$, $\partial\lambda=0$;
pure flatness gives $\partial(2a-k\lambda)=0$ and
$\partial(ka)=(2a-k\lambda)\wedge ka$. Consequently
\[
 \partial k=(4k+2)a+k(k+1)\lambda,\qquad
 0=\partial^2k=4(3k^2+3k+1)\Omega.
\]
The last equation makes $k$ locally constant because it takes values
in a finite set. But the two coefficients of $\partial k$ cannot vanish
together. Hence $\alpha=f\lambda$
where $\mu\ne0$. Now $D=\ker\lambda$ is a preserved hyperplane
containing $\im T$. On $\mu\ne0$, $B_q=\alpha(q)\Pi$ gives
$T(D,N)=0$; on $\mu=0$, scalarization gives $B=0$.
Since $D/N$ is a line, $T(D,D)=0$ throughout.
The hyperplane clause of Lemma~\ref{r17:lem:abelian-ideal-horizontal}
gives $\im T=\im\mathcal X\subset N$ on regular charts,
contrary to the nonzero quotient product.
\end{proof}

\subsection{Prime preserved abelian output without a trace assumption}
The proof uses rank-one rigidity (Theorem~\ref{newunb:r1v3:rankone},
Section~\ref{newunb:sec:general-rankone}).

\begin{theorem}[Abelian output in the original metric]\label{newunb:thm:bare-output}\label{hd:thm:bare-output}
Let $h$ be a smooth positive Hermitian metric of complex dimension $n\ge3$
with constant Chern holomorphic sectional curvature $c\ne0$. Suppose a
smooth subbundle $N\subset T^{1,0}M$ satisfies
\[
 \im T\subset N,\qquad T(N,N)=0,\qquad \nabla'_VN\subset N.
\]
Then $T=P=0$. Balancedness, compactness, $\partial\tau=0$ and barred
preservation of $N$ are not assumptions.
\end{theorem}

\begin{proof}
For $N=0$ or $N=V$, the image or abelian premise gives $T=0$ directly.
Otherwise Lemma~\ref{r17:lem:abelian-ideal-horizontal} gives a
preserved torsion-image line if $H\ne0$, excluded by
Theorem~\ref{newunb:r1v3:rankone}(2). Thus $H=0$.
On nonzero constant-rank torsion charts Corollary~\ref{r48:cor:abelian-output-alignment} gives
$\mathcal X=-2\alpha\otimes T$ and
$\ker\alpha=I\oplus K$, $I=\im T\ne0$, $K=\rad T$.
If $K=0$, Theorem~\ref{u3c:vectorial} excludes the resulting
vectorial torsion; if $K\ne0$, apply
Lemma~\ref{unif:lem:recurrent-projection}. Hence $T=0$ everywhere,
and \eqref{eq:barred} gives $P=0$.
\end{proof}

\subsection{Rank-one rigidity for the prescribed metric}
\label{newunb:r1v3:section}\label{newunb:sec:general-rankone}
To rule out rank-one torsion, we apply the tensor lemmas to the equations
of the original metric. This proof does not use either abelian metric
rigidity result. When we take a quotient, we project the equations,
not the metric.

\begin{theorem}[Rigidity when the torsion has rank one]\label{newunb:r1v3:rankone}\label{hd:thm:rank-one-torsion}
Let $h$ be a smooth positive Hermitian metric of dimension $n\ge3$
with constant real Chern HSC $c\ne0$. The following assertions hold
on open sets:
\begin{enumerate}
\item If $T$ has rank one, $\tau\ne0$, and $U=\im T$ is prime preserved,
then $E(U,V)\subset\Sym^2U$ and the torsion factorization and curvature
identities in \eqref{cut3:full-field} hold.
\item No open rank-one region has prime-preserved image $U$,
even without the nonzero-trace assumption.
\item If $\rank T\le1$, then $T=P=0$.
\end{enumerate}
\end{theorem}

\subsubsection{The complete repeated \ChernPure{} tensor and decomposability}
We use the complete repeated identity \eqref{eq:repeated} and its tensor
preimage \eqref{newunb:r1v3:graph}, valid in every dimension.

\begin{lemma}\label{newunb:r1v3:decomposable}
On an open set of the prescribed metric with $\rank T=1$,
write $T=\psi\otimes u$ with smooth nonzero factors. Then the
alternating form $\psi$ has rank two.
\end{lemma}
\begin{proof}
Put $\pi=q_{\C u}$, $A_p=\pi\nabla'_pu$.
The product rule and \ChernCyclic{} give $\pi\mathcal X_p=\psi A_p$
and $\gamma\wedge\psi=0$ for each component $\gamma$ of $A$.
If $\gamma(v)=1$, then $\psi=\gamma\wedge\iota_v\psi$.
Thus $\rank\psi\ge4$ would force $A=0$.
Both corrected vectors then lie in $\C u$; the
\hyperref[pure:rankone-contraction]{rank-one contraction} makes $\psi$
decomposable, contradicting rank at least four.
\end{proof}

\subsubsection{A direct contradiction on an actual zero-trace region}
\begin{lemma}\label{newunb:r1v3:zero-trace}
There is no actual rank-one region with $\tau=0$ when $n\ge4$.
\end{lemma}
\begin{proof}
Lemma~\ref{newunb:r1v3:decomposable} gives $T=\psi\otimes u$
with $\rank\psi=2$. Thus $K=\rad\psi$ is smooth of rank $n-2\ge2$;
$\tau(v)=\psi(v,u)=0$ puts $N=\C u$ strictly inside $K$, and
$T(K,V)=0$. The Chern curvature factorization
\eqref{r51:eq:barred-image-factor}, with $C=N=\im T$, gives
$(\Sym^2q_N)P=\mathscr B_NT$. All hypotheses of
Lemma~\ref{lem:central-line-kernel} hold, so it excludes the region
without assuming preservation of $N$ or $K$.
\end{proof}

\subsubsection{Prime preservation}
\begin{lemma}\label{newunb:r1v3:prime}
On an actual rank-one region with $\tau\ne0$ and $n\ge4$,
the image line $U=\im T$ is prime preserved.
\end{lemma}
\begin{proof}
The decomposable rank-two form and nonzero trace give a smooth frame with $T(e_0,e_1)=e_0$, all other
products zero, and $K=\langle e_2,\ldots,e_{n-1}\rangle$.
Put $\pi=q_U$, $A=\mathcal X_0(e_0,e_1)$ and $B=\mathcal X_1(e_0,e_1)$.
Image differentiation gives $\mathcal X_p(e_0,e_r)=a_{pr}e_0$ and
$\mathcal X_p(e_1,e_r)=b_{pr}e_0$. Projected \ChernCyclic{} on
$(r,0,1)$ gives $\pi\mathcal X_r(e_0,e_1)=0$. The common factorization
\eqref{r51:eq:barred-image-factor} gives $(\Sym^2\pi)P(K,V)=0$.
For $Q_0=(\Sym^2\pi)P(e_0,e_1)$, complete \ChernPure{} on
$(0,1;0,r),(0,1;1,r)$ gives
\begin{equation}\label{newunb:r1v3:normal-rows}
 \pi A\odot\pi e_r+2a_{0r}Q_0=0,\qquad
 \pi B\odot\pi e_r+(4b_{0r}-2a_{1r})Q_0=0.
\end{equation}
Projecting $4\mathcal C$ gives these canonical terms;
quadratic terms carry $e_0$, and the surviving $\mathcal D$ coefficients
of $e_0\wedge e_1$ are $a_{0r}/2$, $b_{0r}-a_{1r}/2$.
Symmetric multiplication by a nonzero $\pi A$ or $\pi B$ is
injective, so its products with two independent $\pi e_r$ cannot
lie in $\C Q_0$. Thus $\pi A=\pi B=0$; all $\mathcal X_p$
are $U$-valued, which preserves $U$.
\end{proof}

\begin{lemma}[Dichotomy for a line-valued torsion tensor]
\label{unif:lem:line-output-dichotomy}
Let $C=\C e\oplus\C f\oplus K\subset V$, with $\dim K\ge1$,
put $F=\C f$, and suppose
\[
 T(e,f)=f,\qquad T(C,K)=0,\qquad
 X:C\otimes\Lambda^2C\longrightarrow F,\quad X_p(K,K)=0\quad(p\in C).
\]
Let $P:\Lambda^2C\to\Sym^2V$ have alternating inputs and set
$E(x,y)=xy+P(x,y)$. Assume complete \ChernCyclic{} and \ChernPure{} with all
direction and input slots in $C$; all symmetric outputs remain ambient.
Then $E(f,K)\subset F^2$. Moreover either $E(F,C)\subset F^2$, or
\[
 E(f,K)=0,\qquad X_e=-2T,\qquad X_f=X_k=0\quad(k\in K).
\]
This is a pointwise assertion; no metric, \ChernFirst{}, trace, preservation,
or assumption $\operatorname{im}P\subset F\odot V$ is required.
\end{lemma}
\begin{proof}
For $0\ne k\in K$ write
\[
\begin{gathered}
 (a_0,a_1,a_2)=X_e(ef,ek,fk)/f,\quad
 (a_3,a_4,a_5)=X_f(ef,ek,fk)/f,\\
 (a_6,a_7,a_8)=X_k(ef,ek,fk)/f.
\end{gathered}
\]
\ChernCyclic{} gives $a_6=a_4-a_2$. Put $R=E(f,e)$, $W=E(f,k)$.
Four times complete \ChernPure{} on $(e,f; e,f),(e,f;e,k),(e,f;f,k)$,
$(e,k;e,k),(e,k;f,k),(f,k;f,k)$ gives
\begin{equation}\label{fourh:eq:six-internal-rows}
\begin{aligned}
 -4a_3R+4(a_0+2)f^2&=0,\\
 (2a_2-4a_4)R+(2a_0+1)W+2a_1f^2&=0,\\
 -2a_5R+2a_3W+(4a_2-2a_4)f^2&=0,\\
 -4a_7R+4a_1W&=0,\\
 -2a_8R+2(a_2+a_4)W-2a_7f^2&=0,\\
 4a_5W-4a_8f^2&=0.
\end{aligned}
\end{equation}
Including the positive torsion term, the quadratic contributions to the six equations are $8f^2,W,0,0,0,0$.
If $R\in F^2$, rows 3, 1 and 2 exclude $W\notin F^2$ by giving
$a_3=0$, $a_0=-2$ and $W\in F^2$. This gives $E(F,C)\subset F^2$.
If $R\notin F^2$, rows 1--3 give $a_3=a_5=0$, $a_0=-2$,
$a_4=2a_2$ and $W=-2a_2R+(2a_1/3)f^2$. Row 4 then gives
$a_1=a_7=0$, and rows 5--6 give $a_8=-6a_2^2=0$.
Thus $W=0$ and only $a_0=-2$ survives. As $X_p(K,K)=0$,
$X_e=-2T$ and $X_f=0$. \ChernCyclic{} and the zero Jacobiator make
$X_k(e,l)$ and $X_k(f,l)$ symmetric in $k,l\in K$ with zero diagonals, hence zero.
Thus also $X_k=0$.
\end{proof}

\begin{proof}[Proof of Theorem~\ref{newunb:r1v3:rankone}]
For (1), Lemma~\ref{newunb:r1v3:decomposable} and $\tau\ne0$
give $V=\C e\oplus\C f\oplus K$, $U=\C f$, $K=\rad T$,
$T(e,f)=f$. Preservation and differentiation give
$\im\mathcal X\subset U$, $\mathcal X_p(K,K)=0$.
Lemma~\ref{unif:lem:line-output-dichotomy} applies with $C=V$.
Its noncut alternative is $\mathcal X=-2\tau\otimes T$, excluded
by Lemma~\ref{unif:lem:recurrent-projection}, since $I=U$ and
$\rank K=n-2\ge1$. Thus $E(f,x)=\xi(x)f^2$, $\xi(f)=1$;
Lemma~\ref{cut3:full-input} gives \eqref{cut3:full-field}, including
\begin{equation}\label{newunb:r1v3:zero-normal}
 \im P\subset U\odot V.
\end{equation}
For (2), part (1) and Lemma~\ref{unif:lem:prescribed-line-cut}
exclude $\tau\ne0$. The remaining field $\tau=0$ makes $U$ central,
so Lemma~\ref{lem:central-pure}, with $(D,Z)=(V,U)$, gives $T=0$.
For (3), every nonzero torsion point has a rank-one neighborhood.
Its image is preserved in dimension three by
Theorem~\ref{unif:thm:three-prime-images}; in higher dimensions
Lemma~\ref{newunb:r1v3:prime} gives preservation on $\tau\ne0$,
dense by Lemma~\ref{newunb:r1v3:zero-trace}. Continuity and (2)
exclude the neighborhood. Thus $T=0$, and \ChernBarred{} gives $P=0$.
\end{proof}

\section{Vanishing conditions on the first curvature hull}\label{hd:section}
For nonzero constant Chern HSC in dimension $n\ge4$, we exclude
$(\dim H_0,\dim H_1)=(1,2)$ without balancedness. The projected
identity for $H_0\subset H_1$ makes the derivative image a torsion
ideal, contrary to Theorem~\ref{r24:thm:ideal-output-line}.

\subsection{Closure of a first cut without balancedness}
\label{hd:sec:active-action}

We use the complete Chern system of Proposition~\ref{prop:source},
with the prescribed metric and $c\ne0$. The map $q_A$ is the
canonical quotient, not an orthogonal projection.
For the bilinear map $E$, $E(u,v)+E(v,u)=2uv$, and the full \ChernFirst{} operator is
\eqref{newff:eq:E-form-F}. The full repeated row is
\eqref{h2:eq:repeatedE}, where $t=T(u,v)$,
$s_u=\mathcal X_u(u,v)+T(u,t)/2$ and
$s_v=\mathcal X_v(u,v)+T(v,t)/2$.
Compute both corrected vectors with the ambient torsion and inputs.
The quotient is applied only to the resulting symmetric output.

\begin{lemma}[A repeated row with a one-dimensional kernel]
\label{hd:lem:line-kernel}
Let $q:W\to N$ be a surjection of finite-dimensional complex
vector spaces with one-dimensional kernel $F$ and choose a nonzero
$f\in F$. Let $\mathcal E:W\times W\to\Sym^2N$ be bilinear with
\[
 \mathcal E(x,y)+\mathcal E(y,x)=2(qx)(qy).
\]
Suppose $A\in\End W$ satisfies
\begin{equation}\label{hd:eq:line-kernel-row}
 \mathcal E(A^2v,v)+3(qAv)^2=0\qquad(v\in W).
\end{equation}
Then $A(F)\subset F$. Choose a linear section of $q$ and use it to
identify $W=F\oplus N$. In this splitting write
\[
 A=\begin{pmatrix}b&\ell\\0&D\end{pmatrix},
 \qquad \mathcal B(v)=\mathcal E(f,v)\quad(v\in N).
\]
One has
\begin{equation}\label{hd:eq:line-kernel-conclusion}
 \rank D\le1,\qquad D^2=bD.
\end{equation}
If $D\ne0$, write $D=r\otimes\nu$, with $r\ne0$.
Then $\nu(r)=b$, $h:=\ell(r)\ne0$, and
\begin{equation}\label{hd:eq:line-kernel-square}
 \mathcal B(v)=
 \begin{cases}
 -3\nu(v)r^2/h,&b=0,\\
 -2\nu(v)r^2/h,&b\ne0.
 \end{cases}
\end{equation}
These formulas hold for every input $v$.
\end{lemma}

\begin{proof}
The row gives $A(\ker A^2)\subset F$; every eigenvector with
$\lambda\ne0$ satisfies $4\lambda^2(qv)^2=0$.
The eigenvector from each nonzero-eigenvalue block or nontrivial
nilpotent block lies in $F$. Since eigenvectors from distinct blocks
are independent, at most one block is not a zero $1\times1$ block. This also gives $A(F)\subset F$: either $F$ is the
eigenline of the exceptional block or $A=0$.
Expansion at $v+tf$ gives
\begin{align}
 \mathcal B(D^2-b^2\Id)&=0,\label{hd:eq:line-kernel-mixed}\\
 \mathcal E(D^2v,v)+\{b\ell(v)+\ell(Dv)\}\mathcal B(v)+3(Dv)^2&=0.
 \label{hd:eq:line-kernel-full}
\end{align}
For $b\ne0$, a chain $(A-b)x=f$, $(A-b)y=x$ would give
$2b\mathcal B(qx)=0$ from the mixed row and
$4b^2(qx)^2+2b\mathcal B(qx)=0$ from the original row.
For $b=0$, $Ax=f$, $Ay=x$, $Az=y$ would give $(qx)^2=0$
from the polarization of~\eqref{hd:eq:line-kernel-row} in $(x,z)$.
The surviving nontrivial block of $A$ has length at most two.
Passing to the quotient by $F$ gives $\rank D\le1$ and $D^2=bD$.
For $D=r\otimes\nu\ne0$, $h=\ell(r)$, $\nu(r)=b$,
the mixed and full rows give, when $b\ne0$,
$b\mathcal B(v)=\nu(v)\mathcal B(r)$ and
$2bh\mathcal B(r)=-4b^2r^2$; when $b=0$, they give
$\nu(v)\{h\mathcal B(v)+3\nu(v)r^2\}=0$.
Both cases give $h\ne0$ and the formula.
\end{proof}

For the application to the Chern connection of $h$, assume on an
open set that
\begin{equation}\label{hd:eq:first-cut-assumptions}
 L=\im\mathcal X=\C e,\qquad
 L\subset C,\quad \rank C=2,\qquad
 (\Sym^2q_C)P(L,V)=0,
\end{equation}
where $L,C$ are smooth subbundles and $e$ is a nonvanishing local
section of $L$.
Write $\mathcal X=e\chi$ and $A=T(e,\cdot)$.
The assumed cut kills the projected defect if either input lies in $L$,
by alternation. The canonical product also projects to zero for these
inputs, so
$E_C=(\Sym^2q_C)E$ descends in both inputs to $W=V/L$.
Also $A$ descends to $W$, since $Ae=0$.
With derivative output in $L$, the projected identity
\eqref{h2:eq:projected-repeated} descends to \eqref{hd:eq:line-kernel-row}, with kernel $F=C/L$ and $N=V/C$.
Lemma~\ref{hd:lem:line-kernel} therefore gives
\begin{equation}\label{hd:eq:induced-action}
 T(C,C)\subset C,\qquad
 \bar A:=A\bmod L=
 \begin{pmatrix}b&\ell\\0&D\end{pmatrix},
 \quad \rank D\le1,\quad D^2=bD.
\end{equation}
The blocks use a local splitting over $C/L$, with no claim that
the connection preserves it.

Put $\tau(u)=\tr_V T(u,\cdot)=-\eta(u)$.
\begin{lemma}[Ideality of the derivative-image line]
\label{r22:lem:firstcut-action-master}\label{r23:lem:firstcut-line-ideal}
\label{hd:lem:internal-action}\label{hd:lem:active-kernel}\label{hd:lem:active-closure}
Let the given Hermitian metric have complex dimension $n\ge4$
and constant Chern HSC $c\ne0$.
Under \eqref{hd:eq:first-cut-assumptions},
\begin{equation}\label{hd:eq:internal-action-zero}
 \equationalias{hd:eq:active-kernel-result}\equationalias{hd:eq:active-closure}
 T(L,V)\subset L.
\end{equation}
\end{lemma}
\begin{proof}
Put $A=T(e,\cdot)$, $\bar A=A\bmod L$, $\rho=q_CA$,
$\mathcal E=(\Sym^2q_C)E$, $C=\langle e,f\rangle$,
$B=\mathcal E(f,\cdot)$.
On a field $B=0$, Lemma~\ref{lem:containment} puts $\im T\subset C$;
Theorem~\ref{r16:thm:ranktwo-fullimage-terminal} then gives $T=0$.
If $\bar A$ is nonzero somewhere, it stays nonzero on a neighborhood.
The preceding exclusion of $B=0$ on an open set lets us choose a
constant-rank subchart there with $\bar A\ne0$ and $B\ne0$.
Complete repeated \ChernPure{}, with derivative output in $L$, gives
\[
 \mathcal E(A^2v,v)+3(\rho v)^2=0,
 \quad Q(p,v):=\mathcal E(Ap,v)+\mathcal E(Av,p),
 \quad Q(v,Av)=-2(\rho v)^2.
\]
Lemma~\ref{hd:lem:line-kernel} gives
$\bar A\bar f=b\bar f$, $\rank\rho\le1$, $\rho A=b\rho$.
For $\rho\ne0$, it also gives
$b=\tr\bar A/2$, $C/L=\im(\bar A^2-b\bar A)$.
For $\rho=0$, its mixed row gives $b^2B=0$, hence
$b=0$, $q_LA=\bar f\otimes\beta\ne0$, $\beta|_C=0$.
Let $R=\rho$ for $\rho\ne0$, and $R=q_LA$ otherwise.
Put $\sigma_p=q_L\nabla'_pe$. For $p\in\ker\sigma$ write
$\nabla'_pe=\lambda_pe$. The product rule gives
$\nabla'_p\bar A=\lambda_p\bar A$. Thus $C/L$, given by the
corresponding image above, is preserved in that direction, as is $C$. Full
\ChernFirst{} on the cut, including both input motions, gives
\begin{equation}\label{r22:eq:firstcut-directional-FIRST}
 Q(p,v)=\mathcal E(Ap,v)+\mathcal E(Av,p)=0.
\end{equation}
For $\rho\ne0$, $Q(p,Ap)=-2(\rho p)^2$ gives $Rp=0$.
Otherwise
\begin{equation}\label{r23:eq:beta-whole-FIRST}
 Q(p,v)=\beta(p)B(v)+\beta(v)B(p)
\end{equation}
also gives $Rp=0$: if $\beta(p)\ne0$, take $v=p$ and then arbitrary
$v$ to get $B=0$. Thus $\ker\sigma\subset\ker R$, so $\sigma\ne0$.
The normal Hessian $\sigma_p\chi_q=\sigma_q\chi_p$ for
$\mathcal X_p=e\chi_p$ gives a common direction kernel
$\mathcal D=\ker\sigma=\ker(\mathcal X_{\cdot})$ of codimension one;
hence $\mathcal D=\ker R$. For local $p,q\in\mathcal D$, the full Hessian
\eqref{newff:eq:X-pure-Hessian} reduces to
$-\mathcal X_{[p,q]}=0$, so $\mathcal D$ is involutive.
The one-kernel identities give $R(C)=0$, $RA=bR$, while
$\nabla'_pR=\lambda_pR$ for $p\in\mathcal D$ gives geodesy.
Both covariant derivatives and the Lie bracket of sections of
$\mathcal D$ lie in $\mathcal D$, so the torsion identity gives
$T(\mathcal D,\mathcal D)\subset\mathcal D$.
Restrict complete \ChernCyclic{}/\ChernPure{}, retracting only final symmetric
outputs; Lemma~\ref{rf:zero} gives $T(\mathcal D,\mathcal D)=0$.
Since $C\subset\mathcal D$, $Af=0$, hence $b=0$ and $A(V)\subset\mathcal D$.
Choose a defining covector $\kappa$ for $\mathcal D$ and $n$ with $\kappa(n)=1$. Set $z=An\in\mathcal D\setminus L$,
$M=T(n,\cdot)|_{\mathcal D}$. To prove $M(\mathcal D)\subset\mathcal D$, use \ChernCyclic{} to get $J_T(n,u,e)=-\kappa(Mu)z\in L$.
Thus $M(\mathcal D)\subset\mathcal D$, so $J_T=0$ and
\ChernCyclic{} gives $\mathcal X_p(\mathcal D,\mathcal D)=0$.
On a constant-rank chart with $\rank M\ge2$, the final clause of
Lemma~\ref{r37:lem:image-shape}, with $J=0$, preserves $\mathcal D$.
The abelian hyperplane $\mathcal D$ contains $\im T$,
so Theorem~\ref{hd:thm:bare-output} gives $T=0$, contrary to $z\ne0$. The rank-at-most-one field is excluded by
Theorem~\ref{newunb:r1v3:rankone}, since $z\notin L$.
\end{proof}

\subsection{A first-cut obstruction without balancedness}
\label{hd:sec:first-cut}
Retain $n\ge4$, $H_h\equiv c\ne0$, and the quotient notation of
\S\ref{hd:sec:active-action}; balancedness is not assumed.

\begin{lemma}[A trace-free torsion image]\label{hd:lem:image-trace}
Suppose $\im\mathcal X=L$ is a line, $\tau(u)=\tr_VT(u,\cdot)$, and
$\tau\circ T=0$. There is no nonempty constant-rank open set on which
$L\subset J=\im T$.
\end{lemma}
\begin{proof}
Since $\im\mathcal X=L\subset J=\im T$,
Lemma~\ref{r37:lem:image-shape} preserves $J$.
Lemma~\ref{r37:lem:internal-restriction} gives its complete internal
prime source. Its trace is $\tau|_J=0$: all torsion outputs lie in
$J$, and $\tau\circ T=0$.
Differentiate the identically zero internal trace and contract the
restricted identity \ChernCyclic{} as in Proposition~\ref{prop:source}.
Both derivative traces vanish.
The internal derivative image has rank at most one, so
Theorem~\ref{rf:rank-gap} gives $T(J,J)=0$.
Theorem~\ref{hd:thm:bare-output} applied to $J$ gives $T=0$,
contrary to $\im\mathcal X=L\ne0$.
\end{proof}

\begin{lemma}[Trace identities for a derivative-output line]
\label{fourh:line-trace}
In an actual Chern or restricted prime source, use
$\mathcal X=\nabla'T$ and \eqref{eq:cyclic}. Suppose
$\mathcal X=e\chi$ for a local generator $e$ of the line $L$.
Take $\tau(u)=\tr T(u,\cdot)$ on the given bundle
($\tau=-\eta$ in the ambient Chern connection).
Then
\begin{equation}\label{fourh:eq:line-trace}
 J_T(e,u,v)=\tau(T(u,v))e,\qquad
 \tau(T(e,v))=0,\qquad
 (\nabla'_p\tau)(w)=\chi_p(w,e).
\end{equation}
If $T(e,v)=\alpha(v)e$, also
\begin{equation}\label{hd:eq:line-trace}
                         (\alpha+\tau)\circ T=0.
\end{equation}
If $L$ is additionally prime preserved, then
\begin{equation}\label{fourh:eq:ideal-trace}
 (\nabla'_p\alpha)(v)=\chi_p(e,v),\qquad
                         \nabla'(\alpha+\tau)=0.
\end{equation}
\end{lemma}
\begin{proof}
\ChernCyclic{} gives $J_T(p,u,v)=\Omega(p,u,v)e$. The trace of
$p\mapsto J_T(p,u,v)$ is $\tau(T(u,v))$: its other two adjoint
products form a commutator. This proves the first equation in
\eqref{fourh:eq:line-trace}; repetition $u=e$ and the full covariant
trace of $\mathcal X=\nabla'T$ give the other two. For an ideal line,
direct expansion gives $J_T(e,u,v)=-\alpha(T(u,v))e$, proving
\eqref{hd:eq:line-trace}. With prime preservation, the full ideal
product rule cancels the equal moving-line input and output terms,
giving the first equation in \eqref{fourh:eq:ideal-trace}; skewness
and \eqref{fourh:eq:line-trace} give the second.
\end{proof}

\begin{lemma}[Nilpotence of the quotient by a preserved abelian ideal]
\label{r44:lem:nilpotent-quotient}\label{r49:lem:abelian-carrier-nilpotent}
For the Chern tensors of the given metric, let a smooth subbundle $N$ satisfy
\[
 \nabla'_VN\subset N,\qquad T(N,V)\subset N,\qquad
 T(N,N)=0,\qquad\im\mathcal X\subset N.
\]
Then the induced product $\bar T$ on $W=V/N$ is a prime-parallel
nilpotent Lie bracket, in any rank and codimension.
\end{lemma}
\begin{proof}
Ideality and preservation descend the connection with
$\nabla'\bar T=q_N\mathcal X=0$; \ChernCyclic{} gives Jacobi.
In a quotient prime-parallel frame its coefficients are antiholomorphic.
If $\ad_x$ has a nonzero eigenvalue, extend $x$ prime-parallel.
On a regular subchart choose a simple root $\lambda\ne0$ of the
square-free part of the characteristic polynomial and solve the
eigenvector equations using a nonzero minor of constant rank.
This gives antiholomorphic $\lambda,y$ with $y\ne0$ and
$[x,y]=\lambda y$. Then $D=q_N^{-1}\langle x,y\rangle$ is preserved and torsion-closed.
Lemma~\ref{r37:lem:internal-restriction} gives its complete internal
source and inherited pure-flat commutators. Its preserved abelian
ideal $N$ has codimension two and contains the derivative output;
Lemma~\ref{r47:lem:abelian-carriers} gives $T(D,D)\subset N$,
contradicting $[x,y]\ne0$. Every adjoint is therefore nilpotent;
Engel's lemma~\ref{lem:nilpotent-engel} finishes.
\end{proof}

\begin{theorem}[Exclusion of an ideal derivative-output line]
\label{r24:thm:ideal-output-line}\label{fourh:ideal-line}
\label{hd:lem:ideal-line-preparation}\label{hd:eq:ideal-sigma}\label{hd:eq:square-image}
Let $h$ be the original smooth Hermitian metric of complex dimension
$n\ge4$ with constant Chern HSC $c\ne0$.
There is no nonempty open set on which
$L=\im\mathcal X$ is a smooth nonzero line and $T(L,V)\subset L$.
\end{theorem}
\begin{proof}
Write $L=\C e$, $T(e,v)=\alpha(v)e$.
Lemma~\ref{hd:lem:central-line-field} excludes a field $\alpha=0$.
On $\alpha\ne0$, put $\zeta_p=q_L\nabla'_pe$.
The full ideal derivative gives
$(\alpha(v)I+\ad_{[v]})\zeta_p=0$ on the pointwise Lie quotient $W=V/L$.
Commuting and comparing at $T(u,v)$ gives
$\alpha(T(u,v))\zeta_p=0$. Nonzero shape would therefore imply
$\alpha T=\tau T=0$ by \eqref{hd:eq:line-trace}; since
$L\subset\im T$, Lemma~\ref{hd:lem:image-trace} excludes its
constant-rank subcharts. Hence $L$ is prime preserved.
Lemma~\ref{r44:lem:nilpotent-quotient} makes $W$ prime-parallel
nilpotent. On a constant-rank center chart, $Z$ is preserved and
has a prime-parallel frame. If $\alpha|_Z\ne0$, choose parallel
$z\in Z$ with $\alpha(z)\ne0$. The preserved plane
$A=q_L^{-1}\C z$ has $T(A,V)\subset L$ and $T(A,A)\ne0$;
Lemma~\ref{r18:lem:binary-derived-preservation} puts $\im T\subset L$,
contrary to Theorem~\ref{newunb:r1v3:rankone}(2).
Thus $\alpha(Z)=0$.
The preserved ideal $N=q_L^{-1}Z$ satisfies
$T(N,V)\subset L$, $T(N,L)=0$; Lemma~\ref{lem:central-pure}
gives $T(N,N)=0$.
Its horizontal field is zero: otherwise
Lemma~\ref{r17:lem:abelian-ideal-horizontal} gives a preserved
torsion-image line, excluded by rank-one rigidity.
The action on $N$ has image $L$ and trace $-\alpha\ne0$, so
Lemma~\ref{r17:lem:ideal-scalarization} gives
\begin{equation}\label{hd:eq:moving-B}
 N=L\oplus K,\qquad\nabla'_VK\subset K,\qquad
 T(K,V)=0,\qquad E(L,K)=0.
\end{equation}
Full differentiation and \ChernCyclic{} give
$\mathcal X_p(K,V)=\mathcal X_K=0$.
In Lemma~\ref{unif:lem:central-PURE-row}, the Jacobiator and
derivative terms are $L$-valued, hence killed by $E(L,K)=0$:
\begin{equation}\label{hd:eq:central-triples}
 E(T(p,T(q,v)),k)=0\qquad(p,q,v\in V,\ k\in K).
\end{equation}
If $0\ne\bar w\in[W,[W,W]]\cap Z$, lift a sum of nested brackets
to $w=ae+k_0\in L\oplus K$, $k_0\ne0$.
The full ambient identity gives $E(w,k_0)=k_0^2=0$.
Nilpotence therefore forces $[W,[W,W]]=0$: otherwise its last
nonzero lower-central term lies in that intersection. Hence
\begin{equation}\label{hd:eq:rank-star}
 \im T\subset N,\qquad T(N,N)=0.
\end{equation}
Theorem~\ref{hd:thm:bare-output} applied to the preserved $N$ gives
$T=0$, contrary to $L\ne0$. Every proposed $\alpha\ne0$ region contains a center chart;
the opening argument already ruled out the central case.
\end{proof}

\begin{lemma}[Image alignment for a rank-two partial connection]
\label{r50:lem:binary-image}\label{fourh:prime-bcolumn}
On a rank-two bundle $A$ with a partial connection, assume
$\mathcal X=\nabla'T$, equations \eqref{eq:transport}, \eqref{eq:cyclic},
\eqref{eq:pure}, and the pure-flat covector and determinant-line
commutators. Where $T\ne0$, the line $F=\im T$ is prime preserved and
$\im\mathcal X=F$. In a complete restriction to a preserved plane, these mean the
internal images $\im T_A$ and $\im X^A$; they cannot be transverse lines.
\end{lemma}
\begin{proof}
Write $T=B_\alpha$, $F=\ker\alpha$, and
$P(u,v)=(Ku)v-u(Kv)$ with $\tr K=1$.
Put $\delta=-\partial_A\alpha$, $\kappa=\det K$.
Complete repeated \ChernPure{}, \ChernFirst{}, the covector commutator and
determinant differentiation give respectively
\[
 \begin{gathered}
 (\nabla'_u\alpha)(v)=-2\alpha(u)\alpha(v)-\delta(u,Kv),
 \quad\nabla'_uK=u\otimes\alpha-\alpha(u)K,\\
 \nabla'_{Ku}\delta=-2\alpha(Ku)\delta-4\alpha(u)\delta,
 \quad\partial_A\kappa=\alpha\circ\operatorname{adj}K-2\kappa\alpha.
 \end{gathered}
\]
On $\kappa\ne0$, these give
$\nabla'_A\delta=(-10\alpha-4\partial_A\log\kappa)\otimes\delta$.
If $\delta\ne0$, the pure-flat determinant line would have curvature
$-10\partial_A\alpha=10\delta\ne0$. Thus $\delta\kappa=0$.
On $\delta\ne0$, the field $\kappa=0$ gives $K^2=K$, $\alpha K=\alpha$,
hence $K(F)=0$. The first identity preserves $F$ there and on the
interior $\delta=0$; continuity gives preservation everywhere.
Image differentiation gives $\im\mathcal X\subset F$;
Lemma~\ref{rf:zero} gives equality wherever $T\ne0$.
For restrictions, Lemma~\ref{r37:lem:internal-restriction} applies;
its parallel projection also applies to the stated prime sources.
\end{proof}

\begin{lemma}[A preserved nonabelian plane with line-valued action]
\label{r18:lem:binary-derived-preservation}\label{fourh:shape-rank-contradiction}
\label{r17:lem:derived-ideal-mother}\label{fourh:derived-ideal-terminal}
Use the Chern tensors of the given metric, or a complete pure-flat prime source
admitting local prime-parallel frames. If a smooth plane $A$ and
line $F\subset A$ satisfy
\[
 \nabla'_VA\subset A,\qquad T(A,V)\subset F,
 \qquad T(A,A)\text{ is nowhere zero},
\]
then $\nabla'_VF\subset F$ and $\im T\subset F$.

In particular $\rank T\ge2$ is impossible.
\end{lemma}
\begin{proof}
The full binary restriction of
Lemma~\ref{r37:lem:internal-restriction} and
Lemma~\ref{r50:lem:binary-image} preserve $F$ in $A$ directions.
For $x,y\in A$, the full image rule is
$q_F\mathcal X_p(x,y)=(q_F\nabla'_p|_F)T(x,y)$.
In \ChernCyclic{} $(p,x,y)$ the other derivatives and all Jacobi terms lie
in $F$, using this internal preservation and $T(A,V)\subset F$.
Since $T(A,A)$ spans $F$, $F$ is ambient preserved.
On the containment-failure set write $F=\C e$, $A=\langle e,f\rangle$,
$T(e,v)=\alpha(v)e$, $T(f,v)=g(v)e$, $a=\alpha(f)\ne0$, and
$\sigma=\nabla'\alpha$. The preserved ideal gives
$\mathcal X_p(e,v)=\sigma_p(v)e$ and $\mathcal X_p(A,V)\subset F$.
In each $D=A+\C v$ with $v\notin A$, the radical is spanned by
$k=v-\alpha(v)f/a+g(v)e/a$ (since $g(e)=-a$), and $T(-f/a,e)=e$.
Both output images lie in $F$, so the complete identities on $D$
with ambient symmetric output give $E(e,k)\in F^2$ by
Lemma~\ref{unif:lem:line-output-dichotomy}.
A field $R:=E(e,f)\in F^2$ would give $E(F,V)\subset F^2$ and
containment by Proposition~\ref{prop:conceptual-fullcut}.
On the dense noncut failure locus, the other alternative gives
$E(e,k)=0$, $\mathcal X_p|_{\Lambda^2D}=2\alpha(p)T|_{\Lambda^2D}$.
Varying $v$ and polarizing, without differentiating the chosen bases, gives
\begin{equation}\label{r17:eq:derived-sigma}
 \Sym\sigma=2\alpha^2,\quad\sigma_p(f)=2a\alpha(p),\quad
 E(e,v)=\frac{\alpha(v)}aR-\frac{g(v)}ae^2.
\end{equation}
Write $\nabla'_pf=\xi_pe+\zeta_pf$.
Differentiating $a=\alpha(f)$ gives $\partial\log a-\zeta=2\alpha$;
pure flatness of $A/F$ gives $\partial\zeta=0$, hence $\partial\alpha=0$.
With $\Omega=\alpha T$, the torsion-corrected exterior identity gives
\begin{equation}\label{r17:eq:derived-full-sigma}
 \sigma=2\alpha^2-\tfrac12\Omega.
\end{equation}
For a local section $w$ of $\ker\alpha$, differentiate
$E(e,w)\in F^2$.
Modulo $F^2$, complete \ChernFirst{} and the moving input give
\[
 F_p(T)(e,w)=eT(p,w)-\frac{\Omega(p,w)}{2a}R,
 \quad E(e,\nabla'_pw)=\frac{\Omega(p,w)}{2a}R,
\]
since $\alpha(\nabla'_pw)=-\sigma_p(w)=\Omega(p,w)/2$.
Their sum is zero, so $T(p,w)\in F$.
Every $v=w+\alpha(v)f/a$, and $T(p,f)\in F$; hence all torsion
lies in $F$, excluding the failure locus.
\end{proof}

\begin{lemma}[A central output line on an open set]
\label{hd:lem:central-line-field}
There is no nonempty open set with $\im\mathcal X=L=\C e\ne0$
and $T(e,V)=0$.
\end{lemma}
\begin{proof}
Equation~\eqref{hd:eq:line-trace}, whose derivation does not require
$\alpha\ne0$, gives $\tau\circ T=0$ with $\alpha=0$.
Differentiate $T(e,v)=0$, including the derivative of $e$:
$\mathcal X_p(e,v)+T(\nabla'_pe,v)=0.$
Where $\mathcal X_p(e,v)\ne0$ for some slots, this gives $L\subset\im T$.
Every nonempty part of that open set contains a constant-rank image
chart, excluded by Lemma~\ref{hd:lem:image-trace}. Thus
$\mathcal X_p(e,V)=0$ everywhere. Complete \ChernCyclic{} then gives
$\mathcal X_e=0$. Lemma~\ref{rf:central} now applies with all its
hypotheses established and yields $T=\mathcal X=0$, a contradiction.
\end{proof}

\begin{theorem}[The first-cut plane obstruction]\label{hd:thm:first-cut}
\label{hd:lem:abelian-C}\label{hd:lem:transverse-ideal}\label{hd:prop:ideal-C}\label{hd:lem:ideal-line-exclusion}\label{hd:lem:prime-C}
Let $h$ be a smooth Hermitian metric on a complex manifold of dimension
$n\ge4$ with constant Chern HSC $c\ne0$. There is no nonempty open set
carrying a smooth nonzero line $L=\im\mathcal X$ and a smooth plane
$C\supset L$ such that
\begin{equation}\label{hd:eq:first-cut}
                           q_C^2P(L,V)=0.
\end{equation}
\end{theorem}
\begin{proof}
Lemma~\ref{r23:lem:firstcut-line-ideal} gives $T(L,V)\subset L$
for the same original metric and first cut.
Theorem~\ref{r24:thm:ideal-output-line} excludes this nonzero
ideal derivative-output line, proving the obstruction.
\end{proof}

\subsection{Finite-hull growth without balancedness}
\label{hd:sec:short-hull}
Use \eqref{eq:hull}, with $H_0=\im\mathcal X$.

\begin{lemma}[A common finite-hull growth criterion]
\label{unif:lem:finite-hull-growth}
Let $h$ be the original Hermitian metric in dimension $n\ge4$ with
constant Chern HSC $c\ne0$, and fix $1\le k\le n-1$.
Suppose no nonempty open torsion-nonvanishing region with constant
ranks of $H_0,H_1$ has
\[
              0<\dim H_0<\dim H_1\le k.
\]
Let $O$ be open. If $\dim H_1(x)\le k$ for every $x\in O$, or if
$H_{n-k}(x)\ne V_x$ for every $x\in O$, then $T=0$ on $O$.
No constant-rank assumption on $O$ is required.
\end{lemma}
\begin{proof}
On any proposed nonzero-torsion open set, the finite families of curvature compositions defining the hulls and
Lemma~\ref{unif:lem:regular-loci} give a common constant-rank chart
for $H_0,\ldots,H_{n-k}$, with $H_0\ne0$ by Lemma~\ref{rf:zero}.
If $A=H_j=H_{j+1}$ is proper on the common constant-rank chart,
it is a curvature-invariant subbundle and satisfies
$(\Sym^2q_A)P(A,V)=0$ by \eqref{eq:hull} and \eqref{eq:P}.
Since $A\supset\im\mathcal X$, Lemma~\ref{lem:containment} gives
$\im T\subset A$, contradicting Theorem~\ref{sup:thm:support}.
Since consecutive proper subspaces are distinct, the hypothesis gives
$\dim H_1>k$. Every later step from a proper subspace is strict, hence
\[
 \dim H_j\ge\min\{n,k+j\}\qquad(1\le j\le n-k).
\]
Taking $j=1$ contradicts the bound on $\dim H_1$; taking $j=n-k$
contradicts properness of $H_{n-k}$. Hence no
nonzero-torsion open set exists.
\end{proof}

\begin{theorem}[The unbalanced two-dimensional growth criterion]
\label{hd:thm:h1-two}\label{hd:thm:hnminus2}
Let $h$ be a smooth Hermitian metric with constant Chern HSC
$c\ne0$. Either pointwise bound on an open set $O$ gives $T=0$:
\[
 n\ge4:\quad\dim_{\C}H_1(x)\le2\qquad(x\in O),
\]
\begin{equation}\label{hd:eq:hnminus2}
 n\ge5:\quad\dim_{\C}H_{n-2}(x)\le n-1\qquad(x\in O).
\end{equation}
Thus the prescribed metric is K\"ahler and locally a complex space
form of curvature $c$ there, for either sign, without balancedness,
compactness, completeness, or a constant-rank assumption.
\end{theorem}
\begin{proof}
On a regular chart with $(\dim H_0,\dim H_1)=(1,2)$,
the projected vanishing \eqref{h2:eq:physicalcut} contradicts
Theorem~\ref{hd:thm:first-cut}. Thus Lemma~\ref{unif:lem:finite-hull-growth}
with $k=2$ gives $T=0$ under either bound. Vanishing Chern torsion
gives K\"ahlerness, and \eqref{p:eq:Rsf} gives the local space form.
\end{proof}

The theorem assumes the stated bound throughout an open set.
It does not imply that the hull is automatically proper or that
the bound can be replaced by one on $H_{n-3}$ without balancedness.

\section{The first derivative hull in dimension four}
\label{fourh:sec:source}\label{fourh:sec:ideal-line}
When $\dim H_1=3$, the first hull is a hyperplane but need not be
curvature invariant. The first-cut identity \eqref{h2:eq:physicalcut}
puts the torsion image in this hyperplane by
Lemma~\ref{fourh:line-hyperplane-containment}, a step in Theorem~\ref{f4h:thm:first-hull}. The first-hull argument is for
fourfolds; the low-rank rigidity criterion
(Theorem~\ref{fourh:coincident-image}) applies from dimension three onward. Both assume $H_h\equiv c\ne0$, but neither balancedness
nor global hypotheses. The partial-connection statements have their
own rank assumptions. For the partial connection, $X$ denotes the restricted derivative,
and $\tau(u)=\tr T(u,\cdot)$ is the trace on the restricted bundle.
For the original Chern connection, $\tau=-\eta$; see
Lemma~\ref{fourh:line-trace}.

\subsection{Torsion and derivative images for partial connections}
\label{fourh:sec:prime-helpers}
For a subbundle $D$ preserved in its own prime directions with
$T(D,D)\subset D$, the torsion identity gives involutivity.
Use the induced connection and the retraction parallel along $D$
from Lemma~\ref{r37:lem:internal-restriction}. Put
\[
 T_D=T|_{\Lambda^2D},\qquad
 \mathcal X^D=\mathcal X|_{D\otimes\Lambda^2D},\qquad
 P_D=(\Sym^2\pi)P|_{\Lambda^2D},\qquad E_D(u,v)=uv+P_D(u,v).
\]
The lemma restricts complete \ChernFirst{}, \ChernCyclic{} and \ChernPure{} and gives
\begin{equation}\label{fourh:eq:pure-commutator}
 \nabla'_p\mathcal X_q-\nabla'_q\mathcal X_p
                      =-\mathcal X_{T_D(p,q)}.
\end{equation}
We omit the subscript $D$ on the restricted tensors.
The restriction supplies neither a barred identity nor a trace condition, and it does not define a constant-HSC metric.

\begin{lemma}[Alignment for a derivative-image line]
\label{fourh:transverse-column-transport}
In any dimension, let an actual prime source satisfy complete \ChernFirst{},
\ChernCyclic{}, \ChernPure{} and \eqref{fourh:eq:pure-commutator}.
On a constant-rank chart with $T\ne0$, a smooth line $L$
containing $\im\mathcal X$ must lie in $I=\im T$.
\end{lemma}
\begin{proof}
On a transverse chart \ChernCyclic{} separates $J_T=\operatorname{cyc}\mathcal X=0$,
and $\mathcal X_p=e\xi_pT$ for $L=\C e$ by Lemma~\ref{r37:lem:image-shape}.
Put $A=T(e,\cdot)$, $b=\tr A$. Differentiating Jacobi and $\tau T=0$ gives
$(\xi_pT)\wedge A=0$ and $\xi_p(A|_I-b\Id)=0$.
Some $\xi_pT\ne0$ by Lemma~\ref{rf:zero}, so $\rank A\le2$.
At rank two every $\xi_pT$ is proportional to the wedge of the two
independent rows of $A$, hence annihilates $e$; thus $\xi_p A=0$ and $b=0$.
If $\xi_p A=0$ for all $p$, both derivative traces vanish: the
input trace is $-\xi_pA$, and \ChernCyclic{} gives $\mathcal X_e=0$.
Theorem~\ref{rf:rank-gap} gives a contradiction.

On a $b=0$ field with some $\xi_p A\ne0$, $A=x\otimes\lambda$ has rank
one and $\lambda|_{I+L}=0$. The wedge row gives
$\mathcal X_p(Q,Q)=0$ for $Q=\ker\lambda$.
Pointwise restriction of \ChernCyclic{} and complete \ChernPure{} to $Q$,
retracting only the completed outputs, gives $T(Q,Q)=0$ by
Lemma~\ref{rf:zero}. If $\rank T\ge2$, a transverse section $n$ gives
$\rank T(n,Q)=\rank T$, so Lemma~\ref{r37:lem:image-shape} preserves $Q$.
At rank one the torsion form is $\lambda\wedge\eta$, so its radical
lies in $Q$ with codimension two in $V$, and $\mathcal X_p=e\xi_pT$
preserves it by differentiation. Lemma~\ref{lem:central-pure} excludes
$I\subset\operatorname{rad}T$ as a field. Else $Q=\operatorname{rad}T+I$
is preserved by image transport.
In both cases the abelian hyperplane clause of
Lemma~\ref{r17:lem:abelian-ideal-horizontal} contradicts transversality.

If $b\ne0$, write $A=bx\otimes\lambda$, $\lambda(x)=1$.
Jacobi gives $AT(v,x)=bT(v,x)$, making $\C x$ an ideal.
Then $\xi_p=\zeta_p\lambda|_I$. With $C_0=\langle e,x\rangle$,
repeated \ChernPure{} gives
\begin{equation}\label{fourh:eq:transverse-line-repeated}
 \zeta_eE(e,x)-\zeta_xe^2+2bx^2=0,\quad
 E(e,x)=ue^2+v_0x^2,
 \quad\zeta_e\ne0,\quad u=\zeta_x/\zeta_e,\quad v_0=-2b/\zeta_e.
\end{equation}
Differentiating $T(e,x)=bx$ gives $\nabla'_px\equiv\zeta_pe\bmod\C x$,
since both moving inputs pair into $\C x$. The Hessian gives
$\beta_p=q_{C_0}\nabla'_pe=\zeta_pZ$. Full \ChernFirst{} gives
\begin{equation}\label{fourh:eq:line-X-ranktwo-FIRST}
 2ue\beta_p-E(\beta_p,x)=\tfrac b2E(x,p)\pmod{\Sym^2C_0}.
\end{equation}
Lift $Z$ with $\zeta(Z)=0$ and take $p=Z,e$. Adding the two
conclusions by canonical symmetry gives $2(ue-x)Z=0$ modulo
$\Sym^2C_0$, so $Z=0$. The preserved $C_0$ contradicts
Lemma~\ref{fourh:prime-bcolumn}.
\end{proof}

\begin{lemma}[Transverse images for a partial connection]
\label{fourh:prime-transverse-line}
Under the hypotheses of Lemma~\ref{fourh:transverse-column-transport},
a smooth nonzero torsion-image line $I$ equals $L$.
\end{lemma}
\begin{proof}
Apply the alignment lemma with $\rank I=1$.
\end{proof}

\subsection{Rank-three partial connections with small derivative image}
\label{fourh:sec:prime-source}
Use the internal prime sources of \S\ref{fourh:sec:prime-helpers},
writing $M$ for their torsion and $X=\nabla'M$.
Lemma~\ref{fourh:prime-internal-rank-bound} gives the coefficient alternatives and both rank bounds. We use the full
recurrence~\eqref{u3c:actual-recurrence}, vectorial
identity~\eqref{u3c:eq:uniform-vectorial-prime}, and singular full cut
of Lemma~\ref{lem:conceptual-three-singular} with
Proposition~\ref{prop:conceptual-fullcut}. These are prime-source
statements shared by threefolds and internal fourfold sources.
Preservation holds only along the partial connection; no metric or
barred conclusion follows.

\begin{lemma}\label{fourh:prime-ranktwo-line}
There is no three-dimensional internal prime source on a nonempty
open set with
\[
 N=\im M,\quad\rank N=2,\qquad
 \im X=L\subset N,\quad\rank L=1.
\]
\end{lemma}
\begin{proof}
Image differentiation makes $N$ a prime-preserved torsion ideal.
If $M(N,N)\ne0$, restrict completely to $N$ by
Lemma~\ref{r37:lem:internal-restriction}. Binary image alignment
(Lemma~\ref{r50:lem:binary-image}) identifies the nonzero internal
derivative image with $\im(M|_{\Lambda^2N})$, so both equal $L$.
The derived-image rule of Lemma~\ref{r37:lem:image-shape} makes $L$
a prime-preserved ideal, and the codimension-two clause of
Lemma~\ref{r47:lem:abelian-carriers} gives
\begin{equation}\label{fourh:eq:ordered-line-sigma}\equationalias{fourh:eq:ordered-line-character}
 \nabla'_V L\subset L,\qquad \im M\subset L,
\end{equation}
contrary to $\rank M=2$. Thus $N$ is abelian. The hyperplane clause
of Lemma~\ref{r17:lem:abelian-ideal-horizontal} gives $\im M=\im X$,
again a contradiction.
\end{proof}

\begin{theorem}[Image alignment for a rank-three partial connection]
\label{unif:thm:three-prime-images}
Let a rank-three partial connection of zero pure curvature admit local
prime-parallel frames, with $M$ the torsion of its partial connection,
and satisfy complete \ChernFirst{}, \ChernCyclic{}, \ChernPure{} and the full torsion-corrected
covariant commutator. On constant-rank image charts with $M\ne0$,
\[
 I:=\im M=\im X=:J,\qquad 1\le\dim I\le2,\qquad
 \nabla'_V I\subset I.
\]
If $M(I,I)$ is nowhere zero, then $E(I,V)\subset\Sym^2I$.
In particular, a derivative-image line contains the whole torsion image,
and a rank-one torsion region has the same prime-preserved derivative
image. 
Original Chern sources have such frames; for internal sources of
Lemma~\ref{r37:lem:internal-restriction}, project an ambient parallel
frame by its parallel retraction and select independent columns.
\end{theorem}
\begin{proof}
On the interior $q=0$, \eqref{fourh:eq:prime-vectorial} gives equality.
On each regular open chart with $q\ne0$, Lemma~\ref{fourh:prime-internal-rank-bound}
supplies a proper prime-preserved $N\supset I,J$ after shrinking:
$N=I$ is abelian if $\Delta\ne0$, and $E(N,V)\subset\Sym^2N$ if $\Delta=0$.
If $N$ is a line, the two nonzero images it contains must equal $N$.
If $N$ is an abelian plane, it is a hyperplane in this rank-three
source, and the hyperplane clause of Lemma~\ref{r17:lem:abelian-ideal-horizontal} gives equality.
On a nonabelian plane, a line $I$ is preserved by
Lemma~\ref{r18:lem:binary-derived-preservation}, so $J\subset I$, and
if $I=N$, Lemma~\ref{fourh:prime-ranktwo-line} excludes a line $J$.
Thus $I=J$ on a dense subset. By Lemma~\ref{unif:lem:regular-loci}
equality extends at constant ranks, and image differentiation
preserves $I$. On a rank-one torsion region first extend $q_IX=0$,
then $X\ne0$ by Lemma~\ref{rf:zero}.
Where $M(I,I)\ne0$, neither the vectorial nor the $\Delta\ne0$ case
occurs; $N=I$ supplies the ordered cut, which extends by continuity.
\end{proof}

\subsection{Preserved ideal planes and full cuts}
\label{fourh:sec:ideal-plane}

\begin{theorem}[A nonabelian preserved ideal plane]
\label{r38:thm:nonabelian-ideal-plane}
Use the Chern tensors of the given metric in dimension $n\ge3$, or an actual
pure-flat prime source with $T$ the torsion of its partial connection,
satisfying complete \ChernFirst{}, \ChernCyclic{}, \ChernPure{} and the full torsion-corrected
commutator, with local prime-parallel frames. Let a smooth plane $A$ satisfy
\[
 \nabla'_VA\subset A,\qquad T(A,V)\subset A,\qquad
             T(A,A)\text{ is nowhere zero}.
\]
Then $\im T\subset A$. If $\im T=A$, then $T(A,V)=A$ on a dense
open set and $E(A,V)\subset\Sym^2A$ everywhere.
\end{theorem}
\begin{proof}
Put $F=\im T|_{\Lambda^2A}$ and
$K([z])(a)=q_FT(a,z)$ on $V/A$.
On the interior of $K=0$, we have $\im T\subset F$ by
Lemma~\ref{r18:lem:binary-derived-preservation}. At a point with
$K([v])\ne0$, extend $v$ prime-parallel. After shrinking, $D=A+\C v$
is preserved and torsion-closed, and its torsion image is $A$.
Lemma~\ref{r37:lem:internal-restriction} and
Theorem~\ref{unif:thm:three-prime-images} give
$(\Sym^2\pi)E(A,D)\subset\Sym^2A$ for every prime-parallel retraction
onto $D$, by continuity from regular charts.
Fix a projection $r$ onto $A$ with $rv=0$ at the point.
For $\xi(v)\ne0$, the retraction $\pi=r+v\otimes(\xi-\xi r)/\xi(v)$
satisfies $\xi\pi=\xi$ and $r\pi=r$. Extend it prime-parallel;
the cut above gives $\xi^2E=\xi^2(\Sym^2r)E$ on $A\times D$ at the point.
The evaluations agree for a dense set of covectors $\xi$, so the
symmetric tensors agree. Varying $v$ through the complement of
$\ker K$ in $V/A$ gives the full cut on $K\ne0$ by linearity,
so Proposition~\ref{prop:conceptual-fullcut} gives
$\im T\subset A$; continuity covers the boundary.
If $\im T=A$, the zero-$K$ interior is impossible. The dense
nonzero locus has $T(A,V)=A$, and its cut extends everywhere.
\end{proof}

\begin{lemma}[Projected vanishing for a preserved rank-two torsion image]
\label{r7:lem:preserved-ranktwo-fullcut}
For an original positive Hermitian metric of dimension $n\ge3$ and
constant Chern HSC $c\ne0$, suppose $I=\im T$ is a smooth plane with
$\nabla'_VI\subset I$. Then
\[
 E(I,V)\subset\Sym^2I,\qquad (\Sym^2q_I)P(V,V)=0.
\]
\end{lemma}
\begin{proof}
Theorem~\ref{newunb:thm:bare-output} excludes every open set with
$T(I,I)=0$. Theorem~\ref{r38:thm:nonabelian-ideal-plane} gives the first cut
on the dense open locus where $T(I,I)\ne0$, and $T(I,V)=I$ on a dense
subset there. The latter and the
barred-factor clause of Lemma~\ref{r37:lem:image-shape}, with $C=I$,
give the second cut. Both extend by continuity.
\end{proof}

\begin{theorem}[A preserved derivative ideal contains the torsion image]
\label{fourh:ideal-plane-abelian}\label{fourh:ideal-plane-horizontal}
\label{fourh:ideal-plane-terminal}\label{r38:thm:ideal-plane-containment}
For the original Hermitian fourfold with constant Chern HSC $c\ne0$,
let $A$ be a smooth plane satisfying
\[
 \im\mathcal X\subset A,\qquad
 \nabla'_VA\subset A,\qquad T(V,A)\subset A.
\]
Then $\im T\subset A$. In particular, rank-three torsion is impossible.
On a proposed failure-open set the intermediate conclusions are
$T(A,A)=0$ and $\mathcal X_a(q,b)=0$ for $a,b\in A$, $q\in V$.
\end{theorem}
\begin{proof}
Work on an open set where $\im T\not\subset A$ everywhere.
Theorem~\ref{r38:thm:nonabelian-ideal-plane} excludes every nonabelian
open subset, so $A$ is abelian. Lemma~\ref{r47:lem:abelian-carriers}
applies to $N=A$ and gives $\im T\subset A$. On the proposed
failure-open set the quotient product was nonzero, so the horizontal
tensor vanishes by the first step of that lemma's proof.
\end{proof}

\subsection{Internal images in a preserved hyperplane}
\label{fourh:sec:internal-images}

\begin{lemma}[Internal images and ambient image transport]
\label{fourh:internal-ranktwo-identification}\label{fourh:internal-rankone-source}
For the Chern tensors of the given metric, let $C$ be a prime-preserved three-plane
with $T(C,C)\subset C$, and let a smooth plane $A\subset C$ contain
every ambient output of $\mathcal X$. Put $M=T|_{\Lambda^2C}$.
If $\rank M=2$, then $\im M=A$ and $\nabla'_VA\subset A$.
If $\rank M=1$, its image $F$ satisfies
\begin{equation}\label{fourh:eq:rankone-source}
 \im X^C=F\subset A,\qquad
 \nabla'_C F\subset F,\qquad\nabla'_V F\subset A.
\end{equation}
\end{lemma}
\begin{proof}
To compare the restricted torsion image with $A$, we use the tensor
equations and parallel frames provided by Lemma~\ref{r37:lem:internal-restriction}.
Theorem~\ref{unif:thm:three-prime-images} gives $\im M\subset A$
on constant-rank internal derivative charts, hence everywhere by
continuity. In rank two this is $\im M=A$; in rank one Theorem~\ref{unif:thm:three-prime-images} gives $\im X^C=F$ and $\nabla'_CF\subset F$.
The restricted image rule of Lemma~\ref{r37:lem:image-shape} gives
$\nabla'_V(\im M)\subset\im M+A=A$, since $C$ is ambient prime
preserved and every ambient derivative output lies in $A$.
\end{proof}

To prove image alignment, we rule out the following configuration:
\begin{equation}\label{fourh:eq:rankthree-entrance}
 \begin{gathered}
 A=\im\mathcal X,\quad\rank A=2,\quad A\subset C,\quad\rank C=3,\\
 \im T\subset C,\quad\nabla'_VC\subset C,
 \qquad\im T\not\subset A\text{ at every point}.
 \end{gathered}
\end{equation}
This includes both rank-three torsion and a distinct rank-two image.
Every internal projection uses an ambient prime-parallel
retraction $\pi_C$ from Lemma~\ref{r37:lem:internal-restriction}.
Write $E_C=(\Sym^2\pi_C)E$ on all ambient lower inputs; only its
restriction to $C\times C$ gives the restricted tensors. We use no quotient metric or barred preservation.

\begin{lemma}\label{fourh:rankthree-internal-ranktwo-exclusion}
In the fourfold chart \eqref{fourh:eq:rankthree-entrance}, the
restricted torsion $M=T|_{\Lambda^2C}$ has rank at most one.
\end{lemma}
\begin{proof}
Lemma~\ref{fourh:prime-internal-rank-bound} excludes rank three.
On a rank-two minor-open, Lemma~\ref{fourh:internal-ranktwo-identification}
gives $\im M=A$ and $\nabla'_VA\subset A$.
Since $C$ contains $\im T$, it is a torsion ideal; the derived-image
clause of Lemma~\ref{r37:lem:image-shape} gives $T(V,A)\subset A$.
Theorem~\ref{fourh:ideal-plane-terminal} gives $\im T\subset A$,
contradicting the failure hypothesis.
Excluding all rank-two minor-opens gives $\rank M\le1$
throughout the region.
\end{proof}

\subsection{The internal radical and an ambient ideal line}
\label{fourh:sec:rankthree-radical}
Keep~\eqref{fourh:eq:rankthree-entrance} and $\rank M=1$, with
the conclusions of Lemma~\ref{fourh:internal-rankone-source}.
The noncentral locus of $M$ is dense: on an open set of centrality,
the preserved internal image $F$ would satisfy
Lemma~\ref{lem:central-pure}, giving $M=0$.
On this dense locus choose a smooth local frame $e,f,k$ of $C$
such that
\[
 M(e,f)=f,\qquad M(e,k)=M(f,k)=0,\qquad F=\C f.
\]
Thus $k$ spans the radical of the restricted torsion. The restricted tensors $M,X$
have output in $F$ by \eqref{fourh:eq:rankone-source}; every intermediate
input stays in $C$, so Lemma~\ref{unif:lem:line-output-dichotomy}
applies with the complete ambient symmetric output. It gives
\begin{equation}\label{fourh:eq:weak-radical-pair}
 E(F,\operatorname{rad}M)\subset F^2.
\end{equation}
This identity extends to the central locus by continuity.
On the open set where $E(F,C)\not\subset F^2$, Lemma~\ref{unif:lem:line-output-dichotomy} gives
the field identities $E(f,k)=0$ and $X_e=-2M$, $X_f=X_k=0$.

\begin{theorem}[An ideal-line core controls the ambient torsion image]
\label{unif:thm:ideal-line-core-extension}
\label{fourh:rankthree-prime-ideal-line-terminal}
Let $h$ be a Hermitian metric on a fourfold with constant
HSC $c\ne0$. Let $C$ be a smooth prime-preserved hyperplane and
$A\subset C$ smooth of rank at least two, with
\[
 \im T\subset C,\qquad \im\mathcal X=A.
\]
Suppose the complete internal tensors on $C$ have the same nonzero
image line $F$:
\[
 \im(T|_{\Lambda^2C})=
 \operatorname{span}\{\mathcal X_p(u,v):p,u,v\in C\}=F,
 \qquad \nabla'_VF\subset F,\qquad T(F,V)\subset F.
\]
Then $\im T=A$. 
\end{theorem}
\begin{proof}
Put $I=\im T$ and work on a regular chart with $I\ne A$.
The locus where the internal image is central has empty interior by
Lemma~\ref{lem:central-pure}. On its complement choose
$M(e,f)=f$, $M(e,k)=M(f,k)=0$, $F=\C f$, with parallel transverse $n$.
Put $B(v)=(\Sym^2q_F)E(f,v)$.
Where $B|_C\ne0$, Lemma~\ref{unif:lem:line-output-dichotomy}
gives $E(f,k)=0$, $B(e)\ne0$.
Write $\nabla'_nk=ae+bf+dk$ and $T(k,n)=d_ee+d_ff+d_kk$.
Full normal \ChernFirst{} is
\begin{equation}\label{fourh:eq:ideal-line-normal-radical}
 2fT(k,n)=(2a+d_e)E(f,e)+(2b+d_f)f^2.
\end{equation}
Its final double quotient gives $2a+d_e=0$; the whole row gives
$a=d_e=d_k=0$, $b=d_f/2$.
\ChernFirst{} in $C$ directions similarly gives $E(f,\nabla'_pk)=0$;
its $B(e),f^2$ coefficients give $\nabla'_Ck\subset\C k$.
Thus $K=\langle f,k\rangle$ is a preserved abelian ideal with
$T(K,V)\subset F$.
Preservation gives $\mathcal X_p(K,V)\subset F$. As $E(f,k)=0$,
both corrected vectors in \eqref{newunb:r1v3:graph} for $(k,v)$
lie in $F$. Thus $T(k,v)^2=0$ whenever
$E(f,v)\notin F^2$, hence generically by $B(e)\ne0$.
Thus $k$ is central, $b=d_f/2=0$ and $\C k$ is preserved;
differentiation and \ChernCyclic{} give $\mathcal X_p(k,V)=\mathcal X_k=0$.
Since $k$ is central, the torsion values are spanned by $f$ and
$T(e,n)$. Thus $I=\langle f,T(e,n)\rangle$ has rank at most two.
The barred factorization \eqref{r51:eq:barred-image-factor} gives $P(k,V)\subset I\odot V$.
The completed central row~\ref{unif:lem:central-PURE-row}, projected
through $I$, becomes $-2(q_I\mathcal X_p(u,v))(q_Ik)=0$.
Since $I\ne A$ and $\rank I\le2\le\rank A$, the row gives $k\in I$.
Then $I=K$ is preserved and $A\subset I$, a contradiction.
Therefore $B|_C=0$. Differentiate this actual cut along $n$.
In projected \ChernFirst{} on $(f,e)$, ideality and $\im T\subset C$
kill the first and third terms, leaving $-B(n)/2$. Thus
\begin{equation}\label{fourh:eq:whole-ideal-star}
 (\Sym^2q_F)P(F,V)=0.
\end{equation}
The cut extends across the central locus. Take an involutive coordinate
three-plane $D$ transverse to $F$. Via $D\simeq V/F$, the quotient
connection in $D$ directions is pure flat, with projected parallel
frames and torsion $q_F T|_{\Lambda^2D}$ by the bracket identity.
Ideality, preservation, and the double cut descend $T,P$ and the
complete operators. The differentiated ideal kills the lower-$F$
slots of $q_F\mathcal X$, and \ChernCyclic{} kills its family-$F$
slot. Thus the complete equations, including the torsion-direction
term, project to this partial connection. Its preserved abelian
hyperplane $C/F$ has equal torsion and derivative images by
Lemma~\ref{r17:lem:abelian-ideal-horizontal}. These images are $I/F$
and $A/F$, a contradiction.
\end{proof}

\subsection{Normal transport of the internal image line}
\label{fourh:sec:singular-normal}\label{fourh:sec:rankthree-recurrence}
Retain~\eqref{fourh:eq:rankthree-entrance}, $\rank M=1$, and
\eqref{fourh:eq:rankone-source}. Choose an ambient prime-parallel section $n$ transverse to $C$, put
$M=f\theta$, and set
$S=q_F\nabla'_nf$, $g=T(n,f)$.
In \ChernCyclic{} on $(n,c,d)$, the derivative terms lie in $A$,
and the two internal torsion products lie in $F$. Modulo $A$ only
the term proportional to $g$ survives; since $M\ne0$, this gives $g\in A$.
On $S\ne0$ the complete $\mathcal X$ Hessian first gives
$\nabla'_CA\subset A$. Indeed
$\mathcal X_n|_{\Lambda^2C}\bmod F=S\theta$; the Hessian with
families $(n,q)$ modulo $A$ makes the internal shape kill $S$,
while it already kills $F$.
If the remaining normal shape of $A$ is nonzero on an open
set, its kernel is $F$. The same Hessian on all ambient lower inputs,
followed by complete normal \ChernCyclic{}, gives
\begin{equation}\label{fourh:eq:bad-normal-A}
 \mathcal X_q(V,V)\subset F\quad(q\in C),\qquad q_Fg=S.
\end{equation}
These conclusions require nonzero normal shape.

Use $E_C=(\Sym^2\pi_C)E$ for the parallel retraction along $n$,
and $\bar E=(\Sym^2q_F)E_C$. On the dense noncentral locus, the weak radical relation
\eqref{fourh:eq:weak-radical-pair} gives
$\bar E(F,\operatorname{rad}M)=0$.
Test the ordered cut by the smooth map
$F\otimes C\to(\Sym^2C)/F^2$ induced by $E_C$. This quotient kills only $F^2$: it still detects symmetric tensors
with exactly one factor in $F$.
It suffices to treat the nonzero locus and the interior of the zero
locus. Their union is dense, since the boundary of the closed zero
locus has empty interior.
On the nonzero locus $E(F,C)\not\subset F^2$, so Lemma~\ref{unif:lem:line-output-dichotomy}, its trace,
and actual radical differentiation give, with $M(p,f)=\lambda(p)f$
for $p\in C$,
\begin{equation}\label{fourh:eq:internal-character-recurrence}\equationalias{r44:eq:local-character-recurrence}
 X_p|_{\Lambda^2C}=-2\lambda(p)M,\qquad
 \nabla'_C\lambda=-2\lambda\otimes\lambda,
\end{equation}
together with $E(F,\operatorname{rad}M)=0$ as a whole ambient
identity and preservation of the radical along $C$.
For a noncentral rank-one product this is the full recurrence
\eqref{fourh:eq:prime-recurrence}.
On the interior of $E_C(F,C)\subset F^2$, complete normal \ChernFirst{} instead gives
\begin{equation}\label{fourh:eq:ordered-line-normal-cut}
 \bar E(2S-q_Fg,u)+\theta(f,u)Z=0\quad(u\in C),
 \qquad Z=\bar E(f,n).
\end{equation}
Here representatives of $S,q_Fg$ in $C$ may be used because
$E_C(F,C)\subset F^2$. The terms from differentiating the two output factors in
$F^2$ vanish only after projection to $\Sym^2(V/F)$. The derivative of the
$f$ input still contributes; the derivative of the other input lies in
$C$ and gives zero by the vanishing condition. Canonical
symmetry handles the final \ChernFirst{} term with $T(n,u)\in C$.

\begin{lemma}[Normal connection identities under two hypotheses]\label{r40:lem:normal-pencil}
Retain \eqref{fourh:eq:rankthree-entrance} and
\eqref{fourh:eq:rankone-source}. Choose parallel $n$ and actual
fields $f,e,k$ in $C$ such that
\[
 M(f,e)=f,\quad M(f,k)=M(e,k)=0,\quad
 A=\langle f,e\rangle,\quad T(n,f)=e+g_0f,
\]
and suppose $\nabla'_CA\subset A$ and
$\nabla'_nf=\rho f+\sigma e$, with constant $\sigma=0$ or $1$.
Assume either the actual cut $E_C(F,C)\subset F^2$, or the complete
recurrence \eqref{fourh:eq:internal-character-recurrence} and whole
identity $E(f,k)=0$.
Put $R=\bar E(f,e)$, $Z=\bar E(f,n)$, $U=\bar E(e,k)$ and
$E_0=e^2$ modulo $F$. Then $R,Z,U\in\C E_0$.
Moreover $\im T\subset A$ if $\mathcal X_C(V,V)\subset F$
when $\sigma=1$, or if $A$ is ambient prime preserved when $\sigma=0$.
\end{lemma}
\begin{proof}
Under the ordered vanishing hypothesis,
\eqref{fourh:eq:ordered-line-normal-cut} gives
\begin{equation}\label{fourh:eq:ordered-line-negative-Z}
                   R=U=0,\qquad Z=(1-2\sigma)E_0.
\end{equation}
Under the recurrence hypothesis put
$\mathcal X_f(f,n)=xf+ye$, $Y=[\mathcal X_e(f,n)]_e$ and
$V_0=\bar E(e,n)$. The full derivative of $M(f,e)=f$ gives
$[\mathcal X_n(f,e)]_e=\sigma$; \ChernCyclic{} gives
$[\mathcal X_f(e,n)]_e=Y+1-\sigma$.
Four times complete \ChernPure{} on
$(f,n;f,e),(f,n;f,n),(e,n;f,e)$ gives, with congruences only
after completing the output modulo $\C R$,
\begin{gather}
 2yE_0+2(x+Y+\sigma)R=0,\label{fourh:eq:normal-first-gate}\\
 \begin{aligned}
 6E_0+(4x-2)Z+4yV_0&\equiv0,\\
 (2Y-4\sigma+1)E_0+3Z&\equiv0.
 \end{aligned}
 \label{fourh:eq:normal-two-polynomial-rows}
 \equationalias{fourh:eq:recurrent-no-division}
\end{gather}
The first quadratic total is $-2R$; the other two are
$6E_0-2Z$ and $E_0-Z$ modulo $\C R$.
Write $\nabla'_ef=\ell_ef$, $\nabla'_ee=af+\omega e$.
Recurrence with $\lambda(e)=-1$ gives $\omega=-2$; the full derivative of
$T(n,f)=e+g_0f$ gives $\omega=\ell_e-Y$. Thus complete \ChernFirst{} yields
\begin{equation}\label{fourh:eq:recurrent-weights}
 D'_eR=(Y-3)R,\quad D'_eE_0=-4E_0,\quad
 D'_eZ\equiv(Y-\tfrac32)Z-\tfrac12E_0.
\end{equation}
If $R,E_0$ are independent, the first and third \ChernPure{} rows give
$y=0$, $x=-Y-\sigma$, $Z\equiv-(2Y-4\sigma+1)E_0/3$.
The first weight identity shows that $D'_e$ preserves $\C R$. The second \ChernPure{} row gives
\begin{equation}\label{fourh:eq:normal-polynomial}
 9+(2Y+2\sigma+1)(2Y-4\sigma+1)=0.
\end{equation}
The polynomial confines $Y$, and hence $b=2Y-4\sigma+1$, to
a finite set. Smoothness makes $b$ locally constant.
Differentiating $Z\equiv-bE_0/3$ modulo the preserved $\C R$ using
\eqref{fourh:eq:recurrent-weights} gives $(2Y+5)b=-3$.
Subtracting the polynomial row gives $(4-2\sigma)b=6$.
Thus $b>0$ and $2Y+5=b+4\sigma+4>0$, a contradiction.
Hence $R,Z\in\C E_0$, including $R=0$.
The full derivative of $E(f,k)=0$ is
$(2\sigma-1)U+(2a_k-h)R=0$, where
$a_k=[\nabla'_nk]_e$, $h=[T(n,k)]_e$; therefore $U\in\C E_0$.

Under either hypothesis, the $k^2$ coefficient $B$ of $\bar E$ satisfies
$B(A,C)=B(F,n)=0$. For $\sigma=1$, the stipulated
$\mathcal X_C(V,V)\subset F$ gives
$B(\mathcal X_c(n,c),n)=0$.
For $\sigma=0$, ambient preservation of $A,F$ and parallel $n$
keep the full derivative of $Z\in\C e^2$ in that line, including
its moving-$f$ contribution $\rho Z$. Complete \ChernFirst{} gives
\begin{equation}\label{fourh:eq:normal-output-support}
                  \bar F_n(T)(f,n)=-\bar E(g,n).
\end{equation}
Since $g=e+g_0f$, also $\bar E(e,n)\in\C e^2$, so $B(A,n)=0$
and the same vanishing holds.
Put $L=T(n,\cdot)|_C$ and extend the $k$ coordinate $\ell$ by
$\ell(n)=0$. Canonical symmetry and $B(A,C)=0$ also give $B(C,A)=0$. The induced
bilinear form on the line $C/A$ is fixed by its diagonal, so
$B|_{C^2}=\ell^2$.
As $\mathcal X_n(n,c)\in A$ and $T(c,Lc)\in F$, the final
$\ell^2$ contraction of complete repeated \ChernPure{} on $(n,c)$ is
\[
             \ell(c)\ell(L^2c)+3\ell(Lc)^2=0.
\]
For $c\in A=\ker\ell$, this equation gives $\ell(Lc)=0$. Hence
$\ell L=r\ell$ for a scalar $r$. Substituting back gives $r=0$.
Thus $T(n,C)\subset A$. Together with $T(C,C)\subset F\subset A$,
this gives $\im T\subset A$.
\end{proof}

\begin{lemma}[Preservation and ideality of the internal image line]
\label{fourh:recurrent-F-preservation}\label{r44:lem:rankone-F-preservation}
\label{fourh:recurrent-F-ideality}\label{r44:lem:rankone-F-ideality}
Under \eqref{fourh:eq:rankthree-entrance}, $\rank M=1$, and
\eqref{fourh:eq:rankone-source}, $F$ is ambient prime preserved
and $T(V,F)\subset F$.
\end{lemma}
\begin{proof}
Put $S=q_F\nabla'_nf$, $g=T(n,f)\in A$, with parallel $n$.
We first exclude the following conditions:
\[
 g\notin F,\quad\nabla'_CA\subset A,\quad S=\sigma q_Fg,
 \quad\begin{cases}\nabla'_VA\subset A,&\sigma=0,\\
 \mathcal X_C(V,V)\subset F,&\sigma=1.
 \end{cases}
\]
For $\sigma=0$, $\nabla'_CF\subset F$ and $S=0$ preserve $F$ ambiently.
Work on the dense noncentral locus. On an open set with
$E_C(F,C)\subset F^2$, \eqref{fourh:eq:ordered-line-normal-cut}
is $(2\sigma-1)\bar E(g,u)+\theta(f,u)Z=0$;
at $u=g$, the canonical square forces $\theta(f,g)\ne0$.
There, or under \eqref{fourh:eq:internal-character-recurrence} with $\lambda(g)\ne0$,
rescale $f$ and choose $M(f,e)=f$, $g=e+g_0f$.
Lemma~\ref{r40:lem:normal-pencil} gives the forbidden $\im T\subset A$.

It remains to exclude the recurrence field $\lambda(g)=0$.
Choose $M(f,e)=f$ and a radical $k$ with $g=k+g_0f$;
$A=\langle f,k\rangle$ is abelian and $E(f,k)=0$ as a whole tensor.
For $\sigma=1$, put $a=[\nabla'_nk]_e$, $h=[T(n,k)]_e$,
$R=\bar E(f,e)$. Normal \ChernFirst{} and four times \ChernPure{} $(e,n;f,k)$ give
\begin{equation}\label{fourh:eq:recurrent-radical-pair}
 \bar E(S,k)=(h-2a)R,\qquad
 -\bar E(S,k)-(h-2a)R=0.
\end{equation}
The linear block of the latter is $-2\bar E(S,k)+2aR$; its unscaled
quadratic blocks, with signs $+,-,+$, are
$-hR/2$, $-hR/4-\bar E(S,k)/4$, $0$.
The full product rule and recurrence give
$\mathcal X_n(f,k)=-af$, $\mathcal X_e(f,k)=0$.
Thus $\bar E(S,k)=0$, contradicting $S=q_Fk$ and its nonzero square.

For $\sigma=0$, $F,A$ are ambient preserved. Write
$T(n,k)=d_0f+he+ik$, $v=\mathcal X_f(f,n)\in A$.
The moving inputs of the normal derivative of $E(f,k)=0$ lie
in $F,A$ and disappear only after projection; \ChernFirst{} gives
\begin{equation}\label{fourh:eq:zero-character-normal}
       \bar F_n(T)(f,k)=-(k^2+hR)/2,\qquad hR=-k^2.
\end{equation}
In \eqref{newunb:r1v3:graph} for $(f,u+tn)$, $u\in C$, recurrence gives
$s_f=tv$, $q_FT(f,u+tn)=-tk$, and $s_{u+tn}=-t^2he/2\pmod A$
from the term $-t^2T(n,g)/2$. Since $\bar E(A,f)=0$ and
$\bar E(e,f)=-R$, its coefficients of $t,t^2$ give
\begin{equation}\label{fourh:eq:zero-character-three-rows}
 2\bar E(v,u)=0\ (u\in C),\qquad
 4\bar E(v,n)+6k^2-2hR=0.
\end{equation}
Taking $u=v$ in the first row gives $(q_Fv)^2=0$, so $v\in F$.
Taking $u=e$ then gives $v=0$ because $R\ne0$. The second row now
gives $8k^2=0$, which is impossible.
Density excludes the displayed conditions for each value of $\sigma$.

If $S\ne0$ on an open set, the preceding Hessian calculation gives
$\nabla'_CA\subset A$; every nonzero normal-shape subregion of $A$
satisfies \eqref{fourh:eq:bad-normal-A}, the condition just
excluded for $\sigma=1$. Hence $A$ is ambient preserved throughout $S\ne0$.
The full derivative
$\nabla'_ng=\mathcal X_n(n,f)+T(n,\nabla'_nf)$, with $S$ spanning
$A/F$, gives $T(n,A)\subset A$. Thus $A$ is a preserved ideal,
contrary to Theorem~\ref{fourh:ideal-plane-terminal}.
Consequently $S=0$ and $F$ is ambient preserved, without making $A$ an ideal.
Writing $\nabla'_pf=\ell_pf$, a field $g\notin F$ would satisfy
$\nabla'_pg=\mathcal X_p(n,f)+\ell_pg$ and preserve
$A=\operatorname{span}(F,g)$, the excluded case $\sigma=0$.
Thus $g\in F$, and $T(C,F)\subset F$ then gives ideality.
\end{proof}

\begin{corollary}\label{fourh:ordered-line-exclusion}\label{r44:cor:rankone-normal-exclusion}
\label{fourh:rankone-earned-recurrence}
The case of internal rank one
\eqref{fourh:eq:rankthree-entrance}, \eqref{fourh:eq:rankone-source}
cannot occur on a nonempty open set, including its ordered-cut locus.
\end{corollary}
\begin{proof}
Lemma~\ref{fourh:recurrent-F-preservation} gives both ambient
preservation and ideality of $F$; Theorem~\ref{fourh:rankthree-prime-ideal-line-terminal} gives
$\im T\subset A$, contrary to the failure condition.
\end{proof}

\begin{theorem}[Image alignment in a preserved hyperplane]
\label{r38:thm:hyperplane-alignment}\label{fourh:preserved-distinct-images}
\label{fourh:rankthree-exclusion}\label{fourh:ordered-plane-A-exclusion}\label{fourh:ordered-other-plane-ideality}\label{fourh:ordered-other-plane-preservation}\label{fourh:ordered-other-plane-exclusion}\label{fourh:scalar-hyperplane-shape}\label{fourh:eq:hyperplane-shape}\label{fourh:eq:physical-shape-factor}
For the original Hermitian fourfold of constant Chern HSC $c\ne0$,
let $C$ be a smooth prime-preserved hyperplane containing $\im T$,
and suppose $A=\im\mathcal X\subset C$ is smooth of rank one or two.
Then $\rank A=2$ and $\im T\subset A$. In particular, no nonempty
open chart has
\[
 0<\rank\mathcal X\le2,\qquad\rank T=3,\qquad
                  \im\mathcal X\subset\im T.
\]
\end{theorem}
\begin{proof}
Lemma~\ref{fourh:line-X-rankthree}, with $D=C$, excludes rank one,
so $\rank A=2$. On a failure-open set,
\eqref{fourh:eq:rankthree-entrance} holds. For $M=T|_{\Lambda^2C}$,
Lemma~\ref{fourh:rankthree-internal-ranktwo-exclusion} gives $\rank M\le1$.
Lemma~\ref{fourh:internal-rankone-source} and
Corollary~\ref{fourh:ordered-line-exclusion} exclude every rank-one neighborhood of a nonzero-$M$ point;
thus $M=0$ throughout the failure locus, and the preserved abelian subbundle $C$ contradicts
Theorem~\ref{hd:thm:bare-output}.
For the final assertion take a constant-rank derivative chart and
$C=\im T$. Lemma~\ref{r37:lem:image-shape} preserves $C$;
the proved containment contradicts $\rank T=3>\rank A$.
\end{proof}

\subsection{A scalar projection along a hyperplane}
\label{fourh:sec:scalar-source}
Let $A\subset C\subset V$ be smooth bundles of ranks two and three
on a fourfold, with $\operatorname{im}\mathcal X\subset A$ and
$(\Sym^2q_C)P(A,V)=0$.  Choose $C=\ker\nu$ and
$A=\ker\mu\cap\ker\nu$. In this subsection $E_C$ denotes the
scalar projection of $E$ through $V/C$:
\begin{equation}\label{fourh:eq:scalar-E}
 E_C(u,v)=\nu(u)\nu(v)+\kappa\{\mu(u)\nu(v)-\nu(u)\mu(v)\}.
\end{equation}
The terms linear in the derivative vanish separately: each inserts
an output of $\mathcal X$ into $E_C$, which vanishes if either input
lies in $A$.
Writing $F^C$ for the scalar projection of the full $F$, the equation is
\begin{equation}\label{fourh:eq:scalar-pure}
 F_p^C(T)b_q(T)-F_q^C(T)b_p(T)+F_{T(p,q)}^C(T)=0.
\end{equation}
No zero trace is assumed; use Proposition~\ref{prop:source}.

\begin{lemma}[Torsion closure of the hyperplane and a Jacobi identity]
\label{fourh:scalar-closure}
On $\kappa\ne0$, one has $T(C,C)\subset C$ and, for any local basis
$e,f$ of $A$ with $t=T(e,f)$, $T(A,t)\subset A$.  The actual normal
character $\delta\in C^*$ defined by
$\nu T(c,v)=\delta(c)\nu(v)$ satisfies
$\delta(T(C,C))=0$.  Moreover
\begin{equation}\label{fourh:eq:plane-jacobi}
 J_T(e,f,v)=\tau(T(f,v))e-\tau(T(e,v))f,
 \qquad \tau(t)=0.
\end{equation}
On $\kappa=0$ one has $\operatorname{im}T\subset C$ pointwise.
\end{lemma}
\begin{proof}
For $\kappa=0$, use Lemma~\ref{lem:containment}.
Otherwise fix $a\in A$, put $L_a=T(a,\cdot)$,
$\ell=\nu L_a$, $\zeta=\nu L_a^2$.
The scalar repeated row is $E_C(L_a^2v,v)+3\ell(v)^2=0$.
It gives $\ell|_A=0$; write $\ell=b\mu+d\nu$.
On $C$, $\zeta=(3b^2/\kappa)\mu$.
If $b\ne0$, restriction to $A$ gives $\mu L_a|_A=0$,
so $L_a$ descends to $V/A$ with matrix
$\left(\begin{smallmatrix}h&k\\b&d\end{smallmatrix}\right)$.
The $\mu^2,\mu\nu$ coefficients give
$\kappa(h+d)=3b$ and $4b(h+d)=0$, impossible.
Thus $\nu T(a,v)=d(a)\nu(v)$; $T(A,C)\subset C$ and
$\dim C/A=1$ give $T(C,C)\subset C$.
The same row becomes
\begin{equation}\label{fourh:eq:scalar-square-action}
 \mu L_a^2=d(a)^2\mu-4d(a)^2\nu/\kappa.
\end{equation}
At $(a,v)=(e,f),(f,e)$ it gives $T(A,t)\subset A$.
Normal \ChernCyclic{} $(c_1,c_2,n)$ cancels its two normal products,
giving $\delta(T(c_1,c_2))=0$.
By \ChernCyclic{}, $J_T$ is $A$-valued. As in Lemma~\ref{fourh:line-trace},
$p\mapsto J_T(p,u,v)$ has trace $\tau(T(u,v))$ by the adjoint
commutator cancellation. For $(u,v)=(f,v),(e,v)$ this reads off
the $e$ and negative $f$ coefficients of $J_T(e,f,v)$: the other
diagonal entries vanish by alternation or $A$-valuedness.
Thus \eqref{fourh:eq:plane-jacobi} follows, with $\tau(t)=0$ from $v=e$.
\end{proof}

\Needspace{6\baselineskip}
\begin{lemma}[A nonclosed plane is impossible]
\label{fourh:scalar-plane-closed}
Use the scalar first cut \eqref{fourh:eq:scalar-E}, with
$0\ne\operatorname{im}\mathcal X\subset A$ throughout the chart.
Where $\kappa\ne0$, one has $T(A,A)\subset A$.
\end{lemma}
\begin{proof}
On a proposed nonclosed region put $r=T(e,f)\notin A$,
$C=A\oplus\C r$, $T(e,r)=ue+vf$, $T(f,r)=we+zf$.
Equation~\eqref{fourh:eq:plane-jacobi} gives $u+z=0$; the complete repeated identity and the two trace contractions give
\begin{equation}\label{fourh:eq:whole-plane-repeated}
 3r^2-2\sigma P(e,f)\in\Sym^2A,
 \qquad\sigma=(\partial\tau)(e,f)\ne0.
\end{equation}
With $\Sigma=\nabla'\tau$, the $P(e,f)$ coefficient is
$-2\Sigma_e(f)+2\Sigma_f(e)+u+z=-2\sigma$ since $\tau(r)=0$;
the final quotient through $A$ forces $\sigma\ne0$.
The internal $C$ trace of $T(a,\cdot)$ is zero for $a\in A$,
so its ambient trace is $\tau|_A=\delta|_A$.

Use a single table with $\delta(e)=p$, $\delta(f)=0$, taking
$p=1$ on the nonzero-character locus and $p=0$ on the remaining
zero-character field, and reset $r=T(e,f)$.
Equation~\eqref{fourh:eq:scalar-square-action} at $e,f,e+f$ gives
$T(e,r)=p^2f$, $T(f,r)=0$.
Choose $\mu(r)=\nu(n)=1$, $\mu(n)=0$ and write
\[
 T(e,n)=ue+vf+wr+pn,\quad T(f,n)=xe+yf+zr,
 \quad T(r,n)=ae+bf+hr.
\]
The last normal coefficient vanishes since $\delta(r)=0$.
The three $\mu$-\ChernCyclic{} rows give
$a=0$, $u+y=h+pz$, $b=p(h+pz)$.
The complete scalar $4Q$ rows on
$(e,n;e,n),(f,n;f,n),(e,n;f,n),(e,n;r,n)$, including both
compositions and the positive torsion term, are
\[
 2(\kappa pw+\kappa v+4p^2),\quad-2\kappa x,\quad
 \kappa(h+2pz-2u),\quad\kappa p(2h+pz-u).
\]
Thus $x=0$, $v=-pw-4p^2/\kappa$, $u=h/2+pz$, $y=h/2$, $ph=0$.

For $p=1$ one has $h=0$ and the full table becomes
\begin{equation}\label{fourh:eq:nonclosed-character-table}
\begin{array}{lll}
 T(e,f)=r,&T(e,r)=f,&T(f,r)=0,\\
 T(e,n)=me+(-\ell-4/\kappa)f+\ell r+n=:w_0,
 &T(f,n)=mr,&T(r,n)=mf.
\end{array}
\end{equation}
Its abelian image $I=\langle f,r,w_0\rangle$ meets $A$ in $\C f$;
$T(e,\cdot)|_I$ exchanges $f,r$ and is identity on the last quotient,
so is invertible. Full differentiation of $T(I,I)=0$ puts
$\mathcal X_V(I,I)\subset I\cap A$.
Lemma~\ref{r37:lem:image-shape}, with $D=I$, $J=\C f$, preserves
$I$ because the action modulo $J$ has rank two.
Theorem~\ref{hd:thm:bare-output} applied to $I$ gives $T=0$,
contrary to $T(e,f)=r$.

Therefore $\delta|_A=0$ as a field; also $\delta(r)=0$.
The $p=0$ table is
\[
 T(e,f)=r,\quad T(e,r)=T(f,r)=0,\quad
 T(e,n)=\tfrac h2e+wr,\quad T(f,n)=\tfrac h2f+zr,
 \quad T(r,n)=hr.
\]
On $h\ne0$ the actual image is $C$, hence image differentiation
preserves $C$. The trace vanishes on $C$, so preservation of $C$ makes its
covariant derivative vanish on pairs in $C$. This gives $\sigma=0$,
a contradiction.
Thus $h=0$ as a field, and $T=r\theta$ with
$\theta(e,f)=1$, $\theta(r,V)=0$.
Its image is central and $A\cap\operatorname{rad}T=0$;
Lemma~\ref{r37:lem:image-shape} gives $\mathcal X=0$, the final contradiction.
\end{proof}

\begin{lemma}[The normal character and whole torsion containment]
\label{fourh:scalar-normal-character}
Assume $\operatorname{im}\mathcal X$ has constant nonzero rank one
or two on the chart.  On $\kappa\ne0$, the normal character vanishes
on $A$, and
\begin{equation}\label{fourh:eq:scalar-containment}
 T(A,C)\subset A,\qquad \nu T(A,V)=0,
 \qquad \operatorname{im}T\subset C.
\end{equation}
\end{lemma}
\begin{proof}
Lemma~\ref{fourh:scalar-plane-closed} gives $T(A,A)\subset A$.
Choose $r\in C$ and $n$ with $\mu(r)=\nu(n)=1$, $\mu(n)=0$.
The quotient action $M_a=\left(\begin{smallmatrix}h&k\\0&\delta\end{smallmatrix}\right)$
satisfies $\kappa(h^2-\delta^2)=0$ and
$4\delta^2+\kappa k(h+\delta)=0$.
Thus $h=\delta$; on $\delta|_A\ne0$,
$M_a=\delta(a)\left(\begin{smallmatrix}1&-a_0\\0&1\end{smallmatrix}\right)$,
$a_0=2/\kappa$. Choose $e,f\in A$ with $\delta(e)=1$, $\delta(f)=0$
and $R,N$ with $\mu(R)=\nu(N)=1$, $\nu(R)=\mu(N)=0$, $\delta(R)=0$.
Shift $N$ along $e$ so $T(e,N)-N+a_0R\in\C f$.
By $\delta T(C,C)=0$ and the quotient action,
$T(e,f)=bf$, $T(e,R)=R+vf$, $T(f,R)=\alpha f$.
Write $T(f,N)=de+\gamma f$, $T(R,N)=ke+\varepsilon f+jR$.
Equation~\eqref{fourh:eq:plane-jacobi} on $(e,f,N)$ compares $(b+1)d$ with $(b+2)d$
and $\gamma-a_0\alpha$ with $\gamma+j-a_0\alpha$, giving $d=j=0$.
The scalar repeated row on $(R,N)$ is then $-2\kappa k=0$.
Thus $F=\C f$ is an ideal and
\begin{equation}\label{fourh:eq:normal-character-table}
\begin{array}{lll}
 T(e,f)=bf,&T(e,R)=R+vf,&T(e,N)=N-a_0R+xf,\\
 T(f,R)=\alpha f,&T(f,N)=\gamma f,&T(R,N)=\varepsilon f.
\end{array}
\end{equation}
All torsion outputs lie in $I=\langle f,R,N\rangle$.
For $\nabla'_pf=s_pe+t_pR+z_pN\bmod F$, the full derivatives
of the three products containing $f$ give
\begin{equation}\label{fourh:eq:normal-character-motion}
 (b-1)t+a_0z=0,\quad(b-1)z=0,\quad
 s=\alpha t,\quad\alpha z=0,\quad\gamma t+a_0s=0,\quad s=\gamma z.
\end{equation}
The motion equations give $s_p=z_p=0$ and $\alpha t_p=\gamma t_p=0$.
Write $T(i,j)=f\psi(i,j)$ on $I$ and put
$L(i)=q_FT(e,i)$, $\beta_p(i)=[\nabla'_pi]_e$. The full product rule gives
\[
 q_F\mathcal X_p(i,j)=t_p\psi(i,j)\bar R-\beta_p(i)L(j)+\beta_p(j)L(i)
 \quad(i,j\in I).
\]
The right side lies in $I/F$, so $A\cap I=F$ gives
$\mathcal X_p(I,I)\subset F$. At any point with $t_p\ne0$,
$\alpha=\gamma=0$, so $F$ is central in $I$ and both corrected vectors
on $I$ lie in $F$. The \hyperref[pure:rankone-contraction]{rank-one contraction}
gives $\psi=0$ since $\psi(f,I)=0$.
Thus $t_p\psi=0$ in every direction. The rank-two map $L$ now forces
$\beta_p=0$ in the product rule. Therefore $I$ is prime preserved;
differentiating its output inclusion gives $\im\mathcal X\subset I\cap A=F$.
Rank two is impossible, and in rank one the derivative image is
the ideal line $F$, excluded by Theorem~\ref{fourh:ideal-line}.
Thus $\delta|_A=0$.
Now $M_a=\left(\begin{smallmatrix}0&k(a)\\0&0\end{smallmatrix}\right)$,
giving $T(A,C)\subset A$ and $\nu T(A,V)=0$.
For $T(r,n)=t_A+jr+mn$, the complete scalar repeated row
$E_C(T(r,T(r,n)),n)-E_C(T(n,T(r,n)),r)+3m^2=0$
reduces, by these containments, to
$(m^2+\kappa mj)-\kappa jm+3m^2=4m^2=0$.
Hence $\im T\subset C$.
\end{proof}

\subsection{The nonzero scalar cut with torsion rank at most two}
\label{newff:sec:nonzero-scalar-lowrank}
Use the operator \eqref{newff:eq:E-form-F} for the given fourfold metric.

\begin{lemma}[The nonzero scalar cut at low torsion rank]
\label{newff:lem:nonzero-scalar-lowrank}
Let $A=\operatorname{im}\mathcal X$ and $C=H_1$ be actual smooth
bundles of ranks two and three on an open set, with
$R(V,\bar V)A\subset C$.  Write $C=\ker\nu$ and
$A=\ker\mu\cap\ker\nu$.  Suppose
\[
 P^\vee(\nu,\nu)=\kappa\mu\wedge\nu,\qquad \kappa\ne0,
 \qquad \operatorname{rank}T\le2
\]
as identities on the entire open set.  Then this open set is empty.
\end{lemma}
\begin{proof}
Near each rank-two point the rank is constant. Put $I=\im T$ there.
Lemma~\ref{fourh:scalar-normal-character} gives
\begin{equation}\label{newff:eq:initial-containments}
 \im T\subset C,\qquad T(A,C)\subset A.
\end{equation}
The barred factorization \eqref{r51:eq:barred-image-factor}, contracted with $\nu^2$,
gives a nonzero actual $\beta:I\to\C$ with
\begin{equation}\label{newff:eq:beta-factor}
 \kappa\mu\wedge\nu=\beta T.
\end{equation}
Thus $F=\ker\beta$ is a line and $T(A,V)\subset F$.
On $I\ne A$, set $L=A\cap I$, so $C=A+I$.
Lemma~\ref{r37:lem:image-shape} gives $\nabla'_VI\subset C$.
Choose $\nu(n)=1$ and $\ker(\xi|_C)=I$; the normal shape is
$\nu(\nabla'_pc)=\omega_p\xi(c)$.
If $F\not\subset A$, then $T(A,C)\subset A\cap F=0$,
hence $T(C,C)=0$. For $D(c)=T(c,n)$, one has $\im D=I$,
$D(A)=F$ by rank two, and full differentiation gives
\begin{equation}\label{newff:eq:transverse-shape}
 \mathcal X_p(c,d)=\omega_pW(c,d),\qquad
 W(c,d)=\xi(c)Dd-\xi(d)Dc.
\end{equation}
On $\omega\ne0$, $W$ is $L$-valued; taking $d\in I$ gives
$D(I)\subset L$, then $D(L)\subset L\cap F=0$.
Thus $D^3=0$, and the actual trace $\tau=-(\tr D)\nu=0$
supplies all traces~\eqref{h2:eq:traces} by Proposition~\ref{prop:source}.
Since $T(A,A)=0$, Theorem~\ref{h2:thm:abelian-plane} contradicts
rank two. Hence the normal shape vanishes on this locus.

On the interior $F=L$, choose $A=\langle e,f\rangle$, $F=\C f$,
$\xi(e)=1$, $r\in I\setminus F$. The complete table is
\begin{equation}\label{newff:eq:intersection-table}
\begin{array}{lll}
 T(e,f)=af,&T(e,r)=bf,&T(f,r)=df,\\
 T(e,n)=c_0f,&T(f,n)=gf,&T(r,n)=jr+hf,
\end{array}\qquad j\ne0.
\end{equation}
Write $\nabla'_pf=s_pe+u_pf+v_pr$ and
$\nabla'_pe=x_pe+y_pf+\rho_pr+\omega_pn$.
The full $r$ coefficients in the derivatives of
$T(e,f),T(f,r),T(f,n),T(e,r)$ give
$av=0$, $dv=0$, $(g-j)v=0$, $bv+j\omega=0$.
On $\omega\ne0$ this forces
\begin{equation}\label{newff:eq:forced-shape-table}
 b\ne0,\quad v\ne0,\quad a=d=0,\quad g=j\ne0,
 \qquad\omega=-(b/j)v.
\end{equation}
The table is Lie: its only potentially nonzero Jacobi row
$(e,r,n)$ cancels $jb f$ with itself.
Put $z_p=[\nabla'_pr]_e$. The $r$ coefficient of
$\mathcal X_p(e,n)$ gives $j\rho=c_0v$, hence
$q_A\nabla'_p|_A=\omega_pS$ with invertible
$S(e)=q_A(n-c_0r/b)$, $S(f)=-(j/b)q_Ar$.
The full product rule also gives
\[
 \mathcal X_p(e,f)=2j\omega_pf,\qquad
 e^*\mathcal X_p
 =s_p(b e^*\wedge r^*+c_0e^*\wedge n^*+j f^*\wedge n^*)
 +(h s_p+j z_p)r^*\wedge n^*.
\]
By \eqref{r48:eq:output-shape}, $\mathcal X_p=\omega_pZ$,
$\im Z=A$. Thus $\eta=e^*Z\ne0$ has this form with coefficients
$k,\chi$. \ChernCyclic{} gives $\omega\wedge Z=0$, hence $\eta^2=0$;
its volume coefficient $-2bjk^2$ gives $k=0$.
Therefore $\chi\ne0$, $\omega(e)=\omega(f)=0$; the $(e,f,p)$
row of $\omega\wedge Z=0$ becomes $2j\omega_pf=0$, impossible.
Consequently
\begin{equation}\label{newff:eq:intersection-C-prime}
 \nabla'_VC\subset C.
\end{equation}
The two cases give the same vanishing of the normal shape. This
vanishing, but not the intermediate nonzero coefficients, extends
across the boundaries by Lemma~\ref{unif:lem:regular-loci}.
Theorem~\ref{fourh:preserved-distinct-images} gives $I=A$ at every
rank-two point. Theorem~\ref{r16:thm:ranktwo-fullimage-terminal}
then gives $T=0$ on the original rank-at-most-two region,
contrary to $\rank A=2$.
\end{proof}

\begin{lemma}[A common obstruction for line derivative images]
\label{fourh:line-X-ranktwo}\label{fourh:line-X-rankthree}
For the original fourfold metric, no nonempty open set with
$L=\im\mathcal X$ a smooth nonzero line satisfies
$\rank T\le3$.
In particular, neither $\im T\subset D$ for a smooth hyperplane
$D\supset L$, nor $T(B,V)\subset B$ for a smooth plane $B\supset L$,
can occur. 
\end{lemma}
\begin{proof}
On the open rank-three locus, Lemma~\ref{fourh:transverse-column-transport}
gives $L\subset D=\im T$; image differentiation then preserves $D$.
A parallel retraction gives the complete internal system as
in \S\ref{fourh:sec:prime-helpers}. Where $M=T|_{\Lambda^2D}\ne0$,
Lemma~\ref{rf:zero} and Theorem~\ref{unif:thm:three-prime-images}
give $\im M=L$. Lemma~\ref{r37:lem:image-shape} then makes $L$ an ambient
preserved ideal, contradicting Theorem~\ref{fourh:ideal-line}.
Thus $M=0$, and Theorem~\ref{hd:thm:bare-output} excludes the
prime-preserved abelian hyperplane $D=\im T$.
On the now open rank-two locus, the same alignment gives
\begin{equation}\label{fourh:eq:line-X-ranktwo-trace}
\equationalias{fourh:eq:line-X-ranktwo-table}
 L\subset I=\im T,
\end{equation}
Theorem~\ref{r16:thm:ranktwo-fullimage-terminal} now gives $T=0$
on the whole open set, contrary to $L\ne0$.
If $\im T\subset D$, then $\rank T\le3$, so the preceding
obstruction applies. For an ideal plane $B$,
$\rank T\le\dim T(B,V)+1\le3$ since $\dim\Lambda^2(V/B)=1$.
\end{proof}

\subsection{A line derivative image and projected curvature}
\label{fourh:sec:line-contact}
Let $L=\im\mathcal X\ne0$ be a line, $D\supset L$ a smooth hyperplane,
$(\Sym^2q_D)P(L,V)=0$. Put $D=\ker\nu$ and $K=P^\vee(\nu,\nu)$, and set
$E_D(u,v)=\nu(u)\nu(v)+K(u,v)$. Where $\nu\wedge K\ne0$, every
scalar \ChernPure{} derivative term vanishes because its output lies in $L$ and $E_D(L,V)=E_D(V,L)=0$.
\begin{lemma}[A rank bound from the projected bilinear form]\label{r39:lem:scalar-contact-carrier}
On the preceding original-metric first-cut chart, suppose
$\nu\wedge K\ne0$. Put $L=\C e$, $W=V/L$ and $B=\nu^2+K$.
Then $\nu(L)=K(L,V)=0$, \ChernCyclic{} makes $J_T$ $L$-valued,
and scalar \ChernPure{} gives
\[
 \begin{gathered}
 Q^B_{pq}:=F_p^B(T)b_q(T)-F_q^B(T)b_p(T)+F_{T(p,q)}^B(T)=0,\\
 2F_p^B(T)(u,v)=B(T(u,p),v)+B(T(u,v),p)+B(T(p,v),u).
 \end{gathered}
\]
Then $\rank(q_LT)\le2$ at every point of this chart;
no preservation of $L$ is assumed.
\end{lemma}
\begin{proof}
To rule out surjectivity of the projected torsion, we first use
the scalar repeated identity to constrain the action of the
derivative-output line.
Choose a smooth complement $W$ and put
$T(e,u)=\alpha(u)e+Au$, $C(u,v)=q_LT(u,v)$.
The scalar repeated row and projected Jacobi are
\begin{align}
 B(A^2u,u)+3\nu(Au)^2&=0,\label{r39:eq:contact-operator}\\
 AC(u,v)-C(Au,v)-C(u,Av)&=\alpha(u)Av-\alpha(v)Au.
 \label{r39:eq:contact-jacobi}
\end{align}
The remaining algebra is at a fixed point, with no differentiation
of the chosen bases.
On $W$, $B$ is nondegenerate and $K$ restricts nondegenerately to the plane $D_W=\ker(\nu|_W)$.
For $Av=\lambda v$, $\lambda\ne0$, the diagonal gives $v\in D_W$,
and polarizing~\eqref{r39:eq:contact-operator} with $v$ as one argument
gives $K(v,\im(A^2-\lambda^2I))=0$.
If $A$ were invertible, its eigenvectors in $D_W$ would not span $W$.
Take a Jordan chain $v,x$ with $(A^2-\lambda^2I)x=2\lambda v$.
An independent eigenvector $w$ of eigenvalue $\mu$ would have
$K(v,w)\ne0$, hence $\mu^2=\lambda^2$; the row for $w$ kills $K(w,v)$.
Otherwise the chain has a third vector $y$ with
$(A^2-\lambda^2I)y=2\lambda x+v$. Thus $K(v,x)=0$, and the
diagonal at $x$ forces $\nu(x)=0$, impossible on $D_W$.
Thus $A$ is singular.

Assume $q_LT$ is onto on an open set, and shrink to a subchart
where $A$ has constant maximal rank.
A nonzero $\zeta\in\Ann(\im A)$ gives
$\omega=\zeta C\ne0$ and $\omega(Au,v)+\omega(u,Av)=0$ by Jacobi.
Triangularizing $A$ shows that two eigenvalues sum to zero.
If $\lambda=\tr A\ne0$, singularity gives spectrum $(0,0,\lambda)$.
For a $\lambda$-eigenvector $v$, the polarized operator row makes
$K(v,\cdot)$ vanish on the generalized zero eigenspace, where
$A^2-\lambda^2I$ is invertible. It also vanishes on $\C v$,
contrary to $0\ne v\in D_W$ and nondegeneracy of $K|_{D_W}$.
Thus $\tr A=0$. If $\rank A=2$, then $A$ is cyclic (distinct roots
$0,\lambda,-\lambda$, or one nilpotent
block), so $A^3=\sigma A$.
Choose cyclic $z$ and full lifts $y=T(e,z)$, $x=T(e,y)$.
Then $T(e,x)=a_0e+\sigma y$. The projected Jacobi rows give $a_0=0$ and
\[
 C(x,y)=b(x-\sigma z),\quad C(x,z)=qx+(\sigma p+b)y,
 \quad C(y,z)=px+qy+bz,
\]
where surjectivity forces $b\ne0$.
For $\eta=\nu(y)$, $\zeta=\nu(z)$, $\gamma=K(y,z)$,
the $y^2,xz,yz,z^2$ operator coefficients give
\[
 \nu(x)=0,\quad\sigma\eta=0,\quad
 K(x,y)=-\sigma\gamma,\quad K(x,z)=-3\eta^2.
\]
Here $\nu\wedge K\ne0$ means $3\eta^3-\sigma\gamma\zeta\ne0$.
Retain arbitrary $e$ components in lifts of the lower-pair products.
If $\sigma=0$, the repeated \ChernPure{} row on $(x,z)$ has equal
bilinear terms $-3b^2\eta^2$, leaving $3b^2\eta^2=0$,
contrary to $\nu\wedge K\ne0$. If $\sigma\ne0$, then
$\eta=0$ and $\gamma\zeta\ne0$. The $z$ coefficient of
$J_T(x,y,z)$ is $2b\sigma q=0$, and repeated \ChernPure{} on $(x,y)$ gives
$\sigma^2b(2\gamma q+3b\zeta^2)=0$, again a contradiction.

If $\rank A=1$, write $A=x\otimes\beta$. Then $\beta(x)=\tr A=0$,
so $A^2=0$ and~\eqref{r39:eq:contact-operator} gives $\nu(x)=0$.
Equation~\eqref{r39:eq:contact-jacobi} gives
$C(x,u)\equiv\beta(u)v\bmod\C x$ and $\beta(v)=\alpha(x)$.
Jacobi on $(x,r,n)$ gives $0\equiv\beta(r)C(v,n)-\beta(n)C(v,r)
\equiv\beta(v)C(r,n)\mod\langle x,v\rangle$.
Surjectivity gives $\beta(v)=\alpha(x)=0$.
For $\beta(u)\ne0$, the three terms of $4Q^B_{e,u}(x,u)$ are
$2,1,-4$ times $\beta(u)^2K(x,v)$, so $K(x,v)=0$.
Surjectivity gives $v\notin\C x$: otherwise $C(x,W)\subset\C x$,
leaving at most one independent output modulo $\C x$. Hence $v$
represents the radical line modulo $\C x$. Normalize its generator by
$\nu(n)=1$, choose $\nu(r)=0$, $K(x,r)=1$, and write
\[
 C(x,r)=qx+bn,\quad C(x,n)=ix,\quad C(r,n)=mx+or+tn,\qquad b\ne0.
\]
Here $b$ is the displayed $n$ coefficient. Jacobi gives $-b(i+o)=0$ in that coefficient, whereas repeated \ChernPure{}
on $(x,r)$ gives $b(i+o)+3b^2=0$. The $e$ components cancel or
pair to zero in both rows.
If $A=0$ throughout the subchart, $L$ is an ideal, contrary to
Theorem~\ref{fourh:ideal-line} for the original metric and its nonzero
derivative-output line. If the surjectivity locus of $q_LT$ were
nonempty, its openness would give a constant-rank subchart for $A$.
All possible ranks on such a subchart have been excluded.
\end{proof}

\begin{lemma}[A wedge identity for the projected curvature]
\label{fourh:line-contact}
In the setting of Section~\ref{fourh:sec:line-contact}, let
$L=\im\mathcal X\ne0$ be a line and $D=\ker\nu\supset L$ a hyperplane
with $(\Sym^2q_D)P(L,V)=0$.
Then $\nu\wedge K=0$ for $K=P^\vee(\nu,\nu)$.
\end{lemma}
\begin{proof}
Where $\nu\wedge K\ne0$, the scalar output $E_D=\nu^2+K$ gives
$\rank(q_LT)\le2$ by Lemma~\ref{r39:lem:scalar-contact-carrier}.
Hence $\rank T\le3$ on this open set, contrary to
Lemma~\ref{fourh:line-X-ranktwo}.
\end{proof}

\label{fourh:sec:line-hyperplane-completion}

\begin{lemma}[Containment of the torsion image in a hyperplane]
\label{fourh:line-hyperplane-containment}
On an original Hermitian fourfold of constant nonzero Chern HSC,
let $A=\im\mathcal X$ be smooth of rank one or two and let $D\supset A$
be a smooth hyperplane. If $(\Sym^2q_D)P(A,V)=0$, then $\im T\subset D$
throughout the chart.
\end{lemma}
\begin{proof}
Put $D=\ker\nu$ and $K=P^\vee(\nu,\nu)$, so $A\subset\ker K$.
At rank one, Lemma~\ref{fourh:line-contact} gives $\nu\wedge K=0$.
At rank two, both $K$ and $\nu$ descend to $V/A$. Their wedge is a
three-form on that two-dimensional quotient, hence zero.
On $K\ne0$, its radical is a smooth plane $B$ with
$A\subset B\subset D$ and $K=\kappa\mu\wedge\nu$, $\kappa\ne0$.
Lemma~\ref{fourh:scalar-normal-character}, whose premise only requires
$\im\mathcal X\subset B$ of constant nonzero rank one or two,
gives $\im T\subset D$. At every $K=0$ point the full double-$D$ cut
holds, so Lemma~\ref{lem:containment} gives the same conclusion.
\end{proof}

\begin{theorem}[The line--hyperplane first-cut obstruction]
\label{fourh:line-hyperplane}
On a Hermitian fourfold of constant nonzero Chern HSC, no nonempty
open set admits a nonzero line $L=\operatorname{im}\mathcal X$ and
a smooth hyperplane $D\supset L$ with
$(\Sym^2q_D)P(L,V)=0$.
\end{theorem}
\begin{proof}
Lemma~\ref{fourh:line-hyperplane-containment} gives
$\im T\subset D\supset L$, excluded by
Lemma~\ref{fourh:line-X-ranktwo}.
\end{proof}

\subsection{Distinct torsion and derivative images}
\label{fourh:sec:distinct-image}

Use the prescribed fourfold metric and the constant-rank conventions
of \S\ref{sup:sec:terminal}.

\begin{theorem}[Rank-two torsion on a first-hull hyperplane]
\label{fourh:distinct-image}\label{fourh:distinct-shape}\label{fourh:ranktwo-zero-cut}
On an original Hermitian fourfold of constant Chern HSC $c\ne0$,
suppose a nonempty open chart has
\[
 \begin{gathered}
 A=\im\mathcal X,\quad I=\im T,\quad \rank A=\rank I=2,\\
 C=H_1,\quad\rank C=3,\quad(\Sym^2q_C)P(V,V)=0.
 \end{gathered}
\]
Then $C$ is ambient prime preserved on the open locus $I\ne A$,
and no such chart exists.
\end{theorem}
\begin{proof}
Lemma~\ref{lem:containment} gives $I\subset C$.
On $I\ne A$, put $F=A\cap I$. Choose unit normal $n$,
$\nu=h(\cdot,n)$, $\theta(n)=0$, $I=C\cap\ker\theta$,
$e\in A$ with $\theta(e)=1$.
Differentiating $\nu T=0$ gives
$\nu(\nabla'_pc)=\omega(p)\theta(c)$ for $c\in C$.
Suppose $\omega\ne0$. Differentiating both upper slots of the
$\nu\nu$ cut in \ChernFirst{} gives $P^{\theta\nu}=0$, so
$P\bmod I\in\C\bar e^2$. Write $q_I\mathcal X=\bar e\chi\ne0$.
The normal Hessian~\eqref{newff:eq:X-pure-Hessian} gives
$\omega(p)\chi_q=\omega(q)\chi_p$, so $\chi=\omega\otimes\psi$
with $\psi\ne0$. In the $\bar e\bar n$ coefficient of repeated \ChernPure{}
\eqref{eq:repeated}, all $P$ terms vanish and the torsion-only terms
have an $I$ factor. Thus
$\psi(u,v)[\omega(u)\nu(v)-\omega(v)\nu(u)]=0$ for all $u,v$.
Since $\psi\ne0$, this polynomial identity gives $\omega\wedge\nu=0$,
hence $\omega|_C=0$.
Absorb $\omega(n)\ne0$ into $\theta$ and choose $e\in A$ again with
$\theta(e)=1$. This leaves the unit normal $n$ fixed and makes
$\nu(\nabla'_nc)=\theta(c)$. The full image rule gives
\begin{equation}\label{fourh:distinct-shape-data}
 \begin{gathered}
 \nabla'_CI\subset I,\quad\nabla'_CC\subset C,\quad\nabla'_VI\subset C,\\
 q_I\nabla'_pi=\nu(p)\bar e\lambda(i),\quad
 \chi_n=\lambda T\ne0,\quad T(C,C)\subset\ker\lambda.
 \end{gathered}
\end{equation}
Hence $\mathcal X_C(V,V),\mathcal X_n(C,C)\subset F$.

Put $F=\C f$, $D(c)=T(c,n)$, and $B_p=q_A\nabla'_p|_A$.
Three-dimensional alignment (Lemma~\ref{r37:lem:internal-restriction},
Theorem~\ref{unif:thm:three-prime-images}) gives $T|_{\Lambda^2C}=f\tau$
and preserves $F$ along $C$ wherever $\tau\ne0$. Modulo $F$,
\ChernCyclic{} $(c,d,n)$ then gives $\tau\,q_FD(f)=0$, since only
the Jacobi term $T(T(c,d),n)$ can survive. The product rule is
\begin{equation}\label{fourh:ideal-derivative}
 q_F\nabla'_n(T(c,d))=\theta(c)q_FT(n,d)-\theta(d)q_FT(n,c).
\end{equation}
At $d=f$ this gives $\tau(c,f)q_F\nabla'_nf=0$ on $\tau\ne0$.
By relative central \ChernPure{} (Lemma~\ref{lem:central-pure}),
$\tau(C,f)\ne0$ on a dense subset there, so $q_F\nabla'_nf=0$
throughout that locus. Thus \eqref{fourh:ideal-derivative} has zero
left side everywhere, and $D(I)\subset F$.

Rank two gives $r=De\notin F$. By \eqref{fourh:distinct-shape-data}
the product rule gives $q_A\nabla'_n(Di)=\lambda(i)q_Ar$ for $i\in I$:
the $\mathcal X_n$ term lies in $A$ and $T(i,V)\subset F$.
Thus $B_nf\ne0$, ruling out $\tau\ne0$, and $D(I)=F$.
Differentiating $D(I)=F$ along $C$ preserves $F$, since
$\mathcal X_C(I,n)$, $D(\nabla'_CI)$ and $T(I,\nabla'_Cn)$ lie in $F$.
Writing $fW=\mathcal X_n|_{\Lambda^2C}$, the product rule gives
$fW(c,d)=\theta(c)D(d)-\theta(d)D(c)$, hence $W\ne0$.
Put $\mathcal X=e\chi+f\psi$. For $p\in C$, $B_pf=\chi_p=0$, so
\eqref{r48:eq:output-shape} gives $B_p(e)\chi_n=B_n(f)\psi_p$.
Write $\mathcal X_p(n,c)=K(p,c)f$ for $p,c\in C$. The bilinear form
$K$ has rank at most one, and
\ChernCyclic{} gives $K(a,b)-K(b,a)=W(a,b)$.
Complete \ChernPure{} $(a,b;n,d)$ is
\begin{equation}\label{fourh:distinct-three-E}
 W(a,b)E(f,d)+K(a,d)E(f,b)-K(b,d)E(f,a)=0.
\end{equation}
Both intermediate bivectors belong to $\Lambda^2C$, so subsequent
\ChernFirst{} maps vanish; $T(a,b)=0$ kills the torsion term.
Since $W\ne0$, write $K=\xi\otimes\zeta$ with independent factors.
Choose $(\xi(a),\zeta(a))=(1,0)$, $(\xi(b),\zeta(b))=(0,1)$.
This equation becomes $E(f,d)+\zeta(d)E(f,b)=0$. At $d=b$
it gives $E(f,b)=0$. At $d=f$ it then contradicts $E(f,f)=f^2$.
Thus $\omega=0$. Theorem~\ref{fourh:preserved-distinct-images}
excludes $I\ne A$; the residual field $I=A$ is excluded by
Theorem~\ref{r16:thm:ranktwo-fullimage-terminal}.
\end{proof}

\subsection{Low-rank rigidity for the prescribed metric}
\label{fourh:sec:coincident-transverse}
\label{fourh:sec:coincident-Cprime}\label{fourh:sec:coincident-terminal}
\begin{theorem}[A common low-rank rigidity criterion]
\label{fourh:coincident-image}\label{r16:thm:ranktwo-fullimage-terminal}
Let $h$ be the original Hermitian metric in dimension $n\ge3$ with
constant Chern HSC $c\ne0$. Let $O$ be open. If $\rank T_x\le2$ for
every $x\in O$ and $\im\mathcal X_x\subset\im T_x$ whenever
$\rank T_x=2$, then $T=0$ on $O$.
\end{theorem}
\begin{proof}
On the open rank-two locus, image differentiation
(Lemma~\ref{r37:lem:image-shape}) makes $I=\im T$ prime preserved.
Lemma~\ref{r7:lem:preserved-ranktwo-fullcut} gives the ordered and
double quotient cuts. In codimension at least two,
Lemma~\ref{r43:lem:codimension-two-cut} excludes them.
In dimension three, the ordered hyperplane cut gives
$(\nu\otimes\Id)E(u,v)=\nu(u)v$ for $\nu\in\operatorname{Ann}I$,
hence $\mathscr K=\operatorname{Ann}I$ by \eqref{ck:skew}.
Lemma~\ref{tloc:weighted-terminal} then applies with $E=I$ and
$m=2$, independently of the threefold main theorem.

Thus $\rank T\le1$, and Theorem~\ref{newunb:r1v3:rankone}(3) gives $T=0$.
\end{proof}

\subsection{The local first-hull criterion in dimension four}
\label{fourh:sec:first-hull}

\begin{lemma}[The derivative-plane hyperplane chart]
\label{fourh:plane-hyperplane}
There is no nonempty open set of the original Hermitian fourfold
of constant nonzero Chern HSC on which $\dim_{\C}H_0=2$ and
$\dim_{\C}H_1=3$ at every point.
\end{lemma}
\begin{proof}
On such a chart, write $A=H_0=\operatorname{im}\mathcal X$ and $C=H_1$.
The curvature-induced vanishing \eqref{h2:eq:physicalcut} and
Lemma~\ref{fourh:line-hyperplane-containment} give
$\im T\subset C$, hence $\rank T\le3$, throughout the chart.
A nonzero rank-three minor would give a neighborhood of constant
rank three, excluded by
Theorem~\ref{fourh:rankthree-exclusion}. Thus $\rank T\le2$ everywhere.

Take $C=\ker\nu$ and $A=\ker\mu\cap\ker\nu$.
The first cut gives $P^\vee(\nu,\nu)=\kappa\mu\wedge\nu$.
Lemma~\ref{newff:lem:nonzero-scalar-lowrank} excludes the open set
$\kappa\ne0$, so the full double-$C$ cut holds as a field.
Theorem~\ref{fourh:ranktwo-zero-cut} excludes the open rank-two
locus. Hence $\rank T\le1$ everywhere;
Theorem~\ref{hd:thm:rank-one-torsion} gives $T=0$, contrary to $\dim A=2$.
\end{proof}

\begin{theorem}[Local first-hull rigidity on fourfolds]
\label{f4h:thm:first-hull}
Let $h$ be a smooth positive Hermitian metric on a complex fourfold,
with constant real Chern holomorphic sectional curvature $c\ne0$.
On any open subset $O$ on which
\[
                       \dim_{\C}H_1\le3
\]
holds pointwise, $h$ is K\"ahler.  No balancedness, compactness,
completeness, or sign assumption on $c$ is required.
\end{theorem}
\begin{proof}
On a regular chart with $0<\dim H_0<\dim H_1\le3$ and $\dim H_0=1$,
extend $H_1$ locally to a smooth hyperplane $D$. The first cut on $H_1$
descends via $V/H_1\to V/D$, so Theorem~\ref{fourh:line-hyperplane}
excludes the chart. The remaining ranks $(2,3)$ are excluded by
Lemma~\ref{fourh:plane-hyperplane}.
Lemma~\ref{unif:lem:finite-hull-growth} gives $T=0$ throughout $O$, which is exactly
K\"ahlerness of the original metric.
\end{proof}

\section{Assembly of the local rigidity theorems}
\label{sup:sec:finite}
Both local rigidity theorems reduce to $T=0$. Once the metric is K\"ahler,
polarization gives the space-form tensor \eqref{p:eq:Rsf} from \eqref{p:eq:HSC}.
If the hull bound is known only on a set of real Hausdorff
dimension $>2n-1$, Lemma~\ref{sup:lem:hausdorff} extends the bound
to all of $X$ using real analyticity (Proposition~\ref{univ:analyticity}).

\subsection{Threefold rigidity}
\label{tloc:bothsignlocaltheorem}

\begin{proof}[Proof of Theorem~\ref{p:thm:main}]
If $T\not\equiv0$, Lemma~\ref{unif:lem:regular-loci} gives a nonempty
open set where $T\ne0$ and both $T$ and $\mathcal X$ have constant rank. Theorem~\ref{unif:thm:three-prime-images}
gives $\rank T\le2$ and $\im\mathcal X\subset\im T$ on that set.
Theorem~\ref{r16:thm:ranktwo-fullimage-terminal} excludes nonzero
torsion under these two conditions. Thus $T=0$.
\end{proof}

\subsection{Curvature hulls and analytic propagation}

\begin{lemma}[Hausdorff propagation of properness]\label{sup:lem:hausdorff}
Let $(X,h)$ be connected, smooth Hermitian of complex dimension $n$
and constant Chern HSC (possibly zero). For a fixed integer $m\ge0$, if
$H_m$ from \eqref{eq:hull} is proper on an arbitrary subset
$S\subset X$ with $\dim_{\mathcal H}S>2n-1$, then it is proper
everywhere on $X$.
\end{lemma}
\begin{proof}
For $0\le j\le m$, applying $j$ curvature operators in order to
$\mathcal X_v(u,w)$ defines a map $(V\otimes\bar V)^{\otimes j}
\otimes V\otimes\Lambda^2V\to V$. Let $W_m$ be the sum of these maps
on their direct sum.
Then $\im W_m=H_m$ by \eqref{eq:hull}. With the induced $h$-metrics,
$g=\det(W_mW_m^*)$ is real analytic by Proposition~\ref{univ:analyticity}.
Since $\ker(W_mW_m^*)=\ker W_m^*$, the zero set of $g$ is the locus
where $H_m$ is proper. If $g\not\equiv0$, real analyticity and
connectedness make $g$ nontrivial on every chart.
By \cite[preprint, Proposition~3]{Mityagin2015} and a countable chart
cover, $\dim_{\mathcal H}\{g=0\}\le2n-1$, contrary to $S\subset\{g=0\}$.
\end{proof}

\begin{proof}[Proof of Theorem~\ref{thm:main}]
Theorem~\ref{ff:thm:main} proves part~\textup{(\ref{thm:main-fourfold})}.
Each hull bound asserts that the corresponding subspace is proper.
Apply Lemma~\ref{sup:lem:hausdorff} to extend the bounds with $m=1$, $n-3$, and $n-2$ in parts
\textup{(\ref{thm:main-fourfold-hull})},
\textup{(\ref{thm:main-balanced})}, and
\textup{(\ref{thm:main-unbalanced})}, respectively.

For part~\textup{(\ref{thm:main-balanced})}, balancedness supplies the trace hypotheses of
Theorem~\ref{rf:rank-gap}, so $H_0\ne0$ implies $\dim H_0\ge2$.
If $H_0\subsetneq H_1$ and $\dim H_1\le3$, the two ranks must be
$(2,3)$, contrary to Corollary~\ref{h2:cor:growth}.
Thus the small hull cannot grow at its first step. Apply
Lemma~\ref{unif:lem:finite-hull-growth} with $k=3$ and the bound
\eqref{eq:rank} to obtain $T=0$.

For part~\textup{(\ref{thm:main-unbalanced})}, apply Theorem~\ref{hd:thm:hnminus2}
with the bound \eqref{eq:rank-unbalanced}. For part~\textup{(\ref{thm:main-fourfold-hull})},
apply Theorem~\ref{f4h:thm:first-hull} with
\eqref{eq:rank-fourfold-hull}. These two local theorems use no
balancedness hypothesis.
\end{proof}

The set $S$ need be neither dense nor of positive volume. The strict
Hausdorff dimension threshold is needed only to extend the rank bound
from $S$ to $X$; the zero-set estimate does not suffice at equality.
The constant-curvature hypothesis and, in parts~\textup{(\ref{thm:main-fourfold})}
and~\textup{(\ref{thm:main-balanced})}, balancedness hold on all of $X$.
No subbundle identity is differentiated on $S$.

\part{Zero Chern HSC and auxiliary forms}\label{part:zero}
We construct compact examples from an invariant Hermitian metric on the
seven-dimensional quadric. We do not determine the least dimension of a
counterexample. The torsion identity \eqref{extzero:eq:mixed-identity}
is the main step in Theorem~\ref{ext:thm:auxiliary}: if $H_h=0$ and the
first Chern--Ricci form vanishes, the auxiliary-form hypothesis forces
K\"ahler flatness.

\section{Compact examples with zero Chern HSC}\label{sec:high-dimensional-examples}
Invariance reduces the computation of Chern curvature on the complex
affine quadric to one point. Compact Clifford--Klein forms from
\cite[Theorems~1.1.3 and~4.2.1, $q=7$]{CEKY} yield the compact examples
in Theorem~\ref{ce:thm:main}.

\subsection{The homogeneous complex sevenfold}

Let $\mathbb O=\R1\oplus V$ be the real normed octonions with orthonormal
basis $1,e_1,\ldots,e_7$. Fix $e_i^2=-1$, anticommutation of distinct
imaginary units, and cyclic products $e_ae_b=e_c$ for the triples
\begin{equation}\label{ce:eq:fano}
 (123),\quad(145),\quad(176),\quad(246),\quad(257),\quad(347),\quad(365).
\end{equation}
For a real vector space $W$, write $W_{\C}=W\otimes_{\R}\C$.
Extend multiplication and the
Euclidean inner product $\langle\cdot,\cdot\rangle$ complex bilinearly,
calling the latter extension $B$. A bar conjugates complex coefficients,
$a^\star$ is complex-linear octonion conjugation, and $*$ is conjugate
transpose as in \S\ref{sec:conventions}. Set
\begin{equation}\label{ce:eq:quadric}
 \mathscr Q_7=\{z\in\mathbb O_{\C}:B(z,z)=1\}.
\end{equation}
The differential of $B(z,z)$ takes the value $2$ in the direction $z$
at every point of $\mathscr Q_7$, so $\mathscr Q_7$ is
smooth of complex dimension seven with tangent space $z^{\perp_B}$.

Write $L_v(a)=va$ for left multiplication by $v\in V$. Alternativity
and the adjoint identity
\begin{equation}\label{ce:eq:oct-inner}
 \langle ab,c\rangle=\langle b,a^\star c\rangle
                      =\langle a,cb^\star\rangle
\end{equation}
give the Clifford relations
\begin{equation}\label{ce:eq:clifford}
 L_v^{\mathsf T}=-L_v,\qquad
 L_uL_v+L_vL_u=-2\langle u,v\rangle I;
\end{equation}
see \cite[Sections~2.1--2.3]{CEBaez}.
Let $\mathsf G$ be the connected half-spin matrix realization of
$\operatorname{Spin}_0(7,1)$ in $\operatorname{SO}(8,\C)$ whose real
Lie algebra is
\begin{equation}\label{ce:eq:lie-algebra}
 \mathfrak g=
 \Span_{\R}\{L_{e_i}L_{e_j}:i<j\}
 \oplus i\Span_{\R}\{L_{e_j}:1\leq j\leq7\}.
\end{equation}
Its compact spin subgroup is
$K=\operatorname{Spin}(7)\subset\operatorname{SO}(8,\R)$, and the
spin covering $\rho:K\to\operatorname{SO}(V)$ satisfies
\begin{equation}\label{ce:eq:normalizer}
 kL_vk^{-1}=L_{\rho(k)v}.
\end{equation}
This is the spin action, distinct from the octonion automorphism
group $G_2$. The realization, Cartan decomposition $\mathsf G=K\exp(iL_V)$,
and compact quotients are the $q=7$ construction of
\cite[Theorem~4.2.1 and Sections~4.3--4.9]{CEKY}.

\begin{lemma}[Cartan and Gram identities]
\label{ce:lem:cartan-gram}\label{ce:app:construction}
For $g=k\exp(riL_u)$, with $k\in K$, $|u|=1$ and $r\ge0$, put
$s=k1$, $t=ku$ and $g1=x+iy$. Then
\begin{equation}\label{ce:eq:boost}
 \exp(riL_u)=\cosh r\,I+i\sinh r\,L_u,
 \qquad x=\cosh r\,s,\quad y=\sinh r\,t.
\end{equation}
Moreover, with $\mathsf M_z=(|x|^2+|y|^2)I-2iL_{yx^\star}$,
\begin{equation}\label{ce:eq:metric-gram}
 \mathsf M_{g1}=k\exp(-2riL_u)k^{-1}=(g^{-1})^*g^{-1}.
\end{equation}
\end{lemma}

\begin{lemma}\label{ce:lem:orbit}
The action of $\mathsf G$ on $\mathscr Q_7$ is transitive. The stabilizer
$K_{\mathrm{stab}}=\mathsf G_1$ is contained in $K$ and is compact.
\end{lemma}
\begin{proof}[Proof of Lemmas~\ref{ce:lem:cartan-gram} and~\ref{ce:lem:orbit}]
The identity $(iL_u)^2=I$ gives the exponential and the coordinates
of $g1$ in \eqref{ce:eq:boost}.
The vectors $(L_{e_i}L_{e_j})1=e_ie_j$ span $V$ by
\eqref{ce:eq:fano}. These are the derivatives at the identity of
the orbit map $k\mapsto k1$, so its derivative is onto $T_1S^7$.
Translating by the action of $K$ shows that $K1$ is open in $S^7$.
Compactness of $K$ also makes the orbit closed. Since $S^7$ is
connected, $K1=S^7$. For $z=x+iy\in\mathscr Q_7$,
\begin{equation}\label{ce:eq:quadric-real}
 \langle x,y\rangle=0,\qquad |x|^2-|y|^2=1.
\end{equation}
Choose $k1=x/|x|$. If $y\ne0$, set $u=k^{-1}y/|y|\in V$ and
$r=\operatorname{arsinh}|y|$; otherwise set $r=0$.
Then \eqref{ce:eq:boost} gives $k\exp(riL_u)1=z$.
If $g1=1$, its imaginary part $\sinh r\,t$ in \eqref{ce:eq:boost} vanishes. Since $|t|=1$,
we have $r=0$, so the stabilizer is a closed subgroup of $K$.

For the Gram identity, \eqref{ce:eq:normalizer} gives $t=(\rho(k)u)s$.
The associator
$[a,b,c]_{\mathbb O}=(ab)c-a(bc)$ satisfies $[a,s,s^\star]=0$
by alternativity and $s^\star=2B(s,1)1-s$. Thus $ss^\star=1$
implies $ts^\star=\rho(k)u$, and
\[
 \mathsf M_{g1}=\cosh(2r)I-i\sinh(2r)L_{\rho(k)u}
 =k\exp(-2riL_u)k^{-1}=(g^{-1})^*g^{-1},
\]
since $k^*=k^{-1}$ and $iL_u$ is Hermitian.
\end{proof}

\subsection{The positive metric and its Chern curvature}

Use $\mathsf M_z$ from Lemma~\ref{ce:lem:cartan-gram} to define
\begin{equation}\label{ce:eq:ambient-metric}
 \langle u,v\rangle_{h_z}=\bar v^{\mathsf T}\mathsf M_z u
 \quad(u,v\in\mathbb O_{\C}).
\end{equation}
The smooth polynomial $\mathsf M_z$ is positive: for $z=g1$, the Gram identity
\eqref{ce:eq:metric-gram} writes it as $(g^{-1})^*g^{-1}$. Applying the identity at
$az=(ag)1$ gives $a^*\mathsf M_{az}a=\mathsf M_z$ for every $a\in\mathsf G$.
On the tangent space at $1$ this gives
\begin{equation}\label{ce:eq:invariant-metric}
 \langle gu,gv\rangle_{h_{g1}}=\langle u,v\rangle_{h_1}
 \qquad(g\in\mathsf G,\ u,v\in V_{\C}).
\end{equation}
\begin{proposition}\label{ce:prop:curvature}
For the restriction of the metric \eqref{ce:eq:ambient-metric} to $T\mathscr Q_7$,
the Chern curvature at
$1$, in the unitary frame $e_1,\ldots,e_7$, is
\begin{equation}\label{ce:eq:associator-curvature}
 R(u,\bar v)w=-[u,\bar v,w]_{\mathbb O}.
\end{equation}
In particular $H_h\equiv0$, whereas
\begin{equation}\label{ce:eq:nonzero-component}
 R(e_1,\bar e_2)e_4=-2e_7,
 \qquad R_{1\bar2 4\bar7}=-2.
\end{equation}
Moreover $h$ is balanced and all four Chern--Ricci contractions
of $h$ vanish.
\end{proposition}
\begin{proof}
Use local coordinates
$z(w)=\sqrt{1-\sum_a(w^a)^2}\,1+\sum_aw^ae_a$, with the square root
equal to $1$ at zero. Write $uv=-\langle u,v\rangle+u\times v$ and
$C_u(v)=u\times v$, extending bilinearly. Set
$E=\partial z/\partial w$ and $v=\bar w$.
To compute the curvature at $0$, we need only the constant,
linear, and mixed quadratic terms. Up to these terms, we have
$E\equiv\binom{-w^{\mathsf T}}{I}$, $x\equiv1+(w+v)/2$ and $y\equiv(w-v)/(2i)$,
so $|x|^2+|y|^2=1+B(v,w)$ and
$-2iyx^\star=v-w+w\times v$.
The block matrix
$L_a=\left(\begin{smallmatrix}0&-a^{\mathsf T}\\a&C_a\end{smallmatrix}\right)$
therefore gives the first derivatives and mixed second derivatives
of the column metric at zero:
\[
 \mathbf g=E(v)^{\mathsf T}\mathsf M_zE(w)
 \equiv I+C_v-C_w+B(v,w)I+C_{w\times v}-vw^{\mathsf T}.
\]
In the expansion, $v=\bar w$ is a coordinate variable. In the
curvature formula below, $u,v,w$ are tangent vectors at $1$.
The Chern formula
$R(u,\bar v)=-\partial_{\bar v}(\mathbf g^{-1}\partial_u\mathbf g)$
now gives $R(u,\bar v)w=(vu)w-v(uw)$ for real $u,v,w$;
the scalar terms cancel by alternation of
$\langle a\times b,c\rangle$. Polarized alternativity and
complexification give \eqref{ce:eq:associator-curvature}.

For $\Psi(a,b,c,d)=\langle[a,b,c]_{\mathbb O},d\rangle$,
Hermitian curvature symmetry gives
$\Psi_{ijkl}=\Psi_{jilk}=-\Psi_{ijlk}$. Since the associator is
alternating in its three arguments, $\Psi$ changes sign under every
adjacent exchange and is a four-form.
The holomorphic sectional contraction repeats vector arguments, and
each Chern--Ricci contraction repeats a basis vector in the unitary
frame $e_1,\ldots,e_7$. Alternation makes all five contractions zero. The products
$(e_1e_2)e_4=e_7$ and $e_1(e_2e_4)=-e_7$ prove
\eqref{ce:eq:nonzero-component}.
The first derivatives $\mathbf g_i(0)=-C_{e_i}$ give
$T(e_i,e_j)=-2e_i\times e_j$, which has zero trace. By invariance,
these torsion and curvature conclusions hold throughout the quadric.
\end{proof}

\subsection{The compact quotient and the product construction}

\begin{proof}[Proof of Theorem~\ref{ce:thm:allhigher}]
We use the half-spin action of
\cite[Theorems~3.3.1, 3.3.3 and~4.2.1, $q=7$]{CEKY}.
The connected real linear semisimple group $\mathsf G$ admits a
uniform lattice by \cite[Theorem~C]{CEBorel}. Such a lattice is
finitely generated because its quotient of the connected group is
compact. By \cite[Proposition~2.3(2)]{CEBorel} it has a torsion-free
subgroup $\Gamma$ of finite index. Compactness of $K_{\rm stab}$
makes the lattice action on $\mathsf G/K_{\rm stab}$ proper, and
its finite stabilizers are trivial because $\Gamma$ is torsion-free.
Thus $X_7=\Gamma\backslash\mathscr Q_7=\Gamma\backslash\mathsf G/K_{\rm stab}$
is compact and connected. The action preserves the Hermitian metric,
which therefore descends to $X_7$ with $\mathscr Q_7\to X_7$ a holomorphic
local isometry.
The holomorphic form $\Omega_z(u_1,\ldots,u_7)=\det(z,u_1,\ldots,u_7)$
is nonzero by $\mathbb O_\C=\C z\oplus z^{\perp_B}$ and invariant by
$\mathsf G\subset\operatorname{SO}(8,\C)$. Transitivity makes its
norm constant, so the Chern connection form in this holomorphic
frame of the canonical bundle is zero. The form is $\Gamma$-invariant
and thus descends to a parallel holomorphic volume form.
The direct-sum Chern connection on the product with a flat $m$-torus
(a point if $m=0$) has zero sectional numerator, as on both factors. The Ricci contractions, torsion trace, and
\eqref{ce:eq:nonzero-component} are unchanged. The wedge of $\Omega$ with the torus volume form is parallel and
holomorphic in dimension $7+m$.
\end{proof}

\section{Auxiliary curvature results}
\label{extzero:sec:zero}
The proof of Theorem~\ref{ext:thm:auxiliary} pairs the auxiliary form with a nonnegative
torsion term and a mixed Chern--Ricci term. The Stokes argument for complex parallelizable manifolds
\cite[Proposition~4.1]{extzero:AGL} uses a global holomorphic coframe
and a positive definite auxiliary metric. We use local Chern-normal
coframes, which allow unitary holonomy in the flat case and retain the
mixed Ricci term when the curvature is nonzero.

Write $\omega=\omega_h$ and use the sign convention of
Section~\ref{sec:conventions}. In the integral criterion the
auxiliary real $(1,1)$-form $\kappa$ is smooth, semipositive, and
positive definite at some point.

\Needspace{8\baselineskip}
\subsection{The torsion identity}

\begin{lemma}[Detection by exterior derivation]
\label{extzero:lem:exterior}
Let $\theta^1,\ldots,\theta^n$ be a basis of a complex
$n$-dimensional vector space, $n\ge3$, and let $D$ be a degree-one
derivation of its exterior algebra. For an increasing multi-index
$I=(i_1,\ldots,i_{n-2})$, put
$\theta^I=\theta^{i_1}\wedge\cdots\wedge\theta^{i_{n-2}}$.
If $D\theta^I=0$ for every such $I$, then $D=0$.
\end{lemma}
\begin{proof}
By exterior duality, there is a unique $\zeta\in\Lambda^1$ with
$D\gamma=\zeta\wedge\gamma$ for $\gamma\in\Lambda^{n-1}$.
For $\alpha\in\Lambda^1$ and $\beta\in\Lambda^{n-2}$, Leibniz and
$D\beta=0$ give $(D\alpha-\zeta\wedge\alpha)\wedge\beta=0$.
The top-degree wedge pairing of two-forms with $(n-2)$-forms is
nondegenerate, so $D\alpha=\zeta\wedge\alpha$; on
$\Lambda^{n-1}$ this gives $D\gamma=(n-1)\zeta\wedge\gamma$.
Hence $(n-2)\zeta=0$ and $D=0$.
\end{proof}

Set
\[
 \Xi_h=\Ric_h^{(1)}+\Ric_h^{(2)}-s_h\omega,
 \qquad s_h=\tr_h r_h^{(1)}.
\]
Thus $\Xi_h$ is a real $(1,1)$ form. Use the Hodge
normalization $\alpha\wedge *_h\beta=\langle\alpha,\beta\rangle_h\,dV_h$
for real forms of equal degree.

Choose a holomorphic coframe in which the metric matrix is the identity
and its first derivatives vanish at the chosen point. The coframe need
not come from coordinates. Then
\[
 d\theta^a=\sum_{i<j}T^a{}_{ij}\theta^i\wedge\theta^j
\]
at the chosen point. Let $D$ be the corresponding degree-one exterior
derivation. For $1\le p\le n-1$, define
\begin{align}
 &F_R=-\sum_{i,j,a,b}R_{i\bar j a\bar b}
       (i\theta^i\wedge\bar\theta^j)
       \wedge(i\theta^a\wedge\bar\theta^b),
       \label{extzero:eq:curvature-form}\\
 &\mathscr S_p=p!i^{(p+1)^2}\sum_{|I|=p}
       D\theta^I\wedge\overline{D\theta^I}.
       \label{extzero:eq:torsion-square}
\end{align}
These tensorial expressions satisfy
\begin{equation}\label{extzero:eq:normal-square}
 i\partial\bar\partial\omega^p
 =\mathscr S_p+pF_R\wedge\omega^{p-1}.
\end{equation}
Expanding the exterior power gives the $(I,\bar J)$ minors of
the holomorphic-first metric array $H_h^{\mathrm{coef}}$:
\[
 \omega^p=p!i^{p^2}\sum_{|I|=|J|=p}
        \det((H_h^{\mathrm{coef}})_{I\bar J})\theta^I\wedge\bar\theta^J.
\]
Differentiating the coframe factors gives \eqref{extzero:eq:torsion-square},
with phase factor $i^{p^2+1}(-1)^p=i^{(p+1)^2}$. First metric derivatives
vanish at the point, and the mixed metric derivatives are the negatives
of the Chern curvature coefficients. The $p$ metric factors thus give
$pF_R\wedge\omega^{p-1}$.

Every $(n-1,0)$-form is decomposable, so $\mathscr S_{n-2}$ is strongly
positive. Wedging each summand with $\omega$ gives a nonnegative
multiple of $dV_h$, zero exactly when the corresponding $D\theta^I$
vanishes. Lemma~\ref{extzero:lem:exterior} gives
\begin{equation}\label{extzero:eq:torsion-norm}
                 \omega\wedge\mathscr S_{n-2}=0\quad\Longleftrightarrow\quad T=0.
\end{equation}

When $H_h=0$, tracing the polarization identity
\eqref{p:eq:HSC} yields
\begin{equation}\label{extzero:eq:ricci-sum}
 \Ric_h^{(1)}+\Ric_h^{(2)}+\Ric_h^{(3)}+\Ric_h^{(4)}=0,
 \qquad s_h+\widehat s_h=0.
\end{equation}
Test against any real $\kappa=\sum_\ell\lambda_\ell e_\ell$ in a
diagonal unitary coframe, where $e_\ell=i\theta^\ell\wedge\bar\theta^\ell$.
By zero-HSC polarization the $e_i\wedge e_j$ coefficient of $F_R$ is
$-2(R_{i\bar i j\bar j}+R_{j\bar j i\bar i})$ for $i<j$.
The sum of the parenthesized terms over all $i<j$ is $s_h$; summing
only over the pairs containing $\ell$ gives
$r_\ell=(r_h^{(1)}+r_h^{(2)})_{\ell\bar\ell}$,
since $R_{\ell\bar\ell\ell\bar\ell}=0$. Wedging with $e_\ell$ keeps
only the other pairs, each counted $(n-3)!$ times. Hence
\[
 (n-2)\kappa\wedge F_R\wedge\omega^{n-3}
 =2(n-2)!\sum_\ell\lambda_\ell(r_\ell-s_h)\,dV_h
 =2(n-2)!\langle\kappa,\Xi_h\rangle_h\,dV_h.
\]
Since the test form is arbitrary, \eqref{extzero:eq:normal-square}
at $p=n-2$ becomes
\begin{equation}\label{extzero:eq:mixed-identity}
 i\partial\bar\partial\omega^{n-2}
 =\mathscr S_{n-2}+2(n-2)!\,*_h\Xi_h.
\end{equation}

\subsection{The mixed integral criterion}

\begin{theorem}[Mixed Ricci criterion]
\label{extzero:thm:mixed-ricci}
Let $(X^n,h)$ be compact and connected, $n\geq3$, with $H_h=0$.
Suppose that a smooth semipositive real $(1,1)$-form $\kappa$ is
positive definite at some point and satisfies
\begin{equation}\label{extzero:eq:mixed-defect}
 \int_X i\partial\bar\partial\kappa\wedge\omega^{n-2}
 \leq 2(n-2)!\int_X\langle\kappa,\Xi_h\rangle_h\,dV_h.
\end{equation}
Then $h$ is K\"ahler flat. In particular, it suffices that
\begin{equation}\label{extzero:eq:mixed-hypotheses}
 \int_Xi\partial\bar\partial\kappa\wedge\omega^{n-2}\leq0,
 \qquad
 \int_X\langle\kappa,\Xi_h\rangle_h\,dV_h\geq0.
\end{equation}
Plurinegativity of $\kappa$ implies the first of these two sufficient
inequalities.
\end{theorem}
\begin{proof}
Wedge \eqref{extzero:eq:mixed-identity} with $\kappa$ and integrate.
Stokes moves $i\partial\bar\partial$ from $\omega^{n-2}$ to $\kappa$,
and the Hodge star gives the pairing of the forms $\kappa$ and $\Xi_h$. Thus
\begin{align*}
 &\int_Xi\partial\bar\partial\kappa\wedge\omega^{n-2}
 -2(n-2)!\int_X\langle\kappa,\Xi_h\rangle_h\,dV_h=\int_X\kappa\wedge\mathscr S_{n-2}\geq0.
\end{align*}
By \eqref{extzero:eq:mixed-defect}, the nonnegative continuous density
$\kappa\wedge\mathscr S_{n-2}$ has integral zero, so it vanishes everywhere.
Near a point of positive definiteness, $\kappa\ge\epsilon\omega$ for some $\epsilon>0$.
In that neighborhood, $0\le\epsilon\omega\wedge\mathscr S_{n-2}\le
\kappa\wedge\mathscr S_{n-2}=0$. Thus
\eqref{extzero:eq:torsion-norm} gives $T=0$ there. This vanishing extends by Proposition~\ref{univ:analyticity}
and connectedness. The K\"ahler symmetries and zero-HSC polarization
then give $R^h=0$.
\end{proof}

Under \eqref{extzero:eq:mixed-defect}, everywhere positive definite
$\kappa$ gives $T=0$ pointwise by \eqref{extzero:eq:torsion-norm},
without analytic continuation. Plurinegativity alone
does not control the sign of the integral of $\langle\kappa,\Xi_h\rangle_h$
in \eqref{extzero:eq:mixed-defect}.

\subsection{Vanishing of the first Chern--Ricci form}

\begin{lemma}[Integral of the first Chern scalar curvature]
\label{extzero:lem:first-ricci}
On a compact Hermitian manifold with $H_h=0$,
\begin{equation}\label{extzero:eq:scalar-integral}
 2\int_Xs_h\,dV_h=\int_X|\eta|_h^2\,dV_h.
\end{equation}
If $\Ric_h^{(1)}=0$, then $h$ is balanced and all four Chern--Ricci
contractions vanish.
\end{lemma}
\begin{proof}
For $n\ge2$, integrate \eqref{n:eq:ddbar-omega-nminus1} and use $\chi=2s_h$ from \eqref{n:eq:scalar-with-chi}; compare \cite{BroderTang} (Proposition~2.2 in the arXiv version), where the volume form is
$\omega^n=n!\,dV_h$. For $n=1$, both sides vanish.
If $\Ric_h^{(1)}=0$, then
$s_h=0$, hence $\eta=0$ and $h$ is balanced. The Bianchi identities
\eqref{n:eq:ricci13-general} give
$\Ric_h^{(3)}=\Ric_h^{(4)}=0$, and
\eqref{extzero:eq:ricci-sum} gives $\Ric_h^{(2)}=0$.
\end{proof}

The step from the integral identity to balancedness uses compactness;
the integral argument does not establish balancedness from the local
curvature hypotheses alone.

\begin{proof}[Proof of Theorem~\ref{ext:thm:auxiliary}]
Lemma~\ref{extzero:lem:first-ricci} gives $\Xi_h=0$ and balancedness.
For $n\ge3$, plurinegativity gives the first inequality in
\eqref{extzero:eq:mixed-hypotheses}, and $\Xi_h=0$ gives the second.
Thus Theorem~\ref{extzero:thm:mixed-ricci} applies. In dimension two balancedness is K\"ahlerness; in dimension
one K\"ahlerness is automatic. Zero HSC then gives flatness.
\end{proof}

\appendix
\pdfbookmark[0]{Appendices}{global-appendices}
\section{The common boundary-variation calculation}
\label{amp:calc:area}\label{amp:calc:hermitian-area}
We separate the normal, torsion, and bulk terms in the area variation,
then compute the weighted boundary fluxes used for currents.

Let $F(t,z)$ be a holomorphic family of discs with $F(t,0)=F(0,0)$, where
$t=(t_1,\ldots,t_N)$ near $0\in\C^N$ and $z\in\mathbb D$. Put
$f(z)=F(0,z)$, $w=F_z(0,z)$, $V_j=F_{t_j}(0,z)$, and
$\mathscr L_{\!\mathrm{var}}=\sum_j
 \partial_{t_j}\partial_{\bar t_j}|_{t=0}$. For a single parameter
write $V=F_t(0,z)$. For a smooth real $(1,1)$ form $k$, use the
coefficient convention of Section~\ref{sec:conventions} and the bulk
tensor $\mathcal C_k$ of \eqref{amp:eq:area-C}. Set
\[
 S_k(V,w)=-i(\partial\omega_k)(V,w,\bar V)
          =(k_{a\bar b,p}-k_{p\bar b,a})V^p w^a\overline{V^b}.
\]
For positive $k$ this is $\langle T^k(V,w),V\rangle_k$.
The coefficient formula applies to any smooth real $(1,1)$ form.
Differentiation gives
\begin{equation}\label{amp:app:eq:full-area-defect}\equationalias{cur:can:full-defect}
 \partial_t\partial_{\bar t}\langle F_z,F_z\rangle_k|_0
 -\partial_z\partial_{\bar z}\langle V,V\rangle_k
 =\mathcal C_k(V,w,\bar V,\bar w)
       +\partial_{\bar z}S_k+\partial_z\overline{S_k}.
\end{equation}
Write $a,b,c$ for the $t\bar t,z\bar t,z\bar z$ coefficients of
$F^*\omega_k$, omitting the common factor $i$. Since pullback by $F$
commutes with $\partial$ and $\bar\partial$,
$S_k=b_t-a_z$ and
$\mathcal C_k=c_{t\bar t}-b_{t\bar z}-\bar b_{z\bar t}+a_{z\bar z}$.
Substituting these expressions proves the identity. Set
$\mathcal N_k=\sum_j|V_j|_k^2$ and
$\mathcal S_k=\sum_jS_k(V_j,w)$. Summing the identity yields
\begin{equation}\label{amp:app:eq:area-split}
 \mathscr L_{\!\mathrm{var}}\langle F_z,F_z\rangle_k
 =\partial_z\partial_{\bar z}\mathcal N_k
  +\partial_{\bar z}\mathcal S_k+\partial_z\overline{\mathcal S_k}
  +\sum_j\mathcal C_k(V_j,w,\bar V_j,\bar w).
\end{equation}
For continuous $W^{1,p}$ holomorphic families and their
$W^{1,p}$ variation sections, $p>2$, the expanded divergence and
bulk terms have at most two first-derivative factors in the disc
variable, with bounded coefficients. They therefore lie in
$L^{p/2}$. In particular, $\partial_{\bar z}$ does not differentiate
the holomorphic factor $f'$ in $\mathcal S_k$.

The unweighted determinant and trace computations are from
\cite[Lemma~5.1 and Proposition~5.2]{amp:new:OpenAI}. We derive the weighted flux
identities using Green's formula and apply them to the area decomposition
\eqref{amp:app:eq:area-split}, including its torsion and nonclosed bulk terms.

\begin{proposition}[Weighted boundary flux]
\label{an:prop:weighted-boundary}
Fix $2<p<\infty$. Let $Y$ be a Hermitian $n$-fold with metric $k$ and
$f\in C^0(\overline{\mathbb D},Y)\cap W^{1,p}$ be holomorphic inside. Take a holomorphic $W^{1,p}$ frame $B$ of $f^*T^{1,0}Y$,
nonsingular on the closed disc, with $N=B^*kB=I$ on its boundary.
Take a scalar holomorphic $b\in W^{1,p}$ and a finite matrix inner
function $C$
(holomorphic past the circle and unitary there). Set $V_j=(zbBC^{-1})_j$
on a boundary annulus outside the zeros $a$ of $\det C$. Write $m_a$
for their orders.
Let $H_b$ extend the boundary trace of $|b|^2$ harmonically,
$\eta=\eta_k$ be the torsion one-form and
$\Lambda_k=i(\partial\bar\eta-\bar\partial\eta)$. Along common
radii $r_\nu\uparrow1$,
\begin{align}
 &\lim_\nu\frac1{4\pi}\int_{|z|=r_\nu}
       \partial_r\mathcal N_k\,ds
 =n\{\|b\|_{H^2}^2+\mathcal D(b)\}
       -A_{H_b\Ric_k^{(1)}}(f)
       -\sum_a m_aH_b(a),
                    \label{an:eq:normal-flux}\\
 &\lim_\nu\frac1\pi\Rea\int_{|z|=r_\nu}
       e^{i\vartheta}\mathcal S_k\,ds
 =A_{|b|^2\Lambda_k}(f)
       +\frac2\pi\Rea\int_{\mathbb D}
                  b\overline{b'}\eta(f')\,dS.
                    \label{an:eq:torsion-flux}
\end{align}
The last term in \eqref{an:eq:normal-flux} is nonpositive because $H_b\ge0$.
For any fixed smooth closed real $(1,1)$ form $q$, put
$\mathcal N_q=\sum_jq(V_j,\overline{V_j})$ on the collar. The common
radii may depend on this fixed $q$ and can be chosen so that
\begin{equation}\label{an:eq:poisson-flux}
 \lim_\nu\frac1{4\pi}\int_{|z|=r_\nu}
 (G_0\partial_r\mathcal N_q-\mathcal N_q\partial_rG_0)\,ds
 =\left\langle|b|^2(\operatorname{tr}_k q)\circ f\right\rangle.
\end{equation}
More generally let $v\ge0$ be continuous $W^{1,p}$ on the disc,
$\Delta v\in L^{p/2}$, and let $K$ be positive self-adjoint,
continuous $W^{1,p}$ on the closed collar, with boundary value $I$.
In arbitrary rank $m$ put $\mathcal N_k=r^2v\tr K$ and assume
$L:=\log\det K=L_0-\log|I_0|^2$, where $L_0\in C^0(\overline{\mathbb D})\cap W^{1,p}(\mathbb D)$ has
zero boundary trace and $\Delta L_0\in L^{p/2}(\mathbb D)$,
and $I_0$ is finite inner with zeros $a$ of orders $m_a$.
With $H_v$ the harmonic extension of the trace of $v$, the normal
flux is
\[
 m\langle v\rangle+\frac{m}{4\pi}\int_{\mathbb D}\Delta v\,dS
 +\frac1{4\pi}\int_{\mathbb D}H_v\Delta L_0\,dS-\sum_a m_aH_v(a).
\]
If $K$ acts on $f^*TY$ and the variation covariance is $r^2vK$,
the torsion and Poisson formulas use $v,v_{\bar z}$ in place of
$|b|^2,b\overline{b'}$. In this case
$\mathcal N_q=r^2v\tr(q_kK)$, where $q_k$ is the Hermitian
endomorphism representing $q$ for $k$ along $f$. The normal formula alone needs no tangent
bundle or variation.
\end{proposition}
\begin{proof}
We prove the general formula first. Harmonic extension gives
$H_v\in W^{1,2}$. The functions $K-I$, $L$, and $v-H_v$ have zero
boundary trace. Their circular $L^2$ norms at radius $r$ are
$o(\sqrt{1-r})$ by Cauchy--Schwarz along radii: the $L^2$ norms of
the radial derivatives on the annulus from $r$ to $1$ tend to zero. For the normal formula use the list $K_r,v_r,(H_v)_r,L_r$;
for the tangent formulas add $f'$ and $(\mathcal N_q)_r$.
The sum of their squared circular norms is radially integrable,
so choose common radii $r_\nu\uparrow1$ at which $(1-r_\nu)$ times
that sum is bounded. Otherwise the sum would be at least $1/(1-r)$ near $1$,
contrary to radial integrability. For each zero-trace
factor $H$ and each listed derivative $D$,
\begin{equation}\label{amp:app:eq:both-errors}
 \|H(r_\nu)\|_2\|D(r_\nu)\|_2=o(1),\qquad
 (1-r_\nu)\|D(r_\nu)\|_2=o(1).
\end{equation}
In the derivative of $r^2v\tr K$ the $r^2v$ part gives
$m(\langle v\rangle+\int\Delta v/(4\pi))$ by the radial identity
\eqref{amp:app:eq:green-area} with $\Delta=4\partial_z\partial_{\bar z}$.
The error estimates replace $(\tr K)_r$ by $L_r$ (their difference
is $\tr((I-K^{-1})K_r)$), $v$ by $H_v$, and allow subtracting
$L(H_v)_r$. Extend $L$ to the disc by $L_0-\log|I_0|^2$. Since $\Delta H_v=0$,
Green's formula pairs $H_v$ with
$\Delta L=\Delta L_0-4\pi\sum_a m_a\delta_a$ to give the normal flux. The torsion trace of the covariance is
$r^2v\{\eta(f')+O_f(|K-I||f'|_k)\}$; the same errors reduce it to
$v\eta(f')$. Cauchy--Green and $\bar\partial f'=0$ give the asserted
torsion formula. In the Poisson formula, $G_0(r)=O(1-r)$ makes the derivative
term vanish, and $\partial_rG_0=-2/r$ gives the trace of $\mathcal N_q$.
The same radii therefore give all three limits.

To obtain the first formulas in the statement, take $v=|b|^2$,
$K=(BC^{-1})(BC^{-1})^{*_{k}}$, $L_0=\log\det N$, $I_0=\det C$;
$\Delta v=4|b'|^2$ and
$\Delta L_0=-4r_k^{(1)}(f',\overline{f'})$.
This yields \eqref{an:eq:normal-flux}--\eqref{an:eq:poisson-flux}
in the conventions of \eqref{conv:eq:disc-conversion}.
\end{proof}

\begin{lemma}[Closed-disc frames]\label{amp:lem:initial-sobolev-frame}
\label{amp:app:var:initial-frame}\label{amp:app:var:frame}
Let $2<p<\infty$ and $f:\overline{\mathbb D}\to X$ be continuous,
holomorphic inside and $W^{1,p}$. Every pulled-back holomorphic
vector bundle has a holomorphic $W^{1,p}$ frame nonsingular on the
closed disc. For any smooth Hermitian bundle metric the frame may be
chosen unitary on the boundary.
\end{lemma}
\begin{proof}
For the boundary-unitary frame in the unweighted extremal-disc argument,
see \cite[Proposition~4.2]{amp:new:OpenAI}. Here we construct a
$W^{1,p}$ frame for an arbitrary pulled-back holomorphic bundle.
The continuous pullback is trivial since the closed disc is contractible.
Approximate a continuous frame in finitely many bundle charts and
patch with a smooth partition of unity on the disc subordinate to the
chart domains, with supports compactly contained in them.
Uniform approximation gives a nonsingular $W^{1,p}$ frame $e$
with $W^{1,p}$ inverse matrices and
$\bar\partial_E(eu)=e(\bar\partial u+Au)$, $A\in L^p$.
Extend $A$ by zero to $\widetilde A$ on $\C$.
The local Cauchy--Green operator \cite[Proposition~1]{amp:new:BK18}
satisfies
\[
 T_V:L^p\to W^{1,p},\qquad
 \|T_Vq\|_\infty\le C_p(\operatorname{diam}V)^{1-2/p}\|q\|_p.
\]
On sufficiently small discs, the contraction mapping theorem solves
$H_V=I-T_V(\widetilde A H_V)$ with
$\|H_V-I\|_\infty<1/2$ and $H_V,H_V^{-1}\in W^{1,p}$.
On each overlap, $H_V$ and $H_W$ solve the same extended
$\bar\partial$ equation, so $\bar\partial(H_V^{-1}H_W)=0$
in distributions. Weyl's lemma makes the transition maps holomorphic,
including across the unit circle. These transition maps define a holomorphic bundle over $\C$.
It is holomorphically trivial by \cite[Theorem~30.4, p.~229]{amp:new:Forster81}.
A global frame therefore gives $H,H^{-1}\in W^{1,p}_{\rm loc}$,
$\bar\partial H=-\widetilde A H$, and the closed-disc frame $eH$.
The boundary trace of its positive $W^{1,p}$ metric matrix $G$
factors as $G=F^*F$ on $\partial\mathbb D$ by
Lemma~\ref{amp:lem:direct-sobolev-factor}; multiplication by
$F^{-1}$ gives boundary unitarity.
\end{proof}

\begin{lemma}[Factor-preserving variations]\label{amp:lem:factor-preserving-spray}
\label{amp:app:var:spray}\label{amp:app:var:center-derivative}
Fix $2<p<\infty$. Let $f\in\mathcal O(\mathbb D,X)\cap W^{1,p}$ be continuous
on the closed disc, and let $U$ be a holomorphic $W^{1,p}$ frame of
$f^*TX$ nonsingular there. For holomorphic $W^{1,p}$ data $a,v_1,\ldots,v_N$,
with $a$ scalar and $v_j$ sections of $f^*TX$, there is a centered
jointly holomorphic family with $F_0=f$ and
$\partial_{t_j}F_0=za v_j$. In fixed target charts,
$F_t-f=aH_t$, where $H_t=O_{W^{1,p}}(|t|)$ depends holomorphically on $t$.
Each member has finite area and $t\mapsto\partial_zF_t(0)$ is holomorphic.
\end{lemma}
\begin{proof}
The centered variations in \cite[Proposition~4.3]{amp:new:OpenAI}
do not impose the scalar factor used here to preserve prescribed zeros.
Apply Forstneri\v{c}'s tubular-neighborhood theorem
\cite[Theorem~1.2]{amp:new:Forstneric2007MappingSpaces} to the
section $z\mapsto(z,f(z))$ of $\C\times X\to\C$. It gives a
fiber-preserving biholomorphism near the closed graph into a holomorphic
vector bundle. Write its fiber map as $\phi_z$ and set
$w(z)=\phi_z(f(z))$, $\xi_j(z)=(d\phi_z)_{f(z)}v_j(z)$.
Then
\[
 F_t(z)=\phi_z^{-1}\!\left(w(z)+za(z)\sum_jt_j\xi_j(z)\right).
\]
The graph is compact, so we may choose the parameter neighborhood independent of $z$.
Holomorphic composition in the $W^{1,p}$ algebra gives holomorphic
dependence and the stated first variations. In each target chart,
expanding $\phi_z^{-1}$ in its fiber variable about $w(z)$ factors
out $a$, giving $F_t-f=aH_t$ with the stated bound and holomorphic
dependence, including at zeros of $a$.
The factor $z$ fixes the center. The $W^{1,p}$ bound gives finite
area, and Cauchy's formula gives holomorphic center derivatives.
\end{proof}

\begin{lemma}[Boundary normalization]\label{amp:lem:direct-sobolev-factor}
Fix $2<p<\infty$ and let $G$ be a uniformly positive Hermitian $n\times n$ matrix
in $W^{1,p}(\mathbb D)$. Its boundary trace has a factorization
$G=F^*F$, where $F,F^{-1}$ are holomorphic and belong to
$W^{1,p}(\mathbb D)\cap C^0(\overline{\mathbb D})$.
\end{lemma}
\begin{proof}
Let $g$ be the boundary trace, with $cI\le g\le CI$.
The classical analytic factorization follows from the Wiener--Masani
theorem \cite[Theorem~7.13]{WienerMasani1957} applied to $g^{\mathsf T}$.
We use a Toeplitz construction to obtain the Sobolev bounds.
Write $P_+$ for the projection onto nonnegative Fourier modes.
On the vector Hardy space $H^2$, the operator $T_g u=P_+(gu)$
satisfies $cI\le T_g\le CI$. Solve $T_gU=I$ column by column.
The boundary matrix $gU$ has only nonpositive modes and constant
term $I$. Thus $U^*gU$ is Hermitian with only nonpositive modes,
so it is the constant $P=U(0)^*$; the products here are in $L^1$.
Since $P$ is the compression of $T_g^{-1}$ to the constants,
$C^{-1}I\le P\le c^{-1}I$.
Writing $V=(gU)^*$ for the analytic matrix with this trace,
$VU=P$ also holds inside the disc by Fourier uniqueness.
Hence $F=P^{-1/2}V$ and $F^{-1}=UP^{-1/2}$ are analytic,
$F^*F=g$ on the circle, and $F(0)=P^{-1/2}>0$.
Their boundary norm bounds $\sqrt C$ and $1/\sqrt c$ hold inside
by the Poisson formula for $H^2$.

For a real angle $h$, write $g_h(e^{it})=g(e^{i(t+h)})$ and
define $U_h,F_h$ in the same way. Subtract the Toeplitz equations.
Then $g_h(U_h-U)-(g-g_h)U$ has only strictly negative Fourier modes.
Multiplication by $(F_h^{-1})^*$ preserves these modes and turns
$g_h(U_h-U)$ into the analytic function $F_h(U_h-U)$. Applying
$P_+$ and multiplying by $F_h^{-1}$ gives
\[
 U_h-U=F_h^{-1}P_+((F_h^{-1})^*(g-g_h)U).
\]
The bounded factors and $L^p$ boundedness of $P_+$ give
$\|\Delta_h U\|_p\le C\|\Delta_hg\|_p$.
Rotations fix $U(0)$ and hence $P$. The formulas for $F,F^{-1}$
and boundedness of $g,U$ then give $\|\Delta_hj\|_p\le C\|\Delta_hg\|_p$
for either factor trace $j$.
The Sobolev trace theorem gives $g\in B^{1-1/p}_{p,p}$; the
difference estimate gives the same for both factor traces. For either analytic factor $J=F$ or $J=F^{-1}$, the Besov characterization
\cite[\S2.1, (2.5) and the paragraph after (2.7)]{amp:ref:Peller2024}
with $m=1$ gives $J'\in L^p(\mathbb D)$: the radial weight is
$(1-r)^{p(1-(1-1/p))-1}=1$. Thus $F,F^{-1}\in W^{1,p}$.
Sobolev embedding for $p>2$ gives continuity on the closed disc.
\end{proof}

\Needspace{6\baselineskip}
\section*{AI use disclosure}
ChatGPT assisted with exploratory constructions and calculations, as well as the formal verification of selected proof steps and computational results in Lean. The authors conceived and developed the paper’s core ideas and proof strategy and take full responsibility for its mathematical arguments, references, and conclusions.

\makeatletter
\@setaddresses
\let\@setaddresses\relax
\makeatother

\end{document}